\documentclass[twoside,final,titlepage,openany,a4paper,12pt]{book}

\usepackage[utf8]{inputenc}
\usepackage[english]{babel}
\usepackage{enumerate}
\usepackage{latexsym}
\usepackage{amsthm,amsmath}
\usepackage{amssymb}
\usepackage{multicol}
\usepackage{tikz}
\usepackage{float} 

\usepackage[bookmarks]{hyperref} 
\hypersetup{
	pdftitle={2-switch-degree Classification of Split Graphs},
	pdfauthor={V.N. Schvollner},
	colorlinks=true,
	linkcolor=blue,
	citecolor=red,
	filecolor=cyan,
	urlcolor=magenta
}

\usetikzlibrary{babel,decorations.pathreplacing,calc,positioning,shapes.geometric}

\definecolor{10}{RGB}{115,59,171}
\definecolor{8}{RGB}{212,122,240}
\definecolor{7}{RGB}{99,212,119}
\definecolor{6}{RGB}{183,240,164}
\definecolor{D}{RGB}{255,162,79}
\definecolor{E}{RGB}{255,84,0}
\definecolor{F}{RGB}{158,248,255}
\definecolor{G}{RGB}{128,135,255}
\definecolor{I}{RGB}{187,255,0}
\definecolor{A}{cmyk}{.9,.05,.4,0}
\definecolor{B}{RGB}{150,30,150}
\definecolor{C}{RGB}{186,155,189}
\definecolor{9}{RGB}{0,180,60}
\definecolor{0}{RGB}{30,123,191}
\definecolor{1}{RGB}{255,113,102}
\definecolor{2}{RGB}{41,199,92}
\definecolor{3}{RGB}{242,207,16}
\definecolor{5}{RGB}{255,15,154}
\definecolor{4}{rgb}{.8,0,.8}

\definecolor{Red}{rgb}{1,0.4,0.4}
\definecolor{Green}{rgb}{.1,.5,.1}
\definecolor{Blue}{rgb}{.1,.1,.5}
\definecolor{blue}{RGB}{0,0,255}
\definecolor{Yellow}{rgb}{.8,.4,0}
\definecolor{X}{rgb}{.8,.4,0}
\definecolor{H}{rgb}{0,0,1}
\definecolor{light}{rgb}{.67,.84,.90}
\definecolor{Cyan}{rgb}{0,1,1}
\definecolor{Purple}{rgb}{.5,0,.5}
\definecolor{Purple2}{rgb}{.5,.2,.5}
\definecolor{white}{rgb}{1.0,1.0,1.0}
\definecolor{Purple2}{rgb}{.8,.4,0}
\definecolor{Amarillo}{RGB}{225,191,73}
\definecolor{Celeste}{RGB}{117,170,219}
\definecolor{Castano}{RGB}{232,53,17}
\definecolor{Black}{RGB}{0,0,0}
\definecolor{White}{RGB}{255,255,255}
\definecolor{gris}{rgb}{.5,.5,.5}

\newtheorem{theorem}{Theorem}[section]

\newtheorem{conjecture}[theorem]{Conjecture}
\newtheorem{corollary}[theorem]{Corollary}

\newtheorem{lemma}[theorem]{Lemma}

\newtheorem{proposition}[theorem]{Proposition}

\newtheorem{proposicion}[theorem]{Proposition}

\theoremstyle{definition}

\pgfdeclarelayer{background2}
\pgfdeclarelayer{background}
\pgfdeclarelayer{foreground}
\pgfsetlayers{background2,background,main,foreground}

\title{2-switch-degree classification of split graphs}
\author{Victor Nicolas Schvöllner}
\date{\today}

\begin{document}
	\frontmatter
	\begin{titlepage}
		\begin{center}
			\vspace*{0in}
			
			Universidad Nacional de San Luis (UNSL)\\
			\vspace*{1.5cm}
			Facultad de Ciencias Físico-Matemáticas y Naturales (FCFMyN)\\
			\vspace*{1.5cm}
			Department of Mathematics\\
			\vspace*{1.5cm}
			\begin{large}
				\textbf{Thesis submitted for the degree of\\
					Ph.D. in Mathematics}\\
			\end{large}
			\vspace*{1.5cm}
			\rule{80mm}{0.1mm}\\
			\vspace*{1.5cm}
			\begin{Large}
				\textbf{2-switch-degree classification of split graphs} \\
			\end{Large}
			\vspace*{1.5cm}
			\begin{large}
				Thesis presented by: Victor Nicolas Schvöllner\\
			\end{large}
			\vspace*{1.5cm}
			\rule{80mm}{0.1mm}\\
			\vspace*{0.1in}
			\begin{large}
				Supervised by: 
				Dr. Adri\'{a}n Pastine
			\end{large}
		\end{center}
	\end{titlepage}
	
	\newpage
	
	{
		\hypersetup{linkcolor=black}
		\tableofcontents
		\listoffigures
	}


\chapter*{Summary and main objectives}

The degree sequence of a graph $G$ is the list of degrees of its vertices. A 2-switch is an operation on a graph $G$ that replaces two disjoint edges $ab, cd \in G$ with two disjoint edges that are not in $G$ and that are incident to the vertices $a,b,c$, and $d$. This operation preserves the degree sequence. A fundamental result related to this concept, which can be found in \cite{chartrand2010graphs} (page 24), states that, given two graphs with the same degree sequence, it is possible to transform one into the other by means of a sequence of 2-switches. The 2-switch-degree of $G$, or simply the degree of $G$, is the number of 2-switches acting on $G$. The main goal of this work is to study the properties of this new structural parameter (addressed in Chapter \ref{cap:deg(G)}), and to develop tools to classify certain families of graphs based on it (Chapters \ref{cap:grafos.asociados.a.split} and \ref{cap:grafos.activos.con.deg<5}). More specifically, we focus on split graphs, graphs whose vertex set can be partitioned into a clique and an independent set. This choice is motivated by Tyshkevich's Decomposition (see \cite{tyshkevich2000decomposition}, or Theorem \ref{tyshk.decomp}), which states that every graph can be uniquely written as a “product” $G_r\circ\ldots\circ G_2\circ G_1$ of indecomposable graphs, where $G_2,\ldots,G_r$ are all split. In a certain sense, this decomposition plays the same role for graphs that the Fundamental Theorem of Arithmetic plays for integers. Following this analogy, indecomposable split graphs are naturally compared to prime numbers. Moreover, we will prove that the 2-switch-degree distributes over Tyshkevich’s composition; that is, “the degree of the product is the sum of the degrees” (as happens with polynomials). These facts justify the importance of first studying the degree classification of split graphs before tackling the same problem for arbitrary graphs. Here, it is important to emphasize the following: graphs of degree 0 are precisely threshold graphs, a well known subfamily of split graphs.

The most important tool with which we approach this problem is undoubtedly the factor graph $\Phi(S)$ associated with a split graph $S$, which is introduced in Chapter \ref{cap:grafos.asociados.a.split}. This multigraph encodes information about the degree of $S$ through the multiplicity of its edges, using only the independent vertices of $S$. At first glance, listing all indecomposable split graphs of a given degree seems like a task too difficult to accomplish by inspection or brute force. However, the factor graph breaks down this large problem into several simpler ones. Indeed, since the size of $\Phi(S)$ matches the degree of $S$, we will see that it suffices to list all (unlabeled) connected multigraphs of fixed size: to each one of these will correspond a list (possibly empty) of prime split graphs, that is, indecomposable split graphs in which all vertices are involved in some 2-switch. This method will be put into practice in Chapter \ref{cap:grafos.activos.con.deg<5}, where we classify all prime split graphs of degree 1, 2, 3, and 4. The study of the factor graph leads to two other important topics, which may be regarded as sub-objectives of this work: the twin quotient and the $\Delta$ property.

The twin quotient is the main topic of Chapter \ref{cap:cociente.gemelos}. Two vertices are said to be twins in a graph $G$ if their neighborhoods in $G$ are essentially identical. In other words, for a vertex to have twins in $G$ means that it has one or more “copies” of itself in $G$. Since “being twins” is an equivalence relation, we can quotient $G$ by this relation to obtain a quotient graph $[G]$ where such copies have been removed. The study of the twin quotient is important in this work for two main reasons. The first one appears in Section \ref{sec:caract.Phi.simples}: if $u$ and $v$ are independent vertices in $S$ that are twins in $S$, then $u$ and $v$ are non-neighbors and twins in $\Phi(S)$. Moreover, when $S$ is homogeneous (that is, when all its independent vertices have the same degree in $S$), we have that $u$ and $v$ are non-neighbors in $\Phi(S)$ if and only if they are twins in $S$. The second reason the twin quotient matters, lies in its “compatibility” with Tyshkevich's Decomposition. Indeed, in Section \ref{sec:cociente.composicion} we prove that the quotient distributes over $\circ$: $[S\circ G]=[S]\circ [G]$, as long as $S$ is a balanced split graph.

The $\Delta$ property is particularly elegant and interesting, as it represents a connection between Structural Graph Theory and Number Theory. If $n\in\mathbb{N}$, we say that $n$ has the $\Delta$ property if there exists a triple $\{x,y,z\}$ of divisors of $n$ such that $1<x<y\leq z<\sqrt{n}$ and $\frac{n}{x}-x=\frac{n}{y}-y+\frac{n}{z}-z$. Chapter \ref{cap:La.Prop.Delta} is entirely devoted to the study of numbers that satisfy or fail to satisfy this condition. However, the first signs of a connection to Number Theory appear already in Section~\ref{sec:caract.Phi.simples}, when the formula for computing the multiplicity of an edge $uv$ in $\Phi$ is given. The definition of the $\Delta$ property may seem somewhat artificial until one investigates the induced cycles of the factor graph. In Section \ref{sec:caminos,ciclos.en.Phi} we show that such cycles can only have size 3 or 4. In this context, the most “symmetric” situation that can occur in $\Phi$ is that it contains an $n$-simple triangle, that is, a 3-cycle whose edges all have multiplicity $n$. In this case, we show in Section \ref{sec:Graphs<->NumberTheory} that if $\Phi(S)$ is an $n$-simple triangle $abca$ such that $\deg_S(a)<\deg_S(b)<\deg_S(c)$, then $n$ satisfies the $\Delta$ condition. Conversely, if $n$ has the $\Delta$ property, then there exists a balanced split graph $S$ such that $\Phi(S)$ is an $n$-simple triangle $abca$ with $\deg_S(a)<\deg_S(b)<\deg_S(c)$.

Before we begin with Chapter \ref{cap:introduccion}, which introduces all the notation and background knowledge necessary for this work, we would like to clarify that all the results in Chapters \ref{cap:deg(G)}, \ref{cap:cociente.gemelos}, \ref{cap:grafos.asociados.a.split}, \ref{cap:La.Prop.Delta} and \ref{cap:grafos.activos.con.deg<5} are original, unless otherwise specified. Appendix \ref{proof.lema.2switch.preserva.vert.activos} and \ref{familias.inserc.conjuntos}, on the other hand, contain original proofs of already known results.


\mainmatter
\chapter{Introduction} \label{cap:introduccion}
\nocite{*}

\section{Basic concepts and notation}

In this section, we introduce the basic concepts of Graph Theory that we will use throughout this text. Each definition is accompanied by relevant examples, comments, and observations, with the aim of offering better intuition. We begin with the formal definition of a graph, along with those of vertex and edge. From this, we define order, size, and the complement of a graph, as well as the concepts of subgraph and graph equality. Then, we show how to draw a graph in the plane, which allows us to express many concepts more intuitively. Using this language, we define notions such as neighborhood, vertex degree, and degree sequence. This is followed by the definition of path, which opens the door to many other structural notions such as cycles, connectivity, and diameter. Among the remaining concepts introduced, we highlight the following: multigraph, digraph, forest, isomorphism and clique. \\

Before delving into Graph Theory, we would like to clarify the following conventions used throughout this text:
\begin{enumerate}[(1).]
	\item $0\notin \mathbb{N}$;
	\item if $n\in\mathbb{N}$, we use the symbol $[n]$ to denote the set $\{1,2,\ldots,n\}$;
	\item sometimes, we write singleton sets without braces, that is, $a=\{a\}$;
	\item likewise, we sometimes write two-element sets without braces or commas, that is, $ab=\{a,b\}$;
	\item the difference between two sets, $A$ and $B$, is denoted $A-B$; that is: $A-B=\{a \in A \mid a\not \in B\}$.
\end{enumerate}
The context will always help to avoid ambiguities in items (3) and (4). \\

A \textbf{graph} is an ordered pair $G=(V,E)$, where $V$ is a set and
\[ E\subseteq {{V}\choose{2}}= \{W\subseteq V:|W|=2\}. \]
The elements of $V$ are called \textbf{vertices}, and those of $E$ are called \textbf{edges}. In other words, a graph is defined by a set of vertices and a set of edges, a collection of unordered pairs of vertices. Two graphs $G=(V,E)$ and $G'=(V',E')$ are equal if $V=V'$ and $E=E'$. The quantity $|V|$ is called the \textbf{order} of the graph, and $|E|$ is its \textbf{size}. We denote the order by $|G|$ and the size by $\lVert G\rVert$. Usually, a graph $G$ is presented without displaying the pair $(V,E)$ explicitly. In such cases, we use the notation $V(G)$ and $E(G)$ to refer to the vertex set and the edge set of $G$, respectively.

The \textbf{complement} of a graph $G=(V,E)$ is the graph $\overline{G}=(V,E^{c})$, where 
\[ E^{c}={{V}\choose{2}}-E. \]
If $G=(V,E)$ and $H=(V',E')$ is a graph such that $V'\subseteq V$ and $E'\subseteq E$, then we say that $H$ is a \textbf{subgraph} of $G$ and write $H\subseteq G$. If $H\subseteq G$ and $H\neq G$, then $H$ is called a \textbf{proper subgraph} of $G$ and we write $H\subset G$. We use the same notation for proper containment of sets. For example,
\[ G=(\{x,y,z,t\}, \{ \{x,z\}, \{x,t\} \}) \]
is a graph with 4 vertices and 2 edges, and
\[ H=(\{x,y,z\}, \{ \{x,z\} \}) \]
is a proper subgraph of $G$ with order 3 and size 1. The complement of $H$ is
\[ \overline{H}=(\{x,y,z\}, \{ \{x,y\}, \{y,z\} \}). \]
A subgraph $H\subseteq G$ is said to be an \textbf{induced subgraph} of $G$ if, for every $\{u,v\}\subseteq V(H)$, we have the following: if $uv\in E(G)$, then $uv\in E(H)$. We denote this relation by $\preceq$. When $H\preceq G$ and $H\neq G$, we can write $H\prec G$. If $W=\{v_1,\ldots,v_n\}\subseteq V(G)$, we denote by $\langle W\rangle_G$ or $\langle v_1,\ldots,v_n\rangle_G$ the subgraph $H\preceq G$ such that $V(H)=W$ and
\[ E(H)=\{uv\in E(G): u,v\in W\}. \]
In this case, we say that $H$ is the subgraph of $G$ induced by $W$ or by the vertices $v_1,\ldots,v_n$. The subscript “$G$” can be omitted if the context is clear. Obviously, $H\preceq G$ implies $H\subseteq G$, but the converse is not true. This can be quickly verified with the example:
\[ H=(\{a,b\},\varnothing)\subseteq (\{a,b,c\},\{ab,bc\})=G. \]
Here, $H$ is not induced in $G$ because $ab\in E(G)-E(H)$.
If $G=(V,E)$ is a graph and $E'\subseteq\binom{V}{2}$, then 
\[ G+E'=(V,E\cup E'), \quad G-E'=(V,E-E'). \]
If $V'\subseteq V$, then 
\[ G-V'=\langle V-V'\rangle_G. \]
When $E'=\{e\}=\{ab\}$ and $V'=\{v\}$, it is common to write $G\pm e$, $G\pm ab$, and $G-v$.

A graph $G$ has a natural visual representation, where the vertices are considered as distinct, labeled points in the plane (that is, each $v\in V(G)$ corresponds to a point labeled “$v$”). Any two such points $u,v$ are joined by an arc if and only if $uv \in E(G)$. The resulting drawing is called a \textbf{labeled graph}. The same drawing without the vertex labels is called an \textbf{unlabeled graph}. Two unlabeled graphs $G$ and $H$ are considered equal if there is a labeling of both that makes them equal as graphs. Two vertices (points) $a,b$ connected by an edge $ab$ are said to be \textbf{adjacent} (or \textbf{neighbors}), and $ab$ is said to be \textbf{incident} to $a$ and $b$. Formally, $u$ and $v$ are neighbors in $G$ if and only if $uv\in E(G)$. The set
\[ N_{G}(v)=\{x\in V(G): vx\in E(G)\}, \]
of all neighbors of a vertex $v$ in a graph $G$, is called the \textbf{neighborhood} of $v$ in $G$. Clearly, $v\notin N_G(v)$. Two edges $ab, cd$ that do not share a common vertex are said to be \textbf{disjoint} ($ab\cap cd = \varnothing$). The number of edges incident to a vertex $v$ in a graph $G$ is called the \textbf{degree} of $v$ in $G$, denoted by $\deg_G(v)$, $\deg(v)$, or simply $d_v$. It is clear that $0\leq d_v \leq |G|-1$. If $d_v=0$, then $v$ is said to be \textbf{isolated}; if $d_v=1$, then $v$ is a \textbf{leaf}; and if $d_v = |G| - 1$ (i.e., $N_G(v) = V(G) - v$), then $v$ is \textbf{universal}. A $k$-\textbf{regular} graph is a graph in which all vertices have degree $k$. The value of $k$ can be omitted when unnecessary.

Since we usually write $V(G) = [n]$ for some $n$, the degrees of the vertices in $G$ can be arranged as an $n$-tuple
\[ (d_i)_{i=1}^n = (d_1, \dots, d_n), \]
forming what is called the \textbf{degree sequence} of $G$, denoted by $s(G)$ or simply by $s$, if $G$ is clear from the context. For example, if $G = ([4], \{12, 23\})$, then $s(G) = (1, 2, 1, 0)$. If we label the vertices of $G$ so that $d_1 \geq \ldots \geq d_n$, we can rewrite $s$ in the more compact form:
\[ s = d_{n_1}^{\alpha_1} \ldots d_{n_k}^{\alpha_k}, \]
where $\{n_i: i\in[k]\}\subseteq [n]$, $d_{n_1} > \dots > d_{n_k}$ and
\[ \alpha_i = |\{v\in V(G): d_v = d_{n_i}\}|. \]
This form is especially useful for unlabeled graphs, where only the multiplicities $\alpha_i$ of the degrees $d_{n_i}$ matter. Two degree sequences $(d_i)_{i=1}^n$ and $(\delta_i)_{i=1}^m$ of labeled graphs are equal if and only if $n = m$ and $d_i = \delta_i$ for all $i \in [n]$. On the other hand, two degree sequences $d_1^{\alpha_1} \ldots d_n^{\alpha_n}$ and $\delta_1^{\beta_1} \ldots \delta_m^{\beta_m}$ for unlabeled graphs are equal if and only if $n = m$ and
\[ \{(d_i, \alpha_i): i\in[n]\} = \{(\delta_i, \beta_i): i\in[n]\}. \]
A basic but very important fact about degree sequences is the so-called \textbf{Handshake Lemma}, which states that
\[ \sum_{v\in V(G)} \deg_G(v) = 2\Vert G\Vert, \]
for any graph $G$. The symbol $\mathcal{G}(s)$ denotes the set of all graphs whose degree sequence is $s$. If $X$ is a graph, then $\mathcal{G}(X)$ means $\mathcal{G}(s(X))$. Depending on context, elements of $\mathcal{G}(s)$ may be considered labeled or unlabeled. Given an arbitrary sequence of integers $s$, there may or may not exist a graph $G$ such that $s(G) = s$. A \textbf{graphical sequence} is a sequence $s$ of integers such that $\mathcal{G}(s) \neq \varnothing$. In general, multiple graphs may correspond to the same graphical sequence.
In Figure~\ref{grafo0}, we show a graph $G$ of order 12 and size 11. Its degree sequence is
\[ s(G) = (3,3,2,4,3,1,2,1,1,1,1,0). \]
If we imagine $G$ without vertex labels, we could write:
\[ s = 4^1\, 3^3\, 2^2\, 1^5\, 0^1. \]
$G$ has five leaves (vertices $6,8,9,10,11$), one isolated vertex ($12$), and no universal vertices. The edges $\{4,5\}$ and $\{10,11\}$, which are disjoint, belong to $G$, while $56$ does not. Additionally, $87$ and $74$ are not disjoint because they both share vertex 7. The neighborhood of vertex 5 is the set $\{2,5,9\}$, while $N_G(12) = \varnothing$. The graph
\[ H = (\{2,3,4,5\}, \{23,34,45\}) \]
is a proper subgraph of $G$. However, $H$ is not induced in $G$, since $25\notin E(H)$.
\begin{figure}[H]
	\[
	\begin{tikzpicture}
		[scale=.7,auto=left,every node/.style={scale=.7,circle,thick,draw}] 
		\node [label={6}] (n6) at (1,10) {};
		\node [label={4}] (n4) at (4,8)  {};
		\node [label={5}] (n5) at (7,9)  {};
		\node [label={1}] (n1) at (11,8) {};
		\node [label={2}] (n2) at (9,6)  {}; 
		\node [label={3}] (n3) at (5,5)  {};
		\node [label={7}] (n7) at (2,7)  {};
		\node [label={8}] (n8) at (1,5.5)  {};
		\node [label=right:{9}] (n9) at (14,8)  {};
		\node [label=below:{10}] (n10) at (15,10)  {};
		\node [label=left:{11}] (n11) at (17,8)  {};
		\node [label={12}] (n12) at (12,6)  {};
		\foreach \from/\to in {n6/n4,n4/n5,n5/n1,n1/n2,n2/n5,n2/n3,n3/n4,n4/n7,n7/n8,n1/n9,n10/n11}
		\draw (\from) -- (\to);
	\end{tikzpicture}
	\] 
	\caption{Representation of a graph.}
	\label{grafo0}
\end{figure}
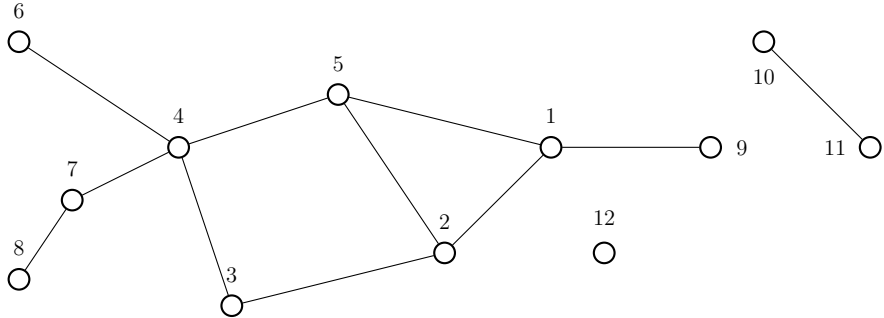
A \textbf{path} in $G$ is a sequence $v_1 \ldots v_n$ of distinct vertices in $G$ such that $v_i v_{i+1} \in E(G)$ for every $i\in [n-1]$, $n \geq 2$. We call $v_1$ the \textbf{initial vertex} of the path and $v_n$ the \textbf{terminal vertex}. The \textbf{length} of a path $v_1 \ldots v_n$ is the number of edges it contains, i.e., $n-1$. Any sequence consisting of a single vertex is considered a path of length 0.

A graph $G$ is \textbf{connected} if for every pair of vertices in $G$ there exists a path connecting them; otherwise, we say that $G$ is \textbf{disconnected} (or not connected). It follows from the definition that every graph of order 0 or 1 is connected. A \textbf{connected component} (or simply, a component) of $G$ is a maximal connected subgraph of $G$, that is, a connected subgraph not properly contained in any other connected subgraph of $G$. Visually, a disconnected graph appears split into two or more disjoint “pieces.” The \textbf{distance} between $u,v\in V(G)$ in $G$, denoted $dist_G(u,v)$, is the length of the shortest path between $u$ and $v$ in $G$. If $u$ and $v$ belong to different components, we define $dist_G(u,v)=\infty$. We may omit the subscript “$G$” when the context is clear. Every graph equipped with $dist(\ast,\ast)$ becomes a metric space, as $dist(\ast,\ast)$ satisfies the usual distance axioms. The \textbf{diameter} of $G$, denoted by $diam(G)$, is the greatest distance between any two vertices of $G$. Clearly, a graph of order 1 has diameter 0, and a disconnected graph has diameter $\infty$. We assume that if $|G|=0$, then $diam(G)=0$.

In Figure \ref{camino0} we can see the path $745219$, of length 5, in a connected graph of order 10. In Figure \ref{camino1}, on the other hand, we show the path $21$, of length 1, in a disconnected graph with 3 components.
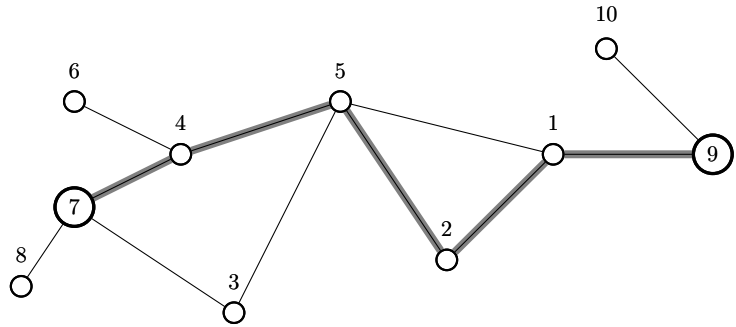
\begin{figure}[H]
	\[
	\begin{tikzpicture}
		[scale=.7,auto=left,every node/.style={scale=.7,circle,thick,draw}] 
		\node [label={6}] (n6) at (2,9) {};
		\node [label={4}] (n4) at (4,8)  {};
		\node [label={5}] (n5) at (7,9)  {};
		\node [label={1}] (n1) at (11,8) {};
		\node [label={2}] (n2) at (9,6)  {}; 
		\node [label={3}] (n3) at (5,5)  {};
		\node  [very thick] (n7) at (2,7)  {7};
		\node [label={8}] (n8) at (1,5.5)  {};
		\node  [very thick] (n9) at (14,8)  {9};
		\node [label={10}] (n10) at (12,10)  {};
		\foreach \from/\to in {n6/n4,n5/n1,n3/n7,n3/n5,n7/n8,n9/n10}
		\draw (\from) -- (\to);
		
		\draw (n7) -- (n4) [line width=3,color=gris];
		\draw (n4) -- (n5) [line width=3,color=gris];
		\draw (n5) -- (n2) [line width=3,color=gris];
		\draw (n2) -- (n1) [line width=3,color=gris];
		\draw (n1) -- (n9) [line width=3,color=gris];
		\draw (n4) -- (n5) ;
		\draw (n5) -- (n2) ;
		\draw (n2) -- (n1) ;
		\draw (n1) -- (n9) ;
		
		\node [label={6}] (n6) at (2,9) {};
		\node [label={4}] (n4) at (4,8)  {};
		\node [label={5}] (n5) at (7,9)  {};
		\node [label={1}] (n1) at (11,8) {};
		\node [label={2}] (n2) at (9,6)  {}; 
		\node [label={3}] (n3) at (5,5)  {};
		\node  [very thick] (n7) at (2,7)  {7};
		\node [label={8}] (n8) at (1,5.5)  {};
		\node  [very thick] (n9) at (14,8)  {9};
		\node [label={10}] (n10) at (12,10)  {};
		\draw (n7) -- (n4);
	\end{tikzpicture}
	\]
	\caption{Un camino de longitud 5 en un grafo conexo.}
	\label{camino0}	
\end{figure}

\begin{figure}[H]
	\[
	\begin{tikzpicture}
		[scale=.7,auto=left,every node/.style={scale=.7,circle,thick,draw}] 
		\node [label={6}] (n6) at (2,9) {};
		\node [label={4}] (n4) at (4,8)  {};
		\node [label={5}] (n5) at (7,9)  {};
		\node  [very thick](n1) at (11,8) {1};
		\node  [very thick](n2) at (9,6)  {2}; 
		\node [label={3}] (n3) at (5,6)  {};
		\node  [label={7}] (n7) at (2,7)  {};
		\node [label={8}] (n8) at (1,5.5)  {};
		\node  [label={9}] (n9) at (14,8)  {};
		\node [label={10}] (n10) at (12,10)  {};
		\foreach \from/\to in {n3/n7,n3/n5,n7/n8,n9/n10,n1/n10,n1/n9}
		\draw (\from) -- (\to);
		
		\draw (n7) -- (n4) ;
		\draw (n4) -- (n5) ;
		\draw (n2) -- (n1) [line width=3,color=gris];
		\draw (n2) -- (n1);
	\end{tikzpicture}
	\]
	\caption{Un camino de longitud 1 en un grafo disconexo.}
	\label{camino1}
\end{figure}
Let $G=(V,E)$ and $G'=(V',E')$ be graphs. If $V\cap V' = \varnothing = E\cap E'$, we say that $G$ and $G'$ are \textbf{disjoint}, and we write $G\cap G' = \varnothing$. In that case, we can “join” $G$ and $G'$ to form the graph
\[ G\dot{\cup}G' = (V\cup V', E\cup E'). \]
Moreover, if $G$ has $k$ components $G_i$, then we may write
\[ G = \dot{\bigcup}_{i=1}^k G_i, \]
since components are pairwise disjoint. When $G$ is unlabeled and all its components are isomorphic to the same unlabeled graph $H$, it is customary to write $G = kH$. For example, if $K_2$ is the unlabeled graph $(\{a,b\},\{ab\})$, then $2K_2$ is the unlabeled version of the graph $([4], \{12,34\})$. If $V\cap V' = \varnothing$, we say that $G$ and $G'$ are vertex-disjoint; if $E\cap E' = \varnothing$, we say they are edge-disjoint. Clearly, if $G$ and $G'$ are vertex-disjoint, they are also edge-disjoint.

A \textbf{cycle} in a graph $G$ is a sequence $v_1 \ldots v_n v_{n+1}$ of vertices of $G$ such that $n\geq 3$, $v_1 \ldots v_n$ is a path, $v_{n+1} = v_1$, and $v_1 v_n \in E(G)$. In other words, a cycle is essentially a path in which the initial and terminal vertices are joined by an edge. Thus, any vertex in the cycle is simultaneously an initial and terminal vertex. The \textbf{length} of a cycle $C=v_1 \ldots v_n v_{n+1}$ is the number of edges (or distinct vertices) of $C$, i.e., $n$. Clearly, a cycle (or path) in $G$ is a subgraph of $G$. Therefore, two cycles (or paths) in $G$ are equal if and only if they are equal as graphs. Graphs containing exactly one cycle are called \textbf{unicyclic}.

The terms “path” and “cycle” are also commonly used to refer to graphs that have such structures. More precisely, we say that $G$ is a path (cycle) if there exists a path (cycle) in $G$ that uses all its edges. Clearly, paths and cycles are connected graphs, and all paths (cycles) in a graph $G$ can be viewed as connected subgraphs of $G$. We denote by $P_n$ the unlabeled path of order $n\geq 2$ (and size $n-1$), and by $C_n$ the unlabeled cycle of order and size $n\geq 3$. Obviously, $C_n$ is unicyclic.

Note that the very definition of graph we gave, forbids the existence of cycles of length 1 and 2 in $G$. These are called, respectively, \textbf{loops} and \textbf{parallel edges}. In the mathematical literature, more general definitions of graph may allow such cases. Therefore, a graph with no loops or parallel edges is often called a \textbf{simple graph}. Throughout this work, the term “graph” will always mean “simple graph,” unless stated otherwise. We denote by $\mathcal{G}_n$ the set of all simple graphs of order $n$. Whether the elements of $\mathcal{G}_n$ are considered labeled or unlabeled will depend on context.
\begin{figure}[H]
	\[
	\begin{tikzpicture}
		[scale=.7,auto=left,every node/.style={scale=.7,circle,thick,draw}] 
		\node  (1) at (0,0) {1};
		\node  [label=below:$M_{0}$](2) at (2.5,0) {2};
		\node  (3) at (2.5,2.5) {3};
		\node  (4) at (0,2.5) {4};
		\foreach \from/\to in {2/3,3/4,4/1}
		\draw (\from) -- (\to);
		\draw (1) [loop right,scale=3.5] to (1) ;
	\end{tikzpicture}
	\begin{tikzpicture}
		\node (1) at (0,0) {};
		\node (1) at (1,0) {};
	\end{tikzpicture}
	\begin{tikzpicture}
		[scale=.7,auto=left,every node/.style={scale=.7,circle,thick,draw}] 
		\node  (1) at (0,0) {1};
		\node  [label=below:$M_{1}$](2) at (2.5,0) {2};
		\node  (3) at (2.5,2.5) {3};
		\node  (4) at (0,2.5) {4};
		\foreach \from/\to in {1/2,2/3,3/4,4/1}
		\draw (\from) -- (\to);
		\draw (2) [bend right=25] to (3) ;
		\draw (2) [bend left=25] to (3) ;
	\end{tikzpicture}
	\begin{tikzpicture}
		\node (1) at (0,0) {};
		\node (1) at (.5,0) {};
	\end{tikzpicture}
	\begin{tikzpicture}
		[scale=.7,auto=left,every node/.style={scale=.7,circle,thick,draw}] 
		\node  (1) at (0,0) {1};
		\node  [label=below:$M_{2}$](2) at (2.5,0) {2};
		\node  (3) at (2.5,2.5) {3};
		\node  (4) at (0,2.5) {4};
		\foreach \from/\to in {2/3,3/4,4/1}
		\draw (\from) -- (\to);
		\draw (1) [loop right,scale=3.5] to (1) ;
		\draw (2) [bend right=25] to (3) ;
	\end{tikzpicture}
	\] 
	\caption{$M_{0}, M_{1}$ y $M_{2}$ are not simple graphs.}
	\label{nosimples}
\end{figure}
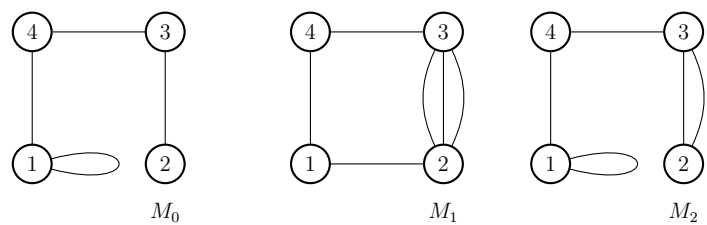
In a more general setting that allows loops and parallel edges, we introduce the concept of a multigraph. Formally, a \textbf{multigraph} $M$ is a pair $(V,E)$ where $V = V(M)$ is a set of vertices and $E = E(M)$ is a \textbf{multiset} of edges—that is, a set of edges that allows repetitions, as well as ordered pairs of the form $(v,v)$, which represent loops on vertex $v$. The number of repetitions of an edge $e \in E(M)$ is called the \textbf{multiplicity} of $e$ and is denoted by $\sigma_e(M)$, or simply $\sigma_e$ if the multigraph is clear from context. Obviously, $e \notin E(M)$ if and only if $\sigma_e(M) = 0$. In these terms, a loopless multigraph $M$ is a simple graph if and only if $\sigma_e(M) = 1$ for all $e \in E(M)$. Alternatively: a loopless multigraph $M$ is simple if
\[ \sigma_{uv}(M) \in \{0,1\}, \]
for all $\{u,v\} \subseteq V(M)$. Two multigraphs $M = (V,E)$ and $M' = (V',E')$ are equal if $V = V'$ and
\[ \sigma_{uv}(M) = \sigma_{uv}(M') \]
for all $\{u,v\} \subseteq V$. We use the notation $me$ to indicate that there are $m \geq 2$ copies of an element $e \in E(M)$. Obviously, $m = \sigma_e(M)$. For example, in Figure~\ref{nosimples} we have:
\[ E(M_2) = \{ (1,1), \{1,4\}, \{3,4\}, 2\{2,3\} \}. \]

A \textbf{digraph} (or \textbf{directed graph}) is a pair $D=(V,E)$ where $V$ is a set of vertices and
\[ E \subseteq \{(x,y) : x,y \in V,\, x \neq y\}. \]
The elements of $E$ are called \textbf{arcs} or \textbf{directed edges}. An arc $(x,y)$ represents a directed edge from $x$ to $y$, and is depicted as an arrow from $x$ to $y$. In simple terms, a digraph is a graph with directed edges and no loops. When both $(a,b)$ and $(b,a)$ are in $E$, we may draw an undirected edge $\{a,b\}$ between $a$ and $b$, and say that $D$ “contains the edge $\{a,b\}$.” When there is no ambiguity, we may write $xy$ instead of $(x,y)$ and $yx$ instead of $(y,x)$. Two digraphs $D=(V,E)$ and $D'=(V',E')$ are equal if and only if $V=V'$ and $E=E'$.

\textbf{Important note:} Unless otherwise specified, throughout this work any concept or operation defined for simple graphs will be considered applicable to multigraphs or digraphs by ignoring multiplicities and directions of edges. \\

A \textbf{forest} is an acyclic graph—that is, a graph with no cycles. A \textbf{tree} is a connected forest. Clearly, $P_n$ is a tree and $C_n$ is not a forest. The following statements are equivalent to saying that $T$ is a tree:
\begin{enumerate}[(1).]
	\item $T$ is connected and $\Vert T \Vert = |T| - 1$;
	\item any two vertices in $T$ are connected by a unique path;
	\item for every $e \in E(T)$, the graph $(V(T), E(T)-e)$ is disconnected;
	\item for every $e \notin E(T)$, the graph $(V(T), E(T) \cup \{e\})$ is unicyclic.
\end{enumerate}

Another important family of trees are the \textbf{stars} $S_n$ of order $n \geq 3$. These are trees with exactly $n-1$ leaves. In Figure~\ref{bosque0}, we show a forest of order 12 with 4 components. Each component is a tree, and the one containing vertex 10 is a star.

\begin{figure}[H]
	\[
	\begin{tikzpicture}
		[scale=.7,auto=left,every node/.style={scale=.7,circle,thick,draw}]
		\node [label={6}] (n6) at (2,9) {};
		\node [label={4}] (n4) at (4,8)  {};
		\node [label={5}] (n5) at (7,9)  {};
		\node [label={1}] (n1) at (11,8) {};
		\node [label={2}] (n2) at (9,7)  {}; 
		\node [label={3}] (n3) at (6,7)  {};
		\node [label={7}] (n7) at (2,7)  {};
		\node [label={13}] (n8) at (13,10)  {};
		\node [label={9}] (n9) at (14,8)  {};
		\node [label={10}] (n10) at (15,10)  {};
		\node [label={11}] (n11) at (17,8)  {};
		\node [label={12}] (n12) at (12.5,7)  {};
		\node [label={8}] (n13) at (8,10)  {};
		
		\foreach \from/\to in {n6/n4,n4/n5,n2/n5,n2/n3,n4/n7,n10/n11,n9/n10,n5/n13,n8/n10}
		\draw (\from) -- (\to);
	\end{tikzpicture}
	\]
	\caption{A forest of order 12 and 4 components.}
	\label{bosque0}
\end{figure}
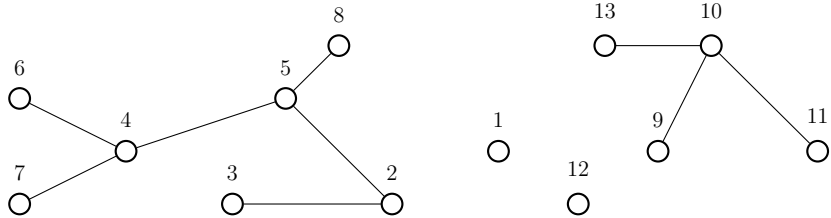

Two graphs $G = (V,E)$ and $G' = (V',E')$ are \textbf{isomorphic} if there exists a bijective function $\varphi: V \to V'$ such that $uv \in E$ if and only if $\varphi(u)\varphi(v) \in E'$. Otherwise, we say that $G$ and $G'$ are not isomorphic. When this condition holds, we say $\varphi$ is an \textbf{isomorphism} between $G$ and $G'$.

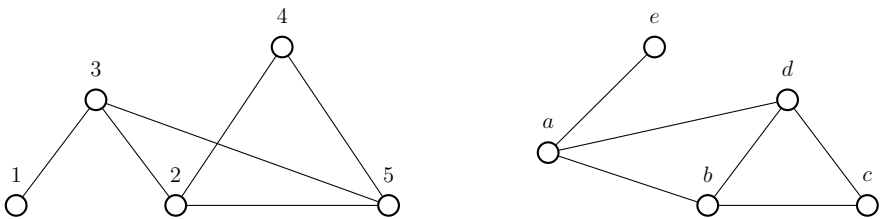
\begin{figure}[H]
	\[
	\begin{tikzpicture}
		[scale=.7,auto=left,every node/.style={scale=.7,circle,thick,draw}]
		\node [label=1] (n1) at (0,0) {};
		\node [label=2] (n2) at (3,0) {};
		\node [label=3] (n3) at (1.5,2) {};
		\node [label=4] (n4) at (5,3) {};
		\node [label=5] (n5) at (7,0) {};
		\node [label=$a$] (n6) at (10,1) {};
		\node [label=$b$] (n7) at (13,0) {};
		\node [label=$c$] (n8) at (16,0) {};
		\node [label=$d$] (n9) at (14.5,2) {};
		\node [label=$e$] (n10) at (12,3) {};
		
		\foreach \from/\to in {n1/n3,n2/n3,n2/n4,n4/n5,n3/n5,n2/n5,n6/n7,n7/n8,n8/n9,n9/n7,n10/n6,n6/n9}
		\draw (\from) -- (\to);
	\end{tikzpicture}
	\]  
	\caption{Two isomorphic graphs.}
	\label{isomorf0}
\end{figure}

The relation of isomorphism, denoted $G \approx G'$, is an equivalence relation and partitions the set of all labeled graphs of order $n$ into disjoint \textbf{isomorphism classes} (i.e., the unlabeled graphs of order $n$). Isomorphic graphs are structurally the same except for vertex labels. Necessary (but not sufficient) conditions for two graphs to be isomorphic include: same order and size, same degree sequence, same number of cycles and components. As illustrated in Figure~\ref{noisomorf0}, even satisfying all these properties does not guarantee isomorphism.

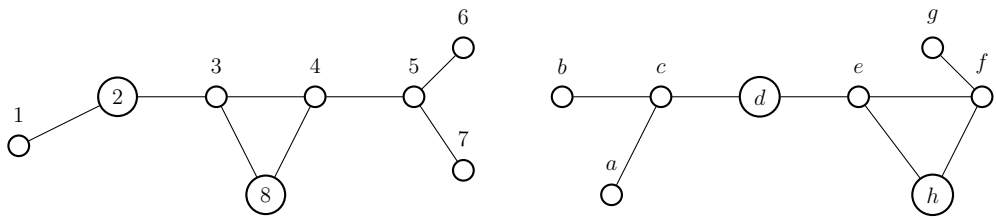
\begin{figure}[H]
	\[
	\begin{tikzpicture}
		[scale=.65,auto=left,every node/.style={scale=.7,circle,thick,draw}]
		\node [label=1] (n1) at (0,1) {};
		\node (n2) at (2,2) {2};
		\node [label=3] (n3) at (4,2) {};
		\node [label=4] (n4) at (6,2) {};
		\node [label=5] (n5) at (8,2) {};
		\node [label=6] (n6) at (9,3) {};
		\node [label=7] (n7) at (9,0.5) {};
		\node (n8) at (5,0) {8};
		
		\node [label=$a$] (n9) at (12,0) {};
		\node [label=$b$] (n10) at (11,2) {};
		\node [label=$c$] (n11) at (13,2) {};
		\node (n12) at (15,2) {$d$};
		\node [label=$e$] (n13) at (17,2) {};
		\node [label=$f$] (n14) at (19.5,2) {};
		\node [label=$g$] (n15) at (18.5,3) {};
		\node (n16) at (18.5,0) {$h$};
		
		\foreach \from/\to in {n1/n2,n2/n3,n3/n4,n4/n5,n5/n6,n5/n7,n3/n8,n8/n4,n11/n10,n11/n9,n11/n12,n12/n13,n13/n16,n14/n16,n13/n14,n14/n15}
		\draw (\from) -- (\to);
	\end{tikzpicture}
	\]
	\caption{Two non-isomorphic graphs.}
	\label{noisomorf0}
\end{figure}

A more general concept than isomorphism is that of a \textbf{homomorphism}. We say that $G=(V,E)$ and $G'=(V',E')$ are \textbf{homomorphic} if there exists a function $\varphi:V\rightarrow V'$ such that $uv\in E$ implies $\varphi(u)\varphi(v)\in E'$. Every isomorphism is a bijective homomorphism whose inverse is also a homomorphism. If $\varphi$ is an isomorphism between $G$ and $G'$, then $H \preceq G$ if and only if $\varphi(H) \preceq G'$, since isomorphisms preserve adjacency and non-adjacency. This is not true in general for homomorphisms: in such cases, $H \preceq G$ only implies $\varphi(H) \subseteq G'$.

A graph $G=(V,E)$ is \textbf{complete} if $V\neq\varnothing$ and
\[ E={{V}\choose{2}}, \]
i.e., every pair of vertices is adjacent. If $E=\varnothing$ and $V\neq\varnothing$, we say that $G$ is \textbf{empty}, consisting only of isolated vertices. Clearly, $G$ is complete if and only if $\overline{G}$ is empty. A graph is \textbf{null} when $V=\varnothing=E$. Up to isomorphism, there is only one complete, empty, or null graph. For $n\geq 1$, $K_n$ denotes the unlabeled complete graph of order $n$, and so $\overline{K_n}$ is the empty graph of the same order. A \textbf{triangle} in $G$ is a subgraph isomorphic to $K_3$. In multigraphs (digraphs), a triangle is a subgraph isomorphic to $K_3$ ignoring multiplicities (or directions). The \textbf{trivial graph} is the complete (or empty) graph of order 1: $K_1$. We denote the null graph by $K_0$.

An \textbf{independent set} of a graph $G$ is a subset of its vertices such that no two of them are adjacent, that is, $I$ is independent if $\langle I\rangle_G$ is empty or null. A \textbf{maximum independent set} is an independent set of largest possible size. The \textbf{independence number} of $G$, denoted by $\alpha(G)$, is the size of a maximum independent set.

A \textbf{clique} $K$ in $G$ is a subset of $V(G)$ such that $\langle K\rangle_G$ is complete. A \textbf{maximum clique} is a clique of maximum possible size. The \textbf{clique number} of $G$, denoted by $\omega(G)$, is the size of a maximum clique. Clearly, every independent set in $G$ is a clique in $\overline{G}$, and vice versa. Thus,
\[ \alpha(G)=\omega(\overline{G}) \quad \text{and} \quad \omega(G)=\alpha(\overline{G}). \]


\section{Preliminaries} \label{sec:preliminares}

In the mathematical literature, a \textbf{2-switch} on $G$ is usually defined as follows (see \cite{chartrand2010graphs}, page 23). Let $G$ be a graph containing four distinct vertices $a,b,c,d$ such that $ab,cd \in G$ and $ac,bd \notin G$. A 2-switch in $G$ is a process that removes the edges $ab$ and $cd$ from $G$ and adds the edges $ac$ and $bd$. However, in this work, we will use another definition (see \cite{vnsigma.2switch.graphs.forests}, page 5, or \cite{vnsigma.2switch.unic.pseudof}, page 3). A 2-switch is a function $\tau = {{a \ b}\choose{c \ d}} : \mathcal{G}_{n}\rightarrow \mathcal{G}_{n}$ such that: 
\begin{enumerate}[(1).]
	\item \begin{equation}
		\label{2switch.def}
		\tau(G)={{a \ b}\choose{c \ d}}G=(G-\{ab,cd\})+ \{ac,bd\},
	\end{equation} 
	if $\{ab,cd\}\subseteq G, ab\cap cd=\varnothing$ and $\{ac,bd\}\subseteq\overline{G}$;
	\item \[ \tau(G)=G, \]
	otherwise.
\end{enumerate}

By defining the 2-switch as a function, the application of a sequence $\theta=(\tau_i)_{i=1}^k$ of 2-switches to $G$ is automatically well-defined as a composition of functions: $\theta(G)=\tau_k\ldots\tau_1(G)$.

When $\tau(G)\neq G$ (i.e., \eqref{2switch.def}), we say that $\tau$ is \textbf{active} in $G$, or, equivalently, that $\tau$ \textbf{acts} on $G$. Otherwise (i.e., if $\tau(G)=G$), we say that $\tau$ is \textbf{inactive} in $G$ (or that it is not active, or that it does not act on $G$). Note that the 2-switch is idempotent, i.e., $\tau^2=\tau$. The matrix $A={{a \ b}\choose{c \ d}}$, of size $2\times 2$, is called the \textbf{action matrix} of $\tau$. Obviously, if $\tau $ is active in $G$, then the entries of $A$ are vertices of $G$. In this case, note that the rows of $A$ represent the edges to be removed, while the columns of $A$ represent the edges to be added. We say that $\tau$ \textbf{activates} the vertex $v$ in $G$ if $\tau$ is active in $G$ and $v$ is an entry of the action matrix of $\tau$. Two action matrices are considered equal if they are equal as functions, i.e., if they represent the same 2-switch. If $P$ is a permutation matrix, it is clear that $PA=A=AP$. However, if $\tau$ is not the identity in $\mathcal{G}_n$, then $A^t\neq A$. The function ${a \ c}\choose{b \ d}$ is called the \textbf{inverse} 2-switch of $\tau$ and is denoted by $\tau^{-1}$. The term ``inverse" comes from the easily verifiable fact that 
\[ \tau^{-1}\tau(G)=G, \]
for all $G\in\mathcal{G}_n$. 

It follows from the definition that the 2-switch preserves the degree sequence, i.e.:
\begin{equation}
	\label{2switch.prop.fund}
	s(\tau(G))=s(G),
\end{equation} 
for every 2-switch $\tau$. Indeed, when $\tau$ is active in $G$, we see from \eqref{2switch.def} that the degrees of the 4 vertices $a,b,c,d$ momentarily decrease by 1 when $\tau$ removes the edges $ab$ and $cd$. But, when adding $ac$ and $bd$, this decrease is compensated. Undoubtedly, \eqref{2switch.prop.fund} is the most important property of the 2-switch. From it, the fundamental theorem of the 2-switch can be proven, which we state below.

\begin{theorem}[\cite{chartrand2010graphs}, page 24]
	\label{berge's.theorem}
	If $G$ and $H$ are two distinct graphs with the same degree sequence, then there exists a sequence of 2-switches $(\tau_i)_{i=1}^k$ such that $H=\tau_k\ldots\tau_1(G)$.
\end{theorem}

Let $Q=\{a,b,c,d\}$ be a subset of 4 vertices of a graph $G$ and let $\tau={{a \ b}\choose{c \ d}}$ be a 2-switch. A key observation is that if $\tau$ is active in $G$, then $\langle Q\rangle_G$ is isomorphic to $P_4, C_4$ or $2K_2$. This can be quickly proven by analyzing the 11 possible isomorphism classes for a graph of order 4 (see Figure \ref{los.11.de.orden4}) and verifying that only these three satisfy all the conditions required by \eqref{2switch.def}.
\begin{figure}[h]
	\[
	\begin{tikzpicture}
		[scale=.7,auto=left,every node/.style={scale=.65,circle,thick,draw}] 
		\node (n1) at (0,0) {};
		\node (n2) at (1,0)  {};
		\node [label=$\overline{K_4}$](n3) at (0,1)  {};
		\node (n4) at (1,1) {};
		\foreach \from/\to in {}
		\draw (\from) -- (\to);
		
		\node (n5) at (3,0) {};
		\node (n6) at (4,0)  {};
		\node [label=$\overline{D_4}$](n7) at (3,1)  {};
		\node (n8) at (4,1) {};
		\foreach \from/\to in {n5/n6}
		\draw (\from) -- (\to);
		
		\node (n9) at (6,0) {};
		\node (n10) at (7,0)  {};
		\node [label=$2K_2$](n11) at (6,1)  {};
		\node (n12) at (7,1) {};
		\foreach \from/\to in {n9/n10,n11/n12}
		\draw (\from) -- (\to);
		
		\node (n13) at (9,0) {};
		\node (n14) at (10,0)  {};
		\node [label=$\overline{U_4}$](n15) at (9,1)  {};
		\node (n16) at (10,1) {};
		\foreach \from/\to in {n13/n14,n14/n16}
		\draw (\from) -- (\to);
		
		\node (n17) at (12,0) {};
		\node (n18) at (13,0)  {};
		\node [label=$P_4$](n19) at (12,1)  {};
		\node (n20) at (13,1) {};
		\foreach \from/\to in {n17/n18,n20/n19,n18/n20}
		\draw (\from) -- (\to);
		
		\node (n21) at (15,0) {};
		\node (n22) at (16,0)  {};
		\node [label=$S_4$](n23) at (15,1)  {};
		\node (n24) at (16,1) {};
		\foreach \from/\to in {n21/n22,n22/n24,n22/n23}
		\draw (\from) -- (\to);
		
	\end{tikzpicture}
	\]
	\[
	\begin{tikzpicture}
		[scale=.7,auto=left,every node/.style={scale=.65,circle,thick,draw}] 
		\node (n1) at (0,0) {};
		\node (n2) at (1,0)  {};
		\node [label=$\overline{S_4}$](n3) at (0,1)  {};
		\node (n4) at (1,1) {};
		\foreach \from/\to in {n1/n2,n2/n4,n4/n1}
		\draw (\from) -- (\to);
		
		\node (n5) at (3,0) {};
		\node (n6) at (4,0)  {};
		\node [label=$U_4$](n7) at (3,1)  {};
		\node (n8) at (4,1) {};
		\foreach \from/\to in {n5/n6,n6/n8,n8/n5,n5/n7}
		\draw (\from) -- (\to);
		
		\node (n9) at (6,0) {};
		\node (n10) at (7,0)  {};
		\node [label=$C_4$](n11) at (6,1)  {};
		\node (n12) at (7,1) {};
		\foreach \from/\to in {n9/n10,n11/n12,n10/n12,n9/n11}
		\draw (\from) -- (\to);
		
		\node (n13) at (9,0) {};
		\node (n14) at (10,0)  {};
		\node [label=$D_4$](n15) at (9,1)  {};
		\node (n16) at (10,1) {};
		\foreach \from/\to in {n13/n14,n14/n16,n16/n15,n15/n13,n14/n15}
		\draw (\from) -- (\to);
		
		\node (n17) at (12,0) {};
		\node (n18) at (13,0)  {};
		\node [label=$K_4$](n19) at (12,1)  {};
		\node (n20) at (13,1) {};
		\foreach \from/\to in {n17/n18,n20/n19,n18/n20,n20/n17,n17/n19,n18/n19}
		\draw (\from) -- (\to);
	\end{tikzpicture} 
	\]
	\caption{The 11 unlabeled graphs of order 4.}
	\label{los.11.de.orden4}
\end{figure}
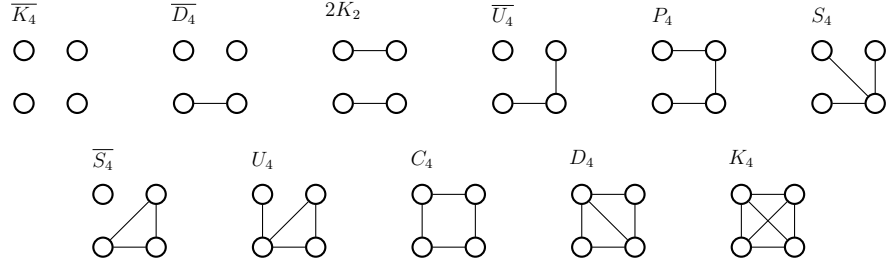

Let $\mathcal{X}$ be an arbitrary family of graphs and let $G,H\in\mathcal{X}, G\neq H$. We say that $H$ is related to $G$ via 2-switch if there exists a 2-switch $\tau$ such that $H=\tau(G)$. We denote this relation by $H\sim G$ or $H\sim_{\tau} G$. If $H\sim G$, it is clear that also $G\sim H$, since $G=\tau^{-1}(H)$. Thus, $\sim$ can be seen as an adjacency relation in $\mathcal{X}$, if we treat the latter as a set of vertices. Being
\[ E(\mathcal{X})=\{ \{G,H\}: G\sim H \}\subseteq {{\mathcal{X}}\choose{2}}, \]
we then have that the pair 
\[ \mathbf{T}(\mathcal{X})=(\mathcal{X},E(\mathcal{X})) \]
is a graph whose vertices are the members of $\mathcal{X}$, and where two vertices (graphs) $G$ and $H$ are neighbors if and only if $G$ and $H$ can be transformed into each other by means of a 2-switch. The graph $\mathbf{T}(\mathcal{X})$ is called the \textbf{transition space} associated with $\mathcal{X}$. In Figure \ref{ejemplo.esp.trans.} we can see a concrete example of this.
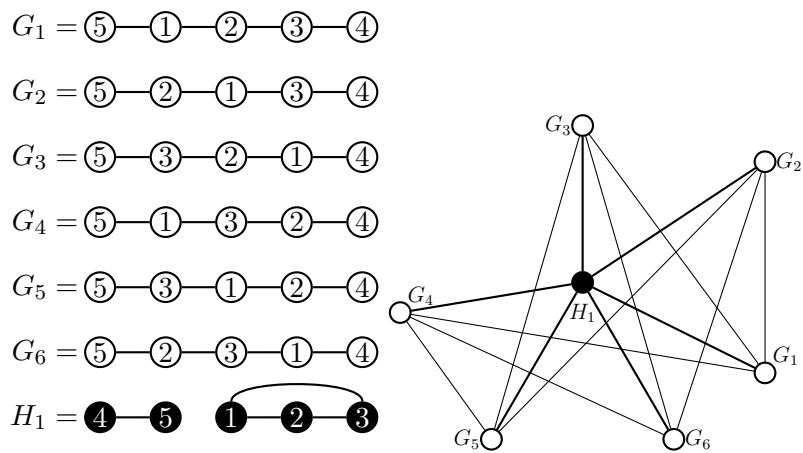
\begin{figure}[h] 
	\label{ejemplo.esp.trans.}
	\centering
	\begin{tikzpicture}[scale=1.15, every node/.style={circle, draw, thick, minimum size=1pt, inner sep=.8pt,font=\small}]
		
		\def\xsep{0.75}  
		\def\ysep{0.75}  
		
		\node (n11) at (1*\xsep, 0) [label=left:{$G_1=$}]{5};
		\node (n12) at (2*\xsep, 0) {1};
		\node (n13) at (3*\xsep, 0) {2};
		\node (n14) at (4*\xsep, 0) {3};
		\node (n15) at (5*\xsep, 0) {4};
		\draw[thick] (n11) -- (n12) -- (n13) -- (n14) -- (n15);
		
		\node (n21) at (1*\xsep, -1*\ysep)[label=left:{$G_2=$}] {5};
		\node (n22) at (2*\xsep, -1*\ysep) {2};
		\node (n23) at (3*\xsep, -1*\ysep) {1};
		\node (n24) at (4*\xsep, -1*\ysep) {3};
		\node (n25) at (5*\xsep, -1*\ysep) {4};
		\draw[thick] (n21) -- (n22) -- (n23) -- (n24) -- (n25);
		
		\node (n31) at (1*\xsep, -2*\ysep)[label=left:{$G_3=$}] {5};
		\node (n32) at (2*\xsep, -2*\ysep) {3};
		\node (n33) at (3*\xsep, -2*\ysep) {2};
		\node (n34) at (4*\xsep, -2*\ysep) {1};
		\node (n35) at (5*\xsep, -2*\ysep) {4};
		\draw[thick] (n31) -- (n32) -- (n33) -- (n34) -- (n35);
		
		\node (n41) at (1*\xsep, -3*\ysep)[label=left:{$G_4=$}] {5};
		\node (n42) at (2*\xsep, -3*\ysep) {1};
		\node (n43) at (3*\xsep, -3*\ysep) {3};
		\node (n44) at (4*\xsep, -3*\ysep) {2};
		\node (n45) at (5*\xsep, -3*\ysep) {4};
		\draw[thick] (n41) -- (n42) -- (n43) -- (n44) -- (n45);
		
		\node (n51) at (1*\xsep, -4*\ysep)[label=left:{$G_5=$}] {5};
		\node (n52) at (2*\xsep, -4*\ysep) {3};
		\node (n53) at (3*\xsep, -4*\ysep) {1};
		\node (n54) at (4*\xsep, -4*\ysep) {2};
		\node (n55) at (5*\xsep, -4*\ysep) {4};
		\draw[thick] (n51) -- (n52) -- (n53) -- (n54) -- (n55);
		
		\node (n61) at (1*\xsep, -5*\ysep)[label=left:{$G_6=$}] {5};
		\node (n62) at (2*\xsep, -5*\ysep) {2};
		\node (n63) at (3*\xsep, -5*\ysep) {3};
		\node (n64) at (4*\xsep, -5*\ysep) {1};
		\node (n65) at (5*\xsep, -5*\ysep) {4};
		\draw[thick] (n61) -- (n62) -- (n63) -- (n64) -- (n65);
		
		\node [color=black,fill=black](n71) at (1*\xsep, -6*\ysep)[label=left:{\textcolor{black}{$H_1=$}}] {\textcolor{white}{4}};
		\node [color=black,fill=black](n72) at (2*\xsep, -6*\ysep) {\textcolor{white}{5}};
		\node [color=black,fill=black](n73) at (3*\xsep, -6*\ysep) {\textcolor{white}{1}};
		\node [color=black,fill=black](n74) at (4*\xsep, -6*\ysep) {\textcolor{white}{2}};
		\node [color=black,fill=black](n75) at (5*\xsep, -6*\ysep) {\textcolor{white}{3}};
		\draw[thick,color=black] (n71) -- (n72)  (n73) -- (n74) -- (n75);
		\draw[thick,color=black] (n73) .. controls +(0,0.45) and +(0,0.45) .. (n75);
		
	\end{tikzpicture} 
	\begin{tikzpicture}[scale=.8, every node/.style={scale=0.7}]
		\node[draw, circle, fill=white, thick] (G_1) at (3.00,-1.50) {};
		\node[above right] at (G_1) {$G_1$};
		\node[draw, circle, fill=white, thick] (G_2) at (3.00,2) {};
		\node at ($(G_2)+(0.4,0)$) {$G_2$};
		\node[draw, circle, fill=white, thick] (G_3) at (0.00,2.60) {};
		\node at ($(G_3)+(-0.4,0)$) {$G_3$};
		\node[draw, circle, fill=white, thick] (G_4) at (-3.00,-0.50) {};
		\node[above right] at (G_4) {$G_4$};
		\node[draw, circle, fill=white, thick] (G_5) at (-1.50,-2.60) {};
		\node at ($(G_5)+(-0.4,0)$) {$G_5$};
		\node[draw, circle, fill=white, thick] (G_6) at (1.50,-2.60) {};
		\node at ($(G_6)+(0.4,0)$) {$G_6$};
		\node[draw, circle, fill=black, thick,color=black] (H_1) at (0.00,0.00) {};
		\node at ($(H_1)+(0,-0.5)$) {\textcolor{black}{$H_1$}};
		\draw (G_1) -- (G_2);
		\draw (G_1) -- (G_3);
		\draw[thick,color=black] (G_1) -- (H_1);
		\draw (G_1) -- (G_4);
		\draw[thick,color=black] (G_2) -- (H_1);
		\draw (G_2) -- (G_5);
		\draw (G_2) -- (G_6);
		\draw[thick,color=black] (G_3) -- (H_1);
		\draw (G_3) -- (G_5);
		\draw (G_3) -- (G_6);
		\draw[thick,color=black] (H_1) -- (G_4);
		\draw[thick,color=black] (H_1) -- (G_5);
		\draw[thick,color=black] (H_1) -- (G_6);
		\draw (G_4) -- (G_5);
		\draw (G_4) -- (G_6);
	\end{tikzpicture}
	\caption{Transition space associated with $\mathcal{G}(2^31^2)$.}
\end{figure}

It follows from the definition that a transition space is completely determined by its vertex set, i.e., $\mathbf{T}(\mathcal{X}_1 )=\mathbf{T}(\mathcal{X}_2 )$ if and only if $\mathcal{X}_1 =\mathcal{X}_2$. Therefore, when there is no ambiguity in the context, we can simply write $\mathcal{X}$, instead of $\mathbf{T}(\mathcal{X})$. Moreover, we can simply say ``space" instead of ``transition space". If $\mathcal{Y}\subseteq\mathcal{X}$, we say that $\mathbf{T}(\mathcal{Y})$ is a (transition) \textbf{subspace} of $\mathbf{T}(\mathcal{X})$. Note that $\mathbf{T}(\mathcal{Y})\preceq\mathbf{T}(\mathcal{X})$. On the other hand, if $\mathbf{S}\preceq\mathbf{T}(\mathcal{X})$, then $\mathbf{S}=\mathbf{T}(V(\mathbf{S}))$. Therefore, the subspaces of a transition space are precisely its induced subgraphs. Since the 2-switch preserves the degree sequence, each connected component $\mathcal{X}_i$ of $\mathcal{X}$ is a subspace of $\mathcal{G}(s_i)$, for some graphical sequence $s_i$. 

If $s$ is an integer sequence, the space $\mathbf{T}(\mathcal{G}(s))$ is usually called the \textbf{realization graph} associated with $s$ (\cite{arikati1999realization}, page 213). A direct consequence of Theorem \ref{berge's.theorem} is that $\mathcal{G}(s)$ is connected, for every integer sequence $s$. In recent years, the problem of determining whether a given transition space is connected or not has been extensively studied. In this direction, many advances have already been made. For example, it is known that the space $\mathcal{F}(s)$, of all forests with degree sequence $s$, is connected. An analogous result also holds for unicyclic graphs, connected graphs, and pseudoforests (the latter are graphs whose components are trees or unicyclic graphs). To delve deeper into these topics, the following papers can be consulted: \cite{vnsigma.2switch.graphs.forests}, \cite{vnsigma.2switch.unic.pseudof}. The most important feature of a connected space $\mathcal{X}$ is that it ensures the possibility of transforming, via 2-switches, a graph $G\in\mathcal{X}$ into another graph $H\in\mathcal{X}$ in such a way that every intermediate graph in the transition also belongs to $\mathcal{X}$. We formalize this concept in the next theorem.

\begin{theorem}
	Let $\mathcal{X}$ be a connected transition space and let $G,H\in\mathcal{X}$. Then there exists a sequence $\theta=(\tau_i)_{i=1}^k$ of 2-switches such that:
	\begin{enumerate}
		\item $H=\theta(G)$;
		\item $\tau_i\ldots\tau_1(G)\in\mathcal{X}$, for every $i\in[k]$.
	\end{enumerate} 
\end{theorem}

\begin{proof}
	Since $\mathcal{X}$ is connected, there exists a path in $\mathcal{X}$ connecting the vertex $G$ with the vertex $H$.
\end{proof}

Let us see an example of how to determine a realization graph. For each $n\in\mathbb{N}$, consider the graph of order $n+3$ obtained by identifying a vertex of $K_2$ with one of the leaves of $S_{n+2}$. The resulting tree, which we denote by $Y_{n+3}$, has degree sequence $s_n =(n+1)^1 2^1 1^{n+1}$. Thanks to the high symmetry of $Y_{n+3}$, it is easy in this case to see that:
\begin{enumerate}[(1).]
	\item there are exactly $n$ 2-switches active on $Y_{n+3}$; 
	\item $\tau(Y_{n+3})\approx Y_{n+3}$, for every 2-switch $\tau$;
	\item $|\mathcal{G}(s_n )|=n+1$.
\end{enumerate}
Therefore, $\mathcal{G}(s_n )\approx K_{n+1}$, for every $n\in\mathbb{N}$.\\

%

Let $s=(d_{v})_{v=1}^{n}$ be a graphical sequence. The sequence 
\[  \bar{s}=(n-1-d_{v})_{v=1}^{n} \]
is called the \textbf{dual sequence} of $s$. This sequence is also graphical, since, being $s=s(G)$ for some $G$, we have that $\bar{s}=s(\overline{G})$. Moreover, $\bar{\bar{s}}=s$. If $\mathcal{X}$ is an arbitrary family of graphs, then the set
\[ \overline{\mathcal{X}}=\{\overline{G}: G\in \mathcal{X}\} \]
is called the \textbf{dual set} of $\mathcal{X}$, and $\mathbf{T}(\overline{\mathcal{X}})$ is called the \textbf{dual space} of $\mathbf{T}(\mathcal{X})$. Clearly, $\overline{\overline{\mathcal{X}}} =\mathcal{X}$. Below, we cite and prove a remarkable fact about dual spaces. 

\begin{theorem}[\cite{arikati1999realization}, page 218]
	\label{dual.spaces.iso}
	Every transition space is isomorphic to its dual: 
	\[ \mathcal{X}\approx\overline{\mathcal{X}}. \]
\end{theorem}

\begin{proof}
	We need to exhibit a graph isomorphism between $\mathcal{X}$ and $\bar{\mathcal{X}}$. Let $\varphi: \mathcal{X} \rightarrow \overline{\mathcal{X}}$ be the function defined by $\varphi(X)=\overline{X}$ and let $G,H\in\mathcal{X}$, $G\neq H$. If $H\sim_{\tau} G$, then it is easy to see that
	\[ \overline{H}=\tau^{-1}(\overline{G}). \]
	Therefore, $\varphi(H)\sim_{\tau^{-1}}\varphi(G)$, which shows that $\varphi$ is a homomorphism. Clearly, $\varphi$ is bijective, and its inverse $\phi:\overline{\mathcal{X}}\rightarrow\mathcal{X}$ is the function defined by $\phi(Y)=\overline{Y}$. Let $A,B\in\mathcal{\overline{X}}$, with $A\neq B$. If $B\sim_{\theta} A$, we know that $\overline{B}=\theta^{-1}(\overline{A})$, i.e., $\phi(B)\sim_{\theta^{-1}} \phi(A)$. Thus, $\phi$ is also a homomorphism, and hence we conclude that $\varphi$ is an isomorphism.  
\end{proof}

In particular, Theorem \ref{dual.spaces.iso} tells us that 
\begin{equation}
	\label{gr.realiz.iso.dual}
	\mathcal{G}(s)\approx\mathcal{G}(\bar{s}),
\end{equation}
for every sequence $s$. Thanks to \eqref{gr.realiz.iso.dual}, we can then disregard investigating the structure of the realization graphs associated with the dual sequences of a given set of sequences. Moreover, as we will see later, \eqref{gr.realiz.iso.dual} allows us to immediately deduce results about the complement of a graph $G$.  \\
%

A \textbf{threshold} graph is a graph that can be constructed from $K_1$ by repeatedly applying either of the following two operations: 
\begin{enumerate}[(1).]
	\item addition of an isolated vertex;
	\item addition of a universal vertex.
\end{enumerate}
Thus, each unlabeled threshold graph of order $n$ can be uniquely identified with a binary sequence $0b_2b_3\ldots b_n$, with $b_i\in\{0,1\}$, where  $0=$``add an isolated vertex" and $1=$``add a universal vertex". For example: $K_1=0, K_2=01, \overline{K_2}=00=0^2, K_3=011=01^2, \overline{K_3}=000=0^3, S_3=001=0^21, S_4=0001=0^31$. It is easy to deduce that, in general, $K_n=01^{n-1}, \overline{K_n}=0^n$ and $S_n=0^{n-1}1$. Conversely, every binary sequence starting (on the left) with 0 defines a unique unlabeled threshold graph. In other words, there is a one-to-one correspondence between unlabeled threshold graphs and binary sequences whose first term is 0. Threshold graphs are characterized as those graphs $G$ that do not contain induced subgraphs isomorphic to $P_4 ,C_4$ or $2K_2$ (\cite{barrus.west.A4}, page 2). From this, it follows that: 
\begin{enumerate}[(1).]
	\item $G$ is threshold if and only if its complement is;
	\item every 2-switch is inactive in $G$;
	\item the property of being threshold is \textbf{hereditary}.
\end{enumerate}
Statement (1) is true because, in general, $H'\preceq H$ if and only if $\overline{H'}\preceq \overline{H}$, and, moreover, because $\overline{P_4}=P_4$ and $\overline{C_4}=2K_2$. Statement (2) is equivalent to $\mathcal{G}(s)=K_1$, if $s$ is the degree sequence of $G$. Finally, (3) means that every induced subgraph of a threshold graph is also threshold. \\

A graph $G\neq K_0$ is said to be \textbf{split} if there exists a clique $K$ in $G$ and an independent set $I$ in $G$ such that 
\[ V(G)=K\dot{\cup}I. \]
In such a case, we say that the pair $(K,I)$ is a \textbf{bipartition} for $G$. Directly from the definition, it follows that the property of being split is hereditary. We call \textbf{clique vertices} the elements of $K$, and \textbf{independent vertices} the elements of $I$. Two bipartitions $(K,I)$ and $(K',I')$ for $S$ are equal if $K=K'$ and $I=I'$. In Figure \ref{ejemplo.split} we can see an example of a split graph.
\begin{figure}[h]
	\label{ejemplo.split}
	\centering
	\begin{tikzpicture}[scale=2, every node/.style={scale=1.3, circle, draw, thick, minimum size=6pt, inner sep=2pt}]
		\foreach \i in {0,...,5} {
			\node[fill=black] (k\i) at ({cos(60*\i)}, {sin(60*\i)}) {};
		}
		
		\foreach \i in {0,...,5} {
			\foreach \j in {\i,...,5} {
				\ifnum\i<\j
				\draw (k\i) -- (k\j);
				\fi
			}
		}
		
		\node[fill=white] (s1) at (-2, 0) {};
		\node[fill=white] (s2) at (-2, -0.6) {};
		\node[fill=white] (s3) at (2, 0.6) {};
		\node[fill=white] (s4) at (2, -0.6) {};
		
		\draw (s3) -- (k1);
		\draw (s3) -- (k0);
		\draw (s4) -- (k0);
		\draw (s1) -- (k2);
		\draw (s1) -- (k3);
		\draw (s1) -- (k4);
		
			\node[fill=white] (k5b) at ({cos(60*5)}, {sin(60*5)}) {};
	\end{tikzpicture}
	\caption{Split graph (clique vertices in black, independent in white).}
\end{figure}
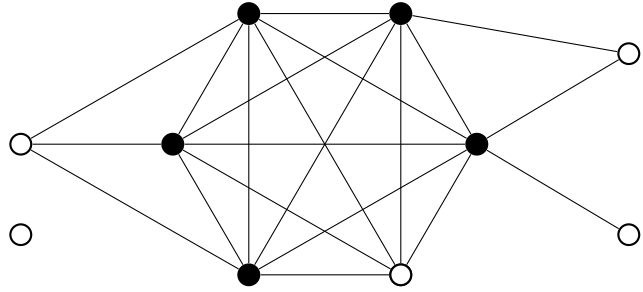

If $S$ is a split graph and $(K,I)$ is a bipartition for $S$, we shall call the triple $(S,K,I)$ a \textbf{splitting} of $S$. This notation, which is due to R. Tyshkevich (\cite{tyshkevich2000decomposition}, page 206), is necessary because there are specific operations/concepts on split graphs that depend on the bipartition. However, when the context allows, we will refer to the triple $(S,K,I)$ simply as ``split graph", to avoid overloading the language. A graph is split if and only if it does not contain induced subgraphs isomorphic to $2K_2, C_4$ or $C_5$ (\cite{barrus.west.A4}, page 14). In particular, threshold graphs are split graphs. Moreover, a graph is split if and only if its complement is. Clearly, $(K,I)$ is a bipartition for $S$ if and only if $(I,K)$ is a bipartition for $\overline{S}$. The complement of a splitting of $S$ is defined as:
\[ \overline{(S,K,I)}=(\overline{S},I,K). \]
If $S\in\mathcal{G}(s)$ and $S$ is split, then all members of $\mathcal{G}(s)$ are also split (\cite{tyshkevich2000decomposition}, page 205).

A split graph $(S,K,I)$ is said to be \textbf{balanced} if $|K|=\omega(S)$ and $|I|=\alpha(S)$. Otherwise, we say that $G$ is \textbf{unbalanced} (or not balanced). Since $\omega(S)=\alpha(\overline{S})$ and $\alpha(S)=\omega(\overline{S})$, it follows that $S$ is balanced if and only if its complement is.

\begin{proposition}
	\label{balanceado.iff.unica.bipart}
	A split graph is balanced if and only if it has a unique bipartition.
\end{proposition}

\begin{proof}
	$(\Rightarrow)$. Suppose that $S$ is balanced but has another bipartition $(K',I')\neq(K,I)$. Then, $|K'|=\omega(S)=|K|, |I'|=\alpha(S)=|I|$, but $K'\neq K$ or $I'\neq I$. 
	
	If $I'\neq I$, then $I-I'\neq\varnothing$. If $x\in I-I'$, it follows that $x\in K'$, whence $N_x\supseteq K'-x$. On the other hand, since $x\in I$, we have that $d_x\leq\omega(S)-1$. Thus, $N_x=K'-x$. But then $I'\cup x$ is an independent set in $S$ of size $\alpha(S)+1$, which is absurd.  
	
	If $K'\neq K$, we can recycle the same reasoning in $(\overline{S},I,K)$.
	
	$(\Leftarrow)$. Suppose that $S$ has a unique bipartition $(K,I)$ but is unbalanced. Then, we have that $|K|<\omega(S)=\omega$ or $|I|<\alpha(S)=\alpha$. 
	
	If $|K|<\omega$, then there exists a clique $K'$ in $S$ such that $|K'|=\omega$. Since $|K|<|K'|$, it is clear that $K'-K\neq\varnothing$. If $x\in K'-K$, then necessarily $x\in I$. Thus, $N_x\subset K$ (if $N_x=K$, then $(K\cup x,I-x)$ would be another bipartition for $S$!), whence $d_x<\omega-1$. On the other hand, since $x\in K'$, we have that $N_x\supseteq K'-x$, and hence $d_x\geq\omega-1$.
	
	If $|I|<\alpha$, we can recycle the same reasoning in $(\overline{S},I,K)$.
\end{proof}

A vertex $w$ of a split graph $(S,K,I)$ is said to be \textbf{swing} if 
\[ N_S (w)=K-w. \]
If $w\in I$, then $w$ is swing if and only if $\deg_S(w) =|K|$. By converting $w$ into a clique vertex, we obtain the split graph $(S,K\dot{\cup}w,I-w)$. If $w\in K$, then $w$ is swing if and only if $\deg_S(w) =|K|-1$. By converting $w$ into an independent vertex, we obtain the split graph $(S,K-w,I\dot{\cup}w,)$. Thus, thanks to Proposition \ref{balanceado.iff.unica.bipart}, it follows that a split graph is balanced if and only if it contains no swing vertices. Every threshold graph is unbalanced. Indeed, if $b_1\ldots b_{n}$ is the binary sequence associated with a threshold graph $G$, then the vertex $v_1$, corresponding to the term $b_1$, is swing in $G$, since it is adjacent to all non-isolated vertices. 

\begin{proposition}
	\label{2switch.preserva.balanceado}
	If $S$ is a balanced split graph, then $\tau(S)$ is also balanced, for every 2-switch $\tau$.
\end{proposition}

\begin{proof}
	We already know that $\tau(S)$ is split. Suppose that $\tau(S)\neq S$ has a swing vertex $w$. If $(K,I)$ is a bipartition for $\tau(S)$, we can assume without loss of generality that $w\in I$, from which $\deg_{\tau(S)}(w)=|K|$. Since the 2-switch preserves clique vertices and independent vertices, it follows that $(K,I)$ is the bipartition of $S$, since $S=\tau^{-1}\tau(S)$. In particular, we note that $w$ is independent in $S$. Since the 2-switch preserves the degree sequence, we have that $\deg_S(w)=|K|$. But then $w$ is swing in $S$, contradicting that $S$ is balanced.   
\end{proof}

Thanks to Proposition \ref{2switch.preserva.balanceado}, we have that, if $S\in\mathcal{G}(s)$ and $S$ is a balanced (unbalanced) split graph, then all members of $\mathcal{G}(s)$ are also balanced (unbalanced). The property of being balanced, as well as being unbalanced, is not hereditary.

Throughout this work, we will use the following standard to visually represent a split graph. The clique vertices will be colored black and arranged in a regular polygon. On the other hand, the independent vertices will be white and positioned outside this polygon. Additionally, edges between clique vertices will be omitted. Instead, the interior of the polygon will be shaded gray. This way, many edge crossings are avoided, favoring the clarity of the drawing. In Figure \ref{ejemplo.grafo.split} we can see on the right an example of an unlabeled split graph $S$, which is represented on the left in the form we have just described, and with labels. Note that $S$ is unbalanced, since vertex 6 is swing in $(S,[6],\{a,b,c,d\})$. The split graph $S+6x$, on the other hand, is balanced for every $x\in\{a,b,c,d\}$. 
\begin{figure}[h]
	\centering
	\includegraphics[scale=0.8]{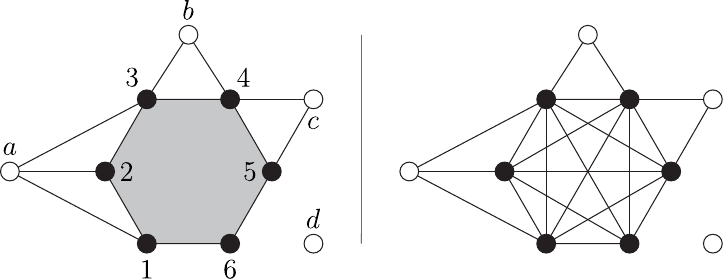}
	\caption{Standard representation of a split graph}
	\label{ejemplo.grafo.split}
\end{figure}

If $(S,K,I)$ is a split graph and $G$ is a graph disjoint from $S$, the \textbf{Tyshkevich composition} $S\circ G$ of $S$ and $G$ is defined as the graph whose vertex set is 
\[ V(S\circ G)=V(S)\cup V(G), \]
and whose edge set is 
\[ E(S\circ G)=E(S)\cup E(G)\cup\{xy:x\in K,y\in G\}. \]
In Figure \ref{ejemplo.composicion} we can see an example of this composition.
\begin{figure}[h]
	\label{ejemplo.composicion}
	\centering
	\begin{tikzpicture}[scale=1.2, every node/.style={circle, draw, minimum size=6pt, inner sep=2pt}]
		
		\node[fill=white, label=above left:$S$] (v1) at (0,0) {};    
		\node[fill=black] (v2) at (0,1) {};    
		\node[fill=black] (v3) at (1,1) {};    
		\node[fill=white] (v4) at (1,0) {};    
		
		\draw[thick] (v1) -- (v2) -- (v3) -- (v4);
		
		\node[fill=gray!60] (g1) at (3,0) {};
		\node[fill=gray!60, label=above right:$G$] (g2) at (4,0) {};
		\node[fill=gray!60] (g3) at (4,1) {};
		\node[fill=gray!60] (g4) at (3,1) {};
		
		\draw[thick] (g1) -- (g2) -- (g3) -- (g4) -- (g1);
		
		\foreach \k in {v2,v3} {
			\foreach \g in {g1,g2} {
				\draw [color=gray!50](\k) -- (\g);
			}
		}
		\draw[color=gray!50] (v3)--(g4);
		\draw[color=gray!50, bend left=15] (v3) to (g3);
		\draw[color=gray!50, bend left=20] (v2) to (g3);
		\draw[color=gray!50, bend left=20] (v2) to (g4);

		%
		
	\end{tikzpicture}
	\caption{Tyshkevich's composition of $S\approx P_4$ and $G\approx C_4$.}
\end{figure}
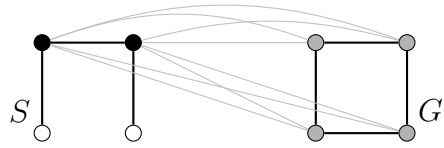

This operation depends on the bipartition of $S$: different bipartitions always lead to different compositions when $S$ is unbalanced. In Figure \ref{fig:compos.distintas.bipart} we can see a concrete example of this. The gray vertices are those of $G$, which is not split and is isomorphic to $2K_2$; the white vertices are those of $S$, which is an unbalanced split graph isomorphic to $P_3$. The gray edges are the edges added by the composition, connecting $G$ with the clique of $S$. On the left side of Figure \ref{fig:compos.distintas.bipart} the composition $S\circ G$ is performed using the bipartition $(67,5)$, while on the right side the same operation is carried out with the bipartition $(6,57)$.
\begin{figure}[H]
	\centering
	\begin{minipage}{0.45\textwidth}
		\centering
		\begin{tikzpicture}[
			vertex/.style = {circle, draw=black, thick, minimum size=3mm, inner sep=0pt},
			gvertex/.style = {vertex, fill=gray!30},
			svertex/.style = {vertex, fill=white},
			edge/.style = {draw=black, thick},
			edge_conn/.style = {draw=gray!50, thin}
			]
			\node[svertex, label=right:5] (5) at (0,2) {};
			\node[svertex, label=left:6] (6) at (-1,1) {};
			\node[svertex, label=right:7] (7) at (1,1) {};
			
			\node[gvertex, label=below:1] (1) at (-2,0) {};
			\node[gvertex, label=right:2] (2) at (-1,-1) {};
			\node[gvertex, label=left:3] (3) at (1,-1) {};
			\node[gvertex, label=below:4] (4) at (2,0) {};
			
			\begin{pgfonlayer}{background}
				\foreach \source in {6,7} {
					\foreach \target in {1,2,3,4} {
						\draw[edge_conn] (\source) -- (\target);
					}
				}
			\end{pgfonlayer}
			
			\draw[edge] (5) -- (6);
			\draw[edge] (6) -- (7);
			\draw[edge] (1) -- (2);
			\draw[edge] (3) -- (4);
		\end{tikzpicture}
	\end{minipage}
	\begin{minipage}{0.45\textwidth}
		\centering
		\begin{tikzpicture}[
			vertex/.style = {circle, draw=black, thick, minimum size=3mm, inner sep=0pt},
			gvertex/.style = {vertex, fill=gray!30},
			svertex/.style = {vertex, fill=white},
			edge/.style = {draw=black, thick},
			edge_conn/.style = {draw=gray!50, thin}
			]
			\node[svertex, label=right:5] (5) at (0,2) {};
			\node[svertex, label=left:6] (6) at (-1,1) {};
			\node[svertex, label=right:7] (7) at (1,1) {};
			
			\node[gvertex, label=below:1] (1) at (-2,0) {};
			\node[gvertex, label=right:2] (2) at (-1,-1) {};
			\node[gvertex, label=left:3] (3) at (1,-1) {};
			\node[gvertex, label=below:4] (4) at (2,0) {};
			
			\begin{pgfonlayer}{background}
				\foreach \target in {1,2,3,4} {
					\draw[edge_conn] (6) -- (\target);
				}
			\end{pgfonlayer}
			
			\draw[edge] (5) -- (6);
			\draw[edge] (6) -- (7);
			\draw[edge] (1) -- (2);
			\draw[edge] (3) -- (4);
		\end{tikzpicture}
	\end{minipage}
	\caption{The Tyshkevich composition depends on the bipartition.}
	\label{fig:compos.distintas.bipart}
\end{figure}
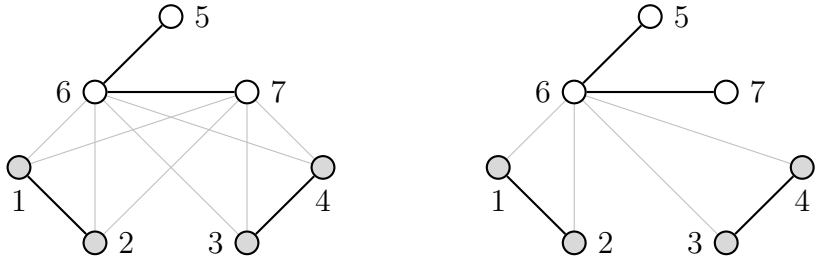

In general, $S\circ G$ is a split graph if and only if $G$ is. This implies that for every split graph $(S',K',I')$, $S'\circ S\circ G$ is well-defined and $(S'\circ S)\circ G=S'\circ(S\circ G)$. When $S$ is a balanced unlabeled split graph, we can use the notation $S^n$ to refer to the unlabeled graph resulting from the Tyshkevich composition of $n$ copies of $S$. Also, the word ``product" can be used instead of ``composition". Thus, depending on the context, the symbol $\prod_{i=1}^{n}G_i$ could replace $G_1\circ G_2\circ\ldots\circ G_n$.  

A graph $G$ is \textbf{decomposable} if there exist non-null graphs $S$ and $H$ such that $S\circ H=G$. Otherwise, $G$ is said to be \textbf{indecomposable}. With these concepts in mind, we can finally cite Tyshkevich's Decomposition Theorem, which, as already announced at the beginning, is key to the development of this entire work.

\begin{theorem}[Tyshkevich, \cite{tyshkevich2000decomposition}]
	\label{tyshk.decomp}
	Every graph $G$ can be written as a (Tyshkevich) composition of indecomposable graphs:
	\begin{equation*}
		G=G_n \circ \ldots\circ G_2 \circ G_1.
	\end{equation*}
	When each $G_r$ is non-null, this decomposition is unique up to isomorphism, where neither the order of the factors nor the choice of bipartitions can vary.
\end{theorem}

Finally, but no less important, we want to cite another previous result that is decisive for this research: the Barrus-West Theorem, which characterizes indecomposable graphs. For the understanding of its statement, it is sufficient to define the graph $A_4(G)$ associated with a graph $G$. The vertex set of $A_4(G)$ coincides with that of $G$, i.e., $V(A_4(G))=V(G)$. Moreover, two vertices $u,v$ are adjacent in $A_4(G)$ if and only if some $P_4, C_4$ or $2K_2$ induced in $G$ contains both $u$ and $v$. In other words, $uv\in E(A_4(G))$ if and only if there exists a 2-switch that activates both $u$ and $v$ in $G$. The notation ``$A_4$" derives from the expression ``alternating 4-cycle", by which Barrus and West refer to the configuration of the 4 edges involved in a 2-switch. Indeed, if in $G$ we replace $\{ab,cd\}$ with $\{ac,bd\}$ via a 2-switch, then the cycle $abdca$ alternates edges of $G$ with edges of $\overline{G}$. As we will see in Section \ref{sec:relacion.Phi-A_4}, the Barrus-West Theorem will strongly inspire the definition of the factor graph $\Phi(S)$ associated with a split graph $S$. 

\begin{theorem} [Barrus and West, \cite{barrus.west.A4}] 
	\label{indecomp.characterization} 
	A graph $G$ is indecomposable if and only if $A_4(G)$ is connected. More generally, the vertex sets of the factors of the Tyshkevich decomposition are the vertex sets of the components of $A_4(G)$.
\end{theorem}


\chapter{The degree of a graph} \label{cap:deg(G)}

In this chapter, we delve into the study of the 2-switch-degree (or simply, degree) of a graph $X$, which is nothing more than the degree of $X$ as a vertex of the realization graph $\mathcal{G}(X)$. Therefore, this is a fundamental concept that quantifies a graph's ability to transform into others via 2-switches. The degree of $X$ not only provides information about its flexibility in terms of 2-switches but is also closely related to other structural and combinatorial properties of $X$. Throughout the chapter, divided into 5 sections, we will explore the characteristics of active and inactive vertices, the basic properties of the degree, explicit formulas for its computation, and its behavior in specific families of graphs, such as trees and unicyclic graphs.

In Section \ref{sec:vertices.activos}, we introduce the concept of an active vertex, i.e., the vertices of a graph $G$ that participate in some 2-switch on $G$. In contrast, we define inactive vertices. We establish that the set of active vertices is invariant under 2-switches. This means that $G$ can be thought of as being divided into an ``active part" and a ``threshold part", and this partition of $V(G)$ is the same for all members of $\mathcal{G}(s)$, where $s=s(G)$. This remarkable fact is a consequence of Theorem \ref{indecomp.characterization} by Barrus and West, since the inactive vertices of $G$ are precisely the isolated vertices of $A_4(G)$, and the active vertices are the rest. However, we provide our proof in Appendix \ref{proof.lema.2switch.preserva.vert.activos}. Additionally, in this section, we analyze the relationship between active/inactive vertices and metric properties such as eccentricity, diameter, and connectivity. Finally, we characterize active vertices in split graphs and explore how this concept intertwines with universal vertices, swing vertices, and the property of being balanced. We emphasize that the introduction of the concepts of degree, active vertices, and their derivatives forms a very natural language for this context and ultimately determines all the important results throughout this research.

In Section \ref{sec:espacio.activo}, we focus on the subgraph $G^*$ of $G$ induced by its active vertices, which we refer to as the active part of $G$. Essentially, what we show is that to fully determine the structure of $\mathcal{G}(s)$ ($s=s(G)$), we can remove all inactive vertices from $ G $ and determine $\mathcal{G}(G^*)$, since both realization graphs turn out to be isomorphic. The graph $\mathcal{G}(G^*)$ is precisely what we call the active space associated with $s$. In other words, inactive vertices are redundant in this context, meaning they do not provide useful information for constructing $\mathcal{G}(s)$. 

In Section \ref{sec:prop.basicas.deg}, we present some fundamental properties of the 2-switch-degree. First, we show that computing the degree of a graph $G$ is essentially equivalent to finding the number of induced subgraphs of $G$ isomorphic to $P_4$, $C_4$, or $2K_2$. Hence, $G$ and $\overline{G}$ have the same degree. Subsequently, we see, among other things, that the 2-switch-degree ``respects" induced inclusion and that computing the degree of a disconnected graph practically reduces to computing the degree of each connected component. Finally, we obtain the two important results of this section: 1) the degree of the ``product" (Tyshkevich composition) is the sum of the degrees; 2) a product is active if and only if each factor is active.    

In Section \ref{sec:fórmulas.explicitas}, we derive explicit formulas to compute the degree of a graph in terms of its degree sequence and other combinatorial or structural invariants. Surprisingly, a connection is established between the 2-switch-degree and two well-studied invariants in Chemical Graph Theory: the Zagreb indices 1 and 2. Furthermore, we deduce that the difficulty of computing the degree of $G$ lies in counting the number of cliques of size 4 in $G$. On the other hand, it is observed that the degree of $G$ is ``easy" to obtain when $G$ has no cycles of order 3 or 4, and it only depends on the second Zagreb index.

Finally, in Section \ref{sec:deg.arboles.unicicl}, we focus on trees and unicyclic graphs, providing explicit formulas to compute their degree in $\mathcal{G}(s)$, $\mathcal{F}(s)$, and $\mathcal{U}(s)$. The latter two are the transition spaces associated, respectively, with forests and unicyclic graphs with degree sequence $s$. The concepts of f-degree and u-degree are introduced in this context. As already mentioned in Section \ref{sec:preliminares}, $\mathcal{F}(s)$ and $\mathcal{U}(s)$ are connected subspaces of $\mathcal{G}(s)$. A notable result of this section is that $\mathcal{F}(s)$ is regular, when $s$ is the sequence of a tree. However, $\mathcal{U}(s)$ is not regular in general, and a counterexample is provided.   


\section{Active vertices} \label{sec:vertices.activos}

Let $G$ be a graph with degree sequence $s$. The \textbf{degree} of $G$ (or \textbf{2-switch-degree} of $G$), denoted by $\deg(G)$, is the degree of $G$ viewed as a vertex of $\mathbf{T}(\mathcal{G}(s))$, i.e., the number of neighbors of the vertex $G$ in the realization graph associated with $s$. Therefore, $\deg(G)$ is the number of active 2-switches in $G$. 

Let $G$ be a graph, and let
\[ Q_G = \{ H \preceq G : |H| = 4 \}, \quad Q_G^* = \{ H \in Q_G : \deg(H) \neq 0 \}. \]
We say that a vertex $ v \in G $ is \textbf{active} in $ G $ if there exists an $ H \in Q_G^* $ containing $ v $. It is easy to deduce from Figure \ref{los.11.de.orden4} that 
\[ Q_G^* = \{ H \preceq G : H \approx X \in \{ P_4, C_4, 2K_2 \} \}. \]
We say that $ v $ is \textbf{inactive} in $ G $ when it is not active in $ G $, i.e., if $ \deg(H) = 0 $ for every $ H \in Q_G $ containing $ v $. We denote by $ act(G) $ the set of all active vertices of $ G $. In other words, a vertex $ v $ is active in $ G $ if there exists some 2-switch that activates it in $ G $. Clearly, $ V(G) - act(G) $ is the set of all inactive vertices of $ G $. 

It is important to note that a vertex $ v $ is active in $ G $ if and only if $ v $ is active in $ \overline{G} $. Indeed, if $ \tau $ is a 2-switch that activates $ v $ in $ G $, then $ \tau^{-1} $ activates $ v $ in $ \overline{G} $. Thus,
\[ act(G) = act(\overline{G}). \]
From this, it immediately follows that universal vertices are inactive, since they are isolated in the complement. 

\begin{lemma}
	\label{inac.G.inac.tauG}
	If $ a $ is an inactive vertex of a graph $ G $, then $ a $ is also inactive in $ \tau(G) $, for every 2-switch $ \tau $.
\end{lemma}

\begin{proof}
	The proof is given in Appendix \ref{proof.lema.2switch.preserva.vert.activos}.
\end{proof}

\begin{theorem}
	\label{2switch.preservs.act}
	If $ G, H \in \mathcal{G}(s) $, then $ act(G) = act(H) $. In other words, the 2-switch preserves active and inactive vertices.
\end{theorem}

\begin{proof}
	The result follows directly from Theorem \ref{berge's.theorem} and Lemma \ref{inac.G.inac.tauG}.
\end{proof}

Given $s = s(G)$, Theorem \ref{2switch.preservs.act} justifies the notation $act(s)$ as an alternative to $act(G)$.

\begin{theorem}
	\label{same.degree.act.inact}
	Let $G\in\mathcal{G}(s)$. If $a$ is an active (inactive) vertex of $G$, then all vertices of $G$ with the same degree (in $G$) as $a$ are active (inactive) in $G$.
\end{theorem}

\begin{proof}
	If $a\in act(G)$, suppose there exists a vertex $x\in G$ such that $d_a = d_x$ but $x\notin act(G)$. Now consider the graph $H$ obtained from $G$ by swapping the neighborhoods of $a$ and $x$: in $G$, every edge of the form $av$ ($v\neq x$) is replaced by $xv$, and every edge of the form $xw$ ($w \neq a$) is replaced by $aw$. Observe that $H\in \mathcal{G}(s)$, and thus, by Theorem \ref{berge's.theorem}, we can transform $G$ into $H$ via a sequence of 2-switches. By Theorem \ref{2switch.preservs.act}, we have $act(H) = act(G)$, so $a \in act(H)$ and $x \notin act(H)$. On the other hand, $H \approx_{\varphi} G$ via the transposition $\varphi = (a \, x)$. Since graph isomorphisms preserve active and inactive vertices, we obtain $x = \varphi(a) \in act(H)$ and $a = \varphi(x)\notin act(H)$, a contradiction.  
\end{proof}

We say that a graph $ G $ is \textbf{active} if
\[ act(G) = V(G). \]
Otherwise, we simply say that $ G $ is \textbf{not active} (or non-active). In particular, if $ act(G) = \varnothing \neq V(G) $, we say that $ G $ is \textbf{inactive}. It is clear that $ |G| \geq 4 $ when $ G $ is active. Some examples of active graphs are $ K_0 $, $ C_n $, and $ P_n $, for $ n \geq 4 $. Also, all disconnected graphs without isolated vertices are clearly active. Inactive graphs, on the other hand, are precisely the graphs of degree 0, i.e., threshold graphs. Thus, the property of being inactive is hereditary. Conversely, the property of being active is clearly non-hereditary, since $ K_1 \preceq G $ for any active graph $ G $, but $ act(K_1) \neq V(K_1) $. 

\begin{corollary}
	\label{inactive.implies.connected}
	Let $s$ be a positive sequence (i.e., every term of $s$ is a positive integer). If $\mathcal{G}(s)$ contains a graph that is not active, then all members of $\mathcal{G}(s)$ are connected. 
\end{corollary}   

\begin{proof}
	Suppose $ \mathcal{G}(s) $ contains a disconnected graph $ G $. Since $ s $ is positive, each component of $ G $ is non-trivial and, therefore, $ G $ is clearly active. Thus, by Theorem \ref{2switch.preservs.act}, we conclude that every member of $ \mathcal{G}(s) $ is active.
\end{proof}

The \textbf{eccentricity} $ ecc_G(x) $ of a vertex $ x $ in a graph $ G $ is a measure describing the greatest distance between $ x $ and any other vertex in $ G $. Formally,
\[ ecc_G(x) = \max\{ dist_G(v, x) : v \in V(G) \}. \]
In the following corollary, we show that in a graph without isolated vertices, inactive vertices have eccentricity 1 or 2.

\begin{proposition}
	\label{eccen.inact.vertex}
	Let $ G $ be a graph without isolated vertices. If $ x \notin act(G) $, then $ dist_G(x, v) \leq 2 $, for every $ v \in V(G) $. In other words:
	\[ ecc_G(x) \in \{1, 2\}. \]
\end{proposition}

\begin{proof}
	If $ v $ is a vertex at distance $ d $ from $ x $ (in $ G $), we can find an induced path of length $ d $ connecting $ x $ to $ v $ in $ G $. If $ d \geq 3 $, then there exist vertices $ y $ and $ z $ such that the 2-switch $ \binom{x \, y}{z \, v} $ is active in $ G $. This means $ x \in act(G) $.
\end{proof}

\begin{corollary}
	\label{inactive.implies.diam<=3}
	Let $ G $ be a graph without isolated vertices. If $ G $ is not active, then $ diam(G) \leq 3 $. Moreover, this bound is sharp. 
\end{corollary}  

\begin{proof}
	By Corollary \ref{inactive.implies.connected}, we know $ G $ is connected. Consider a vertex $ x \notin act(G) $. By the triangle inequality and Corollary \ref{eccen.inact.vertex}, for any pair of vertices $ v, w \in G $, we have (in $ G $) that  
	\begin{equation*}
		dist(v, w) \leq dist(v, x) + dist(x, w) \leq 2 + 2 = 4.
	\end{equation*}
	By Proposition \ref{eccen.inact.vertex}, equality holds if and only if $ dist(v_0, x) = dist(x, w_0) = 2 $ for some $ v_0, w_0 \in V(G) $. But in this case, $ G $ would contain an induced path of the form $v_0yxzw_0 \approx P_5$, for some $y,z\in V(G)$, forcing $ x $ to be active. Therefore, $ dist(v, w) < 4 $ for every pair $ v, w \in V(G) $, and thus $ diam(G) \leq 3 $. Finally, to see that the bound is sharp, consider the degree sequence $ 3^2 2^1 1^2 $. 
\end{proof}

The converses of Corollaries \ref{inactive.implies.connected} and \ref{inactive.implies.diam<=3} are not true: $ P_4 $ is a connected graph with diameter 3 where all its vertices are active. Even if the diameter of a graph $ G $ were 2, we cannot guarantee in general that $ G $ has any inactive vertices; for example, $ C_4 $ is an active graph with diameter 2. It is useful to rephrase Corollary \ref{inactive.implies.diam<=3} in the following equivalent way.

\begin{corollary}
	\label{inactive.implies.diam<=3.contrarr}
	Let $ s $ be a positive sequence. If $ \mathcal{G}(s) $ contains a graph of diameter $ \geq 4 $, then all members of $ \mathcal{G}(s) $ are active.
\end{corollary} 

Since split graphs will soon become central to this work, it is important to establish from the outset the properties of their active and inactive vertices. Before proceeding, note that a vertex $ v $ is active in a split graph $ (S, K, I) $ if and only if there exists an $ H \preceq S $ such that $ v \in H $ and $ H \approx P_4 $. This is because split graphs do not contain induced subgraphs isomorphic to $ C_4 $ or $ 2K_2 $. If $ v \in I $, such an $ H $ will be of the form $vabc$, where $ a, b \in K $ and $ c \in I $. If instead $ v \in K $, $ H $ can be written as $avbc$, where $ a, c \in I $ and $ b \in K $.

\begin{proposition}
	\label{split.caract.vert.activos}
	If $ (S, K, I) $ is a split graph, then the following holds:
	\begin{enumerate}
		\item $ U = \bigcap_{v \in I} N_S(v) $ is the set of universal vertices of $ S $;
		\item $ u \in I $ is active in $ S $ if and only if there exists an $ x \in I $ such that $ N_S(x) - N_S(u) \neq \varnothing $ and $ N_S(u) - N_S(x) \neq \varnothing $;
		\item $ u \in I $ is inactive in $ S $ if and only if, for every $ v \in I $, either $ N_S(v) \subseteq N_S(u) $ or $ N_S(u) \subseteq N_S(v) $;
		\item $ u \in K $ is inactive in $ S $ if and only if, for every $ v \in K $, either $ N_S(v) \cap I \subseteq N_S(u) \cap I $ or $ N_S(u) \cap I \subseteq N_S(v) \cap I $; 
		\item A vertex of $ S $ is inactive in $ S $ if and only if its neighborhood is comparable by inclusion with all other neighborhoods in its partition;
		\item If $ w $ is a swing vertex in the split graph $ (S, K, I) $, then $ w \notin act(S) $.
		\item If $ S $ is active, then $ S $ is balanced.
	\end{enumerate}
\end{proposition}

\begin{proof}
	\begin{enumerate}[(1).]
		\item A vertex $ u $ is universal in $ S $ if and only if it lies in the neighborhood of every vertex $ v $ in $ I $. 
		\item If there exist $ a \in N_x - N_u $ and $ b \in N_u - N_x $, then $ubax\preceq S$, showing that $ u \in act(S) $. Conversely, if $ u \in act(S) $, then there exist $ a, b \in K $ and $ x \in I $ such that $ubax\preceq S$, so $ a \notin N_u $ and $ b \notin N_x $. Thus, both $ N_u - N_x $ and $ N_x - N_u $ are non-empty.
		\item This is simply the contrapositive of (2).
		\item If $ u $ is an inactive clique vertex in $ S $, then, in $\overline{S}$, the same vertex is also inactive but independent. Thus, by (3), we have $ N_{\overline{S}}(v) \subseteq N_{\overline{S}}(u) $ or $ N_{\overline{S}}(u) \subseteq N_{\overline{S}}(v) $ for every $ v \in K $. Equivalently: $ N_v^c - v \subseteq N_u^c - u $ or $ N_u^c - u \subseteq N_v^c - v $, i.e., $ N_v \cup v \supseteq N_u \cup u $ or $ N_u \cup u \supseteq N_v \cup v $. Since $ N_x = (N_x \cap I) \cup (K - x) $ for every $ x \in K $, it follows that $ (N_v \cap I) \cup K \supseteq (N_u \cap I) \cup K $ or $ (N_u \cap I) \cup K \supseteq (N_v \cap I) \cup K $. Therefore, we conclude that $ N_v \cap I \supseteq N_u \cap I $ or $ N_u \cap I \supseteq N_v \cap I $.
		\item Combining (3) and (4) yields the result. 
		\item If $ w \in I $, then $ N_w = K \supseteq N_x $ for every $ x \in I $. If $ w \in K $, then $ N_w = K - w $, so $ N_w \cap I = \varnothing \subseteq N_x \cap I $ for every $ x \in K $. In both cases, (5) holds, so $ w \notin act(S) $.  
		\item If $ S $ is not balanced, then $ S $ has a swing vertex, which is inactive by (6). Thus, $ act(S) \neq V(S) $.
	\end{enumerate}
\end{proof} 

It is important to note that the converse of statement (7) in Theorem \ref{split.caract.vert.activos} is not true in general, as there exist balanced non-active graphs. Indeed, let $ S $ be the split graph with bipartition $ (K, I) = ([3], \{a, b, c\}) $, where $ N_a = \{1, 3\} $, $ N_b = \{2, 3\} $, and $ N_c = \{3\} $. It is easy to see that $ \omega(S) = |K| $ and $ \alpha(S) = |I| $, showing that $ S $ is balanced. However, the vertices $ c $ and $ 3 $ are inactive in $ S $.

\begin{proposition}
	\label{S.balanc.implica.deg>0,K=Union(N_v),|U|<|K|-1}
	Let $ (S, K, I) $ be a balanced split graph. Then, the following holds:
	\begin{enumerate}
		\item $ \deg(S) \geq 1 $, and $ |K|, |I| \geq 2 $; 
		\item $ K = \bigcup_{v \in I} N_S(v) $; 
		\item $ |U| \leq |K| - 2 $.
	\end{enumerate}
\end{proposition}

\begin{proof}
	\begin{enumerate}[(1).]
		\item If $ \deg(S) = 0 $, then $ S $ would be threshold and thus unbalanced. Hence, $ \deg(S) \geq 1 $, which implies that $ S $ contains an induced $ P_4 $. Consequently, $ |I|, |K| \geq 2 $. 
		\item Clearly, $ \bigcup_{v \in I} N_v \subseteq K $. If the inclusion were strict, there would be a vertex $ w \in K $ with no neighbors in $ I $. But then $ w $ would be swing, and thus $ S $ would not be balanced. 
		\item Since $ S $ is balanced, $ d_v < |K| $ for every $ v \in I $. Thus, $ |U| < |K| $ by (2). To complete the proof, it suffices to show that $ |U| \neq |K| - 1 $. Suppose $ |U| = |K| - 1 $. Let $ x $ be the only non-universal clique vertex in $ S $. By (1), there exists $ w \in N_x \cap I $. Then, $ \overline{S} $ has $ |K| - 1 $ isolated vertices and only one non-isolated independent vertex, $ x $. Since $ w x \notin \overline{S} $, $ w $ is swing in $ \overline{S} $. But then $ S $ is unbalanced, a contradiction.
	\end{enumerate}
\end{proof} 

The following theorem is interesting because it establishes that all connected regular graphs are active, except for complete graphs, which are inactive.

\begin{theorem}
	\label{regular.graphs}
	Let $ G $ be a connected regular graph. If $ G $ is not complete, then $ G $ is active. 
\end{theorem}

\begin{proof}
	Let $ G $ be a $ k $-regular graph. If $ k \leq 2 $, it is easy to verify that the statement holds. If $ k \geq 3 $, suppose $ G $ is neither complete nor active. Then, $ 3 \leq k \leq n - 2 $, and $ G $ is inactive by Theorem \ref{same.degree.act.inact}. Moreover, $ diam(G) = 2 $ since $ G $ is not complete and contains no induced $ P_4 $. Now, consider an induced path $abc$ of length 2 in $ G $. Since $d_a = k \geq 3$, there are $ k - 1 $ vertices $ v_i $ in the neighborhood of $ a $. Note that $ v_i \neq c $ for all $ i \in [k - 1] $, because $ dist(a, c) = 2 $. Since $ \deg(G) = 0 $, for each $ i $, the edge $ v_i b $ must be in $ G $. Thus, $ d_b \geq k + 1 $, a contradiction.
\end{proof}


\section{Active space} \label{sec:espacio.activo}

The \textbf{active part} of $G$, denoted by $G^*$, is the subgraph of $G$ induced by $act(G)$. By definition, $G^*$ is active. Thus, $G$ is active if and only if $G^*=G$. If $s=s(G)$, we refer to the degree sequence of $G^*$ as $s^*$.

\begin{lemma}
	\label{tau(G-S)=tau(G)-S}
	Let $G$ be a graph and let $V_{\theta}$ be the set of all vertices activated by some 2-switch in the sequence $\theta =(\tau_i )_{i=1}^k$. If $W\subseteq V(G)$ and $W\cap V_{\theta}=\varnothing$, then
	\[ \theta(G-W)=\theta(G)-W. \]
\end{lemma}

\begin{proof}
	Let $k=1$, that is, $\theta={{a \ b}\choose{c \ d}}$. Since $W\cap V_{\theta}=\varnothing$, it is clear that $\theta$ is active in $G$ if and only if it is active in $G-W$. If $\theta$ is active in $G-W$, then we have
	\[ \theta(G-W)=((G-W)-\{ab,cd\})+\{ac,bd\}. \]
	Since $V_{\theta}=\{a,b,c,d\}$ and $W\cap V_{\theta}=\varnothing$, it follows that
	\[ (G-W)-\{ab,cd\}=(G-\{ab,cd\})-W. \]
	For the same reason, it is easy to see that
	\[ ((G-\{ab,cd\})-W)+\{ac,bd\}= \]
	\[ ((G-\{ab,cd\})+\{ac,bd\})-W=\theta(G)-W. \]
	The rest of the proof is carried out easily by induction on $k$, since it suffices to recycle the case $k=1$ by replacing $G$ with $\tau_1\ldots\tau_{k-1}(G)$ and $\theta$ with $\tau_k$.
\end{proof}

Let $X=(V,E)$ be a graph. For any vertex set $V'$ disjoint from $V$ and any 
\[ E'\subseteq\binom{V\cup V'}{2}-\binom{V}{2}=\binom{V'}{2}\dot{\cup}\{xy:x\in V,y\in V'\}, \]
we define the $(V',E')$\textbf{-extension} of $X$ as the graph 
\[ ext(X,V',E')=(V\cup V',E\cup E'). \]
Clearly, $X\preceq ext(X,V',E')$. 

\begin{lemma}
	\label{ext(theta)=theta(ext)}
	Let $G=(V,E)$ be a graph and let $\theta$ be a sequence of 2-switches. If $W=V-act(G)$ and $E'=\{xy\in E:x\in W\}$, then
	\[ ext(\theta(G^*),W,E')=\theta(ext(G^*,W,E')). \]
\end{lemma}

\begin{proof}
	The first thing to note is that $G^*=G-W$. Hence, it is clear that $ext(G^*,W,E')=G$. Secondly, since each 2-switch in $\theta$ only deletes and adds edges of the form $xy$, with $x,y\in act(G)$, it follows that
	\[ E'=\{xy\in E(\theta(G)):x\in W\} \]
	(remember that $act(G)$ is an invariant associated with $s(G)$, by Theorem \ref{2switch.preservs.act}). With this in mind, and applying Lemma \ref{tau(G-S)=tau(G)-S}, we obtain
	\[ ext(\theta(G^*),W,E')=ext(\theta(G)-W,W,E')= \]
	\[ (V,E(\theta(G)-W)\dot{\cup}E'). \]
	Since
	\[ E(\theta(G)-W)=E(\theta(G))-\{xy\in E(\theta(G)): x\in W\}, \]
	we conclude that $ext(\theta(G^*),W,E')=\theta(G)$.
\end{proof}

If $\mathcal{X}$ is a subset of $\mathcal{G}(s)$, we define
\[ \mathcal{X}^* =\{G^* :G\in\mathcal{X}\}. \]
Let $\psi :\mathcal{X}\rightarrow\mathcal{X}^*$ be the function defined by $\psi(X)= X^*$. Choose arbitrarily a graph $X_0\in\mathcal{G}(s)$ and define
\[ E_0=\{xy\in E(X_0):x\in W\}, \]
where $W=V(X_0)-act(s)$. Thanks to Theorem \ref{2switch.preservs.act}, it is easy to verify that $E_0$ is an invariant associated with $s$. In other words, replacing $X_0$ with any other member of $\mathcal{G}(s)$, the set $E_0$ does not change. Now suppose that $G^* =H^*$. Then,
\[ G=ext(G^*,W,E_0)=ext(H^*,W,E_0)=H, \]
which shows that $\psi$ is injective. Since $\mathcal{X}^*=\psi(\mathcal{X})$, $\psi$ is obviously surjective. Therefore, $\psi$ is a bijection. 

Let $G^*,H^*\in\mathcal{X}^*$, $G^*\neq H^*$. According to Theorem \ref{berge's.theorem}, there exists a sequence $\theta$ of 2-switches such that $H=\theta(G)$, since $G,H\in\mathcal{G}(s)$ and $G\neq H$. Then, applying Theorem \ref{2switch.preservs.act} and Lemma \ref{tau(G-S)=tau(G)-S} for $W=V(G)-act(s)$, we deduce that 
\[ \theta(G^*)=\theta(G-W)=\theta(G)-W=H-W=H^*, \]
which means that $G^*,H^*\in\mathcal{G}(s^*)$. In other words, $\mathcal{X}^*$ is a subset of $\mathcal{G}(s^*)$. 

The transition space $\mathbf{T}(\mathcal{X}^*)$ is called the \textbf{active space} associated with $\mathcal{X}$ (or with $\mathbf{T}(\mathcal{X})$). Thanks to Lemma \ref{tau(G-S)=tau(G)-S} and Theorem \ref{2switch.preservs.act}, it is clear that $\psi$ is a homomorphism between the transition spaces $\mathcal{X}$ and $\mathcal{X}^*$. On the other hand, observe that the function $\zeta:\mathcal{X}^*\rightarrow\mathcal{X}$, defined by $\zeta(Y)=ext(Y,W,E_0)$, is the inverse of $\psi$, and it is also a homomorphism by Lemma \ref{ext(theta)=theta(ext)}. Then, $\psi$ is an isomorphism. 

\begin{theorem}
	\label{isomorfismo.espacio.activo}
	Every transition subspace $\mathcal{X}$ of $\mathcal{G}(s)$ is isomorphic to its associated active space:
	\begin{equation*}
		\mathcal{X}\approx\mathcal{X}^* .
	\end{equation*}
\end{theorem}

\begin{proof} 
	It follows from the previous discussion.
\end{proof}


\section{Basic properties of the degree} \label{sec:prop.basicas.deg}

At this point, it should come as no surprise that the number of 2-switches acting on a certain graph $G$ can be obtained by summing the degree of all induced subgraphs of order 4 in $G$. In other words,
\begin{equation*}
	\deg(G)=\sum_{H\in Q_G}\deg(H),
\end{equation*}
where $\deg(H)=2$ if $H\approx 2K_2$ or $C_4$, $\deg(H)=1$ if $H\approx P_4$, and $\deg(H)=0$ in the remaining cases (see Figure \ref{los.11.de.orden4}). If $X$ is one of the 11 graphs in Figure \ref{los.11.de.orden4}, we define
\[ Q_G (X)=\{H\in Q_G :H\approx X\}. \]
Then, we have the following result.

\begin{theorem}
	\label{degreeofG}
	For every graph $G$,
	\begin{equation}
		\label{eq18}
		\deg(G)=2|Q_G (2K_2 )|+2|Q_G (C_4 )|+|Q_G (P_4 )|.
	\end{equation}
	In particular, if $G$ is a split graph, then $\deg(G)=|Q_G (P_4 )|$.
\end{theorem}

\begin{proof}
	Equality \eqref{eq18} follows from the previous discussion. When $G$ is split, we know it does not contain induced subgraphs isomorphic to $C_4$ or $2K_2$. Therefore, $|Q_G (2K_2 )|=|Q_G (C_4 )|=0$.
\end{proof}

From \eqref{eq18} we observe that
\[ \deg(G)\equiv|Q_G(P_4)|\pmod{2}, \]
that is, the degree of $G$ has the same parity as the number of induced $P_4$'s in $G$. In particular, every graph of odd degree contains at least one induced $P_4$.

\begin{corollary}
	\label{degG=deg(G.complemento)}
	For any graph $G$,
	\begin{equation*}
		\deg(G)=\deg(\overline{G}).
	\end{equation*}
\end{corollary}

\begin{proof}
	Let $G\in\mathcal{G}(s)$. Then, $\overline{G} \in\mathcal{G}(\bar{s})$. By Theorem \ref{dual.spaces.iso}, we have $\mathcal{G}(s)\approx\mathcal{G}(\bar{s})$ via the isomorphism $G\mapsto\overline{G}$, and therefore $\deg(G)=\deg(\overline{G})$.
\end{proof}

It is worth mentioning that Corollary \ref{degG=deg(G.complemento)} can also be proven using the fact that $H\preceq G$ if and only if $\overline{H}\preceq\overline{G}$. This, together with the fact that $\overline{P_4}=P_4$ and $\overline{C_4}=2K_2$, implies that $|Q_G(P_4)|=|Q_{\overline{G}}(P_4)|$, $|Q_G(C_4)|=|Q_{\overline{G}}(2K_2)|$, and $|Q_G(2K_2)|=|Q_{\overline{G}}(C_4)|$.

\begin{proposition}
	\label{deg.respects.ind.inc}
	If $H\preceq G$, then 
	\begin{equation*}
		\deg(H)\leq\deg(G).
	\end{equation*}
\end{proposition}

\begin{proof}
	This is true because $Q^*_H \subseteq Q^*_G$.
\end{proof}

Proposition \ref{deg.respects.ind.inc} is not generally true if the subgraph is not induced. For example, $P_4 \subseteq K_4$ but $\deg(P_4)=1$ and $\deg(K_4)=0$. Proposition \ref{deg.respects.ind.inc} can be easily generalized as follows.

\begin{proposition}
	\label{collH.ind.sub.G.degG>=sum.degH}
	Let $\{ H_i :i\in[k]\}$ be a collection of induced subgraphs of $G$ such that $|E(H_i )\cap E(H_j )|\leq 1$ for $i\neq j$. Then, 
	\begin{equation*}
		\sum_{i=1}^{k}\deg(H_i )\leq\deg(G).
	\end{equation*}
\end{proposition}

\begin{proof}
	If $\tau$ were an active 2-switch on both $H_i$ and $H_j$ for certain distinct $i,j$, then we would have $|E(H_i )\cap E(H_j )|\geq 2$, contradicting the hypothesis. Therefore, each active 2-switch in $H_i$ is inactive in $H_j$, for every $j\neq i$.
\end{proof}

\begin{proposition}
	If $G'\preceq G^{*}$ and $\deg(G')=\deg(G)>0$, then $G'=G^{*}$. In other words, $G^{*}$ is the smallest induced subgraph (in terms of order) of $G$ with the same degree as $G$.
\end{proposition}

\begin{proof}
	Suppose that $G'\preceq G^{*}$ and $\deg(G')=\deg(G)\geq 1$, but $G'\neq G^{*}$. If $v\in V(G^{*})-V(G')$, then there is at least one 2-switch $\tau$ that acts on $G$ activating $v$. However, $\tau$ is clearly inactive on $G'$, and therefore $\deg(G')\leq\deg(G)-1$, which is a contradiction.
\end{proof}

The next result states that computing the degree of a graph essentially reduces to computing the degree of its components.

\begin{theorem}
	\label{degree.disconn.G}
	Let $G$ be a graph such that $\kappa(G)=k\geq 2$. If $G_{i}$ is a component of $G$, then 
	\begin{equation}
		\deg(G)=\sum_{i=1}^{k}\deg(G_{i}) +\sum_{1\leq i<j\leq k}2\rVert G_{i}\rVert \rVert G_{j}\rVert.	
	\end{equation}	
\end{theorem}

\begin{proof}
	Suppose $k=2$. Since $G_1$ and $G_2$ are disjoint induced subgraphs of $G$, the degree of $G$ is at least $\deg(G_1)+\deg(G_2)$, by Proposition \ref{collH.ind.sub.G.degG>=sum.degH}. The remaining 2-switches that transform $G$ are those that replace edges from different components. More specifically: if $ab\in G_1$ and $cd\in G_2$, then $\begin{pmatrix} a & b \\ c & d \end{pmatrix}$ and $\begin{pmatrix} a & b \\ d & c \end{pmatrix}$ are inactive in $G$. Therefore, $\deg(G)=\deg(G_1)+\deg(G_2) +2\Vert G_1 \Vert\Vert G_2 \Vert$. The rest of the proof follows easily by induction on $k$.
\end{proof}

If $G_1 =(V_1 ,E_1 )$ and $G_2 =(V_2 ,E_2 )$ are disjoint, then we denote by $G_1 +G_2$ the graph whose vertex set is $V_1 \cup V_2$ and whose edge set is
\[ E_1 \cup E_2 \cup\{xy:x\in G_1 ,y\in G_2 \}. \]
This binary operation is well known and is called the \textbf{join} of $G_1$ and $G_2$. If $V_2 =\{v\}$ (i.e., $G_2 \approx K_1$), we can write $G_1 +G_2$ as $G_1 +v$. The join operation is clearly commutative and associative. Also, $G_1, G_2 \preceq G_1+G_2$.

\begin{lemma}
	\label{deg(G+Kn)=degG}
	For any graph $G$ and any $n\geq 1$,
	\begin{equation}
		\deg(G +K_n )=\deg(G).
	\end{equation}
\end{lemma}

\begin{proof}
	Since $G$ and $K_n$ are disjoint induced subgraphs of $G+K_n$, we have that $\deg(G)+\deg(K_n )=\deg(G)\leq\deg(G+K_n)$, by Proposition \ref{collH.ind.sub.G.degG>=sum.degH}. If $\tau$ is a 2-switch between $ab\in G$ and $cd\in K_n$, then $\tau$ is inactive in $G+K_n$ because $a,b,c$, and $d$ form a clique there. If $a,b\in V(G)$ and $c,d\in V(K_n)$, then every 2-switch involving these 4 vertices is inactive in $G+K_n$, because $a,b,c,d$ induce a $K_4$ or a $D_4$ (see Figure \ref{los.11.de.orden4}) in $G+K_n$.
\end{proof}

By taking $G+K_n$, we can thus obtain a graph with a certain fixed degree (that of $G$), but with arbitrarily large order and size.

\begin{theorem}
	\label{deg(SoG)=deg(S)+deg(G)}
	If $(S,I,K)$ is a split graph and $G$ is a graph, then 
	\begin{equation}
		\label{eq17}
		\deg(S\circ G)=\deg(S)+\deg(G).
	\end{equation}
\end{theorem}

\begin{proof}
	Since $S$ and $G$ are disjoint induced subgraphs of $S\circ G$, we have that $\deg(S\circ G)\geq \deg(S)+\deg(G)$, by Proposition \ref{collH.ind.sub.G.degG>=sum.degH}. Thanks to Lemma \ref{deg(G+Kn)=degG}, to obtain the equality \eqref{eq17} we only need to focus on an arbitrary pair of disjoint edges of the form $\{ik,ab\}$, where $\{ik,ab\}\subseteq E(S\circ G)$, $i\in I,k\in K$ and $ab\notin E(S)$. Since $H=\langle a,b,i,k\rangle_{S\circ G}$ contains the triangle $kabk$, we have that $\deg(H)=0$. Thus, equality \eqref{eq17} is finally proven.
\end{proof}

\begin{theorem}
	\label{SoG.act.iff.S,G.act}
	$S\circ G$ is active if and only if $S$ and $G$ are active.
\end{theorem}

\begin{proof}
	For simplicity, let $H=S\circ G$.
	
	($\Leftarrow$) Since $S,G\preceq H$, it is clear that every active vertex in $S$ or in $G$ remains active in $H$. Since $V(H)=V(S)\cup V(G)$ by the definition of $\circ$, we conclude that $H$ is active.
	
	($\Rightarrow$) If $x\in V(S)$ is active in $H$, then by applying Theorem \ref{deg(SoG)=deg(S)+deg(G)} we obtain that 
	\[ \deg(S)+\deg(G)=\deg(H)>\deg(H-x)=\deg(S-x)+\deg(G). \]
	Hence, $\deg(S)>\deg(S-x)$, which shows that $x$ is active in $S$. If $x\in V(G)$ is active in $H$, we repeat the argument and complete the proof.
\end{proof}


\section{Explicit formulas} \label{sec:fórmulas.explicitas}

If $G$ is a graph, we define
\[ dpe(G)=|\{ \{e,f\}: e,f\in E(G), e\cap f=\varnothing \}|. \]
In other words, $dpe(G)$ is the number of unordered pairs of disjoint edges in $G$. The abbreviation “$dpe$” stands for “\textbf{disjoint pairs of edges}”.

\begin{proposition}
	\label{dpe.fórmula}
	If $G$ is a graph with degree sequence $s=(d_{v})_{v=1}^{n}$, then:
	\begin{equation*}
		dpe(G)=\binom{\Vert G\Vert}{2}-\sum_{v=1}^{n} \binom{d_{v}}{2}=\binom{\Vert G\Vert +1}{2}-\frac{1}{2}|s|^{2},
	\end{equation*}
	where $|s|^2=d_1^2+\ldots+d_n^2$ and $\binom{i}{j}=0$ if $i<j$. In particular, $dpe(G)$ is constant on $\mathcal{G}(s)$, i.e., it is an invariant associated with $s$.
\end{proposition}

\begin{proof}
	We can obtain $dpe(G)$ by computing the number of pairs of non-disjoint edges in $G$ and then subtracting this number from $\binom{\Vert G\Vert}{2}$, i.e., the total number of edge pairs in $G$. Since there are $d_{v}$ edges incident to each $v\in V(G)$, we get $\binom{d_{v}}{2}$ pairs of edges sharing $v$, and hence $\sum_{v=1}^{n} \binom{d_{v}}{2}$ edge pairs with a common vertex.
	
	The second equality follows from expanding the binomial coefficients and applying simple manipulations.
\end{proof}

Motivated by Proposition \ref{dpe.fórmula}, we may use the notation $dpe(s)$ instead of $dpe(G)$ whenever $s=s(G)$. Moreover, note that $dpe(G)=|\{H\subseteq G:H\approx 2K_2 \}|$. The next proposition gives a formula expressing the $dpe$ of a graph in terms of the $dpe$ of its connected components.

\begin{proposition}
	\label{dpe.components}
	If $G$ is a graph with $k$ components $G_i$, then
	\begin{equation*}
		dpe(G)=\sum_{i=1}^{k}dpe(G_i )+\sum_{1\leq i<j\leq k}\Vert G_i \Vert\Vert G_j \Vert .
	\end{equation*}
\end{proposition}

\begin{proof}
	Let $k=2$, i.e., $G=G_1\dot{\cup}G_2$. If $e_1\in E(G_1)$ and $e_2\in E(G_2)$, clearly $e_1\cap e_2=\varnothing$ and there are $\Vert G_2\Vert$ pairs of the form $\{e,e_2\}$ for each $e\in E(G_1)$. Then,
	\[ dpe(G)=dpe(G_1)+dpe(G_2)+\Vert G_1\Vert\Vert G_2\Vert. \]
	The rest follows easily by induction on $k$.
\end{proof}

\begin{proposition}
	For any graph $G$ of order $n$ and size $m\geq 0$,
	\begin{equation*}
		\deg(G)\leq 2dpe(G)\leq m(m-1),
	\end{equation*}
	and equality holds if and only if $G\approx (mK_2) \dot{\cup}\overline{K}_{n-2m}$.
\end{proposition}

\begin{proof}
	For each pair of disjoint edges in $G$, there are at most two 2-switches that replace them. Hence,
	\begin{equation*}
		\deg(G)\leq 2dpe(G)=m(m-1)-\sum_{v=1}^{n}(d_v^{2}-d_v )\leq m(m-1).
	\end{equation*}
	Clearly, equality is achieved precisely when $d_v \in\{0,1\}$ for each vertex $v\in G$. Under this condition, we must have $d_u =d_v =1$ for every $uv\in E(G)$, and the remaining $n-2m$ vertices of $G$ are isolated.
\end{proof}

\begin{proposition}
	\label{number.of.P4.in.G}
	
	Let $G$ be a graph. If $p_4 (G)=|\{H\subseteq G: H\approx P_4\}|$ and $k_3 (G)=|\{H\subseteq G: H\approx K_3\}|$, then
	\begin{equation}
		\label{eq43}
		p_4 (G)+3k_3 (G)=\sum_{uv\in E(G)}(d_{u}-1)(d_{v}-1) .	
	\end{equation}
\end{proposition}

\begin{proof}
	For each edge $uv\in G$, we have $d_{u}-1$ options to attach an edge $ux\neq uv$ at $u$, and $d_{v}-1$ options to attach an edge $vy\neq uv$ at $v$. If $G$ has no triangles, there are exactly $(d_{u}-1)(d_{v}-1)$ different $P_4$'s for each edge $uv$ in $G$. If $G$ has triangles, then
	\begin{equation}
		\label{eq44}
		(d_u -1)(d_v -1) = |P_4^*(uv)|+|K_3(uv)|,
	\end{equation} 
	where $P_4^*(uv)$ is the set of all $P_4$'s in $G$ with $uv$ as the central edge, and $K_3(uv)$ is the set of all triangles in $G$ containing $uv$. Since
	\[ \{H\subseteq G: H\approx P_4\} = \dot{\bigcup}_{e\in E(G)} P_4^*(e), \]
	it is clear that $\sum_{e\in E(G)} |P_4^*(e)|=p_4(G)$. On the other hand, let $M$ be the multiset given by the union, over $e\in E(G)$, of all $K_3(e)$. Since each triangle $T\in\{H\subseteq G:H\approx K_3\}$ appears in $M$ with multiplicity 3, it follows that $\sum_{e\in E(G)} |K_3(e)|=3k_3(G)$. Finally, \eqref{eq43} follows from \eqref{eq44} by summing over all $uv\in E(G)$.
\end{proof}

We denote by $c_4(G)$ and $k_4(G)$ the number of cycles and cliques of order 4 in $G$, respectively. More precisely:
\[ c_4(G) = |\{H \subseteq G : H \approx C_4\}|, \]
\[ k_4(G) = |\{H \subseteq G : H \approx K_4\}| = |Q_G(K_4)|. \]

\begin{theorem}
	\label{deg.general.fórmula}
	If $G \in \mathcal{G}(s)$, then
	\begin{equation}
		\label{deg.fórmula}
		\deg(G) = 2dpe(s) + 2c_4(G) - p_4(G) - 4k_4(G).
	\end{equation}
	Letting $\Vert G\Vert = m$, \eqref{deg.fórmula} can be rewritten as:
	\begin{equation*}
		\deg(G) = m(m+1) - |s|^2 + 2c_4(G) + 3k_3(G) - 4k_4(G) - \sum_{uv \in E(G)} (d_u - 1)(d_v - 1).
	\end{equation*}
\end{theorem}

\begin{proof}
	First, recall that $dpe(G) = |\{H \subseteq G : H \approx 2K_2\}|$. If $Y \preceq G$ and $|Y| = 4$, observe that
	\begin{equation}
		\label{contar.subgrafos}
		|\{H \subseteq G : H \approx Y\}| = \sum_{i=1}^{11} |Q_G(X_i)| \cdot |\{H \subseteq X_i : H \approx Y\}|,
	\end{equation}
	where the sum in \eqref{contar.subgrafos} ranges over the 11 graphs $X_i$ in Figure \ref{los.11.de.orden4}. To simplify notation, we write $|X|$ instead of $|Q_G(X)|$ and use $p_4, c_4, k_4, dpe$ instead of $p_4(G), c_4(G), k_4(G), dpe(s)$. Applying formula \eqref{contar.subgrafos} for $Y \in \{C_4, P_4, 2K_2\}$, we obtain:
	\begin{align*}
		c_4 &= |D_4| + |C_4| + 3k_4, \\
		p_4 &= 4|C_4| + |P_4| + 2|U_4| + 6|D_4| + 8k_4, \\
		dpe &= 2|C_4| + |P_4| + |2K_2| + |U_4| + 2|D_4| + 3k_4.
	\end{align*}
	Solving for $|C_4|$, $|2K_2|$, and $|P_4|$ from these, and applying Theorem \ref{degreeofG}, we find:
	\begin{align*}
		\deg(G) &= 2(c_4 - |D_4| - 3k_4) + \\
		&\quad 2(dpe - 2|C_4| - |U_4| - |P_4|) + \\
		&\quad p_4 - 4|C_4| - 2|U_4| - 6|D_4| - 8k_4 \\
		&= 2dpe + 2c_4 + p_4 - 12|D_4| - 20k_4 - 4|U_4| - 2(4|C_4| + |P_4|) \\
		&= 2dpe + 2c_4 - p_4 - 4k_4.
	\end{align*}
\end{proof}

Theorem \ref{deg.general.fórmula} suggests that the difficulty of computing the 2-switch-degree of $G$ essentially lies in determining the number of 4-cliques in $G$, i.e., $k_4(G)$. This is because it is relatively “easy” to compute $c_4(G)$ and $k_3(G)$ (from which $p_4(G)$ depends). Indeed, it is well known that these can be obtained from the \textbf{adjacency matrix} $A$ of $G$, an $n \times n$ matrix (where $n = |G|$), with entry $[A]_{ij} = 1$ if $ij \in E(G)$ and $0$ otherwise. In fact:
\[ k_3(G) = \frac{1}{6} \, \text{tr}(A^3), \]
where $\text{tr}(\ast)$ is the trace, and
\begin{equation*}
	c_4(G) = \frac{1}{2} \sum_{1 \leq i < j \leq n} \binom{[A^2]_{ij}}{2}.
\end{equation*}
We also have:
\begin{equation*}
	\text{tr}(A^4) = 8c_4(G) + 2\Vert G\Vert + 4\sum_{v=1}^{n} \binom{d_v}{2} = 8c_4(G) - 2\Vert G\Vert + 2|s|^2,
\end{equation*}
\begin{equation*}
	c_4(G) = \frac{1}{8} \text{tr}(A^4) + \frac{1}{4} \Vert G\Vert - \frac{1}{4} |s|^2 = \frac{1}{8} \text{tr}(A^4) + \frac{1}{2} dpe(s) - \frac{1}{4} \Vert G\Vert^2.
\end{equation*}

No general closed formulas (like the above) are currently known for $k_4$. In fact, the problem of counting cliques of a given size in a graph is generally of extreme computational complexity, as it is NP-complete. However, for certain families of graphs it might be “easy” to determine the degree of their members (or the degree might even be constant in such families), allowing $k_4$ to be computed easily via formula \eqref{deg.fórmula}.

As immediate corollaries of Theorem \ref{deg.general.fórmula}, we can derive simplified formulas for specific graph classes. For example, if $G$ is triangle-free (i.e., $G$ contains no 3-cycles), then $k_3(G) = 0 = k_4(G)$. If $G$ is $C_4$-free, then $c_4(G) = 0 = k_4(G)$. The \textbf{girth} of a graph $G$, denoted by $g(G)$, is the length of the shortest cycle contained in $G$. If $G$ is acyclic, $g(G)$ is defined as $\infty$. Note that $g(G) \geq 5$ if and only if $k_3(G) = 0 = c_4(G)$.

\begin{corollary}
	\label{deg.girth>4}
	If $G \in \mathcal{G}(s)$ and $g(G) \geq 5$, then
	\begin{equation}
		\label{deg.girth>4.fórmula}
		\deg(G) = 2dpe(s) - p_4(G).
	\end{equation}
	Letting $\Vert G\Vert = m$, formula \eqref{deg.girth>4.fórmula} can be rewritten as
	\begin{equation*}
		\deg(G) = m(m+1) - |s|^2 - \sum_{uv \in E(G)} (d_u - 1)(d_v - 1).
	\end{equation*}
\end{corollary}

\begin{proof}
	It follows from the previous discussion.
\end{proof}

As an application of Corollary \ref{deg.girth>4}, we can compute the degree of paths and cycles, since $g(P_n) = \infty$ and $g(C_n) = n$.

Paths of order 1 or 2 clearly have degree 0. For $n \geq 3$, $s = s(P_n) = 2^{n-2}1^2$. Since $dpe(s) = \binom{n-2}{2}$ and $p_4(P_n) = n - 3$, we get
\begin{equation}
	\label{path.deg}
	\deg(P_n) = (n - 3)^2.
\end{equation}
For each $n \geq 3$, $s=s(C_n) = 2^n$ and $dpe(s) = \binom{n}{2} - n$. Clearly, $\deg(C_3) = 0$ and $\deg(C_4) = 2$. If $n \geq 5$, then $p_4(C_n) = n$ and therefore,
\begin{equation}
	\label{ciclo.deg}
	\deg(C_n) = n(n - 4).
\end{equation}
Note that $\deg(C_n) \geq \deg(P_n)$ for all $n \geq 3$. We can apply formulas \eqref{path.deg} and \eqref{ciclo.deg} to obtain a lower bound on the degree of a graph $G$ that contains some path $P$ or cycle $C$ as an induced subgraph. Indeed, from Proposition \ref{deg.respects.ind.inc} we have $\deg(G) \geq \deg(P), \deg(C)$.

Let $G$ be a graph with degree sequence $s = (d_v)_{i=1}^n$. The numbers
\[ \zeta_{1}(G)=\sum_{v=1}^n d_{v}^{2}=|s|^{2}, \]
\[ \zeta_{2}(G)=\sum_{uv\in E(G)} d_{u}d_{v} \]
are called the \textbf{first} and \textbf{second Zagreb indices} of $G$, respectively. These are well-known parameters in Chemical Graph Theory and widely studied in the literature. These so-called topological invariants are often used to quantify structural characteristics of chemical compounds. Both $\zeta_1$ and $\zeta_2$ were first introduced in 1972 by Gutman and Trinajstić and have been used to relate molecular properties such as chemical stability and $\pi$-electron energy. The latter refers to the electrons involved in $\pi$ bonds, a type of covalent bond formed by the lateral overlap of $\pi$ atomic orbitals. There is a relationship between the total energy of $\pi$ electrons and the Zagreb indices, since the latter capture aspects of molecular connectivity and branching that influence electron delocalization. For example, in certain studies, it has been found that greater branching (reflected in higher Zagreb index values) can lead to lower stability and thus higher total $\pi$-electron energy. Therefore, Zagreb indices can be useful tools for estimating and understanding how molecular structure affects $\pi$-electron energy.

It is a curious fact that these chemical indices arise naturally in our context. Indeed, since $\sum_{uv \in E(G)} (d_u + d_v) = \zeta_1(G)$, we have
\begin{equation}
	\label{sum.zagreb.1,2}
	\sum_{uv \in E(G)} (d_u - 1)(d_v - 1) = \zeta_2(G) - \zeta_1(G) + \Vert G\Vert.
\end{equation}
We know that the left-hand sum in \eqref{sum.zagreb.1,2} (which equals $p_4(G) + 3k_3(G)$ by Proposition \ref{number.of.P4.in.G}) is part of the formula in Theorem \ref{deg.general.fórmula} and Corollary \ref{deg.girth>4} for computing $\deg(G)$. Hence, we see a link between the 2-switch and the Zagreb indices. More specifically, the next theorem tells us how the 2-switch-degree is tied to $\zeta_2$.

\begin{theorem}
	\label{relacion.deg.zeta2}
	If $G$ is a graph, then
	\[ \deg(G) + \zeta_2(G) = \Vert G\Vert^2 + 2c_4(G) + 3k_3(G) - 4k_4(G). \]
\end{theorem}

\begin{proof}
	It follows from the previous discussion.
\end{proof}

When $g(G) \geq 5$, we obtain a remarkable corollary of Theorem \ref{relacion.deg.zeta2}: the sum of the 2-switch-degree and the second Zagreb index equals the square of the size.

\begin{corollary}
	If $g(G) \geq 5$, then
	\begin{equation}
		\label{eq.deg.zeta2.girth>4}
		\deg(G) + \zeta_2(G) = \Vert G\Vert^2.
	\end{equation}
\end{corollary}

\begin{proof}
	It follows immediately from Theorem \ref{relacion.deg.zeta2}, since $c_4(G) = k_3(G) = k_4(G) = 0$ when $g(G) \geq 5$.
\end{proof}

Now suppose we are studying extremal values of the second Zagreb index (and which graphs attain them) in a family of graphs with fixed size. Equality \eqref{eq.deg.zeta2.girth>4} may be very useful in this case, as it tells us that minimizing (maximizing) $\zeta_2$ is equivalent to maximizing (minimizing) the degree.


\section{Degree of trees and unicyclic graphs} \label{sec:deg.arboles.unicicl}

Recall that we denote by $\mathcal{F}(s)$ the family of all forests with degree sequence $s$. Since all members of this family have the same number of connected components, it follows that if one of them is a tree, then all the others are trees as well. A 2-switch $\tau$ on a forest $F \in \mathcal{F}(s)$ is called an \textbf{f-switch} on $F$ if $\tau(F) \in \mathcal{F}(s)$. When $F$ is a tree, we can use the term \textbf{t-switch} instead of f-switch. The \textbf{f-degree} of a forest $F$, denoted by $\deg_f(F)$, is the degree of the vertex $F$ in the graph $\mathbf{T}(\mathcal{F}(s))$. Equivalently, $\deg_f(F)$ is the number of f-switches that act on $F$. Clearly, $\deg_f(F) \leq \deg(F)$, since $\mathcal{F}(s) \preceq \mathcal{G}(s)$.

The following result provides a formula to compute the f-degree of a tree.

\begin{theorem}
	\label{deg_f.tree}
	Let $T$ be a tree in $\mathcal{F}(s)$, where $s=(d_{v})_{v=1}^{n}$. Then,
	\begin{equation*}
		\deg_f(T) = dpe(s) = \binom{n-1}{2} - \sum_{v=1}^{n} \binom{d_v}{2} = \binom{n}{2} - \frac{1}{2}|s|^2,
	\end{equation*}
	where $\binom{i}{j}=0$ if $i<j$.
\end{theorem}

\begin{proof}
	Consider the 2-switches $\tau=\binom{a \ b}{c \ d}$ and $\tau'=\binom{a \ b}{d \ c}$. Assuming that $ab\ldots cd$ is the path connecting $a$ to $d$ in $T$, we observe that:
	\begin{enumerate}
		\item $\tau$ is a t-switch on $T$;
		\item if $bc \notin E(T)$, then $\tau'$ disconnects $T$, and therefore is not a t-switch on $T$;
		\item if $bc \in E(T)$, then $\tau'$ is inactive on $T$.
	\end{enumerate}
	With this in mind, it is now easy to see that we can perform exactly one t-switch in $T$ for each unordered pair of disjoint edges of $T$. Then, $\deg_f(T)=dpe(s)$, and we apply Proposition \ref{dpe.fórmula}.
\end{proof}

It is quite surprising that the f-degree of a tree $T \in \mathcal{F}(s)$ is an invariant associated with $s$. Theorem \ref{deg_f.tree} has a significant impact on the global structure of $\mathbf{T}(\mathcal{F}(s))$. Indeed, we see that every vertex of this transition space has the same degree, namely $dpe(s)$. Moreover, we already know from Section \ref{sec:preliminares} that $\mathcal{F}(s)$ is connected (\cite{vnsigma.2switch.graphs.forests}, page 9). Hence, we obtain the following corollary.

\begin{corollary}
	If $s$ is the degree sequence of a tree, then $\mathcal{F}(s)$ is a connected and $k$-regular graph, where $k=dpe(s)$.
\end{corollary}

\begin{proof}
	It follows from the preceding discussion.
\end{proof}

Previously, we stated that the degree of a tree is not generally an invariant associated with its degree sequence. The next corollary justifies this claim.

\begin{corollary}
	If $T$ is a tree of order $n$, then
	\begin{equation}
		\deg(T) = 2\deg_f(T) - \sum_{uv \in E(T)} (d_u - 1)(d_v - 1) = (n - 1)^2 - \zeta_2(T).
	\end{equation}
\end{corollary}

\begin{proof}
	The first equality is an immediate consequence of Theorems \ref{deg.general.fórmula} and \ref{deg_f.tree}, where $k_3(T) = c_4(T) = k_4(T) = 0$ and $dpe(T) = \deg_f(T)$. The second equality follows from Corollary \ref{deg.girth>4}.
\end{proof}

If $T$ is a tree in $\mathcal{G}(s)$, note that $\deg(T) - \deg_f(T)$ is the number of disconnected neighbors of $T$ in $\mathcal{G}(s)$.\\

We denote by $\mathcal{U}(s)$ the family of all unicyclic graphs with degree sequence $s$. It is worth noting that $\mathcal{U}(s)$ is connected (\cite{vnsigma.2switch.unic.pseudof}, page 8), something already mentioned in Section \ref{sec:preliminares}. A 2-switch $\tau$ on a unicyclic graph $U \in \mathcal{U}(s)$ is said to be a \textbf{u-switch} on $U$ if $\tau(U) \in \mathcal{U}(s)$. The \textbf{u-degree} of a unicyclic graph $U$, denoted by $\deg_u(U)$, is the degree of the vertex $U$ in the graph $\mathbf{T}(\mathcal{U}(s))$. Equivalently, $\deg_u(U)$ is the number of u-switches acting on $U$. Clearly, $\deg_u(U) \leq \deg(U)$, since $\mathcal{U}(s) \preceq \mathcal{G}(s)$.

In a unicyclic graph $U$, we can always identify the following two subgraphs: the unique cycle $C$ of $U$ and the forest  
\[ F = U - E(C) - \{v \in V(C) : \deg_U(v) = 2 \}. \]
Then, $E(U) = E(C) \dot{\cup} E(F)$ and $|V(F) \cap V(C)|$ is the number of components of $F$. If $s(U) = 2^n$, then $U \approx C_n$ and $F = K_0$. But if $d_v \neq 2$ for some $v \in V(U)$, then $E(F) \neq \varnothing$. Note that $C$ is an induced subgraph of $U$, but $F$, in general, is not.

\begin{corollary}
	If $U \in \mathcal{U}(s)$ is a unicyclic graph of order $n$ with a cycle of length $c$, then 
	\begin{equation*}
		\deg(U) =
		\begin{cases}
			n^2 - \zeta_2(U) + 3, & \text{if } c = 3 \\
			n^2 - \zeta_2(U) + 2, & \text{if } c = 4 \\
			n^2 - \zeta_2(U), & \text{if } c \geq 5 \\
		\end{cases}
	\end{equation*}
\end{corollary}

\begin{proof}
	It follows immediately from Theorem \ref{relacion.deg.zeta2}.
\end{proof}

\begin{theorem}[\cite{vnsigma.2switch.unic.pseudof}, page 5]
	\label{u-switch.caract}
	Let $\tau$ be a 2-switch acting on the edges $e_1$ and $e_2$ of a unicyclic graph $U$. If $C$ and $F$ are, respectively, the cycle and the forest of $U$, then we have the following:
	\begin{enumerate}
		\item if $e_1, e_2 \in F$ and $e \in E(C)$, then $\tau$ is a u-switch on $U$ if and only if it is a t-switch on $U - e$;
		\item if $e_1 \in C$ and $e_2 \in F$, then $\tau$ is a u-switch on $U$;
		\item if $e_1, e_2 \in C$, then $\tau$ is a u-switch on $U$ if and only if $\tau(C) \approx C$.
	\end{enumerate}
\end{theorem}

\begin{theorem}
	If $U$ is a unicyclic graph with cycle $C$ and forest $F$, then 
	\begin{align*}
		\deg_u(U) = & \deg(U) - \deg(C) + dpe(C) - dpe(F) + p_4(F) \\
		= & \deg(U) - \deg(C) - \deg(F) + dpe(C) + dpe(F).
	\end{align*}
\end{theorem}

\begin{proof}
	We can compute the required formulas by computing the number $\bar{u}(U)$ of u-switches that do not preserve the unicyclic structure of $U$, since $\deg_u(U) = \deg(U) - \bar{u}(U)$. Using Theorem \ref{u-switch.caract}, it follows that these 2-switches fall into two disjoint categories:
	\begin{enumerate}[(1).]
		\item those between edges in $F$ that are not t-switches on $U - e$, for any $e \in E(C)$;
		\item those between edges in $C$ that split $C$ into two disjoint cycles.
	\end{enumerate}
	If $F$ has $k$ components $T_i$, then there are $\deg(T_i) - \deg_f(T_i)$ 2-switches of type (1) acting on each $T_i$. Moreover, there is exactly one 2-switch of type (1) for each unordered pair of edges lying in different components of $F$. Using Theorem \ref{deg.girth>4}, we deduce that
	\[ \deg(T_i) - \deg_f(T_i) = dpe(T_i) - p_4(T_i). \]
	Furthermore, it is clear that
	\[ \sum_{i=1}^k p_4(T_i) = p_4(F). \]
	Finally, since 
	\[ dpe(F) = \sum_{i=1}^k dpe(T_i) + \sum_{1 \leq i < j \leq k} \Vert T_i \Vert \Vert T_j \Vert \]
	by Proposition \ref{dpe.components}, we find that there are
	\begin{equation}
		\label{dpe(F)-p_4(F)}
		dpe(F) - p_4(F)
	\end{equation}   
	active 2-switches of type (1) in $U$. Using Theorem \ref{deg.girth>4}, we see that \eqref{dpe(F)-p_4(F)} is equal to $\deg(F) - dpe(F)$.
	
	To determine the number of 2-switches of type (2), first note that $\deg_u(C) = dpe(C)$. Then, there are 
	\begin{equation*}
		\deg(C) - dpe(C)
	\end{equation*} 
	active u-switches of type (2) in $U$, and therefore,
	\begin{equation*}
		\bar{u}(U) = \deg(C) - dpe(C) + dpe(F) - p_4(F).
	\end{equation*}
\end{proof}

\begin{corollary}
	Let $U \in \mathcal{U}(s)$ be a unicyclic graph of order $n$ with cycle $C$ and forest $F$. If $c = |C|$, then we have:
	\begin{equation*}
		\deg_u(U) =
		\begin{cases}
			2dpe(s) + dpe(F) - \deg(F) - p_4(U), & \text{if } c = 3 \\
			2dpe(s) + dpe(F) - \deg(F) - p_4(U) + 2, & \text{if } c = 4 \\
			2dpe(s) + dpe(F) - \deg(F) - p_4(U) - \frac{c}{2}(c - 5), & \text{if } c \geq 5 \\
		\end{cases}
	\end{equation*}
	where
	\begin{equation*}
		p_4(U) =
		\begin{cases}
			\zeta_2(U) - |s|^2 + n - 3, & \text{if } c = 3 \\
			\zeta_2(U) - |s|^2 + n, & \text{if } c \geq 4 \\
		\end{cases}
	\end{equation*}
\end{corollary}

Contrary to the transition spaces of trees with the same degree sequence, $\mathcal{U}(s)$ is not regular in general. To see this, consider the unicyclic graph $U$ obtained by identifying a vertex of a triangle with a leaf of a $P_4$, and the unicyclic graph $U'$ obtained by identifying a vertex of a $C_4$ with a leaf of a $P_3$. We see that $s(U) = 3^1 2^4 1^1 = s(U')$. However, $\deg_u(U) = 11$ and $\deg_u(U') = 10$.


\chapter{Quotient by twins} \label{cap:cociente.gemelos}

Let $u$ and $v$ be vertices of a graph $G$. We say that $u$ is a \textbf{twin} of $v$ in $G$ if
\[ N_G(u)-v = N_G(v)-u. \]
In other words, twin vertices in $G$ appear as ``copies" of each other in $G$, since their neighborhoods are essentially identical. In this chapter, we develop a structural simplification tool for graphs: the quotient by twins. Informally, this operator identifies twin vertices in $G$, that is, for each $v \in V(G)$, it removes all its twins. Thus, it reduces the redundant information in $G$ that comes from having copies of a certain vertex. We denote by $[G]$ the graph obtained from $G$ through this process. When passing from $G$ to $[G]$, it is clear that some information is generally lost. One of the main objectives of this chapter is to quantify this loss by studying which essential properties of $G$ are preserved in $[G]$. Another fundamental objective is to understand how the quotient by twins behaves in relation to Tyshkevich composition. Finally, it should be mentioned that this chapter also serves as preparation for Chapter \ref{cap:grafos.asociados.a.split}, where we will see that twin vertices arise naturally in the context of split graphs, playing a very important role. We divide this analysis into 5 sections.

In Section \ref{sec:prop.basicas.cociente}, we present a series of basic results about the quotient by twins. We show that ``being twins in $G$" is an equivalence relation on $V(G)$, which we denote by $\cong_G$, and from this we rigorously define $[G]$ as a graph whose vertices are the equivalence classes. Subsequently, we study how $\cong_G$ relates to $\cong_H$ when $H \preceq G$; among other things, we prove that $[H]$ is essentially an induced subgraph of $[G]$. Then, we show that $[G]$ is isomorphic to the induced subgraph of $G$ obtained by removing all its twin vertices until only one representative of each class remains. Next, we show that $[\ast]$ preserves cliques and independent sets in a ``natural" way. We conclude the section by proving that inactive vertices of the same degree are twins.

In Section \ref{sec:cociente.en.split}, we focus on the quotient by twins in split graphs. Here, the key result is that $[S]$ is balanced if and only if $S$ is balanced.

In Section \ref{sec:cociente,diametro,activos}, the primary objective is to study under what conditions $[\ast]$ preserves the property of being an active graph. In fact, we see that this is generally not the case, i.e., $G$ being active does not always imply that $[G]$ is active. Conversely, an active quotient forces the original graph to be active. In Section \ref{sec:vertices.activos}, we saw that there is a relationship between $diam(G)$ and whether $G$ is active or not. This becomes important again here. Indeed, we show that if $diam(G) \geq 4$, then $[G]$ is active if and only if $G$ is. An identical result is also obtained for active split graphs, whose diameter is always 3. All these results are due to the fact that the quotient preserves distances $\geq 3$.

In Section \ref{sec:cociente.composicion}, we seek to understand how the quotient behaves with respect to Tyshkevich composition. First, we prove that if $H = S \circ G$, then a vertex $s \in S$ cannot be a twin in $H$ of a vertex $g \in G$. Then, through this and other auxiliary results, we manage to show that $[\ast]$ distributes over $\circ$. In other words, the quotient of the product is the product of the quotients: $[H] = [S] \circ [G]$, whenever $S$ is balanced.

Finally, in Section \ref{sec:indice.cociente}, we introduce and study a new invariant: the quotient index. We first observe that $[G]$ may still have twins. However, if the quotient is iterated sufficiently many times, we always eventually arrive at a graph that no longer has any pair of twin vertices. Such a graph is said to be twin-free. The minimum number of quotient iterations needed for $G$ to become twin-free is what we define as the quotient index of $G$, denoted by $i(G)$. This is a parameter that, in a certain sense, measures how ``far" a graph is from being twin-free. After finding a sharp upper bound for $i(\ast)$, we determine it precisely for certain specific classes, such as trees, unicyclic graphs, and balanced split graphs. Finally, we show that $i(S \circ G)$ essentially depends on $G$.


\section{Basic Properties} \label{sec:prop.basicas.cociente}

If $u$ and $v$ are twin vertices in a graph $G$, we denote this relation with the symbol $u\cong_G v$ (or simply $u\cong v$, if the graph in question is clear from context). It is straightforward to verify that ``$\cong_G$" is an equivalence relation on the set $V(G)$. Regarding this, we want to emphasize the following fact: if $u\in N_G(v)$, $u\cong_G v$ and $v\cong_G w$, then $u\in N_G(w)$. 

Therefore, we can quotient $V(G)$ by the relation $\cong_G$ and thus obtain a partition of $V(G)$ into equivalence classes 
\[ [v]_G=\{x\in V(G): x\cong_G v\} \]
(or simply $[v]$, if the graph in question is clear from context). If $W\subseteq V(G)$, then $[W]_G=\{[v]_G:v\in W\}$. We now define the \textbf{twin quotient} $[G]$ of $G$ as the graph whose vertex set is
\[ V[G]=V([G])=\{[v]_G: v\in V(G)\}, \]
and where $[u]$ and $[v]$ are neighbors if and only if the following two conditions hold:
\begin{enumerate}
	\item some $u'\in[u]$ is adjacent (in $G$) to some $v'\in[v]$;
	\item $[u]\neq[v]$.
\end{enumerate}
We want to note that condition (2) serves to prevent loops from appearing in $[G]$. Indeed, this would happen precisely when a certain class $[v]_G$ is a clique of size $\geq 2$ in $G$.

\begin{proposition}
	\label{lema.fundamental.uv.[u][v]}
	Let $G$ be a graph and let $u,v\in V(G)$. Then:
	\begin{enumerate}
		\item If $[u][v]\in E[G]$, then $uv\in E(G)$.
		\item If $uv\in E(G)$ and $u$ is not a twin of $v$, then $[u][v]\in E[G]$.
	\end{enumerate}
\end{proposition}

\begin{proof}
	\begin{enumerate}[(1).]
		\item If $[u][v]\in E[G]$, then there exist vertices $u',v'\in G$ such that $u'\in[u]$, $v'\in[v]$ and $u'v'\in G$. Since $u'\in N_{v'}$ and $v'\cong v$, it follows that $u'\in N_v$ and thus $v\in N_{u'}$. As $u'\cong u$, we conclude that $v\in N_u$. Therefore, $uv\in G$.
		\item If $uv\in E(G)$ and $[u]\neq[v]$, then it is obvious by definition of the quotient graph that $[u][v]\in E[G]$, since $u\in [u]$ and $v\in [v]$.
	\end{enumerate}
\end{proof}

The following lemma tells us that two twin vertices in a graph $G$ remain twins in every induced subgraph $H$ of $G$ that contains them. Indeed, if $u,v\in V(H)$ and $u\cong_G v$, then
\[N_H(u)-v=(N_u-v)\cap V(H)=(N_v-u)\cap V(H)=N_H(v)-u.\] 

\begin{proposition}
	\label{gem.G.entonces.gem.H<G}
	Let $G$ be a graph and $H\preceq G$. If $u,v\in V(H)$ and $u\cong_G v$, then $u\cong_H v$. In other words: $[v]_G\subseteq [v]_H$, for every $v\in V(H)$.
\end{proposition}

\begin{proof}
	It follows from the previous discussion.
\end{proof}

It can be easily verified that the converse of Proposition \ref{gem.G.entonces.gem.H<G} is not true, i.e.: if two vertices are twins in $H\preceq G$, they are not necessarily twins in $G$ as well. In fact, taking $G=abcd\approx P_4$ and $H=abc\approx P_3$, we have that $a$ and $c$ are twins in $H$ but not in $G$. Since $[a]_G=\{a\}$ and $[a]_H=\{a,c\}$, this example also shows us that the class of a vertex may change when viewed in an induced subgraph.

If $H\preceq G$, it is sometimes necessary to consider in $[G]$ the induced subgraph on the set $[V(H)]_G=\{[v]_G: v\in V(H)\}$. This subgraph should be denoted by the symbol $\langle [V(H)]_G \rangle_{[G]}$, which can be somewhat tedious to write down. Therefore, we replace it with a more concise one: $[H]_G$. When $H=G$, we recover the usual quotient, since $[G]_G$ is what we already denoted by $[G]$. Attention! If $H$ is an induced subgraph of a graph $G$, then in general it is not true that $[H]\approx [H]_G$. Indeed, taking 
\begin{equation}
	\label{[H].no-isomorfo.a.[H]_G}
	G=abcd\approx P_4, H=ab\approx K_2,
\end{equation}
we have that $[H]=[a]_H\approx K_1$ and $[H]_G=[a]_G[b]_G\approx K_2$. Note that in this particular example $[H]$ is essentially an induced subgraph of $[H]_G$. Will this be true in general?

\begin{proposition}
	\label{homomorfismo:[H]--->[H]_G}
	Let $G$ be a graph. If $H\preceq G$, then the function $\phi: V[H]\rightarrow [V(H)]_G$, defined by $\phi([x]_H)=[x]_G$, is a homomorphism between $[H]$ and $[H]_G$. Moreover, $\phi[H]\preceq [H]_G$ (i.e., $[H]$ is essentially an induced subgraph of $[H]_G$).
\end{proposition}

\begin{proof}
	If $[u]_H[v]_H\in E[H]$, then $uv\in H$, by Proposition \ref{lema.fundamental.uv.[u][v]}. Thus, $uv\in G$, since $H\subseteq G$. Note that $[u]_G\neq[v]_G$, because otherwise, by Proposition \ref{gem.G.entonces.gem.H<G}, we would have $u\cong_H v$, contradicting that $[u]_H\neq [v]_H$. Consequently, $[u]_G[v]_G\in E[G]$, by Proposition \ref{lema.fundamental.uv.[u][v]}. Since $[u]_G=\phi([u]_H)$ and $[v]_G=\phi([v]_H)$, we finally conclude that $\phi$ is a homomorphism.
	
	Since $\phi$ is a homomorphism, we have $\phi[H]\subseteq [H]_G$. To see that $\phi[H]\preceq [H]_G$, we will prove that $[u]_G[v]_G\in [H]_G$ implies $[u]_H[v]_H\in [H]$. By the definition of $[H]_G$, it is clear that $u,v\in H$. Applying Proposition \ref{lema.fundamental.uv.[u][v]}, we see that $[u]_G[v]_G\in [H]_G$ implies $uv\in G$. Thus, $uv\in H$, since $H\preceq G$. Since $\phi([u]_H)=[u]_G\neq[v]_G=\phi([v]_H)$, it is obvious that $[u]_H\neq[v]_H$. Therefore, $[u]_H[v]_H\in [H]$, again by Proposition \ref{lema.fundamental.uv.[u][v]}.
\end{proof}

Before continuing, it is important to note that the homomorphism $\phi$ from Proposition \ref{homomorfismo:[H]--->[H]_G} is not an isomorphism because, although it preserves adjacencies and non-adjacencies, it is not bijective. More specifically, $\phi$ is injective thanks to Proposition \ref{gem.G.entonces.gem.H<G}, but in general it is not surjective, as can be seen in example \eqref{[H].no-isomorfo.a.[H]_G}.

It is easy to intuit that $[G]$ is isomorphic to the induced subgraph of $G$ obtained by removing all its twin vertices until only one representative of each class remains.

\begin{proposition}
	\label{isomorfismo.Q=phi[G]}
	Let $G$ be a graph and let $[V(G)]=\{[v_i]:1\leq i\leq k\}$. Then $[G]$ is isomorphic to $Q=\langle \{v_i:1\leq i\leq k\}\rangle_G$.
\end{proposition}  

\begin{proof}
	Let $\varphi:V[G]\rightarrow V(Q)$ be the function defined by $\varphi[x]=x$. If $[x]\neq[y]$, it is clear that $\varphi[x]=x\in[x]$, $\varphi[y]=y\in[y]$ and $[x]\cap[y]=\varnothing$. Thus, $x\neq y$, so $\varphi$ is injective. On the other hand, we see that $V(Q)=\varphi(V[G])$, by definition of $Q$. Therefore, $\varphi$ is surjective.
	
	If $[u][v]\in E[G]$, then $uv\in G$, by Proposition \ref{lema.fundamental.uv.[u][v]}. By definition of $Q$, we have $u,v\in Q$. Thus, $\varphi[u]\varphi[v]\in E(Q)$, since $Q\preceq G$ by definition of $Q$. 
	
	Conversely, suppose $\varphi[u]\varphi[v]=uv\in Q$. Obviously, $\varphi[u]\neq\varphi[v]$. It follows that $[u]\neq[v]$, since $\varphi$ is bijective. As $uv\in G$, we conclude by Proposition \ref{lema.fundamental.uv.[u][v]} that $[u][v]\in E[G]$. Finally, we have shown that $\varphi$ is an isomorphism.     
\end{proof}

\begin{proposition}
	\label{gem.G.iff.gem.G^c}
	Two vertices $u$ and $v$ are twins in $G$ if and only if they are twins in $\overline{G}$.
\end{proposition}

\begin{proof}
	First, note that $N_x^c=N_{\overline{G}}(x)\dot{\cup}x$, for every $x\in V(G)$. Then, the following equivalent equalities prove the required result:  
	
	\[ N_u-v=N_v-u \]
	\[ N_u^c\cup v=N_v^c\cup u \]
	\[(N_{\overline{G}}(u)\cup v)\dot{\cup}u=(N_{\overline{G}}(v)\cup u)\dot{\cup}v\]
	\[(N_{\overline{G}}(u)-v)\dot{\cup}\{u,v\}=(N_{\overline{G}}(v)-u)\dot{\cup}\{u,v\}\]
	\[ N_{\overline{G}}(u)-v=N_{\overline{G}}(v)-u. \]
\end{proof}

Thanks to Proposition \ref{gem.G.iff.gem.G^c}, it is clear that $V\left[\overline{G}\right]= V\left(\overline{[G]}\right)$. Let us see that $E\left[\overline{G}\right] =E\left(\overline{[G]}\right)$ also holds. Indeed, if $[a][b]\in \left[\overline{G}\right]$, then $ab\in \overline{G}$ by Proposition \ref{lema.fundamental.uv.[u][v]}, i.e., $ab\notin G$. Thus, $[a][b]\notin [G]$, again by Proposition \ref{lema.fundamental.uv.[u][v]}, which is equivalent to $[a][b]\in \overline{[G]}$. Conversely, let $[a][b]\in\overline{[G]}$, i.e., $[a][b]\notin[G]$. Since $[a]\neq[b]$, it follows by Proposition \ref{lema.fundamental.uv.[u][v]} that $ab\notin G$, which is equivalent to $ab\in\overline{G}$. Therefore, $[a][b]\in [\overline{G}]$, again by Proposition \ref{lema.fundamental.uv.[u][v]}. We have just proved that quotienting by twins commutes with complementation.

\begin{proposition}
	\label{cociente.y.compl.conmutan}
	If $G$ is a graph, then $\left[\overline{G}\right]=\overline{[G]}$.
\end{proposition}

\begin{proof}
	It follows from the previous discussion. 
\end{proof}

\begin{proposition}
	\label{ind.clique.G.[G]}
	Let $G$ be a graph. 
	\begin{enumerate}
		\item If $v\in V(G)$, then $[v]$ is a clique or an independent set in $G$;
		\item if $I$ is an independent set in $G$, then $[I]$ is an independent set in $[G]$;
		\item if $K$ is a clique in $G$, then $[K]$ is a clique in $[G]$.
	\end{enumerate}
\end{proposition}

\begin{proof}
	\begin{enumerate}[(1).]
		\item Let $H$ be the subgraph of $G$ induced by $[v]$. If $E(H)=\varnothing$, then we are done because $[v]$ is an independent set. Therefore, suppose from now on that $E(H)\neq\varnothing$. If $|[v]|\leq 2$ the result is obvious. If $|[v]|\geq 3$, then we can affirm, thanks to the independence of the representative for equivalence classes, that there exists a vertex $u\in[v]$ such that $uv\in H$. Let $x$ be any member of $[v]-\{u,v\}$. Since $x\cong v$, we have $ux\in H$. Since $x\cong u$, then $vx\in H$. Thus, $N_{H}(v)=[v]-v$. If $w\cong v$, then $N_H(w)-w=N_H(v)-v$, by Proposition \ref{gem.G.entonces.gem.H<G}. In other words: $N_H(w)=V(H)-w$ for every $w\in H$, which means that $H$ is complete.
		\item If $[u]$ and $[v]$ are two distinct vertices in $[I]=\{[x]:x\in I\}$, then there exist $u',v'\in I$ such that $[u]=[u']$ and $[v]=[v']$. Since $u'v'\notin G$, it follows by Proposition \ref{lema.fundamental.uv.[u][v]} that $[u'][v']=[u][v]\notin E[G]$. Thus, $[I]$ is an independent set in $[G]$.
		\item Since $K$ is an independent set in $\overline{G}$, it follows that $[K]$ is an independent set in $[\overline{G}]$, by (1). As $[\overline{G}]=\overline{[G]}$ by Proposition \ref{cociente.y.compl.conmutan}, we conclude that $[K]$ is a clique in $[G]$.
	\end{enumerate}
\end{proof}

\begin{proposition}
	\label{inactivos.mismo.deg.son.gemelos}
	Let $u$ and $v$ be inactive vertices of a graph $G$. If $\deg_G(u)=\deg_G(v)$, then $u\cong_G v$. 
\end{proposition}

\begin{proof}
	Suppose $u$ and $v$ are two inactive vertices of $G$ with the same degree in $G$, but $N_u-v\neq N_v-u$. Since $|N_u-v|=|N_v-u|$, we can find vertices $a,b$ such that $a\in N_u-N_v$ and $b\in N_v-N_u$. But then $\langle a,b,u,v\rangle_G \in Q_G^*$, which contradicts that $u$ and $v$ are inactive. Thus, $N_u-v=N_v-u$.
\end{proof}

The previous proposition ceases to be true if we replace the word ``inactive" with ``active". In fact, it suffices to consider the leaves $u$ and $v$ of a graph $G\approx P_4$.


\section{Quotient in split graphs} \label{sec:cociente.en.split}

In this section, we establish that the quotient by twins is an efficient tool for simplifying split graphs and studying their fundamental properties. In fact, recall that the basic idea of the quotient by twins is to eliminate the “redundancy” created by vertices that have the same neighborhood. The important result we reach here is that $[S]$ is balanced if and only if $S$ is also balanced. This is achieved through several intermediate steps. We begin below by reflecting on some concrete examples.\\

Since the quotient by twins is essentially an induced subgraph of the original graph (Proposition \ref{isomorfismo.Q=phi[G]}), we can immediately affirm that $[S]$ is a split graph if $S$ is, as we already know that this property is hereditary. On the other hand, $[G]$ may be a split graph even if $G$ is not. In Figure \ref{[G].split.pero.G.no} we exhibit the corresponding counterexample.
\begin{figure}[h]
	\centering
	\includegraphics[scale=0.8]{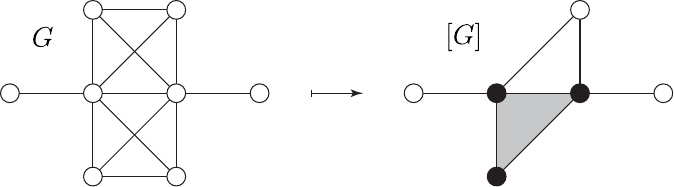}
	\caption{$[G]$ is split (unbalanced) but $G$ is not ($2K_2\preceq G$).}
	\label{[G].split.pero.G.no}
\end{figure}
Observing the counterexample in Figure \ref{[G].split.pero.G.no}, one might conjecture that adding the hypothesis that $[G]$ is balanced or active could finally allow us to infer that $G$ is split. However, the counterexample in Figure \ref{[G].split.activo.pero.G.no-split} dashes all hopes in this direction (recall that active implies balanced). Therefore, $[G]$ can be a split active (or balanced) graph without $G$ being split.
\begin{figure}[h]
	\centering
	\includegraphics[scale=0.8]{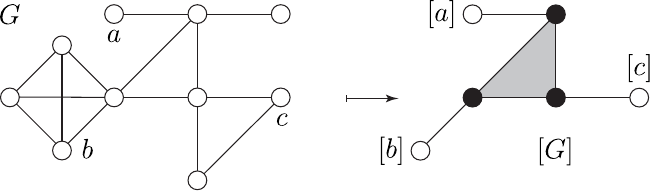}
	\caption{$[G]$ is split and active but $G$ is not split.}
	\label{[G].split.activo.pero.G.no-split}
\end{figure}
Recall that a split graph is balanced if and only if it has no swing vertices. We want to show that if $S$ is balanced, then $[S]$ is also balanced. As a first step, let us show that in balanced split graphs, independent vertices cannot be twins of clique vertices, and vice versa.

\begin{lemma}
	\label{v.ind.nogemelo.v.clique}
	If $(S,K,I)$ is balanced, then $([K],[I])$ is a bipartition for $[S]$. In particular: $[v]\in [I]$ ($[v]\in [K]$) if and only if $v\in I$ ($v\in K$).
\end{lemma}

\begin{proof}
	We already know that $[S]$ is split, so we only need to prove that $[I]\dot{\cup}[K]$ is a partition of $V[S]$ such that $[I]$ is an independent set in $[S]$ and $[K]$ is a clique in $[S]$. Applying Proposition \ref{ind.clique.G.[G]}, it follows that $[I]$ and $[K]$ are respectively an independent set and a clique in $[S]$. Since $V(S)=I\cup K$, it is obvious that $V[S]=[V(S)]=[I]\cup[K]$. 
	
	Suppose $[x]\in[I]\cap[K]$. Since $[x]\in [I]$, there exists a vertex $u\in I$ such that $[x]=[u]$. Since $[x]\in [K]$, there exists a vertex $v\in K$ such that $[x]=[v]$. Since $[u]=[v]$ and $I\cap K=\varnothing$, it follows that $u\cong v$ and $u\neq v$. Therefore,
	\[ (K-v)\dot{\cup}(N_v\cap(I-u))=N_v-u=N_u-v\subseteq K-v. \]
	This implies $N_v=K-v$ and $N_u=K$. But then $u$ and $v$ are swing. This is impossible since $S$ is balanced. Thus, $[I]$ and $[K]$ are disjoint.
	
	If $v \in I$, it is obvious that $[v] \in [I]$, by the definition of $[I]$. If $[v] \in [I]$ but $v \notin I$, then $v \in K$. Hence, $[v] \in [K]$, and therefore $[v] \in [I] \cap [K]$, contradicting that $[I] \cap [K] = \varnothing$.
\end{proof}

Lemma \ref{v.ind.nogemelo.v.clique} is not generally true if $(S,K,I)$ is not balanced. Indeed, let $S=abc\approx P_3$, with $I=\{a\}$ and $K=\{b,c\}$. Then, $a\cong c$, $[I]=\{[c]\}$ and $[K]=\{[b],[c]\}$, so
\begin{equation}
	\label{[I]&[K]no.vacio.contraejemplo}
	[I]\cap [K]\neq\varnothing.
\end{equation}
This means that: 1) there is an independent vertex that is a twin of a clique vertex ($a\in I,a\cong c\in K$); 2) $[I]\cup[K]$ is not a partition of $[S]=[a][b]\approx K_2$.

Through this same example, we also observe another counterintuitive phenomenon: a swing vertex may have a twin that is not swing. In fact, $c$ is swing in $S$ since $N_c =b=K-c$, and is a twin of $a$ in $S$. However, $a$ is not swing in $S$, since $N_a=b\neq \{b,c\}=K-a$. This is because the interchangeability of a vertex in $S$ is tied to the bipartition of $S$, which is not unique when $S$ is unbalanced.

Next, we make the first important step toward the goal of this section.

\begin{lemma}
	\label{S.balanc=>[S].balanc}
	If $S$ is balanced, then $[S]$ is balanced.    
\end{lemma}

\begin{proof}
	Thanks to Lemma \ref{v.ind.nogemelo.v.clique}, we know that $([K],[I])$ is a bipartition for $[S]$. Suppose that $[S]$ has a swing vertex $[w]$. Without loss of generality, we may assume that $[w]\in [I]$, so that $N_{[S]}[w]=[K]$. From Lemma \ref{v.ind.nogemelo.v.clique}, we have $w\in I$. If there were a vertex $k\in K$ such that $wk\notin S$, then by Proposition \ref{lema.fundamental.uv.[u][v]} we would have $[w][k]\notin E[K]$. This contradicts $N_{[S]}[w]=[K]$, since $[k]\in [K]$, by Lemma \ref{v.ind.nogemelo.v.clique}. Therefore, $N_w=K$, which means $w$ is swing in $S$, a contradiction.
\end{proof}

The next technical lemma essentially tells us that the quotient “preserves” the neighborhood of a swing independent vertex, provided its class does not belong to the set of clique vertex classes.

\begin{lemma}
	\label{lema.N_[w]=[K]}
	Let $(S,K,I)$ be a split graph. If $S$ has a swing vertex $w\in I$ such that $[w]\notin[K]$, then $N_{[w]}=[K]$.
\end{lemma}

\begin{proof}
	If $[x]\in N_{[w]}$, that is, if $[w][x]\in E[S]$, then $wx\in S$ by Proposition \ref{lema.fundamental.uv.[u][v]}. Since $w\in I$, it follows that $x\in K$, so $[x]\in[K]$.
	
	Conversely, suppose $[x]\in[K]$. Then there exists a vertex $t\in K$ such that $[t]=[x]$. Since $N_w=K$, we have $wt\in S$. Hence $wx\in S$ also, because $x\cong t$. Since $[w]\notin[K]$ by hypothesis, it is obvious that $[w]\neq[x]$. Therefore, $[x]\in N_{[w]}$ by Proposition \ref{lema.fundamental.uv.[u][v]}.
\end{proof}

Lemma \ref{lema.N_[w]=[K]} now helps us characterize when $S$ is balanced, assuming that $[S]$ also is.

\begin{lemma}
	\label{S.balanc.iff.[I]&[K]=vacio}
	Let $(S,K,I)$ be a split graph such that $[S]$ is balanced. Then, $S$ is balanced if and only if $[I]\cap[K]=\varnothing$.
\end{lemma}

\begin{proof}
	Suppose $[I]\cap[K]=\varnothing$ but $S$ has a swing vertex $w\in I$. Then $([K],[I])$ is a bipartition for $[S]$ and $[w]\in[I]-[K]$. Then, by Lemma \ref{lema.N_[w]=[K]}, we have $N_{[w]}=[K]$, which contradicts the hypothesis that $[S]$ is balanced.
	
	The reciprocal statement has already been proved in Lemma \ref{v.ind.nogemelo.v.clique}.
\end{proof}

The next is the final lemma before proving that $S$ is balanced if and only if $[S]$ is balanced as well. This reaffirms something already seen in example \eqref{[I]&[K]no.vacio.contraejemplo}: when $S$ is unbalanced, the class of a swing vertex may prevent $[I]\cup[K]$ from being a partition in the quotient.

\begin{lemma}
	\label{w.interc.=>[I]&[K]<{[w]}}
	Let $(S,K,I)$ be a split graph. If $S$ has a swing vertex $w$, then $[I]\cap[K]\subseteq\{[w]\}$.
\end{lemma}

\begin{proof}
	Assuming without loss of generality that $w\in I$, it suffices to show that $[I-w]\cap[K]=\varnothing$, since
	\[[(I-w)\cup w]\cap[K]=([I-w]\cap[K])\cup(\{[w]\}\cap[K]).\] 
	If $[x]\in[I-w]\cap[K]$, then there exist vertices $u\in I-w$ and $v\in K$ such that $[u]=[x]=[v]$. Since $N_w=K$ and $w\notin\{u,v\}$, it follows that $w\in N_v-u=N_u-v$. But then $uw\in S$, which is absurd.
\end{proof}

\begin{theorem}
	If $S$ is a split graph, then $[S]$ is balanced if and only if $S$ is balanced.
\end{theorem}

\begin{proof}
	We already know that $[S]$ is balanced if $S$ is, by Lemma \ref{S.balanc=>[S].balanc}. Now suppose that $[S]$ is balanced and that $(S,K,I)$ has a swing vertex $w\in I$. Then, combining Lemmas \ref{S.balanc.iff.[I]&[K]=vacio} and \ref{w.interc.=>[I]&[K]<{[w]}}, we deduce that $[I]\cap[K]=\{[w]\}$. Consequently, $V[S]=[I]\dot{\cup}([K]-[w])=([I]-[w])\dot{\cup}[K]$. Since $[S]$ is balanced, we have $|[I]|=|[I]-[w]|=|[I]|-1$, which is absurd.
\end{proof}


\section{Quotient, diameter, and active graphs} \label{sec:cociente,diametro,activos}

Let $P$ be an induced path in a graph $G$ connecting vertex $u$ to vertex $v$. Observe that if $u\cong v$ and $uv\notin G$, then $\Vert P\Vert =2$. Indeed, if $\Vert P\Vert \geq 3$, there would be (without loss of generality) a vertex $w$ adjacent to $u$ but not to $v$, contradicting $u\cong v$. With this in mind, we easily obtain the following lemma.

\begin{lemma}
	\label{prop.gem.en.caminos}
	Let $P$ be an induced path in a graph $G$. If $P$ contains two twin vertices $u$ and $v$, then $u$ and $v$ are the endpoints of $P$ and $\Vert P\Vert\in\{1,2\}$.
\end{lemma}

\begin{proof}
	First suppose $uv\in G$, but $\Vert P\Vert \geq 2$. If $w\neq v$ is adjacent to $u$ in $P$, then $vw\notin G$, contradicting $u\cong v$. Then $P=uv$.
	
	If instead $uv\notin G$, then from the previous observation we can write $P=a\ldots uwv\ldots b$, for some $a,w,b\in V(G)$. If $x\neq w$ is a neighbor of $u$ in $P$, then $vx\notin G$, again contradicting $u\cong v$. Therefore, $P=uwv$.
\end{proof}

\begin{lemma}
	\label{P.camino.ind.long3omás=>[P].camino.ind.long3omás}
	Let $n\geq 4$. If $P=v_1v_2\ldots v_n$ is an induced path in $G$, then $[v_1][v_2]\ldots[v_n]$ is an induced path in $[G]$.
\end{lemma}

\begin{proof}
	If $n=4$, it follows from Lemma \ref{prop.gem.en.caminos} that $P$ does not contain any twins in $G$. Hence, $[V(P)]=\{[v_i]:i\in[4]\}$ and $|[V(P)]|=4$. Moreover, combining Lemmas \ref{prop.gem.en.caminos} and \ref{lema.fundamental.uv.[u][v]}, we have that $[v_1][v_3], [v_1][v_4]$, and $[v_2][v_4]$ are not edges of $[G]$. Also, we see that $[v_i][v_{i+1}]\in E[G]$ for every $i\in[3]$, again by Proposition \ref{lema.fundamental.uv.[u][v]}. Thus, $[v_1]\ldots[v_4]$ is an induced path in $[G]$.
	
	If $n\geq 5$, simply apply the $n=4$ case to the induced paths $v_iv_{i+1}v_{i+2}v_{i+3}$, for $i=2,3,\ldots,n-3$.
\end{proof}

A \textbf{geodesic} between two vertices $u$ and $v$ in a graph $G$ is a path in $G$ of length $dist_G(u,v)$. In other words, a geodesic from $u$ to $v$ in $G$ is a shortest path between $u$ and $v$ in $G$. Every geodesic between $u$ and $v$ in $G$ is an induced path, but not every induced path from $u$ to $v$ is a geodesic.

Lemma \ref{P.camino.ind.long3omás=>[P].camino.ind.long3omás} essentially tells us that the quotient by twins preserves the length of every induced path in $G$ of length 3 or more. Clearly, this also applies to geodesics. Consider two vertices $u,v\in G$ such that $dist(u,v)=r\geq 3$. If $u\ldots v$ is a geodesic in $G$, then $P=[u]\ldots[v]$ is an induced path of length $r$ in $[G]$. Hence, $dist([u],[v])\leq r$. Now suppose that $P$ is not a geodesic between $[u]$ and $[v]$ in $[G]$. Then, there exists a path $R\preceq[G]$ connecting $[u]$ and $[v]$ of length $r'<r$. Since the isomorphism preserves induced subgraphs, it follows that $\varphi(R)$ (see Proposition \ref{isomorfismo.Q=phi[G]}) is an induced path in $\varphi[G]$ of length $r'$ connecting $u$ and $v$. As $\varphi[G]\preceq G$, we have by the definition of $dist(\ast)$ that $r\leq r'$, a contradiction. Therefore, $dist([u],[v])=r$. We formalize this below.

\begin{theorem}
	\label{cociente.preserva.dist>=3}
	Let $G$ be a graph and let $[u]$ and $[v]$ be two distinct vertices in $[G]$. If $dist(u,v)\geq 3$, then
	\[ dist([u],[v]) = dist(u,v). \]
\end{theorem}

\begin{proof}
	It follows from the previous discussion.
\end{proof}

Using similar arguments, we easily obtain the next corollary. But first, note that if a graph $G$ contains at most one isolated vertex, then the quotient “distributes” over each connected component of $G$. That is, if $G$ has $k$ components $G_i$, then
\begin{equation}
	\label{eq42}
	[G]=\dot{\bigcup}_{i=1}^k [G_i].
\end{equation}

\begin{corollary}
	\label{diam[G]=diam(G)}
	Let $G$ be a graph different from $\overline{K_n}$, for $n\geq 2$. If $diam(G)\geq 3$, then 
	\[diam[G] = diam(G).\]
\end{corollary}

\begin{proof}
	Since $G\not\approx\overline{K_n}$ for $n\geq 2$, it is clear from \eqref{eq42} that $[G]$ is disconnected if $G$ is.  
	
	If $G$ is connected and if $u\ldots v$ is a longest geodesic in $G$, then $d=dist([u],[v])=dist(u,v)=diam(G)$, by Theorem \ref{cociente.preserva.dist>=3}. Then, $d'=diam[G]\geq d$. Suppose now that $d'>d$. Then, there exists a geodesic $P$ in $[G]$ of length $d'$. Since the isomorphism preserves induced subgraphs, it follows that $\varphi(P)$ (see Proposition \ref{isomorfismo.Q=phi[G]}) is an induced path in $\varphi[G]$ of length $d'$. As $\varphi[G]\preceq G$, we have by the definition of $diam(\ast)$ that $d'\leq d$, which is absurd. Thus, $d'=d$.
\end{proof}

If $G$ is an active graph, it is easy to see that $[G]$ is not necessarily active. For instance, $[C_4]=K_2$ and $[2K_2]=2K_1$. What can we say about the converse? Below, we see that it is true.

\begin{proposition}
	\label{[G].activo=>G.activo}
	Let $G$ be a graph. If $[G]$ is active, then $G$ is active.    
\end{proposition}

\begin{proof}
	Given $v\in V(G)$, we will show that $v\in act(G)$. Since $[G]$ is active, all its vertices are active. In particular, $[v]$. Then, there exist vertices $[a], [b]$, and $[c]$ in $[G]$ such that the 2-switch $\binom{[a] \ [v]}{[b] \ [c]}$ is active in $[G]$. Hence: 
	\begin{enumerate}
		\item $[a][v],[b][c]\in E[G]$;
		\item $[a][b],[v][c]\notin E[G]$;
		\item $\{[a],[v]\}\cap\{[b],[c]\}=\varnothing$.
	\end{enumerate}
	Thanks to (1) and (2), it follows by Proposition \ref{lema.fundamental.uv.[u][v]} that $av,bc\in E(G)$ and $ab,vc\notin E(G)$. From (3), we immediately deduce that $[a],[v]\not\in\{[b],[c]\}$, hence $a,v\notin\{b,c\}$, i.e., $\{a,v\}\cap\{b,c\}=\varnothing$. Consequently, $\binom{a \ v}{b \ c}$ is an active 2-switch in $G$, which shows that $v\in act(G)$.
\end{proof}

\begin{theorem}
	Let $G$ be a graph with diameter $\geq 4$. If no connected component of $G$ is complete, then $[G]$ is active if and only if $G$ is active.
\end{theorem}

\begin{proof}
	We already know from Proposition \ref{[G].activo=>G.activo} that if $[G]$ is active, then $G$ is active. 
	
	For the converse, suppose $G$ is an active graph with diameter $d\geq 4$ and no complete components. By Theorem \ref{diam[G]=diam(G)}, we have $diam[G]=d$. Since $d\geq 4$, it follows from Corollary \ref{inactive.implies.diam<=3.contrarr} that $[G]$ is active, as $[G]$ has no isolated vertices. Indeed, if $[x]$ is isolated in $[G]$, then necessarily $H=\langle [x]\rangle_G$ is disjoint from $G$ and either empty or complete. Clearly, $H$ cannot be empty because $G$ is active. But then $H$ is complete, contradicting the hypothesis.
\end{proof}

If $G$ is an active graph with diameter 3, then $[G]$ may not be active. Figure \ref{G.activo.diam=3.pero.[G].no} shows a counterexample.

\begin{figure}[h]
	\centering
	\includegraphics[scale=0.8]{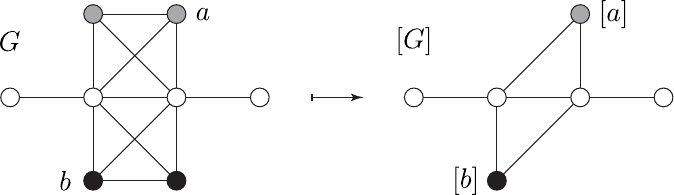}
	\caption{$G$ is active and $diam(G)=3$, but $[G]$ is not active.}
	\label{G.activo.diam=3.pero.[G].no}
\end{figure}

The counterexample in Figure \ref{G.activo.diam=3.pero.[G].no} is unfortunately not isolated. In fact, we can generate infinitely many similar counterexamples by simply replacing, in $G$, the two $K_4$'s attached to the central $P_4$ with any number $m$ of $K_n$'s. It is easy to verify that for all $m\geq 2$ and $n\geq 4$, $G$ remains active and with diameter 3, and $[G]$ still has inactive vertices (in Figure \ref{G.activo.diam=3.pero.[G].no}, the vertices $[a]$ and $[b]$ are inactive in $[G]$).

Every connected split graph has diameter at most 3. In particular, every active split graph has diameter 3. The following result extends Theorem \ref{diam[G]=diam(G)} to split graphs.

\begin{theorem}
	If $S$ is a split graph, then $S$ is active if and only if $[S]$ is active.
\end{theorem}   

\begin{proof}
	Thanks to Proposition \ref{[G].activo=>G.activo}, we already know that $S$ is active if $[S]$ is.
	
	To prove the converse, suppose that $S$ is active but $[S]$ is not. In other words, $[S]$ has an inactive vertex $[x]$. Since $[S]$ is split, this means that there are no induced $P_4$’s in $[S]$ containing $[x]$. Given that $x\in act(S)$, $x$ is contained in a path $P\in Q_S(P_4)$, because $S$ is split. But then, by Proposition \ref{P.camino.ind.long3omás=>[P].camino.ind.long3omás}, we have that $[P]_S$ is an induced $P_4$ in $[S]$ containing $[x]$, a contradiction.
\end{proof}


\section{Quotient and composition} \label{sec:cociente.composicion}

As announced in the introduction of this chapter, in this section we study how the quotient by twins relates to Tyshkevich's composition. The main result we reach is that $[S\circ G]=[S]\circ [G]$. To prove this, we first need to establish a few preliminary lemmas.

\begin{lemma}
	\label{s.no-gemelo.g.en.SoG}
	Let $(S,K,I)$ be a balanced split graph and let $G$ be a graph. If $s\in V(S)$ and $g\in V(G)$, then $s$ and $g$ cannot be twins in $S\circ G$.
\end{lemma}

\begin{proof}
	For simplicity, we will use $N_x$ instead of $N_{S\circ G}(x)$. Suppose that $N_s-g=N_g-s$. Since $K-s\subseteq N_g-s$ by definition of $\circ$, it follows that $K-s\subseteq N_s$, and hence $s\in K$. Moreover, $N_g\cap I=\varnothing$ by definition of $\circ$. Then, $N_s=K-s$, which means that $s$ is swing in $S$, a contradiction.
\end{proof}

\begin{lemma}
	\label{[S]=[S]_H,[G]=[G]_H}
	Let $S$ be a balanced split graph and let $G$ be a graph. If $H=S\circ G$, then $[S]=[S]_H$ and $[G]=[G]_H$.
\end{lemma}

\begin{proof}
	Since $S\preceq H$, it follows from Proposition \ref{homomorfismo:[H]--->[H]_G} that $\phi[S]\preceq [S]_H$, where $\phi:[S]\rightarrow[S]_H$ is the homomorphism defined by $\phi([x]_S)=[x]_H$. From Proposition \ref{gem.G.entonces.gem.H<G}, we have that $[x]_H\subseteq [x]_S$ for all $x\in V(S)$. Now suppose there exist vertices $u,v\in S$ such that $u\in [v]_S-[v]_H$. Then,
	\[ N_S(u)-v=N_S(v)-u, \ N_H(u)-v\neq N_H(v)-u. \]
	Let $(K,I)$ be the bipartition of $S$. It follows from Lemma \ref{v.ind.nogemelo.v.clique} that $u,v\in K$ or $u,v\in I$. If $u,v\in I$, then
	\[ N_H(u)=N_S(u)=N_S(v)=N_H(v), \]
	a contradiction. If $u,v\in K$, then
	\[ N_H(u)-v=(N_S(u)-v)\dot{\cup}V(G)=(N_S(v)-u)\dot{\cup}V(G)=N_H(v)-u, \]
	another contradiction. Therefore, we conclude that $[x]_S=[x]_H$ for all $x\in V(S)$, so $V[S]\subseteq V([S]_H)=[V(S)]_H$. If there existed $a\in V(S)$ such that $[a]_H\in [V(S)]_H-V[S]$, then clearly $[a]_H\neq [x]_S$ for all $x\in V(S)$. In particular, this would imply $[a]_H\neq[a]_S$, contradicting what we just proved. Hence, $[S]=[S]_H$.
	
	We obtain $[G]=[G]_H$ using the same reasoning as for $S$: the only detail that differs is that $N_H(g)=N_G(g)\dot{\cup}K$ for all $g\in V(G)$.
\end{proof}

\begin{lemma}
	\label{V[SoG]=...}
	If $G$ is a graph and $S$ is a balanced split graph, then $V[S\circ G]=V[S]\dot{\cup}V[G]$.
\end{lemma}

\begin{proof}
	Let $H=S\circ G$. If $[x]_H\in[V(S)]_H\cap[V(G)]_H$, then there exist vertices $s\in S$ and $g\in G$ such that $[s]_H=[x]_H=[g]_H$. This contradicts Lemma \ref{s.no-gemelo.g.en.SoG}. Then,
	\[ V[S\circ G]=[V(S)]_H\dot{\cup}[V(G)]_H. \]
	Finally, from Lemma \ref{[S]=[S]_H,[G]=[G]_H} we get that
	\[ [V(S)]_H=[V(S)]_S=V[S], \ [V(G)]_H=[V(G)]_G=V[G]. \]
\end{proof}

We now use the above lemmas to establish the main result of this section.

\begin{theorem}
	\label{[S°G]=[S]°[G]}
	If $G$ is a graph and $(S,K,I)$ is a balanced split graph, then 
	\[ [S\circ G]=[S]\circ [G]. \]
\end{theorem} 

\begin{proof}
	Let $H=S\circ G$ and $X=[S]\circ [G]$. It is clear that $V(X)=V[S]\dot{\cup}V[G]$ by the definition of $\circ$. Then, we immediately derive from Lemma \ref{V[SoG]=...} that $V[H]=V(X)$.
	
	Now we need to show that $E[H]=E(X)$. Clearly, every edge of $[S]_H$ or $[G]_H$ belongs to $[H]$, since $[S]_H,[G]_H\preceq [H]$. Then, every edge of $[S]$ or $[G]$ is also an edge of $[H]$ by Lemma \ref{[S]=[S]_H,[G]=[G]_H}. If $e$ is an edge in $X$ connecting $[S]$ and $[G]$, then $e=[s]_S[g]_G$, where $[s]_S\in [K]_S$ and $[g]_G\in [G]$ (the fact that $[K]_S$ is the clique of $[S]$ comes from Lemma \ref{v.ind.nogemelo.v.clique}). Clearly, $s\in S$ and $g\in G$. More specifically, by Lemma \ref{v.ind.nogemelo.v.clique}, $s\in K$. Then $sg\in H$ by the definition of $\circ$, and thus $[s]_H[g]_H\in [H]$ by Proposition \ref{lema.fundamental.uv.[u][v]}. Since $[s]_S=[s]_H$ and $[g]_G=[g]_H$ by Lemma \ref{[S]=[S]_H,[G]=[G]_H}, we conclude that $e\in [H]$. Therefore, $E(X)\subseteq E[H]$.
	
	Now let us prove the reverse inclusion. If $e=xy\in E[H]$, then by Lemma \ref{V[SoG]=...} we distinguish 3 cases: 1) $x,y\in V[S]$; 2) $x,y\in V[G]$; 3) $x\in V[S], y\in V[G]$. 
	
	For case (1), we can write $x=[a]_S$ and $y=[b]_S$ for some $a,b\in V(S)$. By Lemma \ref{[S]=[S]_H,[G]=[G]_H}, we have $[a]_S=[a]_H$ and $[b]_S=[b]_H$, so $x,y\in [S]_H$. Since $[S]_H\preceq [H]$, it follows that $e\in [S]_H$. As $[S]_H=[S]$ by Lemma \ref{[S]=[S]_H,[G]=[G]_H}, we conclude that $e\in X$, since $[S]\preceq X$. The same reasoning applies to case (2).
	
	In case (3), we can write $x=[a]_S$ and $y=[b]_G$ for some $a,b\in V(H)$, where $a\in S$ and $b\in G$. By Lemma \ref{[S]=[S]_H,[G]=[G]_H}, we have $[a]_S=[a]_H$ and $[b]_G=[b]_H$. Since $e\in [H]$, it follows that $ab\in H$ by Proposition \ref{lema.fundamental.uv.[u][v]}. Then $a\in K$ by the definition of $\circ$, and hence $[a]_S\in [K]_S$ by Lemma \ref{v.ind.nogemelo.v.clique}. Finally, $e=[a]_S[b]_G\in X$ by the definition of $\circ$.
	
	After analyzing all 3 cases, we conclude that $E[H]\subseteq E(X)$. Therefore, $[H]=X$.
\end{proof}

\begin{corollary}
	Let $G$ be a graph and let $(S,K,I)$ be a split graph. If $[S]$ and $[G]$ are active, then:
	\begin{enumerate}
		\item $S$ is balanced and $[S\circ G]=[S]\circ[G]$.
		\item $S,G,S\circ G$, and $[S\circ G]$ are active.
	\end{enumerate}
\end{corollary}

\begin{proof}
	\begin{enumerate}[(1).]
		\item If $[S]$ is active, then $[S]$ is balanced. Therefore, we can apply Theorem \ref{[S°G]=[S]°[G]}.
		\item The graph $[S\circ G]$ is active thanks to Theorems \ref{SoG.act.iff.S,G.act} and \ref{[S°G]=[S]°[G]}. Then $S\circ G$ is active by Proposition \ref{[G].activo=>G.activo}. Finally, $S$ and $G$ are active by Theorem \ref{SoG.act.iff.S,G.act}.
	\end{enumerate}
\end{proof}


\section{Quotient index} \label{sec:indice.cociente}

Once we take the quotient by twins of a graph $G$, nothing prevents us from quotienting $[G]$ again, obtaining the graph $[ [G] ]$, and so on. For this, it will be more convenient to use the notation $[G]^2$ instead of $[[G]]$ and, in general, $[G]^k$ to indicate that the quotient has been iterated $k$ times on $G$. Naturally, $[G]^1=[G]$ and $[G]^0=G$.

We say that a graph $G$ is \textbf{twin-free} if $G\approx [G]$ (that is: $N_G(u)-v\neq N_G(v)-u$ for every pair of distinct vertices $u,v\in V(G)$). It is easy to find examples of graphs $G$ in which there are still pairs of twin vertices in $[G]$, or even in $[G]^2$ or $[G]^3$. This motivates the creation of a parameter that measures “how far” $G$ is from being twin-free: the \textbf{quotient index} of $G$, denoted $i(G)$. For $k\geq 0$, we say that $i(G)=k$ if $k$ is the smallest integer such that $[G]^k \approx [G]^{k+1}$. Therefore, $G$ is twin-free if and only if $i(G)=0$. This parameter is obviously finite, since every time we quotient a graph with twins, we remove at least one vertex in the process. From Proposition \ref{cociente.y.compl.conmutan} and the definition of $i(\ast)$ it follows immediately that $i(G)=i(\overline{G})$.

\begin{proposition}
	\label{familia.i(G_n)=n-1}
	For every natural number $n$, there exists a graph $G_n$ of order $n$ such that $i(G_n)=n-1$. In particular, the inequality
	\begin{equation}
		\label{i(G)<n}
		i(G)\leq |G|-1
	\end{equation}
	holds for every graph $G\neq K_0$, and it is sharp.
\end{proposition}

\begin{proof}
	If $V(G_n)=\{v_i:i\in[n]\}$, consider the threshold graph $G_n$ defined by the following binary sequence:
	\[ \beta_n = 
	\begin{cases}
		00101010101\ldots, &\text{if } n \text{ is odd}, \\
		010101010101\ldots, &\text{if } n \text{ is even},
	\end{cases}\] 
	where digit $b_i$ corresponds to vertex $v_i$ ($b_1$ is the leftmost digit of $\beta_n$). Let $V_n^0=\{v_i: i \text{ is even}\}$ and $V_n^1=\{v_i: i \text{ is odd}\}$. Since $G_n$ is split, it is easy to see that when $n$ is odd, $(V_n^1,V_n^0)$ is a bipartition for $G_n$. On the other hand, when $n$ is even, $(V_n^0,V_n^1)$ is a bipartition for $G_n$. Other key observations involve the degree sequence of $G_n$:
	\begin{enumerate}[(1).]
		\item $d_{i+2}-d_{i}=1$, if $v_i$ and $v_{i+2}$ are clique vertices;
		\item $d_i - d_{i+2}=1$, if $v_i$ and $v_{i+2}$ are independent vertices;
		\item $v_n$ is universal and $v_{n-1}$ is a leaf;
		\item $d_1=\lfloor n/2 \rfloor = d_2$;
		\item $v_1$ and $v_2$ are the only vertices with the same degree.
	\end{enumerate}
	Thanks to (5) and Proposition \ref{inactivos.mismo.deg.son.gemelos}, we deduce that $\{v_1,v_2\}$ is the only pair of twins in $G_n$. Therefore, $[G_n]\approx G_n - v_1$ for $n\geq 2$. But the operation $G_n - v_1$ corresponds in $\beta_n$ to replacing $b_1b_2$ with a 0, so that the binary sequence of $G_n - v_1$ becomes exactly $\beta_{n-1}$. Hence, $G_n - v_1 \approx G_{n-1}$, from which $[G_n]^k \approx G_{n-k}$ for $k\geq 0$, and consequently $i(G_n)=n-1$.
\end{proof}

Let $\{G_n:n\in\mathbb{N}\}$ be the family of graphs constructed in Proposition \ref{familia.i(G_n)=n-1}. Since $\deg(G_n)=0$ and $diam(G_n)\leq 2$ for all $n$, it is natural to ask whether the inequality \eqref{i(G)<n} can be improved if $\deg(G)\geq 1$ or if $diam(G)\geq 3$. Obviously, $|G|\geq 4$ if $G$ is active. With this in mind, and assuming that $G$ is a connected graph with diameter $\geq 3$, what can be quickly deduced using Corollary \ref{diam[G]=diam(G)} is that $i(G)\leq |G|-4$, since every graph of order $\leq 3$ has diameter $\leq 2$ or $\infty$.

Using the next two results, we precisely describe the quotient index of a tree.

\begin{lemma}
	\label{gem.tree=leaves}
	Let $T$ be a tree. If $u\cong_T v\neq u$, then $u$ and $v$ are leaves.
\end{lemma}

\begin{proof}
	Since the result is trivially true for $|T|\leq 2$, we can assume $|T|\geq 3$. Suppose $\deg_T(u)\geq 2$. Then $u$ has a neighbor $x$ that is not part of the $u \ldots v$ path. Since $N_u-v=N_v-u$, it follows that $vx\in T$. But then $xu \ldots vx$ forms a cycle in $T$, a contradiction.    
\end{proof}

\begin{theorem}
	\label{index.trees}
	Let $T$ be a tree and let $d=diam(T)$.
	\begin{enumerate}
		\item If $d\leq 2$, then $i(T)\leq 3$.
		\item If $d\geq 3$, then $i(T)\leq 1$.
		\item A tree not isomorphic to $K_2$ is twin-free if and only if each of its vertices has at most one leaf in its neighborhood.
	\end{enumerate}
\end{theorem}

\begin{proof}
	\begin{enumerate}[(1).]
		\item If $|T|\leq 3$, it's easy to check by inspection that $i(T)\leq 2$. If $|T|\geq 4$, then $d\leq 2$ implies that $T$ is a star, which has index 3.
		\item Suppose $[T]$ contains two distinct twins $[u]$ and $[v]$. Since $diam[T]=d$ by Corollary \ref{diam[G]=diam(G)}, we have by Proposition \ref{isomorfismo.Q=phi[G]} that $[T]$ is a tree. Then $[u]$ and $[v]$ are leaves, by Lemma \ref{gem.tree=leaves}. 
		
		If $dist(u,v)\geq 3$, then by Theorem \ref{cociente.preserva.dist>=3} we get $dist(u,v)=dist([u],[v])=2$, a contradiction. Hence, $dist(u,v)=2$, since by Proposition \ref{lema.fundamental.uv.[u][v]} we have $uv\notin T$. Let $N_{[u]}=\{[x]\}=N_{[v]}$. Since $N_u-v\neq N_v-u$, there exists a vertex $a\notin uxv$ such that $au\in T$ and $av\notin T$. Since $auxv\subseteq T$ and $T$ is acyclic, we have $auxv\preceq T$. It follows by Lemma \ref{P.camino.ind.long3omás=>[P].camino.ind.long3omás} that $[a][u][x][v]\preceq[T]$. But then $[u]$ is not a leaf in $[T]$, which is a contradiction.
		\item It follows directly from Lemma \ref{gem.tree=leaves}.   
	\end{enumerate} 
\end{proof}

Just as we did for trees, we now aim to characterize the quotient index of unicyclic graphs. Intuitively, this parameter should behave in unicyclic graphs almost the same as in trees. This turns out to be true, but requires more effort to prove.

\begin{proposition}
	\label{[G+-{xy:...}]=[G]}
	Let $G$ be a graph and let $[v]\in V[G]$ with $|[v]|\geq 2$. Then
	\[[G-\{xy:x,y\in[v]\}] =[G]= [G+\{xy:x,y\in[v]\}].\]
	In particular,
	\[i(G-\{xy:x,y\in[v]\}) =i(G)= i(G+\{xy:x,y\in[v]\}).\]
\end{proposition}

\begin{proof}
	We know that every equivalence class (by twins) in $G$ is either a clique or an independent set in $G$.
	
	Suppose first that $[v]=[v]_G$ is a clique. Let $H=G-\{xy:x,y\in[v]\}$ and let $u\in V(G)$. Since $N_H(u)=N_u-[v]$ for all $u\in V(G)=V(H)$, it is easy to see that any pair of vertices that were twins (or not) in $G$ remain so in $H$. In other words, $[u]_H=[u]$ for all $u$, and hence $V[H]=V[G]$.
	
	If $e=[a][b]$ were an edge in $E[G]-E[H]$, then $ab\in G$ and $a\notin [v]$. But then $ab\in H$ and $e\in [H]$ by Proposition \ref{lema.fundamental.uv.[u][v]}, a contradiction. Thus, $E[G]\subseteq E[H]$. If $e=[a][b]\in E[H]$, then $ab\in H\subset G$ and $e\in [G]$ by Proposition \ref{lema.fundamental.uv.[u][v]}, so $E[H]\subseteq E[G]$.
	
	If $[v]$ is an independent set, we consider $\overline{G}$ and apply the same reasoning. 
\end{proof}

\begin{lemma}
	\label{gem.unicyclic.C>4}
	Let $U$ be a unicyclic graph with a cycle of order $\geq 5$. If $a\cong_U b\neq a$, then $a$ and $b$ are leaves.
\end{lemma}

\begin{proof}
	Let $C$ be the cycle of $U$. Since $|C|\geq 5$, there exists an edge $e\in C$ such that $N_U(a)=N_T(a)$ and $N_U(b)=N_T(b)$, where $T=U-e$. This means $a\cong_T b$. Since $T$ is a tree, $a$ and $b$ must be leaves in $T$ by Lemma \ref{gem.tree=leaves}. Because removing $e$ does not affect the neighborhoods of $a$ and $b$, we conclude that both vertices are also leaves in $U$.
\end{proof}

\begin{lemma}
	\label{gem.unicyclic.C<5}
	Let $U$ be a unicyclic graph with a cycle $C$ of order $\leq 4$, and suppose $a\cong_U b\neq a$. Then only one of the following cases may occur:
	\begin{enumerate}
		\item $a,b\in C$ and $d_a=2=d_b$; 
		\item $a,b\notin C$ and $d_a=1=d_b$.
	\end{enumerate}
\end{lemma}

\begin{proof}
	Let $C$ be the cycle of $U$. If $a\notin C$ or $b\notin C$, we may repeat the argument used in Lemma \ref{gem.tree=leaves} to deduce $d_a=1=d_b$. The same holds if $ab\in C\approx C_4$. It remains to consider: 1) $ab\in C\approx K_3$; 2) $a,b\in C\approx C_4, ab\notin C$. But in both cases it is clear that $d_a=2=d_b$ since otherwise $U$ would have more than one cycle. 	
\end{proof}

\begin{lemma}
	\label{U.with.C<5.no.twins.in.C}
	Let $U$ be a unicyclic graph with a cycle $C$ of order $\leq 4$. If $C$ contains no twins, then $i(U)\leq 1$.
\end{lemma}

\begin{proof}
	If $C$ has no twins and $U$ is not twin-free, then by Lemma \ref{gem.unicyclic.C<5} every pair of twins in $U$ must be a pair of leaves outside $C$. Moreover, the fact that $C$ has no twins implies there exist vertices $a,b\in C$ such that $ab\in U$ and $\deg_U(a),\deg_U(b)\geq 3$. Hence, $diam(U)\geq 3$. 
	
	Suppose $[x]$ and $[y]$ are distinct twins in $[U]$. Although $a$ and $b$ may lose neighbors in the quotient, they still retain at least one neighbor outside $C$ afterward. So both $[a]$ and $[b]$ still have degree $\geq 3$ in $[U]$. Therefore, as $[U]$ is a unicyclic graph with a cycle of the same length as $U$, Lemma \ref{gem.unicyclic.C<5} implies $[x]$ and $[y]$ are leaves in $[U]$. Note that $[a][b]$ is an edge in the cycle of $[U]$ and $[U-ab]=[U]-[a][b]$. Since $U-ab$ is a tree with $diam\geq 3$, Theorem \ref{index.trees} tells us that $[U-ab]$ is twin-free. However, we reach a contradiction, because $[x]$ and $[y]$ are twins in $[U]-[a][b]$.   
\end{proof}

\begin{theorem}
	Let $U$ be a unicyclic graph with diameter $d$ and let $C$ be its cycle. 
	\begin{enumerate}
		\item If $|C|\neq 4$ and $d\geq 3$, then $i(U)\leq 1$.
		\item If $|C|=4$ and $U\not\approx C_4$, then $i(U)\leq 1$.
		\item If $U\approx C_4$, then $i(U)=2$.
		\item If $|C|=3$ and $d=2$, then $i(U)=3$.
		\item If $U\approx K_3$, then $i(U)=1$.
	\end{enumerate}
\end{theorem}

\begin{proof}
	Recall that by Lemma \ref{gem.unicyclic.C<5}, each pair of twins in $U$ is either inside $C$ or in $U-V(C)$.
	\begin{enumerate}[(1).]
		\item By Lemma \ref{gem.unicyclic.C>4}, we only need to handle the case $|C|=3$. If $C=abca$ contains twins, we may assume $[a]=[b]\neq[c]$ without loss of generality (as $d\geq 3$ implies $U\not\approx K_3$). By Proposition \ref{[G+-{xy:...}]=[G]}, we get $i(U)=i(U-ab)$. Since $U-ab$ is a tree with diameter $d$, we have $i(U-ab)\leq 1$, by Theorem \ref{index.trees}. If instead $C$ contains no twins, we use Lemma \ref{U.with.C<5.no.twins.in.C}.
		\item If $|C|=4$ and $U\not\approx C_4$, then clearly $d\geq 3$. Hence, at most one pair of twins is in $V(C)$. If $C=xyztx$ contains twins, we may assume $[x]=[z]$. Observe that $[U]=[U-x]$. Since $U-x$ is a tree, $i[U-x]=0$ by Theorem \ref{index.trees}, so $i(U)\leq 1$. If instead $C$ has no twins, we use Lemma \ref{U.with.C<5.no.twins.in.C}.
		\item Immediate to verify.
		\item In this case, it is easy to see that $U$ must be isomorphic to a $K_3$ with $k\geq 1$ leaves attached to one of its vertices.
		\item Immediate to verify.
	\end{enumerate}   
\end{proof}

We now conclude this section with the two main theorems that describe the quotient index of a balanced split graph and of a composition.

\begin{theorem}
	\label{i(split)<2}
	If $S$ is balanced, then $i(S)\leq 1$.
\end{theorem}  

\begin{proof}
	If $S$ is twin-free, we are done. If $(S,K,I)$ contains twins, suppose $[S]$ has a pair of distinct vertices $[u]$ and $[v]$ such that $[u]\cong_{[S]}[v]$. By Lemma \ref{v.ind.nogemelo.v.clique}, we know that either $\{[u],[v]\}\subseteq[I]$ or $\{[u],[v]\}\subseteq[K]$, since $[S]$ is balanced by Theorem \ref{S.balanc=>[S].balanc}.
	
	If $[u],[v]\in[I]$, then $u,v\in I$ by Lemma \ref{v.ind.nogemelo.v.clique}. Since $N_u\neq N_v$, there exists, without loss of generality, a vertex $x\in K$ such that $ux\in S$ but $vx\notin S$, so $[u][x]\in[S]$ and $[v][x]\notin[S]$ by Proposition \ref{lema.fundamental.uv.[u][v]}. Therefore $N_{[u]}\neq N_{[v]}$, contradicting the assumption that $[u]$ and $[v]$ are twins in $[S]$.
	
	Then $[u]$ and $[v]$ must belong to $[K]$. But $[K]$ is the maximum independent set in $\bar{[S]}$, where $[u]$ and $[v]$ remain twins by Proposition \ref{gem.G.iff.gem.G^c}. So we can repeat the previous arguments and reach another contradiction. Therefore, we conclude that $[S]$ is twin-free, i.e., $i(S)=1$.
\end{proof}

\begin{theorem}
	\label{i(S°G)=max(i(S),i(G))}
	If $S$ is a balanced split graph and $G$ is a graph, then $i(S\circ G)=\max\{i(S),i(G)\}$. In particular, $i(S\circ G)=0$ if and only if $i(S)=0=i(G)$.
\end{theorem}

\begin{proof}
	Let $H=S\circ G, i(S)=s, i(G)=g$ and $i(H)=h$. By Theorem \ref{[S°G]=[S]°[G]}, we have $[H]^h=[S]^h\circ[G]^h$. Since $i([H]^h)=0$, it follows that $i([S]^h)=0=i([G]^h)$. Then $s,g\leq h$, and hence $\max\{s,g\}\leq h$.
	
	To obtain the reverse inequality, we consider two cases: 1) $g\geq 1$; 2) $g=0$. Suppose $g\geq 1$. By Theorem \ref{i(split)<2}, we have $s\leq g$. Applying Theorem \ref{[S°G]=[S]°[G]}, we get $[H]^{g+1}=[S]^{g+1}\circ [G]^{g+1}\approx [S]^g\circ[G]^g=[H]^g$. So $h\leq g=\max\{s,g\}$. If $g=0$, then again by Theorem \ref{[S°G]=[S]°[G]}, we get $[H]^{s+1}=[S]^{s+1}\circ [G]^{s+1}\approx [S]^s\circ[G]^s=[H]^s$. Then $h\leq s=\max\{s,g\}$.
\end{proof}


\chapter{Graphs associated with split graphs} \label{cap:grafos.asociados.a.split}

In this chapter, we introduce and study two objects associated with a split graph $S$: the factor graph of $S$ and the flow configuration of $S$. The former, denoted by $\Phi(S)$, is a loopless multigraph, while the latter, denoted by $\vec{\Phi}(S)$, is a digraph. Both are isomorphic to each other as simple graphs (ignoring multiplicities and arc directions), but they encode different combinatorial and structural information about $S$.

This chapter is divided into 5 sections. In the first, we define $\Phi(S)$ and relate its connectedness to that of $A_4(S)$, obtaining an analogue of Theorem \ref{indecomp.characterization} for $\Phi$, and an analogue of Theorem \ref{tyshk.decomp} for twin-free graphs.

In the second section, we study the multigraph $\Phi$ under the assumption that it is a simple and connected graph. We show that this condition greatly restricts the structure of both $\Phi(S)$ and $S$. This phenomenon becomes even more pronounced if $\Phi(S)$ is simple and complete. In such cases, using results from Appendix \ref{familias.inserc.conjuntos}, we obtain a very precise description of the possible realizations of $S$, especially when it is active.

In Section \ref{sec:caminos,ciclos.en.Phi}, we focus on investigating the length and multiplicities of the induced paths and cycles of $\Phi$. Two results stand out: 1) induced cycles can only have length 3 or 4; 2) only the terminal edges of an induced path can be simple. Thanks to (2), we obtain an upper bound for $diam(\Phi(S))$ in terms of the degree of $S$.

In Section \ref{sec:config.flujo}, we define the digraph $\vec{\Phi}$ and determine which directed configurations are admissible for its induced cycles. In this context, we also introduce the concept of square-free split graphs.

In the final section, we define and study a very special class of split graphs: linear split graphs. These share a number of very peculiar properties in their flow configurations, which ultimately allow us to prove that $[\Phi(S)]$ is a path when $S$ is linear.


\section{Relationship between $\Phi$ and $A_4$} \label{sec:relacion.Phi-A_4}

Let $(S,K,I)$ be a split graph. The \textbf{factor graph} of $S$ is the loopless multigraph $\Phi(S)$ with $V(\Phi(S))=I$, such that there is an edge between $u$ and $v\neq u$ for each $H\in Q_S(P_4)$ containing both $u$ and $v$.

We aim to investigate how the connectedness of $\Phi=\Phi(S)$ is related to the connectedness of $A_4=A_4(S)$ (see Theorem \ref{indecomp.characterization}) and vice versa. The connectedness of $\Phi$ is equivalent to the existence of a path in $A_4$ between any pair of vertices $u,v\in I$ that only involves vertices from $I$. This is the general idea behind Lemmas \ref{lema.conexionlocal.H7}, \ref{lema.conexionlocal.uxv}, and \ref{lema.conexionlocal.uxz}, which are used to prove that $\Phi$ is connected if $A_4$ is. However, we then note that the converse does not hold in general, unless $S$ is active. Finally, we obtain one of the most important results in this work: if $S$ is active and $\Phi(S)$ is connected, then $S$ is indecomposable.

\begin{lemma}
	\label{lema.conexionlocal.H7}
	Let $(H,K,I)$ be a split graph of order 7 with $I=\{a,b,u,v\}$ and $K=\{x,y,z\}$. If $H_1 =\langle a,x,y,u\rangle_H \approx P_4$ and $H_2 =\langle b,x,z,v\rangle_H \approx P_4$, then:
	\begin{enumerate}
		\item $u$ and $z$ are connected by a path $P$ in $A_4(H)$ such that $V(P)-z\subseteq I$;
		\item $\Phi(H)$ is connected.
	\end{enumerate}
\end{lemma}

\begin{proof}
	For simplicity, we use $N_*,d_*,A_4$, and $\Phi$ to denote $N_H(*)$, $\deg_H(*)$, $A_4(H)$, and $\Phi(H)$, respectively.
	First, observe that $V(H_1)$ and $V(H_2)$ are cliques in $A_4$ sharing only the vertex $x$. To show that $\Phi$ is connected, it suffices to find an edge in $\Phi$ between a vertex in $\{a,u\}$ and a vertex in $\{b,v\}$, since $au, bv\in\Phi$ by hypothesis. Given that $\{H_1,H_2\}\subseteq Q_H(P_4)$, $H_1\cap H_2=x$, and $H$ is active, we have $|N_x\cap \{a,u\}|=|N_x\cap \{b,v\}|=1$, and hence necessarily $d_x=4$. More specifically,
	\[ N_x\cap K=\{y,z\}, \ N_x\cap I\in\{ \{u,b\},\{u,v\},\{a,v\},\{a,b\} \}. \]
	Furthermore, the degree in $H$ of each vertex in $I$ must be 1 or 2.
	
	Without loss of generality, assume $ux,bx\in H$. If $d_b=1$, then $ayxb\preceq H$, because $ax\notin H$. Thus, $uabv\subseteq\Phi$ and $uabz\subseteq A_4$. If $d_u=1$, then $uxzv\preceq H$ because $vx\notin H$. Then $uv\in \Phi$ and $uz\in A_4$. If $d_b=2=d_u$, then $uz,by\in H$ and therefore $uzyb\preceq H$. Hence, $ubv\subseteq\Phi$ and $uz\in A_4$.
\end{proof}

\begin{lemma}
	\label{lema.conexionlocal.uxv}
	Let $(S,K,I)$ be a split graph. Suppose $u,v\in I$ are connected in $A_4 =A_4 (S)$ by the path $uxv$, for some $x\in K$. Then, there exists another path $P\subseteq A_4$ between $u$ and $v$ such that $V(P)\subseteq I$ and $|P|\leq 4$.
\end{lemma}

\begin{proof}
	Since $ux,xv\in A_4$, there exist $H_1,H_2 \in Q_S (P_4 )$ such that $u,x\in H_1$ and $x,v\in H_2$. Let $H_1 =\langle a,u,x,y\rangle_S$ and $H_2 =\langle b,v,x,z\rangle_S$, where $a,b\in I$ and $y,z\in K$. Observe that $V(H_1)$ and $V(H_2)$ are cliques in $A_4$, and $H=\langle V(H_1 )\cup V(H_2 )\rangle_S$ is active. Moreover, with $K'=\{x,y,z\}$ and $I'=\{a,b,u,v\}$, we have that $(K',I')$ is the bipartition of $H$.
	
	Clearly, $ub,av\in A_4$. Thus, $uav\subseteq A_4$ if $a=b$. Otherwise, if $a\neq b$, we consider two cases: 1) $y\neq z$; 2) $y=z$. In both cases we prove that $\Phi(H)$ is connected, which ensures the existence of the desired path $P$. If $y\neq z$, we apply Lemma \ref{lema.conexionlocal.H7}. If $y=z$, then
	\[ |N_H(x)\cap I'|=2=|N_H(y)\cap I'|, \  N_H(x)\cap N_H(y)\cap I'=\varnothing. \]
	Therefore, $\Phi(H)\approx C_4$.
\end{proof}

\begin{lemma}
	\label{lema.conexionlocal.uxz}
	Let $(S,K,I)$ be a split graph. Suppose $u\in I$ and $z\in K$ are connected in $A_4 =A_4 (S)$ by the path $uxz$, for some $x\in K$. Then $u$ and $z$ are connected in $A_4$ by a path $P$ such that $V(P)-z\subseteq I$ and $|P|\leq 4$.
\end{lemma}

\begin{proof}
	If $uz$ is an edge in $A_4$, then there is nothing to prove. So assume $uz\not\in A_4$. Since $ux,xz\in A_4 $, there exist $H_1 ,H_2 \in Q_S (P_4 )$ such that $u,x\in H_1$ and $x,z\in H_2$. Let $H_1 =\langle a,x,y,u\rangle_S$ and $H_2 =\langle b,x,z,v\rangle_S$, with $a,b,v\in I$ and $y\in K$. Observe that $V(H_1)$ and $V(H_2)$ are cliques in $A_4$. In particular, since $\{bz,vz,uy\}\subseteq E(A_4)$, it follows that $u\notin\{b,v\}$ and $y\neq z$, since $uz\notin A_4$ by hypothesis.
	
	Let $H=\langle V(H_1 )\cup V(H_2 )\rangle_S$ be the active split graph with bipartition $(\{x,y,z\}, \{a,b,u,v\})$. The above observations reduce the proof to two cases: 1) $a\in\{b,v\}$; 2) $a\notin\{b,v\}$. For the first case, we have that $uaz$ is the desired path. In case (2), we apply Lemma \ref{lema.conexionlocal.H7}.
\end{proof}

\begin{theorem}
	\label{A_4.conexo.implica.Phi.conexo}
	Let $(S,K,I)$ be a split graph. If $A_4(S)$ is connected, then $\Phi(S)$ is connected.
\end{theorem}

\begin{proof}
	Suppose $A_4$ is connected and choose arbitrary vertices $u$ and $v\in I$. To show that $\Phi$ is connected, we must find a path in $A_4$ from $u$ to $v$ that includes only independent vertices.
	
	Since $A_4$ is connected, there exists a path $P\subseteq A_4$ between $u$ and $v$. If $V(P)\subseteq I$, we are done. Otherwise, $P$ alternates sequences of $I$-vertices and $K$-vertices. It suffices to consider the case $P=ux_1 \ldots x_n v$, where $x_i \in K$ for all $i\in[n]$, and then construct the desired alternative path $P'=u\ldots v\subseteq A_4$ consisting only of independent vertices. We do this by induction on $n$.
	
	The base case $n=1$ is covered by Lemma \ref{lema.conexionlocal.uxv}. For each $m$ such that $1\leq m <n$ we assume the following: for any pair $u',v'\in I$ connected in $A_4$ by a path $u'x_1 \ldots x_m v'$ ($x_i \in K$ for all $i\in[m]$), there exists a path $u'\ldots v'\subseteq A_4$ with $V(u'\ldots v')\subseteq I$. Applying Lemma \ref{lema.conexionlocal.uxz} to $vx_n x_{n-1}$, we find a path in $A_4$ of the form $v\ldots w x_{n-1}$ with $V(v\ldots wx_{n-1})-x_{n-1}\subseteq I$ (possibly $w=v$). Then the inductive hypothesis applies to $ux_1 \ldots x_{n-1}w$ and we obtain the desired path $P'$.
\end{proof}

\begin{corollary}
	\label{S.indecomp.implies.PhiS.conn}
	If $S$ is a split graph and $\Phi(S)$ is disconnected, then $S$ is decomposable.
\end{corollary}

\begin{proof}
	If $\Phi$ is disconnected, then Theorem \ref{A_4.conexo.implica.Phi.conexo} implies that $A_4$ is disconnected. Thus, $S$ is decomposable by Theorem \ref{indecomp.characterization}.
\end{proof}

The converse of Theorem \ref{A_4.conexo.implica.Phi.conexo} does not hold in general: the connectedness of $\Phi$ does not necessarily imply the connectedness of $A_4$. For this, we present a counterexample in Figure \ref{phi.conexo.pero.A_4.no}.

\begin{figure}[h]
	\centering
	\includegraphics[scale=0.8]{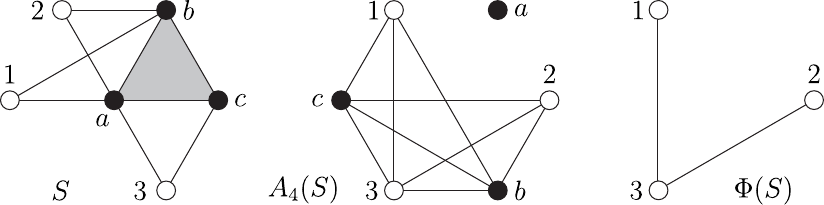}
	\caption{$\Phi$ is connected but $A_4$ is not.}
	\label{phi.conexo.pero.A_4.no}
\end{figure}

A careful look at the split graph $S$ in Figure \ref{phi.conexo.pero.A_4.no} reveals that it is balanced but not active (the vertex $a$ is universal). Then, a natural question arises: does the converse of Corollary \ref{S.indecomp.implies.PhiS.conn} hold if we add the hypothesis that $S$ is active? The answer is affirmative.

\begin{theorem}
	\label{thm.actividad_y_PhiS.conn}
	If $S$ is an active split graph and $\Phi(S)$ is connected, then $S$ is indecomposable.
\end{theorem}

\begin{proof}
	Ignoring parallel edges in $\Phi$, we clearly have $\Phi\subseteq A_4$. Furthermore, every clique vertex must be connected in $A_4$ to some independent vertex, since $S$ is active. Hence, $A_4$ is connected. Applying Theorem \ref{indecomp.characterization}, we conclude that $S$ is indecomposable.    
\end{proof}

We say that a graph $G$ is \textbf{prime} if $G$ is active and indecomposable. Basic examples of prime graphs include $C_4$ and $P_4$. Other examples of prime graphs are trees with diameter $\geq 3$, unicyclic graphs with girth $\geq 4$, and disconnected graphs with no isolated vertices.

\begin{corollary}
	\label{S.primo.iff.Phi(S).conexo}
	Let $S$ be an active split graph.  
	\begin{enumerate}
		\item $S$ is prime if and only if $\Phi(S)$ is connected.
		\item If $S=\prod_{k=1}^{n}S_k$, where each $S_k$ is prime, then $\Phi(S)=\dot{\bigcup}_{k=1}^n\Phi(S_k)$.
		\item If $G$ is an active graph, then each factor in the Tyshkevich decomposition of $G$ is prime.
	\end{enumerate}
\end{corollary}

\begin{proof}
	\begin{enumerate}[(1).]
		\item If $\Phi(S)$ is connected, then $S$ is prime by Theorem \ref{thm.actividad_y_PhiS.conn}. Conversely, if $S$ is prime, then $A_4(S)$ is connected by Theorem \ref{indecomp.characterization}. Hence, $\Phi(S)$ is connected by Theorem \ref{A_4.conexo.implica.Phi.conexo}.
		
		\item It suffices to consider $n=2$, since the general case follows by induction. Let $(K,I)$ be the bipartition of $S$ and let $\Phi=\Phi(S)$. If $(S_1,K_1,I_1)$ and $(S_2,K_2,I_2)$ are prime, then $\Phi(S_1)$ and $\Phi(S_2)$ are connected by (1). If $u,v\in I_1\subset I_1\dot{\cup} I_2=I$, then $\sigma_{uv}(S_1)=\sigma_{uv}(S)$ since $N_{S_1}(x)=N_S(x)$ for all $x\in I_1$. Similarly, $\sigma_{uv}(S_2)=\sigma_{uv}(S)$ for all $u,v\in I_2$. Hence, $\Phi(S_1)=\langle I_1\rangle_{\Phi}$ and $\Phi(S_2)=\langle I_2\rangle_{\Phi}$. If $u\in I_1$ and $v\in I_2$, then $\sigma_{uv}(S)=0$ by (1) (otherwise $\Phi$ would be connected). Therefore, $\Phi(S_1)\dot{\cup}\Phi(S_2)=\langle I_1\rangle_{\Phi}\dot{\cup}\langle I_2\rangle_{\Phi}=\langle I_1\cup I_2\rangle_{\Phi}=\Phi(S)$.
		
		\item It follows immediately from Theorems \ref{tyshk.decomp} and \ref{SoG.act.iff.S,G.act}.
	\end{enumerate}
\end{proof}

If $(S,K,I)$ is a split graph with $K\neq\varnothing$ and $G$ is a nontrivial graph, then $diam(S\circ G)\leq 3$. Hence, any graph of diameter $\geq 4$ with no isolated vertices is prime.

We say that a graph $G$ is \textbf{elementary} if $G$ is prime and twin-free.

\begin{theorem}
	Every active and twin-free graph can be uniquely expressed as a composition of elementary graphs.
\end{theorem}

\begin{proof}
	This result follows easily by combining Corollary \ref{S.primo.iff.Phi(S).conexo} with Theorems \ref{[S°G]=[S]°[G]} and \ref{i(S°G)=max(i(S),i(G))}.
\end{proof}


\section{Characterization of simple factor graphs} \label{sec:caract.Phi.simples}

Given a split graph $(S,K,I)$, we aim to study the neighborhoods in $S$ of the vertices in $I$ in the case where $\Phi(S)$ is simple and connected, with the goal of extrapolating structural properties of both $\Phi$ and $S$. For two vertices $u,v\in I$, we first obtain a formula to compute the multiplicity $\sigma_{uv}(S)$ of the edge $uv\in\Phi(S)$, for arbitrary $S$.

From this, we then describe the basic properties of $\sigma_{uv}$. This is where some important connections arise. One is with Number Theory, which will be explored in the next chapter. The other is about the neighborhoods in $S$ and $\Phi$ of $u$ and $v$, and how they intersect. Subsequently, we introduce the concept of a homogeneous split graph and analyze its general properties. When $\Phi(S)$ is simple and connected, we show that $S$ is homogeneous.

It is in this context that it becomes natural to view the set $\{N_S(v):v\in I\}$ as an intersecting family. This abstraction proves decisive when $\Phi(S)$ is simple and complete, and it is precisely in such a case that we apply the entire framework developed in Appendix \ref{familias.inserc.conjuntos}. In this scenario, there are essentially two possible behaviors: either each vertex in $I$ has a specific neighbor in $K$ that characterizes it, or each vertex in $I$ has a specific non-neighbor in $K$ that characterizes it. Thanks to the tools from intersecting families of sets, we ultimately obtain results that allow us to fully understand the structure of a balanced or active split graph $S$ when $\Phi(S)$ is simple and connected. \\

As we already anticipated in the introduction of this section, we define $\sigma_{uv}(S)$ as the multiplicity of the edge $uv$ in the multigraph $\Phi(S)$. Recall that, by the definition of $\Phi$, $\sigma_{uv}(S)$ also equals: 1) the number of 2-switches that simultaneously activate $u$ and $v$ in $S$; 2) the number of induced $P_4$’s in $S$ that contain $u$ and $v$. Hence,
\begin{equation*}
	\deg(S)=|Q_S(P_4)|=|E(\Phi(S))|=\sum_{\{u,v\}\subseteq I} \sigma_{uv}(S).
\end{equation*}
As usual, we omit the ``$(S)$'' symbol when the split graph $S$ under consideration is clear from context. This means, for instance, that we may write $\sigma_{uv}$ and $\Phi$ instead of $\sigma_{uv}(S)$ and $\Phi(S)$. Also, recall the following shorthand notations: $N_v, N_i, d_v, d_i$ instead of $N_G(v), N_G(v_i), \deg_G(v), \deg_G(v_i)$, respectively, as long as it is clear from the context which graph $G$ we are referring to.

\begin{figure}[h]
	\centering
	\includegraphics[scale=0.8]{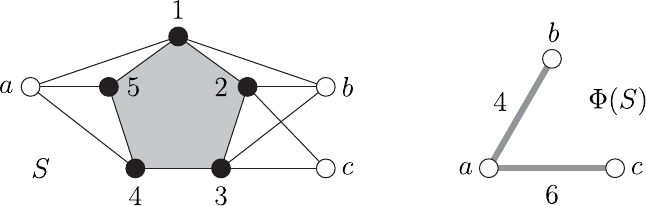}
	\caption{$\deg(S)=\sigma_{ab}+\sigma_{ac}+\sigma_{bc}=4+6+0=10$.}
	\label{ejemplo.cantidad.P4.ind.split}
\end{figure}

In Figure \ref{ejemplo.cantidad.P4.ind.split} we see an example of an active split graph $S$ where $\deg(S)$ is computed by counting the number $|Q_S(P_4)|$ of induced $P_4$’s in $S$. By inspection, we see that there are four induced $P_4$’s connecting $a$ with $b$ ($a52b, a53b, a42b, a43b$), six between $a$ and $c$ ($a12c, a13c, a52c, a53c, a42c, a43c$), and none between $b$ and $c$. In other words, $\deg(S)=\sum_{\{u,v\}\subseteq I} \sigma_{uv} = \sigma_{ab} + \sigma_{ac} + \sigma_{bc} = 4 + 6 + 0 = 10$. A closer look at this example makes it clear that the vertices in $N_u\cap N_v$ do not participate in any induced $P_4$ from $u$ to $v$. That is, the only clique vertices that appear in these paths are those in $N_u-N_v$ and $N_v-N_u$. Indeed, every induced $P_4$ between $u$ and $v$ must use a neighbor of $u$ that is not a neighbor of $v$, say $x$, and a neighbor of $v$ that is not a neighbor of $u$, say $z$. Since both $x$ and $z$ are in the clique of $S$, they are adjacent. Therefore, any $x\in N_u-N_v$ and any $z\in N_v-N_u$ yields an induced $P_4$ between $u$ and $v$. This gives us $|N_u-N_v||N_v-N_u|$ ways to form such a path. Letting $\eta_{uv}(S)=|N_S(u)\cap N_S(v)|$, we obtain the following formula:

\begin{equation}
	\label{fórmula.sigma_uv}
	\sigma_{uv}(S)=(\deg_S(u)-\eta_{uv}(S))(\deg_S(v)-\eta_{uv}(S)),
\end{equation}
As always, if no ambiguity arises in the context, we may write simply $\eta_{uv}$ instead of $\eta_{uv}(S)$. Hence, it will be quite common from now on to simplify the expression in \eqref{fórmula.sigma_uv}, using instead $\sigma_{uv}=(d_u-\eta_{uv})(d_v-\eta_{uv})$. We anticipate a very simple yet crucial observation that will be essential in the next chapter: $d_u -\eta_{uv}$ and $d_v -\eta_{uv}$ are complementary divisors of $\sigma_{uv}$.

The following proposition lists the most important properties of a split graph involving $\sigma_{uv}$. 

\begin{proposicion}
	\label{prop.basicas.sigma_uv}
	Every split graph $S$ satisfies the following properties:
	\begin{enumerate}
		\item $\sigma_{uv}=0$ (i.e., $uv\notin\Phi$) and $d_v \leq d_u$ if and only if $N_v \subseteq N_u$; 
		\item $\sigma_{uv}=0$ and $d_u =d_v$ if and only if $N_u = N_v$;
		\item if $N_u = N_v $, then $N_{\Phi}(u)=N_{\Phi}(v)$; in other words: if $u$ and $v$ are twins in $S$, then they are also twins in $\Phi(S)$;
		\item if $\sigma_{uv}=1$, then $d_u =d_v$ and $|N_u -N_v |=|N_v -N_u |=1$;
		\item if $d_u \geq d_v$ and $\sigma_{uv}=p$ is prime, then $|N_u -N_v |=p, |N_v -N_u |=1$ and $d_u-d_v=p-1$;
		\item $d_u =d_v$ if and only if $|N_u-N_v|=|N_v-N_u|$;
		\item if $d_u =d_v$, then $\sigma_{uv}$ is a perfect square;
		\item if $S$ has no isolated vertices, then $\sigma_{uv}=d_u d_v$ if and only if $N_u \cap N_v =\varnothing$.
	\end{enumerate}
\end{proposicion}

\begin{proof}
	\begin{enumerate}[(1).]
		\item If $\sigma_{uv}=|N_u-N_v||N_v-N_u|=0$, then $|N_u-N_v|=0$ or $|N_v-N_u|=0$. Thus, $N_u\subseteq N_v$ or $N_v\subseteq N_u$.
		\item It follows from (1).
		\item Let $x$ be an independent vertex in $S$. If $N_u=N_v$, then $d_u=d_v$ and $N_u\cap N_x=N_v\cap N_x$. Hence, $\sigma_{ux}=\sigma_{vx}$.
		\item If $\sigma_{uv}=|N_u-N_v||N_v-N_u|=1$, then clearly $|N_u-N_v|=|N_v-N_u|=1$ since $\sigma_{uv}\in\mathbb{Z}$. Also, $d_u-\eta_{uv}=|N_u-N_v|=|N_v-N_u|=d_v-\eta_{uv}$ implies $d_u=d_v$.
		\item It is derived in the same way as (4).
		\item It follows from the equalities $d_v-|N_v-N_u|=\eta_{uv}=d_u-|N_u-N_v|$.
		\item Straightforward.
		\item If $N_u\cap N_v=\varnothing$, then $\eta_{uv}=0$ and thus $\sigma_{uv}=(d_v-0)(d_u-0)$. For the converse, assume that $\sigma_{uv}=d_ud_v$. We can rewrite \eqref{fórmula.sigma_uv} as $\sigma_{uv}=d_ud_v-\eta_{uv}(d_u+d_v-\eta_{uv})$. Then, $\eta_{uv}(d_u+d_v-\eta_{uv})=0$. If $d_u=\eta_{uv}-d_v$, then $d_u=0$, since $\eta_{uv}-d_v\leq 0$, which contradicts the assumption that $S$ has no isolated vertices. Hence, $\eta_{uv}=0$, that is, $N_u\cap N_v=\varnothing$.
	\end{enumerate}
\end{proof}

In Figure \ref{gemelos.Phi(S).pero.no.S} we show that $N_\Phi(u)=N_\Phi(v)$ does not necessarily imply $N_u=N_v$. In other words, the converse of statement (3) of Proposition \ref{prop.basicas.sigma_uv} does not hold. Furthermore, the converse of (7) also does not hold. To verify this, it is enough to take the split graph $(S,[5], \{a,b\})$ where $N_a=[4]$ and $N_b=\{5\}$: we see that $\sigma_{ab}=2^2$, but $d_a=4\neq 1=d_b$.

\begin{figure}[h]
	\centering
	\includegraphics[scale=0.8]{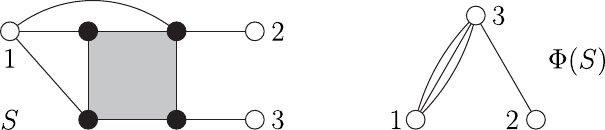}
	\caption{$N_{\Phi}(1)=N_{\Phi}(2)$ but $N_1\neq N_2$.}
	\label{gemelos.Phi(S).pero.no.S}
\end{figure}

It is also important to note that $N_u\subset N_v$ does not generally imply $N_{\Phi}(u)\subseteq N_{\Phi}(v)$, as shown in Figure \ref{contenido.S.pero.no.Phi}.

\begin{figure}[h]
	\centering
	\includegraphics[scale=0.8]{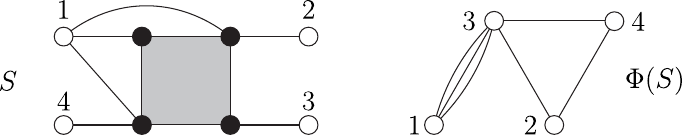}
	\caption{$N_2\subset N_1$ but $N_{\Phi}(2)\nsubseteq N_{\Phi}(1)$.}
	\label{contenido.S.pero.no.Phi}
\end{figure}

Let $(S,K,I)$ be a split graph. We say that $S$ is \textbf{homogeneous} if $S$ is balanced (recall this means that $S$ has no swing vertices, i.e., vertices $w$ such that $N_w=K-w$) and $\deg_S(v)=\deg_S(u)$ for all $u,v\in I$. We now apply Proposition \ref{prop.basicas.sigma_uv} to this class of split graphs, obtaining substantial understanding of the structure of both $S$ and $\Phi(S)$.

\begin{theorem}
	\label{S.homogeneo.implica...}
	If $(S,K,I)$ is a homogeneous split graph, then:
	\begin{enumerate}
		\item $\sigma_{uv}=0$ if and only if $u$ and $v$ are twins in $S$;
		\item $\sigma_{uv}$ is a perfect square for all $\{u,v\}\subseteq I$;
		\item $\deg(S)\geq 1$; in particular, if $d$ is the degree in $S$ of the independent vertices, then $d\notin\{0,|K|\}$;
		\item $\Phi=\Phi(S)$ cannot contain induced subgraphs isomorphic to $K_2\dot{\cup}K_1$;
		\item $diam(\Phi)\leq 2$;
		\item $\Phi$ contains no induced cycles of length 5 or greater.
	\end{enumerate}
\end{theorem}

\begin{proof}
	\begin{enumerate}[(1).]
		\item It follows immediately from Proposition \ref{prop.basicas.sigma_uv}.
		\item Same as (1).
		\item Suppose $\deg(S)=0$. Then $N_u=N=N_v$ for all $u,v\in I$, by (1) (for some $N\subseteq K$). Since $S$ is balanced, it follows that $N\neq K\neq\varnothing$ (otherwise there would be a swing vertex). Hence, there must be $w\in K-N$. But then $w$ would be swing, which contradicts the fact that $S$ is balanced. Thus, $\deg(S)\geq 1$. If $d\in\{0,|K|\}$, then $S$ is clearly inactive.
		\item Suppose $H=(\{a,b,c\},\{\sigma_{ab}\{a,b\}\})\preceq\Phi$, where $\sigma_{ab}\neq 0$. Since $d_a =d_c$ and $\sigma_{ac}=0$, Proposition \ref{prop.basicas.sigma_uv} implies $N_a =N_c$. Since $d_c =d_b$ and $\sigma_{bc}=0$, we get $N_c =N_b$. Thus $N_a =N_b$, and hence $\sigma_{ab}=0$, a contradiction.
		\item If $diam(\Phi)\in [3,\infty)$, then there is an induced $P_4$ in $\Phi$. Since $K_2\dot{\cup}K_1\preceq P_4$, we have a contradiction with (4). If $diam(\Phi)=\infty$, then there is an induced $K_2\dot{\cup}K_1$ in $\Phi$, since $\Vert\Phi\Vert=\deg(S)\geq 1$ by (3). Again, this contradicts (4). Therefore, $diam(\Phi)\leq 2$.
		\item It is an immediate consequence of (4).
	\end{enumerate}
\end{proof}

\begin{theorem}
	\label{Phi.simple.conexo.1}
	Let $(S,K,I)$ be a split graph such that $K=\bigcup_{v\in I}N_v$. If $\Phi=\Phi(S)$ is simple and connected and $|I|\geq 2$, then $S$ is homogeneous.
\end{theorem}

\begin{proof}
	The fact that $\Phi$ is simple means $uv\in E(\Phi)$ if and only if $\sigma_{uv}=1$. If $xyz\subseteq\Phi$, then $\sigma_{xy}=1$ implies $d_x =d_y$, and $\sigma_{yz}=1$ implies $d_y =d_z$, by Proposition \ref{prop.basicas.sigma_uv}. Therefore, there exists $d\geq 0$ such that $d_v =d$ for all $v\in I$, because $\Phi$ is connected. Also, $d\notin\{0,|K|\}$, since otherwise $\Phi$ would have no edges. Then, $S$ has no swing vertices in $I$. If there were such a vertex in $K$, say $w$, the hypothesis $K=\bigcup_{v\in I}N_v$ would force $w$ to be adjacent to some $u\in I$, contradicting $N_w=K-w$. Thus, $S$ has no swing vertices, i.e., it is balanced.
\end{proof}

\begin{lemma}
	\label{Phi.completo.implica.inactivos.universales}
	Let $(S,K,I)$ be a balanced split graph such that $|I|\geq 2$ and $\Phi(S)$ is complete. If $x$ is an inactive vertex in $S$, then $x$ is universal.
\end{lemma}

\begin{proof}
	The assumptions on $S$ guarantee $I\subseteq act(S)$. Hence, if $x\notin act(S)$, then $x\in K$. Suppose $x$ is not universal. Since $S$ has no swing vertices, it follows that $N_x\cap I\neq\varnothing$. Then there exist $a,b\in I$ such that $a\in N_x$ and $b\notin N_x$. On the other hand, since $\sigma_{ab}>0$, there exist $y,z\in K$ such that $ayzb\preceq S$. But then $axzb\preceq S$, contradicting that $x$ is inactive.
\end{proof}

To prove the next theorem (Theorem \ref{Phi.simple.completo.implica...}), we need auxiliary results about certain intersecting families of sets, which are independent of the graph structure itself. As announced in the introduction to this section, these tools are developed extensively in Appendix \ref{familias.inserc.conjuntos}. When $\Phi(S)$ is simple and complete, we know that all vertices of $\Phi$ have fixed degree $d$ in $(S,K,I)$. Translating this into set-theoretic language: $|N_S(v)|=d$ for all $v\in I$. Additionally, since $\Phi$ is simple, we also have $|N_u\cap N_v|=d-1$ for all $\{u,v\}\subseteq I$. As can be seen in Appendix \ref{familias.inserc.conjuntos}, $\{N_v:v\in I\}$ turns out to be an intersecting family with very interesting properties, which fit perfectly into this context. Below, we quote the three results from Appendix \ref{familias.inserc.conjuntos} that will be used to prove Theorem \ref{Phi.simple.completo.implica...}. The respective proofs can be found in Appendix \ref{familias.inserc.conjuntos}. \\

\textbf{Proposition} \ref{lema.aureo.conjuntista}.
Let $\mathcal{N}=\{N_v :v\in I\}$ be a finite family of finite sets of size $d$ such that $|N_u \cap N_v |=d-1$ for all $\{u,v\}\subseteq I$. If $W$ is a non-empty subset of $I$, then 
\begin{equation*}
	\left|\bigcap_{v\in W} N_v \right| \in\{d+1-|W|,d-1\}.
\end{equation*}

\textbf{Theorem} \ref{equivalencias.d-1.conjuntistas}. Let $\{N_v :v\in [\alpha],\alpha\geq 3\}$ be a family of finite sets of size $d\geq 1$ such that $|N_x \cap N_y |=d-1$ for all $\{x,y\}\subseteq [\alpha]$. Let also $\omega=|\bigcup_{v=1}^{\alpha}N_v|$. The following statements are equivalent:
\begin{enumerate}
	\item there exists a triple $T\subseteq[\alpha]$ such that $|\bigcap_{v\in T}N_v|=d-1$;
	\item \[ \left|\bigcap_{v\in A}N_v \right|=d-1, \] for all $A\subseteq[\alpha]$ with $|A|\geq 2$;
	\item \[ \omega=\alpha+d-1. \] 
\end{enumerate}

\textbf{Theorem} \ref{equivalencias.d+1-n.conjuntistas}. Let $\{N_v :v\in [\alpha],\alpha\geq 3\}$ be a family of finite sets of size $d\geq 1$ such that $|N_x \cap N_y |=d-1$ for all $\{x,y\}\subseteq [\alpha]$. Let also $\omega=|\bigcup_{v=1}^{\alpha}N_v|$. The following statements are equivalent:
\begin{enumerate}
	\item there exists a triple $T\subseteq[\alpha]$ such that $|\bigcap_{v\in T}N_v|=d-2$;
	\item \[ \left|\bigcap_{v\in A}N_v \right|=d+1-|A|, \] for all $A\subseteq[\alpha]$ with $A\neq\varnothing$;
	\item \[ \omega=d+1. \] 
\end{enumerate}

\begin{theorem}
	\label{Phi.simple.completo.implica...}
	Let $(S,K,I)$ be a split graph such that $\Phi(S)$ is simple and complete, and let $U$ be the set of universal vertices of $S$. If $\bigcup_{v\in I}N_v = K$, $\omega=|K|$ and $\alpha=|I|\geq 2$, then:
	\begin{enumerate}
		\item $S$ is homogeneous;
		\item if $d$ is the degree in $S$ of the vertices in $I$, then $1\leq d\leq\omega -1$;
		\item $|U|\in\{d-1,d+1-\alpha\}$;
		\item $|U|=d-1$ if and only if $\omega=\alpha +d-1$;
		\item $|U|=d+1-\alpha$ if and only if $\omega =d+1$;
		\item $\omega =\alpha +|U|$ (in particular, $\omega\geq\alpha$);    
		\item $S$ is active if and only if $U=\varnothing$;
		\item If $S$ is active, then $\omega=\alpha$ and $d\in\{1,\omega-1\}$;
		\item if $\omega=\alpha$ or $d=1$, then $S$ is active; 
	\end{enumerate} 
\end{theorem}

\begin{proof}
	\begin{enumerate}[(1).]
		\item It is a particular case of Theorem \ref{Phi.simple.conexo.1}.
		\item It is an immediate consequence of (1) and Theorem \ref{S.homogeneo.implica...}.
		\item It is a particular case of Proposition \ref{lema.aureo.conjuntista}, since $U=\bigcap_I N_v$, according to Proposition \ref{S.balanc.implica.deg>0,K=Union(N_v),|U|<|K|-1}.
		\item It is a particular case of Theorems \ref{equivalencias.d-1.conjuntistas} and \ref{equivalencias.d+1-n.conjuntistas}, but translated into the language of Graph Theory.
		\item Same as (4).
		\item We know that $|U|\in\{d-1,d+1-\alpha\}$, by (3). The statement is evident if $|U|=d-1$, thanks to (4). If instead $|U|=d+1-\alpha$, then $\omega=(d+1-\alpha)+\alpha$, by (5).
		\item If $S$ is active, it is clear that $U=\varnothing$, since $U\subseteq act(S)^c=\varnothing$. Conversely, if $U=\varnothing$, then $S$ is active thanks to Lemma \ref{Phi.completo.implica.inactivos.universales}.
		\item If $S$ is active, then $U=\varnothing$, by (7). Thus, $\omega=\alpha$, by (6), and $0\in\{d-1,d+1-\omega\}$, by (3). Therefore, $d\in\{1,\omega-1\}$.
		\item If $\omega=\alpha$, then $U=\varnothing$, by (6). Thus, $S$ is active, by (7). If $d=1$, then $|U|\in\{0,2-\alpha\}$, by (3). If $|U|=2-\alpha$, then $\omega=2=\alpha$, by (6) and because $\alpha\geq 2$. Therefore, $S\approx P_4$, where obviously there are no universal vertices. Consequently, $d=1$ implies $U=\varnothing$. Finally, $S$ is active thanks to (7).
	\end{enumerate}
\end{proof}

Regarding item (9) of Theorem \ref{Phi.simple.completo.implica...}, it is important to note that $d=\omega-1$ does not generally imply that $S$ is active. For example, consider the split graph $(S,[4],\{a,b\})$ where $N_a=\{1,2,3\}$ and $N_b=\{2,3,4\}$. Here, we see that $d_a=3=\omega-1=d_b$ and $\sigma_{ab}=1$, so $\Phi(S)$ is simple and complete. However, $S$ is not active, since vertices 2 and 3 are universal.\\

Given a split graph $(S,K,I)$, the \textbf{inversion} (or \textbf{inverse}) $(S,K,I)^{\iota}$ of $S$ is defined as the split graph $(S^{\iota},I,K)$, where 
\[ E(S^{\iota})=(E(S)-\{ab:a,b\in K\})\cup\{ab:a,b\in I\}. \]
When inverting $S$, $K$ becomes an independent set, $I$ becomes a clique, but all edges between $K$ and $I$ remain. This concept is not new and has been widely used by R. Tyshkevich in \cite{tyshkevich2000decomposition} (page 14). It should be noted that inversion, like composition, depends entirely on the bipartition of $S$. Indeed, it does not make sense in general to speak of ``the inverse" of $S$, since $S$ will not have a unique bipartition if it is unbalanced. Obviously, inversion is an involutive operation, i.e., $((S,K,I)^{\iota})^{\iota}=(S,K,I)$. Another important property is that inversion and complementation commute with each other.

\begin{lemma}
	\label{inv.compl.conmutan}
	If $(S,K,I)$ is a split graph, then
	\begin{equation}
		\label{eq40}
		\overline{(S,K,I)^{\iota}} = \Big(\overline{(S,K,I)}\Big)^{\iota}.
	\end{equation}  
	Moreover, $N_{\overline{S^{\iota}}}(v)=K-N_S(v)$, for all $v\in I$.
\end{lemma}

\begin{proof}
	Since both inversion and complementation interchange the roles of $K$ and $I$ as clique/independent sets, it follows that $G=\overline{S^{\iota}}$ and $H=(\overline{S})^{\iota}$ both have the same bipartition as $S$. Therefore, to obtain \eqref{eq40}, it suffices to prove that $N_G(v)=N_H(v)$ for all $v\in I$. First, note that: 
	\begin{enumerate}[(1).]
		\item $v$ is a clique vertex in $S^{\iota}$ and in $\overline{S}$;
		\item $N_{S^{\iota}}(v)=(I-v)\cup N_S(v)$;
		\item $N_{\overline{S}}(v) = (I-v)\cup(K-N_S(v))$.
	\end{enumerate}
	From (2), we deduce that 
	\[ N_G(v)\cup v =  \]
	\[ N_{S^{\iota}}(v)^c = (I-v)^c\cap N_S(v)^c = (K\cup v)\cap(I\cup(K-N_S(v)))=\]
	\[ ((K\cup v)\cap I)\cup((K\cup v)\cap(K-N_S(v)))= \]
	\[ v\cup(K-N_S(v)), \]
	that is, $N_G(v)=K-N_S(v)$. On the other hand, by (1) and by the definition of inversion, it is clear that $N_H(v)=N_{\overline{S}}(v)-I$. Finally, using (3), we obtain that $N_H(v)=K-N_S(v)$. 
\end{proof}

We now use Lemma \ref{inv.compl.conmutan} to make an important observation. Given a pair $\{u,v\}\subseteq I$, we have 
\[ \eta_{uv}(\overline{S^{\iota}})=|(K-N_u)\cap(K-N_v)|=|N_u^c\cap N_v^c|= \]
\[ |K-(N_u\cup N_v)|=|K|-|N_u\cup N_v|= \]
\[ |K|-d_u-d_v+\eta_{uv}. \]
Combining this with the fact that $\deg_{\overline{S^{\iota}}}(x)=|K|-d_x$, it quickly follows, applying \eqref{fórmula.sigma_uv}, that
\begin{equation}
	\label{eq41}
	\sigma_{uv}(\overline{S^{\iota}})=\sigma_{uv}(S),
\end{equation}
for all $\{u,v\}\subseteq I$. Equality \eqref{eq41} implies the following result.

\begin{proposition}
	\label{Phi(S)=Phi(S^-1)}
	For every split graph $S$,  
	\[ \Phi(S)=\Phi(\overline{S^{\iota}}). \]
\end{proposition}

\begin{proof}
	It follows from the previous discussion. 
\end{proof}

\begin{proposition}
	\label{S.activo.iff.S^-1.activo}
	Let $(S,K,I)$ be a split graph.
	\begin{enumerate}
		\item If $w$ is a swing vertex in $S$ and $w\in I$, then $w$ is isolated in $S^{\iota}$.
		\item If $w$ is a swing vertex in $S$ and $w\in K$, then $w$ is universal in $S^{\iota}$.
		\item If $S$ is non-empty and $w$ is a universal or isolated vertex in $S$, then $w$ is swing in $S^{\iota}$.
		\item $act(S) = act(S^{\iota}) = act(\overline{S^{\iota}})$.
	\end{enumerate}
\end{proposition}

\begin{proof}
	\begin{enumerate}[(1).]
		\item If $w\in I$, then $N_w=K$. Thus, by definition of inverse, $N_{S^{\iota}}(w)=(I-w)\cup N_w=V(S)-w$, which shows that $w$ is universal in $S^{\iota}$.
		\item If $w\in K$, then $N_w=K-w$. Thus, by definition of inverse, $N_{S^{\iota}}(w)=N_w-K=\varnothing$, which shows that $w$ is isolated in $S^{\iota}$.
		\item We need to prove that $N_{S^{\iota}}(w)=I-w$. If $w$ is universal in $S$, then $w\in K$, since $S$ is non-empty. Thus, $N_{S^{\iota}}(w)=N_w-K=(V(S)-w)-K=I-w$. If $N_w=\varnothing$, then $w\in I$, since $S$ is non-empty. Therefore, $N_{S^{\iota}}(w)=(I-w)\cup N_w=I-w$.  
		\item We know that a vertex is active in a graph if and only if it is active in its complement. Therefore, $act(S^{\iota})=act(\overline{S^{\iota}})$. If $a\in act(S)$, then there exists an $H\in Q_S(P_4)$ such that $a\in V(H)$. If $a\in I$, then $H=axyz$, where $z\in I$ and $x,y\in K$. Since $\{ax,az,yz\}\subseteq E(S^{\iota})$ and $xy\notin E(S^{\iota})$, it follows that $xazy\subseteq S^{\iota}$. Moreover, since $ay,xz\notin S$, we have that neither $ay$ nor $xz$ belong to $S^{\iota}$. Therefore, $xazy\preceq S^{\iota}$, which shows that $a\in act(S^{\iota})$. If instead $a\in K$, then $a$ is active and independent in $\overline{S}$. Thus, $a$ is active in $(\overline{S})^{\iota}=\overline{S^{\iota}}$, where the last equality is valid thanks to Lemma \ref{inv.compl.conmutan}. Since $act(\overline{S^{\iota}})=act(S^{\iota})$, we obtain that $act(S)\subseteq act(S^{\iota})$. In other words, if a vertex is active in a split graph, then it is also active in its inverse. This means that if $a$ is active in $S^{\iota}$, then $a$ is active in $(S^{\iota})^{\iota}=S$; that is, $act(S^{\iota})\subseteq act(S)$.
	\end{enumerate}
\end{proof}

Proposition \ref{S.activo.iff.S^-1.activo} implicitly tells us that inversion does not generally preserve the property of being balanced. In fact, it is easy to find a balanced non-empty split graph $S$ that contains isolated or universal vertices. In such a case, by item (3) of Proposition \ref{S.activo.iff.S^-1.activo}, $S^{\iota}$ will be unbalanced. As a concrete example, consider the graph $(S,[3],\{a,b\})$ such that $N_a=\{1,2\}$ and $N_b=\{2,3\}$, which is balanced, and in which vertex 2 is universal. By inverting $S$, we obtain that $E(S^{\iota})=\{1a,2a,2b,3b,ab\}$. Since $N_{S^{\iota}}(2)=\{a,b\}$, we see that, indeed, vertex 2 is swing in $(S^{\iota},\{a,b\},[3])$. An analogous example can be constructed by taking $\overline{S}$, which is also balanced. Here, vertex 2 is isolated in $\overline{S}$, but swing in $((\overline{S})^{\iota},[3],\{a,b\})$.

\begin{corollary}
\label{act.Phi.simple.compl.iff...}
Consider the split graph $(G,K,I)$ with $|K|=|I|$, $K=\{x_v:v\in I\}$, and $N_G(v)=\{x_v\}$ for all $v\in I$. Let $S$ be a split graph such that $|\Phi(S)|=|I|$. Then, $S$ is active and $\Phi(S)$ is simple and complete if and only if $S$ is isomorphic to $G$ or $\overline{G^{\iota}}$.
\end{corollary}

\begin{proof}
($\Leftarrow$). By the definition of $G$, it is very easy to verify that it is active and $\Phi(G)$ is simple and complete. Verifying the same properties for $\overline{G^{\iota}}$ is immediate thanks to Proposition \ref{Phi(S)=Phi(S^-1)}.  

($\Rightarrow$). By Theorem \ref{Phi.simple.completo.implica...} we know that the independent vertices of $(S,K',I')$ all have the same degree $d$ in $S$, $d\in\{1,\omega-1\}$ and $\omega=|K'|=|I'|$. If $d=1$, the fact that $S$ is active and $|K'|=|I'|=|I|$ forces $S$ to be isomorphic to $G$. If $d=\omega-1\geq 2$, observe that all vertices of $I$ are leaves in $\overline{S^{\iota}}$. Then, since $\overline{S^{\iota}}$ is active by Proposition \ref{S.activo.iff.S^-1.activo}, and $|K'|=|I'|=|I|$, it is clear that $\overline{S^{\iota}}\approx G$. Thus, $S\approx (\overline{G})^{\iota}=\overline{G^{\iota}}$, by Lemma \ref{inv.compl.conmutan}.
\end{proof}

The following is a technical result that relates the subgraph of $\Phi(S)$, induced by a subset $A$ of its vertices, to the factor graph of the corresponding induced subgraph in $S$ that has $A$ as an independent set. It also shows that, in some way, $\Phi$ does not depend on universal vertices, i.e., removing universal vertices in the underlying split graph does not alter the structure of the factor graph. 

\begin{proposition}
\label{subgrafos.de.Phi}
Let $(S,K,I)$ be a split graph where $K=\bigcup_{u\in I}N_u$. For $A\subseteq I$, define 
\[ S_A =\left\langle\bigcup_{v\in A}(v\dot{\cup}N_v )\right\rangle_{S}, \ S_A'=S_A-\bigcap_{v\in A}N_v.\]
Then 
\[ \left( \bigcup_{v\in A} N_v, A \right), \ \left( \bigcup_{v\in A} N_v -\bigcap_{v\in A} N_v, A \right) \]
are bipartitions for $S_A$ and $S_A'$, respectively, and
\[ \Phi(S_A)=\langle A\rangle_\Phi =\Phi(S_{A}'). \] 
\end{proposition}

\begin{proof}
By definition, $\Phi(S_A), \langle A\rangle_\Phi$ and $\Phi(S_{A}')$ have the same vertex set, namely $A$. Then, to obtain the required equalities, it suffices to show that
\[ \sigma_{ab}(S_A)=\sigma_{ab}=\sigma_{ab}(S'_A), \]
for every pair of vertices $a,b\in A$. On one hand, we know that $\sigma_{ab}=(d_a-\eta_{ab})(d_b-\eta_{ab})$ (see \eqref{fórmula.sigma_uv}). On the other hand: $N_{S_A}(v)=N_v$, for all $v\in A$, by the definition of $S_A$. With this in mind, it follows immediately that $\sigma_{ab}(S_A)=\sigma_{ab}$.

To prove that $\sigma_{ab}(S'_A)=\sigma_{ab}$, observe that 
\[ N_{S_A'}(u)=N_u-\bigcap_{v\in A}N_v, \]
for all $u\in A$, and that 
\[ N_{S_A'}(a)\cap N_{S_A'}(b) = N_a\cap N_b-\bigcap_{v\in A}N_v. \]
Letting $\eta_A=|\bigcap_{v\in A}N_v|$, $d_u'=\deg_{S_A'}(u)$ and $\eta_{ab}(S_A')=\eta_{ab}'$, we then obtain that $d_u'=d_u-\eta_A$ and $\eta_{ab}'=\eta_{ab}-\eta_A$. Using this in the equality
\[ \sigma_{ab}(S_A')=(d_a'-\eta_{ab}')(d_b'-\eta_{ab}'), \]
we finally prove that $\sigma_{ab}(S_A')=\sigma_{ab}$.
\end{proof}

The following result is very important because it provides structural information about both $S$ and $\Phi(S)$ when the latter is simple and connected, depending on whether it is complete or not. In short, we show that $S$ and $\Phi(S)$ have the same clique number. 

\begin{theorem}
\label{Phi.simple.conexo.caracterizacion} 
Let $(S,K,I)$ be an active split graph such that $\Phi=\Phi(S)$ is simple and connected. Then one and only one of the following two possibilities occurs:
\begin{enumerate}
\item $\Phi$ is complete and $|K|=|I|$.
\item $|K|=\omega(\Phi)<|I|$ and every vertex of $\Phi$ belongs to some maximum clique of $\Phi$; in particular:
\[ |K|\leq 1+\min\{\deg_{\Phi}(v):v\in I\}. \]
\end{enumerate} 
\end{theorem}

\begin{proof}
	\begin{enumerate}[(1).]
		\item Consider a maximum clique $Q\subseteq I$ in $\Phi$. If $Q=I$, we are done because $\Phi$ is complete and, therefore, $|I|=|K|$, according to Theorem \ref{Phi.simple.completo.implica...}.
		\item Then, we assume that $Q\neq I$, i.e., $|Q|=\omega(\Phi)<|I|$. Thus, for each $x\in I-Q$ there must be a vertex $q\in Q$ such that $\sigma_{qx}=0$, i.e., $N_x =N_q$ (otherwise, $Q$ would not be a maximum clique). This fact has two important consequences. First, that $N_{\Phi}(x)=N_{\Phi}(q)$; this means, in particular, that $x$ forms another maximum clique in $\Phi$ together with the vertices of $Q-q$. Second, that $\bigcup_{Q}N_v =\bigcup_{I}N_v =K$ and $\bigcap_{Q}N_v =\bigcap_{I}N_v =\varnothing$. If $S_Q=( \langle \bigcup_{Q}(v\dot{\cup} N_v)\rangle_{S}, Q, K ) $, then $\Phi(S_Q)=\langle Q\rangle_{\Phi}$, by Proposition \ref{subgrafos.de.Phi}. Clearly, we have that $|Q|\geq 2$, because $S$ is active, and hence $\Phi$ has at least one edge. Since $\Phi(S_Q)$ is simple and complete, it follows from Theorem \ref{Phi.simple.completo.implica...} that $S_Q$ is balanced. But then, by Lemma \ref{Phi.completo.implica.inactivos.universales}, $S_Q$ is active, since $\bigcap_{Q}N_v =\varnothing$. Finally, we can again apply Theorem \ref{Phi.simple.completo.implica...} to infer that $|K|=|Q|$.
		
		Let $v$ be any vertex of $\Phi$. Since $v$ is part of a clique of size $|K|$ in $\Phi$, it follows that $\deg_{\Phi}(v)\geq |K|-1$.
	\end{enumerate}
\end{proof}

\begin{corollary}
\label{Phi.simple.conexo.deg.cotas}
Let $(S,K,I)$ be an active split graph and let $\alpha=|I|$ and $\omega=|K|$. If $\Phi(S)$ is simple and connected, then:
\begin{enumerate}
\item $\omega\leq\alpha$;
\item $S$ and $\Phi(S)$ have the same clique number;
\item if $\Phi(S)$ is complete, then \[ \deg(S)=\binom{\alpha}{2}; \]
\item if $\Phi(S)$ is not complete, then \[ \left\lceil\frac{\alpha}{2}(\omega-1)\right\rceil\leq\deg(S)<\binom{\alpha}{2}. \]
\end{enumerate}
\end{corollary}

\begin{proof}
\begin{enumerate}[(1).]
\item It is part of the content of Theorem \ref{Phi.simple.conexo.caracterizacion}.
\item Same as (1).
\item $\deg(S)=\sum_{\{u,v\}\subseteq I}\sigma_{uv}(S)=\sum_{\{u,v\}\subseteq I}1=\binom{\alpha}{2}$, since $\Phi(S)$ is simple and complete.
\item The upper bound is a trivial consequence of assuming that $\Phi$ is not complete. On the other hand, the lower bound is deduced from Theorem \ref{Phi.simple.conexo.caracterizacion} via the Handshake Lemma. Indeed,
\[ |\Phi|\min\{\deg_{\Phi}(v):v\in I\}=\alpha(\omega-1)\leq \]
\[ \sum_{v\in I}\deg_{\Phi}(v)=2\Vert\Phi\Vert=2\deg(S). \]
\end{enumerate}
\end{proof}

\begin{corollary}
\label{Phi(S).simple.conexo.no-compl.implica.Phi(S^c)...}
Let $S$ be an active split graph. If $\Phi(S)$ is simple, connected, but not complete, then $\Phi(\overline{S})$ is complete but not simple.
\end{corollary}

\begin{proof}
Being $(K,I)$ the bipartition of $S$, we know that $|K|=\omega(\Phi(S))<|I|$, by Theorem \ref{Phi.simple.conexo.caracterizacion}. This means that $\Vert\Phi(S)\Vert >\binom{|K|}{2}$. Thus, $\Phi(\overline{S})$ cannot be simple, since $\deg(\overline{S})=\deg(S)$, by Corollary \ref{degG=deg(G.complemento)}. 

We will now prove that $\Phi(\overline{S})$ is complete. To this end, observe the following: if for each $v\in I$ we remove all $u\in I-v$ that are twins of $v$ in $S$, it is clear (by Theorem \ref{Phi.simple.conexo.1} and by item (1) of Theorem \ref{S.homogeneo.implica...}) that we obtain an active split graph $(H,J,K)\preceq S$ such that $\Phi(H)$ is simple and complete. Consequently, $\Phi(\overline{H})$ is also simple and complete, by Corollary \ref{act.Phi.simple.compl.iff...}. This means that $\sigma_{xy}(\overline{H})\neq 0$ for all $\{x,y\}\subseteq K$. Since $\overline{H}\preceq\overline{S}$, then $\sigma_{xy}(\overline{S})\neq 0$ for all $\{x,y\}\subseteq K$ as well, which shows that $\Phi(\overline{S})$ is complete.  
\end{proof}

Finally, with the next Theorem, all active split graphs $(S,K,I)$ with $|K|\leq|I|$ such that $\Phi(S)$ is simple and connected are completely characterized.

\begin{theorem}
\label{caract.Phi.simple.conexo}
Given a split graph $(S,K,I)$ with $2\leq|K|\leq |I|$ and $K=\bigcup_{v \in I}N_S(v)$, consider the associated split graph $(R,K,I)$ such that $\deg_R(v)=1$ for all $v\in I$, where $u\cong_R v$ if and only if $\sigma_{uv}(S)=0$, for all $\{u,v\}\subseteq I$. Then $S$ is active and $\Phi(S)$ is simple and connected, if and only if $S$ is isomorphic to $R$ or $\overline{R^{\iota}}$.
\end{theorem}

\begin{proof}
($\Leftarrow$). Thanks to Proposition \ref{Phi(S)=Phi(S^-1)}, it suffices to consider the case $S\approx R$. Since $\deg_S(v)=1$ for all $v\in I$, it is obvious that $\sigma_{uv}(S)\leq 1$ for all $\{u,v\}\subseteq I$. Therefore, $\Phi(S)$ is simple. Moreover, since $2\leq|K|\leq |I|$ and $K=\bigcup_{v \in I}N_S(v)$ by hypothesis, it is evident that $S$ is active. Since $S$ is homogeneous, we conclude that $\Phi(S)$ is connected, by Theorem \ref{S.homogeneo.implica...}. 

($\Rightarrow$). If for each $v\in I$ we remove all $u\in I-v$ that are twins of $v$ in $S$, it is clear (by Theorem \ref{Phi.simple.conexo.1} and by item (1) of Theorem \ref{S.homogeneo.implica...}) that we obtain an active split graph $(H,J,K)\preceq S$ such that $\Phi(H)$ is simple and complete. Let $(G,J,K)$ be the split graph such that $K=\{x_j:j\in J\}$ and $N_G(j)=\{x_j\}$, for all $j\in J$. Then, $H$ is isomorphic to $G$ or $\overline{G^{\iota}}$, by Corollary \ref{act.Phi.simple.compl.iff...}. It is important to note that, obviously, $N_H(v)=N_S(v)$, for all $v\in I$. If $H\approx G$, let $i\in I-J$ be a twin of $j$ in $S$. Since $N_S(i)=N_S(j)=N_H(j)=\{x_j\}$, we conclude that $\deg_S(v)=1$, for all $v\in I$. In other words, $S\approx R$. If $H\approx\overline{G^{\iota}}$, let $i\in I-J$ be a twin of $j$ in $S$. Since $N_S(i)=N_S(j)=N_H(j)=K-x_j$, we conclude that $\deg_S(v)=|K|-1$, for all $v\in I$. This shows that $S\approx\overline{R^{\iota}}$.
\end{proof}


\section{Induced paths and cycles in $\Phi$} \label{sec:caminos,ciclos.en.Phi}

In this section, we continue studying the implications of Proposition \ref{prop.basicas.sigma_uv} on the structure of the factor graph $\Phi(S)$ of a split graph $S$. 

We begin by studying induced paths in $\Phi$. We note that the presence of such a path $P\preceq\Phi$ induces certain inclusion chains in the neighborhoods in $S$ of the vertices of $P$. Thanks to this, we also see how the degree in $S$ of the vertices of $P$ tends to increase along $P$, moving from one end to the other. These results are then applied to study the length of induced cycles in $\Phi$ (since these contain induced paths), discovering that it is at most 4. 

Subsequently, we return to induced paths, now analyzing the multiplicity of their edges. We find that only the first or last edge can be simple (i.e., of multiplicity 1). This result allows us to conclude the section with an interesting upper bound for the diameter of $\Phi(S)$, in terms of the degree of $S$.

\begin{lemma}
	\label{lema.camino.inducido.en.Phi}
	Let $(S,K,I)$ be a split graph and $P=v_1\ldots v_n$ an induced path in $\Phi(S)$, with $n\geq 2$. If $d_i=\deg_S(v_i)$, $N_i=N_S(v_i)$ and $d_1=\max\{d_i:i\in[n]\}$, then for each $i\in[n]$ we have
	\begin{equation*}
		d_i\geq d_j \ \text{and} \ N_i\supseteq N_j \ \text{for all} \ j\geq i+2.
	\end{equation*}
\end{lemma}

\begin{proof}
	We proceed by induction on $i$. Since $\sigma_{1j}=0$ and $d_1 \geq d_j$ for all $j\geq 3$, we have $N_1\supseteq N_j$, for all $j\geq 3$. For $i\geq 2$, we assume that $d_{i-1}\geq d_j$ and $N_{i-1}\supseteq N_j$, for all $j\geq i+1$. If for some $k\geq i+2$ we had $d_i<d_k$, then we would have $N_i\subset N_k\subseteq N_{i-1}$, where the last containment follows from the inductive hypothesis. This would imply that $\sigma_{i-1,i}=0$, which contradicts that $P$ is a path.
\end{proof}

\begin{lemma}
	\label{no.existe.Phi=P_5.tal.que...}
	There is no split graph $S$ such that $\Phi(S)$ contains an induced path $v_5 v_4 v_1 v_2 v_3$ where $\max\{ d_i=\deg_S(v_i):i\in[5]\}=d_1$.
\end{lemma}

\begin{proof}
	Suppose such a split graph $S$ exists. Then: $N_5,N_3\subseteq N_1$, since $\sigma_{13}=0=\sigma_{15}$ and $d_3,d_5\leq d_1$. As $\sigma_{24}=0$, we can assume without loss of generality that $N_4\subseteq N_2$. Then, it must be that $N_4\subseteq N_3$, since otherwise $N_3\subseteq N_4$ would imply $\sigma_{23}=0$. But then $N_4\subseteq N_1$, and therefore $\sigma_{14}=0$, which is a contradiction.
\end{proof}

\begin{theorem}
	\label{caminos.inducidos.en.Phi}
	Let $S$ be a split graph and $P=v_1\ldots v_n$ an induced path in $\Phi(S)$, with $n\geq 2$. If $d_i=\deg_S(v_i)$, then
	\begin{equation}
		\max\{ d_i:i\in[n] \}\in\{d_1,d_2,d_{n-1},d_n\}.
		\label{max.d_i}
	\end{equation}
	Moreover, if $\max\{ d_i:i\in[n] \}=d_1$, then:
	\begin{enumerate}
		\item for each $i\geq 1$: $d_i\geq d_j \ \text{and} \ N_i=N_S(v_i)\supseteq N_j \ \text{for all} \ j\geq i+2$;
		\item for each $i\geq 1$: $N_i\supseteq\bigcup_{j= i+2}^n N_j$; in particular, $\bigcup_{i=1}^n N_i=N_1\cup N_2$;
		\item $d_{2k}\geq d_{2k+2}$ and $d_{2k-1}\geq d_{2k+1}$ for all $k\geq 1$;
		\item $\min\{ d_i:i\in[n] \}\in\{d_{n-1},d_n\}$.
	\end{enumerate} 
\end{theorem}

\begin{proof}
	We begin by proving (\ref{max.d_i}). If $n\leq 4$, there is nothing to do. If $n\geq 5$, suppose that $\max\{ d_i:i\in[n] \}=d_j$, where $j\notin\{1,2,n-1,n\}$. Then $v_{j-2}v_{j-1}v_jv_{j+1}v_{j+2}\preceq \Phi(S)$, but by Lemma \ref{no.existe.Phi=P_5.tal.que...} we know this is impossible. 
	\begin{enumerate}[(1).]
		\item It is the content of Lemma \ref{lema.camino.inducido.en.Phi}.
		\item It is an immediate consequence of (1).
		\item From (1) it immediately follows that $d_1\geq d_3\geq d_5\geq d_7\geq\ldots$ and that $d_2\geq d_4\geq d_6\geq d_8\geq\ldots$.
		\item If $\min\{ d_i:i\in[n] \}=d_j$ for some $j<n-1$, then $d_j\leq d_n$, which contradicts (1).
	\end{enumerate}  
\end{proof}

The following is a direct application of Theorem \ref{caminos.inducidos.en.Phi}. Through it, we find an interesting divisibility relationship involving the multiplicity of one of the terminal edges of an induced path in $\Phi$. 

\begin{corollary}
	\label{K-d_max.divide.sigma}
	Let $(S,K,I)$ be a split graph such that $K=\bigcup_{v\in I}N_v$. If 
	$P=v_1\ldots v_n$ is an induced path in $\Phi(S)$ where $d_1=\max\{d_i:i\in[n],n\geq 2\}$, then $|\bigcup_{i=1}^n N_i|-d_1$ is a divisor of $\sigma_{12}$ less than or equal to $\sqrt{\sigma_{12}}$. In particular, if $\Phi=P$, then $|K|-d_1$ is a divisor of $\sigma_{12}$ less than or equal to $\sqrt{\sigma_{12}}$.
\end{corollary}

\begin{proof}
	Since $P$ is an induced path in $\Phi$, we have that $\bigcup_{i=1}^n N_i$ $=N_1\cup N_2$, by Theorem \ref{caminos.inducidos.en.Phi}. Applying formula \eqref{formulafamosa.union.intersec}, it follows that $|\bigcup_{i=1}^n N_i|-d_1=d_2-\eta_{12}$. We complete the proof using that $\sigma_{12}=(d_1-\eta_{12})(d_2-\eta_{12})$ and that $d_1\geq d_2$.
\end{proof}

Since every induced cycle in a graph (or multigraph) contains induced paths, we can apply Theorem \ref{caminos.inducidos.en.Phi} to such paths in $\Phi$. In this way, we obtain an extremely interesting result: $\Phi$ can only contain induced cycles of length 3 or 4. This is precisely the content of Theorem \ref{ciclos.inducidos.en.Phi}, whose proof is essentially divided between the following two lemmas. 

\begin{lemma}
	\label{no.hay.C_5.en.Phi.con.|C|<6}
	If $S$ is a split graph and $C$ is an induced cycle in $\Phi(S)$, then $|C|\leq 5$.
\end{lemma}

\begin{proof}
	Assume that $|C|\geq 6$. If $C=v_1\ldots v_nv_1$, we can assume without loss of generality that $d_1=\max\{d_i:i\in[n]\}$ ($d_i=\deg_S(v_i)$). Then, $P=v_1\ldots v_{n-1}$ and $P'=v_1v_n v_{n-1}\ldots v_3$ are induced paths of $\Phi$. Since $n\geq 6$, we obtain that $N_3\supseteq N_{n-1}$ and $N_{n-1}\supseteq N_3$, applying Theorem \ref{caminos.inducidos.en.Phi} to $P$ and $P'$ respectively. Consequently, $N_3=N_{n-1}$ and, therefore, also $N_{\Phi}(v_3)=N_{\Phi}(v_{n-1})$. This means that $v_2v_{n-1} \in\Phi$, which is a contradiction.
\end{proof}

\begin{lemma}
	\label{no.hay.C_5.en.Phi}
	If $S$ is a split graph and $C$ is an induced cycle in $\Phi(S)$, then $|C|\neq 5$.
\end{lemma}

\begin{proof}
	Suppose that $\Phi(S)$ has an induced cycle $C$ of length 5. If $C=v_1\ldots v_5v_1$, we can assume without loss of generality that $\max\{d_i\in[5]\}=d_1$. Since $v_1v_2v_3v_4$ and $v_1v_5v_4v_3$ are induced paths in $S$, we can use Theorem \ref{caminos.inducidos.en.Phi} to infer that $N_4\subseteq N_2$ and that $N_3\subseteq N_5$. Since $\sigma_{25}=0$, we have $N_2\subseteq N_5$ or $N_5\subseteq N_2$. If $N_2\subseteq N_5$, then $N_4\subseteq N_2$, and hence $\sigma_{45}=0$. If $N_5\subseteq N_2$, then $N_3\subseteq N_5$, which implies $\sigma_{23}=0$. Both are contradictions.
\end{proof}

\begin{theorem}
	\label{ciclos.inducidos.en.Phi}
	If $S$ is a split graph and $C$ is an induced cycle in $\Phi(S)$, then $|C|\leq 4$.
\end{theorem}

\begin{proof}
	It follows directly from Lemmas \ref{no.hay.C_5.en.Phi.con.|C|<6} and \ref{no.hay.C_5.en.Phi}.
\end{proof}

As announced in the introduction to this section, we will now see a very interesting result about the simple edges of a $P_4\preceq\Phi$. 

\begin{theorem}
	\label{prohibido.P_4.sigma_23=1}
	If $S$ is a split graph, then $\Phi(S)$ cannot contain an induced path $v_1 v_2 v_3 v_4$ where $\sigma_{23}=1$.
\end{theorem}

The proof of Theorem \ref{prohibido.P_4.sigma_23=1} requires several steps. Therefore, we postpone it momentarily, as it will be obtained more easily by first proving some preliminary lemmas.  

\begin{lemma}
	\label{prohibido.P_4.sigma_23=1.d_1<d_2>d_4}
	If $S$ is a split graph, then $\Phi(S)$ cannot contain an induced path $v_1 v_2 v_3 v_4$ where $\sigma_{23}=1$ and $d_1\leq d_2\geq d_4$ ($d_i=\deg_S(v_i)$).
\end{lemma}

\begin{proof}
	Suppose there exists a split graph $S$ such that $v_1v_2v_3v_4\preceq\Phi=\Phi(S), \sigma_{23}=1$ and $d_1\leq d_2\geq d_4$. Since $\sigma_{23}=1$, by Proposition \ref{prop.basicas.sigma_uv} we have that $d_2=d_3$, which implies that $d_3\geq d_4$ and $N_1\subseteq N_3$. Since $\sigma_{14}=0$, we can assume without loss of generality that $d_4\leq d_1$ and that $N_4\subseteq N_1\subseteq N_3$. But this contradicts that $\sigma_{34}\neq 0$. 
\end{proof}

\begin{lemma}
	\label{prohibido.P_4.sigma_23=1.d_1>d_2<d_4}
	If $S$ is a split graph, then $\Phi(S)$ cannot contain an induced path $v_1 v_2 v_3 v_4$ where $\sigma_{23}=1$ and $d_1\geq d_2\leq d_4$.
\end{lemma}

\begin{proof}
	The proof is entirely analogous to that of Lemma \ref{prohibido.P_4.sigma_23=1.d_1<d_2>d_4}.
\end{proof}

The following proposition allows, among other things, to obtain a quick proof of Lemma \ref{prohibido.P_4.sigma_23=1.d_1<d_2<d_4}. 

\begin{proposition}
	\label{P_3.d_1<d_2.sigma_23=1.implica...}
	Let $S$ be a split graph. If $\Phi(S)$ contains an induced path $v_1v_2v_3$ such that $d_1\leq d_2$ and $\sigma_{23}=1$, then:
	\begin{enumerate}
		\item $N_1-N_2=\{x\}=N_3-N_2$, for some $x\in V(S)$.
		\item $N_3=x\dot{\cup}(N_2\cap N_3)\subset N_1\cup N_2$.
		\item $\sigma_{12}=d_2-d_1+1$.
	\end{enumerate} 
\end{proposition}

\begin{proof}
	First, recall that $d_2=d_3$ and $|N_2-N_3|=1=|N_3-N_2|$, since $\sigma_{23}=1$. Moreover, $N_1\subseteq N_3$, since $\sigma_{13}=0$ and $d_1\leq d_3$. If $|N_1-N_2|>1$, it is clear that we would also have $|N_3-N_2|>1$, since $N_1\subseteq N_3$. Therefore, it must be that $N_1-N_2=\{x\}=N_3-N_2$, for some $x\in V(S)$. 
	
	Since $N_3-N_2=N_3-N_2\cap N_3$, it follows that $N_3=x\dot{\cup}(N_2\cap N_3)$. In particular, $N_3\subset N_1\cup N_2$, and the inclusion is strict because $N_2-N_3\neq\varnothing$.
	
	Finally, since $\sigma_{12}=(d_1-\eta_{12})(d_2-\eta_{12})$ and $d_1-\eta_{12}=|N_1-N_2|=1$, we easily obtain that $\sigma_{12}=d_2-d_1+1$.  
\end{proof}

\begin{lemma}
	\label{prohibido.P_4.sigma_23=1.d_1<d_2<d_4}
	If $S$ is a split graph, then $\Phi(S)$ cannot contain an induced path $v_1 v_2 v_3 v_4$ where $\sigma_{23}=1$ and $d_1\leq d_2\leq d_4$.
\end{lemma}

\begin{proof}
	Suppose such a path exists in $\Phi(S)$. Since $d_1\leq d_2=d_3\leq d_4$ and $\sigma_{14}=\sigma_{24}=0$, it follows that $N_1\cup N_2\subseteq N_4$. But thanks to Proposition \ref{P_3.d_1<d_2.sigma_23=1.implica...} we know that $N_3\subseteq N_1\cup N_2$. Thus, $\sigma_{34}=0$, a contradiction.
\end{proof}

We are now ready to prove Theorem \ref{prohibido.P_4.sigma_23=1}.

\begin{proof}[Proof of Theorem \ref{prohibido.P_4.sigma_23=1}]
	Suppose there exists a split graph $S$ such that
	\[ v_1v_2v_3v_4\preceq\Phi=\Phi(S) \ \text{and} \ \sigma_{23}=1. \]
	Since $d_2=d_3$ (by Proposition \ref{prop.basicas.sigma_uv}), we only have 3 possibilities for the degrees of $v_1, v_2$ and $v_4$ in $S$: 1) $d_1\leq d_2\geq d_4$; 2) $d_1\geq d_2\leq d_4$; 3) $d_1\leq d_2\leq d_4$. However, these cases conflict with Lemmas \ref{prohibido.P_4.sigma_23=1.d_1<d_2>d_4}, \ref{prohibido.P_4.sigma_23=1.d_1>d_2<d_4} and \ref{prohibido.P_4.sigma_23=1.d_1<d_2<d_4}.
\end{proof}

A very important immediate consequence of Theorem \ref{prohibido.P_4.sigma_23=1} is that, in any induced path of $\Phi$, only the first or last edge can be simple. This possibility is indeed realized; that is: there exist split graphs $S$ such that $\Phi(S)$ contains an induced $P_4$ $1234$ with $\sigma_{12}=1$. As an example, we can take the (active) split graph $(G,\{x,y,z,t\},[4])$ where $N_1=x, N_2=y, N_3=\{x,z\}$ and $N_4=\{x,y,t\}$ (in this case: $\Phi\approx P_4$ and $\sigma_{23}=2=\sigma_{34}$).

\begin{corollary}
	\label{no.aristas.simples.internas.en.P_n}
	Let $S$ be a split graph and let $v_1\ldots v_n$ ($n\geq 2$) be an induced path in $\Phi(S)$. If $\sigma_{i,i+1}=1$ for some $i\in[n-1]$, then $i\in\{1,n-1\}$. 
\end{corollary}

\begin{proof}
	It follows immediately from Theorem \ref{prohibido.P_4.sigma_23=1}.
\end{proof}

Recall that in this work, distances between vertices in a multigraph are measured ignoring edge multiplicities. In other words, the metric in $\Phi$ is the usual metric of a simple graph. On the other hand, the fact that paths maximize the diameter of connected graphs with fixed size is a classical result. Therefore, the diameter of a connected $\Phi$ cannot exceed the length of a path multigraph $P$ with the same size, i.e., $\deg(S)$. Clearly, the multiplicity of each edge of $P$ should be as small as possible, if we want to maximize its length in terms of the degree of $S$. Using Corollary \ref{no.aristas.simples.internas.en.P_n}, it is then easy to conclude that the internal edges of $P$ must all have multiplicity 2, while its terminal edges (i.e., the first and last) must have multiplicity 1 or 2. We have thus obtained the following inequality.

\begin{theorem}
	If $S$ is a split graph such that $\Phi=\Phi(S)$ is connected, then
	\[ diam(\Phi)\leq\lceil (\deg(S)+1)/2 \rceil . \] 
\end{theorem}

\begin{proof}
	It follows from the previous discussion.
\end{proof}


\section{Flow configuration} \label{sec:config.flujo}

The \textbf{flow configuration} of a split graph $(S,K,I)$, denoted by $\vec{\Phi}(S)$, is defined as the digraph with $I$ as vertex set, where there is an arc $(u,v)$ from $u$ to $v$ if and only if $\deg_S(u)\leq\deg_S(v)$ and $\sigma_{uv}(S)>0$. Clearly, $\vec{\Phi}(S)\approx \Phi(S)$ as simple graphs, ignoring multiple edges and arc directions. Furthermore: $uv\in\vec{\Phi}(S)$ and $vu\notin\vec{\Phi}(S)$ if and only if $\sigma_{uv}(S)>0$ and $\deg_S(u)<\deg_S(v)$.

Let $(S,K,\{a,b,c\})$ be a split graph such that $\vec{\Phi}=\vec{\Phi}(S)$ is a triangle, and consider the following 4 digraphs (see Figure \ref{triangulos.permitidos}):
\begin{enumerate}
	\item $\Delta_0 =(\{a,b,c\},\{ab,bc,ac\})$ (type 0);
	\item $\Delta_1^+ =(\{a,b,c\},\{ab,cb,ac,ca\})$ (type $1^+$);
	\item $\Delta_1^- =(\{a,b,c\},\{ba,bc,ac,ca\})$ (type $1^-$);
	\item $\Delta_3 =(\{a,b,c\},\{ab,ba,bc,cb,ac,ca\})$ (type 3). 
\end{enumerate}

\begin{figure}[h]
	\centering
	\includegraphics[scale=0.8]{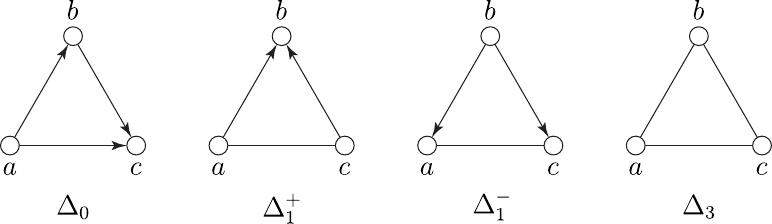}
	\caption{Permitted triangles in $\vec{\Phi}$.}
	\label{triangulos.permitidos}
\end{figure}

If $\{ab,ba,bc,cb\}\subseteq\vec{\Phi}$ (that is, $\vec{\Phi}$ has at least two edges), then $d_a=d_b$ and $d_b=d_c$. Hence, $d_a=d_c$ as well, and therefore $\vec{\Phi}=\Delta_3$. If $\vec{\Phi}$ has exactly one edge, we may assume without loss of generality that $ac,ca\in \vec{\Phi}$. If $ab\in \vec{\Phi}$, then $cb$ must also belong to $\vec{\Phi}$, because if $bc\in \vec{\Phi}$, we would get $d_a\leq d_b\leq d_c$, which would again force $\vec{\Phi}=\Delta_3$. Thus, $\vec{\Phi}=\Delta_1^+$. On the other hand, if $ba\in \vec{\Phi}$, similar arguments yield $\vec{\Phi}=\Delta_1^-$. Finally, suppose $\vec{\Phi}$ has no edges. Without loss of generality, assume $ab,bc\in\vec{\Phi}$. If $ca\in \vec{\Phi}$, then $d_a\leq d_b\leq d_c\leq d_a$, hence $\vec{\Phi}=\Delta_3$, a contradiction. Thus, $ac\in \vec{\Phi}$, and so $\vec{\Phi}=\Delta_0$. We have just proved the following theorem.

\begin{theorem}
	\label{triang.dir.en.Phi_lujo}
	Let $(S,K,\{a,b,c\})$ be a split graph such that $\vec{\Phi}(S)$ is a triangle. Then $\vec{\Phi}(S)$ is one of the 4 digraphs in Figure \ref{triangulos.permitidos}.
\end{theorem}

\begin{proof}
	It follows from the previous discussion.
\end{proof}

We say that $S$ is \textbf{square-free} if $\sigma_{uv}(S)$ is not a positive square for any pair $\{u,v\}\subseteq I$. In this case, we can also say that $\Phi(S)$ is square-free.

\begin{proposition}
	\label{squarefree&d_u=d_v.implica.N_u=N_v}
	Let $S$ be a square-free split graph. If $\deg_S(u)=\deg_S(v)$, then $\sigma_{uv}(S)=0$, $N_S(u)=N_S(v)$, and $N_{\Phi(S)}(u)=N_{\Phi(S)}(v)$.   
\end{proposition}

\begin{proof}
	If $d=\deg_S(u)$, then $\sigma_{uv}(S)=(d-\eta_{uv})^2$. Since $S$ is square-free, we must have $\sigma_{uv}=0$, which implies $d=\eta_{uv}$, that is, $|N_u|=|N_v|=|N_u\cap N_v|$. Hence, $N_u =N_v$, and therefore $N_{\Phi}(u)=N_{\Phi}(v)$, by Proposition \ref{prop.basicas.sigma_uv}.
\end{proof}

In other words, Proposition \ref{squarefree&d_u=d_v.implica.N_u=N_v} states that when $S$ is square-free, its independent vertices with the same degree cannot be adjacent in $\Phi$. Moreover, such vertices are twins in both $S$ and $\Phi$.   

\begin{corollary}
	\label{S.squarefree.implica.triang_tipo0}
	Let $S$ be a square-free split graph. If $\vec{\Phi}(S)$ contains a triangle $T$, then $T$ is of type 0.
\end{corollary}

\begin{proof}
	It follows immediately from Theorem \ref{triang.dir.en.Phi_lujo} using Proposition \ref{squarefree&d_u=d_v.implica.N_u=N_v}.
\end{proof}

\begin{theorem}
	Let $S$ be a split graph such that $\vec{\Phi}(S)=\Delta_1^+$ (see Figure \ref{triangulos.permitidos}). Then $\sigma_{ab}=\sigma_{bc}$ if and only if $\eta_{ab}=\eta_{bc}$.
\end{theorem}

\begin{proof}
	Note that $\sigma_{ab}=\sigma_{bc}$ is equivalent to
	\[(d_a-\eta_{ab})(d_b-\eta_{ab})=(d_b-\eta_{bc})(d_c-\eta_{bc}).\]
	But since $ac,ca\in \Delta_1^+$, we have $d_c=d_a$ in $S$. Then the equivalence becomes
	\[(d_a-\eta_{ab})(d_b-\eta_{ab})=(d_b-\eta_{bc})(d_a-\eta_{bc}).\]	
	Which is equivalent to:
	\begin{align*}
		d_ad_b-(d_a+d_b)\eta_{ab}+\eta_{ab}^2-d_ad_b+(d_a+d_b)\eta_{bc}-\eta_{bc}^2=&0\\
		(d_a+d_b)(\eta_{bc}-\eta_{ab})+\eta_{ab}^2-\eta_{bc}^2=&0\\
		(d_a+d_b)(\eta_{bc}-\eta_{ab})+(\eta_{ab}+\eta_{bc})(\eta_{ab}-\eta_{bc})=&0\\
		(\eta_{bc}-\eta_{ab})(d_a+d_b-\eta_{ab}-\eta_{bc})=&0.
	\end{align*}	
	Since $\sigma_{ab},\sigma_{bc}>0$, it follows that $d_a-\eta_{ab}>0$ and $d_b-\eta_{bc}>0$, so $d_a-\eta_{ab}+d_b-\eta_{bc}>0$. Thus, $\sigma_{ab}=\sigma_{bc}$ if and only if $\eta_{ab}=\eta_{bc}$.
\end{proof}

\begin{lemma}
	\label{lema.5.digrafos.C4.prohibidos}
	Let $S$ be a split graph and let $C$ be an induced 4-cycle in $\vec{\Phi}(S)$. Then $C$ cannot be any of the 5 digraphs shown in Figure \ref{5.digrafos.C4.proibidos}.
	\begin{figure}[H]
		\centering
		\includegraphics[scale=0.8]{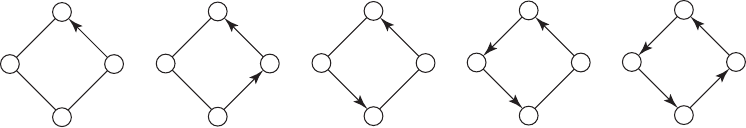}
		\caption{The forbidden $C_4$ digraphs in $\vec{\Phi}$.}
		\label{5.digrafos.C4.proibidos}
	\end{figure}
\end{lemma}

\begin{proof}
	Since all digraphs in Figure \ref{5.digrafos.C4.proibidos} have at least one arc, let $V(C)=\{a,b,c,d\}$ and let $ab$ be the arc. Moreover, for $v\in V(C)$, let $d_v=\deg_S(v)$. Note that for all the considered digraphs we have $d_a<d_b\leq d_c\leq d_d\leq d_a$. This gives $d_a<d_a$, which is absurd.
\end{proof}

\begin{theorem}
	\label{C_4.dir.en.Phi_lujo}
	Let $S$ be a split graph and let $C$ be an induced 4-cycle in $\vec{\Phi}(S)$. Then $C$ is one of the 10 digraphs shown in Figure \ref{10.digrafos.C4.posibles}.
	\begin{figure}[H]
		\centering
		\includegraphics[scale=0.8]{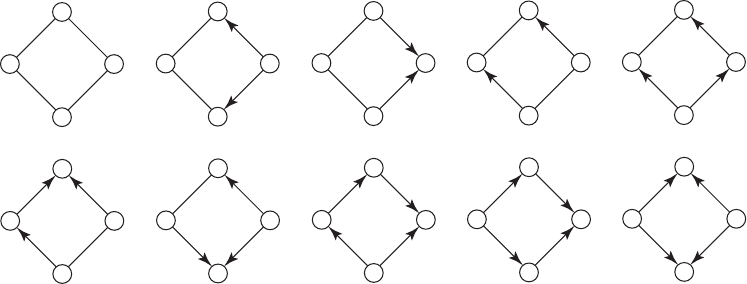}
		\caption{The allowed $C_4$ digraphs in $\vec{\Phi}$.}
		\label{10.digrafos.C4.posibles}
	\end{figure}
\end{theorem}

\begin{proof}
	Up to relabeling, there are 15 digraphs isomorphic to $C_4$: the 10 shown in Figure \ref{10.digrafos.C4.posibles} plus the 5 in Figure \ref{5.digrafos.C4.proibidos}. The result follows from Lemma \ref{lema.5.digrafos.C4.prohibidos}.
\end{proof}


\section{Linear split graphs}

We say that a split graph $(S,K,I)$ is $\varepsilon$\textbf{-linear} if there exists some constant $\varepsilon>0$ such that $|\deg_S(u)-\deg_S(v)|=\varepsilon$ for every edge $uv\in\Phi(S)$. In other words, it is required that for every pair of adjacent vertices in $\Phi$, the absolute difference of their respective degrees in $S$ is constant. Alternatively, when specifying the value of $\varepsilon$ is not important, we may simply say that $S$ (or $\Phi$, or $\vec{\Phi}$) is \textbf{linear}.

From the definition, it follows that if $S$ is $\varepsilon$-linear, then $\vec{\Phi}$ contains no edges. Indeed, if $uv$ is an arc in $\vec{\Phi}$, we have $d_v=d_u+\varepsilon>d_u$ (i.e., the degree strictly increases in the direction of the arc). The term ``linear" is inspired by the well-known fact that a linear function has a constant incremental ratio. We will also see that if $S$ is linear, then $\vec{\Phi}$ is ``essentially" a path, reinforcing this concept. Indeed, one of the central results of this section is that the twin quotient of $\Phi(S)$ is a path when $S$ is a linear split graph. Notably, the vertices in each equivalence class all have the same degree in $S$. As a consequence, the metric in $\Phi$ is surprisingly simple. In fact, we find that the distance between $u$ and $v$ in $\Phi$ depends only on $\varepsilon$ and the difference between the respective degrees of $u$ and $v$ in $S$.

We also study the induced cycles in $\vec{\Phi}$, showing that they can only be of order 4 and of two types (recall that, in principle, there are 10 possible types; see Figure \ref{10.digrafos.C4.posibles}). Last but not least, we again relate split graphs to Number Theory, through results on those that are $p$-simple (we will define this concept later), for $p$ prime, and with an interesting theorem about $\Phi(S)$, when $S$ is linear. \\

The first lemma of this section tells us that a cycle $C$ of order 4 induced in $\vec{\Phi}(S)$ can only be of two types when $S$ is $\varepsilon$-linear. If $\{u,x,y,z\}$ is the vertex set of $C$, then these are the only two possibilities for the arc set of $C$: 1) $\{ ux,uz,xy,zy \}$; 2) $\{ ux,uz,yx,yz \}$. These two digraphs can be seen in Figure \ref{2.digrafos.C4.posibles.split.lineal}. The digraph of the first case is called $C_4(0121)$, or we say that $C$ is of type 0121. The digraph of the second case is called $C_4(0101)$, or we say that $C$ is of type 0101. The reason for this nomenclature lies in the degrees in $S$ of the vertices of $C$. Indeed, in $C_4(0121)$, we have $d_u=d_u+0\varepsilon, d_x=d_u+1\varepsilon, d_y=d_u+2\varepsilon$, and $d_z=d_u+1\varepsilon$. Similarly, 0101 is derived in the other case.

\begin{figure}[h]
	\centering
	\includegraphics[scale=0.8]{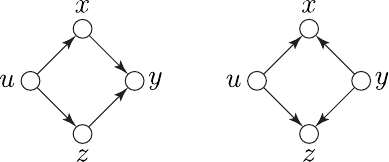}
	\caption{$C_4(0121)$ (left) and $C_4(0101)$ (right).}
	\label{2.digrafos.C4.posibles.split.lineal}
\end{figure}

\begin{lemma}
	\label{C_4.dir.split.lineal}
	Let $S$ be an $\varepsilon$-linear split graph and let $C$ be a cycle of order 4 in $\vec{\Phi}(S)$. Then $C$ is isomorphic to $C_4(0121)$ or $C_4(0101)$ (see Figure \ref{2.digrafos.C4.posibles.split.lineal}).
\end{lemma} 

\begin{proof}
	By Theorem \ref{C_4.dir.en.Phi_lujo}, it suffices to rule out that $C$ can have $\{ ux,xy,yz,uz \}$ as its arc set. Indeed, if this were the case, we would have $d_u+\varepsilon=d_z=d_u+3\varepsilon$, which is absurd.  
\end{proof}

If $S$ is linear, we can easily rule out the possibility that $\vec{\Phi}(S)$ contains triangles. Thus, we confirm that the induced cycles in $\vec{\Phi}(S)$, if any, can only be of length 4 and of type 0121 or 0101.

\begin{theorem}
	\label{ciclos.en.Phi(split.lin)}
	Let $S$ be an $\varepsilon$-linear split graph and let $C$ be an induced cycle in $\vec{\Phi}(S)$. Then, $C$ is isomorphic to $C_4(0121)$ or $C_4(0101)$ (see Figure \ref{2.digrafos.C4.posibles.split.lineal}).
\end{theorem}   

\begin{proof}
	Applying Theorems \ref{ciclos.inducidos.en.Phi} and \ref{triang.dir.en.Phi_lujo}, and Lemma \ref{C_4.dir.split.lineal}, we only need to verify that $C$ is not a triangle of type 0 (see Figure \ref{triangulos.permitidos}). Indeed, if it were, we would have $d_a+\varepsilon =d_b=d_a+2\varepsilon$, a clear absurdity. 
\end{proof}

We now turn our attention to the induced paths of $\vec{\Phi}(S)$, still in the case where $S$ is linear. For this, we first need a lemma about induced paths of length 3.  

\begin{lemma}
	\label{lema.P_4.lineal}
	Let $S$ be an $\varepsilon$-linear split graph and let $uxyz\preceq\Phi(S)$. If $d_v=\deg_S(v)$ and $d_u<d_x$, then $d_x<d_y<d_z$.
\end{lemma}

\begin{proof}
	If it were not true that $d_x<d_y<d_z$, then we would face three possible scenarios: 
	\begin{enumerate}[(1).]
		\item $d_y=d_u<d_x=d_u+\varepsilon=d_z;$ 
		\item $d_u<d_x=d_u+\varepsilon=d_z<d_y=d_u+2\varepsilon$;
		\item $d_u-\varepsilon=d_z<d_y=d_u<d_x=d_u+\varepsilon$.
	\end{enumerate}
	We see that $d_x=d_z$ in cases (1) and (2), while $d_u=d_y$ in the third. This implies, thanks to Proposition \ref{prop.basicas.sigma_uv}, that $x\cong_{\Phi}z$ or $u\cong_{\Phi}y$, respectively. But then $uz\in\Phi$, which contradicts $uxyz\preceq\Phi$.
\end{proof}

By repeatedly applying the previous lemma, we find that in an induced path $P\preceq\vec{\Phi}(S)$ of length 3 or more, the degrees in $S$ of the vertices of $P$ grow strictly and regularly from one end to the other of $P$.   

\begin{theorem}
	\label{los.caminos.largos.son.dir.en.Phi_lujo}
	Let $S$ be an $\varepsilon$-linear split graph and let $P=v_1v_2 \ldots v_n$ be an induced path in $\Phi(S)$ of order $n\geq 4$. If $d_i=\deg_S(v_i)$, then $d_1<d_2<\ldots<d_n$. Moreover, $d_i=d_1+(i-1)\varepsilon$, for all $i\in[n]$.  
\end{theorem} 

\begin{proof}
	We obtain $d_1<d_2<\ldots<d_n$ by applying Lemma \ref{lema.P_4.lineal} to each path $v_kv_{k+1}v_{k+2}v_{k+3}$ of order 4 contained in $P$, for each $k\in[n-3]$. The fact that $d_i=d_1+(i-1)\varepsilon$, for all $i\in[n]$, is straightforward to derive.  
\end{proof}

The following lemma essentially says that if the twin quotient of a graph $G$ is a path, then we can easily reconstruct much of the original structure of $G$. Indeed, we take each pair of consecutive equivalence classes in $[G]$ and perform a join (recall that we denote this operation by the symbol $+$) between their respective induced subgraphs in $G$, which we already know can only be cliques or independent sets (in $G$), by Proposition \ref{ind.clique.G.[G]}. Thus, to completely determine $G$, we only need to specify the cardinality of each class and whether it is a clique or an independent set.  

\begin{lemma}
	\label{aristas.de.G,si.[G].es.camino}
	Let $G$ be a graph. Then $[G]$ is the path $[v_1]\ldots[v_q]$ ($q\geq 2$) if and only if
	\[ E(G)=\bigcup_{i=1}^{q-1}E\big(\langle[v_i]\rangle_G+\langle[v_{i+1}]\rangle_G\big). \] 
\end{lemma}

\begin{proof}
	If $q=2$, the result is obvious. From now on, we assume $q\geq 3$.
	
	($\Rightarrow$). If $a\in[v_i]$ and $b\in[v_{i+1}]$, then $ab\in G$ by Proposition \ref{lema.fundamental.uv.[u][v]}, since $[v_i]$ and $[v_{i+1}]$ are neighbors in $[G]$ by hypothesis. Thus,
	\[ E(G)\supseteq\bigcup_{i=1}^{q-1}E\big(\langle[v_i]\rangle_G+\langle[v_{i+1}]\rangle_G\big). \]
	To prove that this inclusion cannot be proper, suppose that 
	\[ ab\in E(G)-\bigcup_{i=1}^{q-1}E\big(\langle[v_i]\rangle_G+\langle[v_{i+1}]\rangle_G\big). \]
	Then, $a\in[v_i], b\in[v_j]$ and $|j-i|\geq 2$. Since $[G]$ is the path $[v_1]\ldots[v_q]$, it follows that $[a][b]=[v_i][v_j]\notin E[G]$. But then $ab\notin E(G)$, by Proposition \ref{lema.fundamental.uv.[u][v]}, and we get a contradiction.
	
	($\Leftarrow$). Straightforward.
\end{proof}

We want to craft a theorem that includes the most relevant properties of a linear split graph $S$ and its associated graphs. To this end, we first prove two auxiliary results relating the twin vertices of $S$, their degrees in $S$, and the multiplicities in $\Phi(S)$ of the edges containing them.  

\begin{lemma}
	\label{lineal.mismo.grado.implica.gem}
	Let $(S,K,I)$ be a linear split graph and let $a,b\in I$. If $d_a=d_b$, then $\sigma_{ab}=0$. Moreover, $a$ and $b$ are twins in $S$ and in $\Phi(S)$.  
\end{lemma}

\begin{proof}
	If $ab\in E(\Phi)$ (i.e., $\sigma_{ab}\neq 0$), then $|d_a-d_b|>0$, since $S$ is linear. Thus, $d_a\neq d_b$. The fact that $a$ and $b$ are twins in $S$ and in $\Phi$ follows immediately from Proposition \ref{prop.basicas.sigma_uv}. 
\end{proof}

\begin{lemma}
	\label{sigma_xz=sigma_yt}
	Let $S$ be a split graph. If $x\cong_S y$ and $z\cong_S t$, then $\sigma_{xz}=\sigma_{yt}$.   
\end{lemma}

\begin{proof}
	Since $x\cong_S y$, we have $N_S(x)=N_S(y)$. Since $z\cong_S t$, we have $N_S(z)=N_S(t)$. Thus,
	\[ \sigma_{xz}=(\deg_S(x)-|N_S(x)\cap N_S(z)|)(\deg_S(z)-|N_S(x)\cap N_S(z)|)= \]
	\[ (\deg_S(y)-|N_S(y)\cap N_S(t)|)(\deg_S(t)-|N_S(y)\cap N_S(t)|)=\sigma_{yt}. \]
\end{proof}

\begin{theorem}
	\label{split.lineal.propiedades}
	Let $(S,K,I)$ be a balanced and $\varepsilon$-linear split graph such that $\Phi=\Phi(S)$ is connected and $[\Phi]\neq K_1$. Then $S$ has the following properties:   
	\begin{enumerate}
		\item if $v_i$ is a vertex in $I$ whose degree in $S$ is $\delta_i=\delta_1+(i-1)\varepsilon$, for some $i,\delta_1\geq 1$, then each class $[v_i]_{\Phi}$ is an independent set in $\Phi$;
		\item $[v_i]_S=\{ v\in I:\deg_S(v)=\delta_i \}= [v_i]_{\Phi}$;
		\item $[\Phi(S)]$ is the path $[v_1]_{\Phi}\ldots[v_q]_{\Phi}$, where $q=|[I]_S|=|[I]_{\Phi}|$;
		\item $E(\Phi)=\bigcup_{i=1}^{q-1}E\big(\langle[v_i]_{\Phi}\rangle_{\Phi}+\langle[v_{i+1}]_{\Phi}\rangle_{\Phi}\big)$, ignoring multiplicities; 
		\item if $\sigma_{ab}\neq 0$, then $a\in[v_i]_{\Phi}, b\in[v_{i+1}]_{\Phi}$ and $\sigma_{ab}=\sigma_{i,i+1}$, for some $i\in[q-1]$;
		\item $\deg(S)=\sum_{i=1}^{q-1}\sigma_{i,i+1}|[v_i]_{\Phi}||[v_{i+1}]_{\Phi}|$.
	\end{enumerate}
\end{theorem}

\begin{proof}
	For simplicity, throughout the proof, the symbol $[v_i]_{\Phi}$ will be replaced by $[v_i]$.
	\begin{enumerate}[(1).]
		\item If there exists a $u\in [v_i]$ such that $uv_i\in\Phi$, then both $u$ and $v_i$ have a common neighbor $w$ in another equivalence class, since $[\Phi]\neq K_1$ and $\Phi$ is connected (recall Proposition \ref{lema.fundamental.uv.[u][v]}). Hence, $\Phi$ contains the triangle $uv_iwu$, contradicting Theorem \ref{ciclos.en.Phi(split.lin)}.
		
		\item By Lemma \ref{lineal.mismo.grado.implica.gem}, we have $[v_i],[v_i]_S\supseteq\{ v\in I:\deg_S(v)=\delta_i \}$. If $x\in [v_i]_S$, then $x\in I$, by Lemma \ref{v.ind.nogemelo.v.clique}, and consequently $N_S(x)=N_S(v_i)$. Thus, $x\in [v_i]$, by Proposition \ref{prop.basicas.sigma_uv}. Moreover, it is clear that $\deg_S(x)=\delta_i$. Hence, $[v_i]_S\subseteq [v_i],\{ v\in I:\deg_S(v)=\delta_i \}$. 
		
		To prove that $[v_i]=[v_i]_S$, suppose there exists a vertex $x\in [v_i]-[v_i]_S$, i.e., $x$ is a twin of $v_i$ in $\Phi$ but not in $S$. This means that $N_S(x)\neq N_S(v_i)$. Since $[v_i]$ is an independent set in $\Phi$, by (1), it follows that $v_ix\notin\Phi$, and thus $N_S(x)\subset N_S(v_i)$ or $N_S(v_i)\subset N_S(x)$, by Proposition \ref{prop.basicas.sigma_uv}. Hence, $\deg_S(x)\neq\delta_i$. Since $\Phi$ is connected, it must be that $\deg_S(x)=\delta_j$ for some $j\neq i$, and thus $x\in[v_j]_S\subseteq [v_j]$, by the first part of (2). But then $x\in [v_i]\cap[v_j]=\varnothing$, absurd. 
		
		\item Since $\delta_{i+1}-\delta_i=\varepsilon$, it follows that $v_iv_{i+1}\in E(\Phi)$ for all $i\in[q-1]$, by the definition of an $\varepsilon$-linear split graph. Thus, $[v_i][v_{i+1}]\in E[\Phi]$ for all $i\in[q-1]$, by Proposition \ref{lema.fundamental.uv.[u][v]}. Now suppose that $[v_i][v_j]\in E[\Phi]$, for some $i,j\in[q]$ such that $|j-i|\geq 2$. This implies that $v_iv_j\in E(\Phi)$, by Proposition \ref{lema.fundamental.uv.[u][v]}. On the other hand, we have $|\delta_j-\delta_i|\geq 2\varepsilon$. But then $v_iv_j\notin E(\Phi)$, by the definition of an $\varepsilon$-linear split graph. This contradiction shows that $[\Phi]$ is indeed the path $[v_1]\ldots[v_q]$. 
		
		\item It is a direct application of Lemma \ref{aristas.de.G,si.[G].es.camino}, thanks to (3).   
		
		\item We know that $\sigma_{ab}\neq 0$ means $ab\in E(\Phi)$. Thus, $ab\in E\big(\langle[v_i]\rangle_{\Phi}+\langle[v_{i+1}]\rangle_{\Phi}\big)$ for some $i\in[q-1]$, by (4). More specifically, we have $a\in[v_i]$ and $b\in[v_{i+1}]$, by (1). Finally, combining Lemma \ref{sigma_xz=sigma_yt} with (2), we obtain $\sigma_{ab}=\sigma_{i,i+1}$.
		
		\item Since $G_i=\langle[v_i]\rangle_{\Phi}+\langle[v_{i+1}]\rangle_{\Phi}$ is a complete bipartite graph, by (1), we have $\Vert G_i\Vert=|[v_i]||[v_{i+1}]|$, for each $i\in[q-1]$. On the other hand, each edge of $G_i$ has multiplicity $\sigma_{i,i+1}$, by (5). The required formula then follows immediately from (4), since $E(G_i)\cap E(G_j)=\varnothing$ for all $j\neq i$, by (1).
	\end{enumerate}	  
\end{proof}

As a corollary of the previous theorem, we obtain that the metric of $\Phi(S)$, when $(S,K,I)$ is $\varepsilon$-linear, is completely determined by $\varepsilon$ and the degrees in $S$ of the vertices in $I$.

\begin{corollary}
	\label{distancias.en.split.lineal}
	Let $S$ be an $\varepsilon$-linear split graph such that $\Phi=\Phi(S)$ is connected, and let $u$ and $v$ be two distinct vertices in $I$. If $u\cong v$, then $dist_{\Phi}(u,v)=2$; otherwise:
	\[ dist_{\Phi}(u,v)=\frac{1}{\varepsilon}\Big|\deg_S(u)-\deg_S(v)\Big|. \]
	In particular, if $diam(\Phi)\geq 3$, then
	\[ diam(\Phi)=\frac{1}{\varepsilon}\Big(d_{\max}-d_{\min}\Big) =|[\Phi]|-1, \]
	where $d_{\max}=\max\{\deg_S(v):v\in I\}$ and $d_{\min}=\min\{\deg_S(v):v\in I\}$. 
\end{corollary}

\begin{proof}
	This is an immediate consequence of items (1), (3), and (4) of Theorem \ref{split.lineal.propiedades}. 
\end{proof}

We say that $S$ is $n$\textbf{-simple} if there exists a positive integer $n$ such that $\sigma_{uv}(S)\in\{0,n\}$ for all $\{u,v\}\subseteq I$. In this case, we may also say that $\Phi(S)$ (or $\vec{\Phi}$) is $n$-simple. When $n=1$, we recover for $\Phi$ the property of being a simple graph and use the term ``simple" instead of ``1-simple".

When $n$ is prime, it is easy to see that $S$ is $(n-1)$-linear. This is precisely the content of the next theorem.

\begin{theorem}
	\label{p-simple.implica.p-1-lineal}
	Let $p$ be a prime number. If $S$ is a $p$-simple split graph, then $S$ is $(p-1)$-linear.
\end{theorem}

\begin{proof}
	If $S$ is $p$-simple, let $uv$ be an arbitrary edge of $\Phi(S)$. Then, $\sigma_{uv}=(d_u-\eta_{uv})(d_v-\eta_{uv})=p$, and thus $(d_u-\eta_{uv},d_v-\eta_{uv})\in\{(p,1),(1,p)\}$, since $p$ is prime. Hence, $|(d_u-\eta_{uv})-(d_v-\eta_{uv})|=|d_u-d_v|=p-1$.
\end{proof}

When $p$ is an odd prime, Theorems \ref{p-simple.implica.p-1-lineal} and \ref{los.caminos.largos.son.dir.en.Phi_lujo} imply that in a $p$-simple split graph $S$, the degrees in $S$ of its independent vertices are all even or all odd, since $p-1$ is even.

\begin{corollary}
	Let $(S,K,I)$ be a $p$-simple split graph. If $p$ is prime and $p\neq 2$, then $d_u\equiv d_v\pmod 2$, for all $\{u,v\}\subseteq I$.
\end{corollary}

\begin{proof}
	It follows from the previous discussion.
\end{proof}

\begin{theorem}
	\label{lineal.K-d_max.divide.sigma}
	Let $(S,K,I)$ be a linear split graph such that $K=\bigcup_{v\in I}N_v$, where the independent vertices have the following degrees in $S$: $\delta_1>\delta_2>\ldots >\delta_q$, for some integer $q$. If $a$ and $b$ are two vertices in $I$ such that $d_a=\delta_1$ and $d_b=\delta_2$, then $|K|-d_a$ is a divisor of $\sigma_{ab}$ less than $\sqrt{\sigma_{ab}}$.
\end{theorem}

\begin{proof}
	For this proof, it is essential to first review the statement of Theorem \ref{split.lineal.propiedades}. With this result in mind, let $v_i$ be an independent vertex of $S$ such that $\deg_S(v_i)=\delta_i$. Clearly, $N_u=N_i$ for all $u\in[v_i]$. Thus, $K=\bigcup_{i=1}^q N_i$. Since, by Theorem \ref{split.lineal.propiedades}, $v_1\ldots v_q$ is an induced path in $\Phi$, it follows by Corollary \ref{K-d_max.divide.sigma} that $|K|-d_a$ is a divisor of $\sigma_{ab}$ less than or equal to $\sqrt{\sigma_{ab}}$. More specifically, we have $|K|-d_a<\sqrt{\sigma_{ab}}$, since $d_a>d_b$. 
\end{proof}

\begin{corollary}
	\label{S.p-simple.Phi.camino.implica.d_max=|K|-1=d_min+(|I|-1)(p-1)}
	Let $(S,K,I)$ be a $p$-simple split graph, and let $d_{\min}$ and $d_{\max}$ be the minimum and maximum degrees in $S$ of the vertices in $I$, respectively, with $|I|\geq 2$. If $K=\bigcup_{v\in I}N_v$, $p$ is prime, and $\Phi(S)$ is a path, then
	\[ d_{\max}=|K|-1=d_{\min}+(|I|-1)(p-1). \] 
	In particular, if $p=2$ and $d_{\min}=1$, then
	\[ |I|=d_{\max}=|K|-1. \]
\end{corollary}

\begin{proof}
	Thanks to Theorem \ref{p-simple.implica.p-1-lineal}, we know that $S$ is linear. Thus, Theorem \ref{lineal.K-d_max.divide.sigma} tells us that $|K|-d_{\max}$ divides $p$ and is $\leq\sqrt{p}$. Since $p$ is prime, it follows that $|K|-d_{\max}=1$. The equality $d_{\max}=d_{\min}+(|I|-1)(p-1)$, on the other hand, is a particular case of Corollary \ref{distancias.en.split.lineal}, since $S$ is $(p-1)$-linear.   
\end{proof}


\chapter{The $\Delta$ property}\label{cap:La.Prop.Delta}

In this chapter, we address a problem that lies at the intersection of Graph Theory and Number Theory: establishing conditions on $n\in\mathbb{N}$ for the existence of a split graph $S$ such that $\vec{\Phi}(S)$ is an $n$-simple triangle of type 0 (see Figure \ref{triangulos.permitidos}). This problem may initially seem overly specific, but in fact, considering that the cycles in $\Phi$ can only be of size 3 or 4, it represents a significant step forward in understanding the 2-switch structure of split graphs. Moreover, the bridge it creates with Number Theory opens many doors to future research in this area.

The core of this chapter lies in the definition and study of the $\Delta$ property and of those natural numbers that do or do not satisfy it. We denote by $\mathbb{N}(\Delta)$ the set of natural numbers with the $\Delta$ property. This is a property involving the complementary divisors of a number and certain sets of their differences and sums. In this context, we prove two fundamental theorems. The first one tells us that if $\vec{\Phi}$ is an $n$-simple triangle of type 0, then $n\in\mathbb{N}(\Delta)$. The corollaries of this result generate numerous interesting connections between the sets that define the $\Delta$ property and $n$-simple and linear split graphs. The second theorem is a kind of converse of the first: if $n\in\mathbb{N}(\Delta)$, then we can construct a split graph $S$ such that $\vec{\Phi}(S)$ is an $n$-simple triangle of type 0. This, in summary, is the content of Section \ref{sec:Graphs<->NumberTheory}, the first of the three sections of this chapter.

In the following three sections, we momentarily step away from Graph Theory to focus our efforts on the $\Delta$ property. In Section \ref{sec:Delta.prim}, we introduce the concept of a $\Delta$-primitive number, an analogue of the concept of a prime number, but for the members of $\mathbb{N}(\Delta)$. Indeed, these numbers are essentially the ``atoms" that compose each member of $\mathbb{N}(\Delta)$. A notable results we will prove is the fact that there are infinite $\Delta$-primitive numbers. We note that perfect squares and square-free numbers take on a certain importance in this theoretical framework. In this regard, it is shown that the squares with the $\Delta$ property form an abelian semigroup under the usual multiplication, and we also conjecture the infinitude of $\Delta$-primitive squares.

In Section \ref{sec:polinom.generadores}, we develop a very effective method to generate numbers with the $\Delta$ property. To this end, we construct an infinite family of polynomials $f(x)$ of degree 3, such that $f(x)\in\mathbb{N}(\Delta)$ for all $x\in[n_0,\infty)\cap\mathbb{N}$, for some explicit $n_0\in\mathbb{N}$.

Finally, in the last section, we find a long list of families of numbers that do not satisfy the $\Delta$ condition. A key insight in this direction is that $pk\notin\mathbb{N}(\Delta)$ when $p$ is a ``large" prime compared to $k$. Another observed behavior is that those numbers with very few prime factors in their factorization (for example, powers of a prime or products of certain powers of two primes) are less likely to satisfy the $\Delta$ condition. We conjecture that if $n$ is a $\Delta$-primitive of the form $p^xq^y$ ($p,q$ primes), then $n\in\{24,40\}$, and we characterize when a product of three primes $pqr\in\mathbb{N}(\Delta)$. After all this analysis, we return to split graphs, establishing a result about the $n$-simple induced cycles of $\Phi$, when $n$ does not satisfy the $\Delta$ condition and is not a square.


\section{From Graph Theory to Number Theory}\label{sec:Graphs<->NumberTheory}

If $D_n$ is the set of positive divisors of a natural number $n$, we define:
\begin{equation*}
	D^*_n =\{|a-b|: a,b\in D_n, ab=n\}.
\end{equation*}
That is, $D^*_n$ is the set of non-negative differences between complementary divisors of $n$. For example, we have $D_1^*=\{0\}, D_7^*=\{6\}, D_9^*=\{8,0\}, D_{10}^*=\{9,3\}, D_{16}^*=\{15,6,0\}$ and $D_{18}^*=\{17,7,3\}$. Clearly, $\max(D_n^*)=n-1$ for all $n\in\mathbb{N}$. Also note that $0\in D^*_n$ if and only if $n$ is a perfect square, and that $D_p^*=\{p-1\}$ if and only if $p$ is prime or $p=1$. A basic and well-known property of the set $D_n$ is that
\begin{equation}
	|\{ x\in D_n:x<\sqrt{n} \}|=|\{x\in D_n:x>\sqrt{n}  \}|.
	\label{particion.D_n}
\end{equation}
This can be proven by showing that if $x<\sqrt{n}$, then the function mapping $x$ to its complement $n/x$ is a bijection between the sets in (\ref{particion.D_n}). Then, letting $k=|\{x\in D_n:x<\sqrt{n}\}|$, it follows that $|D_n-\sqrt{n}|=2k$ and $|D_n^* -0|=k$. In particular, the following three statements are equivalent: 1) $\sqrt{n}\in\mathbb{N}$; 2) $|D_n|=2k+1$; 3) $|D_n^*|=k+1$.

From $D^*_n$, we now define the following set:
\begin{equation*}
	D^+_n =\{x+y: x,y\in D^*_n-\{0\}\}.
\end{equation*}
For example, $D_1^+=\varnothing, D_7^+=\{12\}, D_9^+=\{16\}, D_{10}^+=\{18,12,6\}, D_{16}^+=\{30,21,12\}$ and $D_{18}^+=\{34,24,20,14,10,6\}$. Clearly, $\max(D_n^+)=2(n-1)$ for all $n\geq 2$.
We say that a natural number $n$ has the \textbf{$\Delta$ property} (or that $n$ satisfies the $\Delta$ condition) if
\begin{equation*}
	D^*_n \cap D^+_n \neq\varnothing.
\end{equation*}
We denote by $\mathbb{N}(\Delta)$ the set of all natural numbers that satisfy the $\Delta$ condition. In the previous examples we can see that, if $n\in\{1,7,9,10,16,18\}$, then $n\notin\mathbb{N}(\Delta)$. By inspection, we can quickly verify that 24 is the first natural number that satisfies the $\Delta$ condition. Indeed, since $D_{24}=\{1,2,3,4,6,8,12,24\}$, we have
\[ D_{24}^*=\{23,10,5,2\}, \]
\[ D_{24}^+=\{46,33,28,25,20,15,12,10,7,4\}, \]
and thus,
\[ D_{24}^*\cap D_{24}^+=\{10\}. \]
The second element of $\mathbb{N}(\Delta)$ is 40. This means that no number in the interval $[25,39]$ satisfies the $\Delta$ condition. The positive divisors of 40 are 1, 2, 4, 5, 8, 10, 20, and 40. Then,
\[ D_{40}^*=\{39,18,6,3\}, \]
\[ D_{40}^+=\{6,9,12,21,24,36,42,45,57,78\}, \]
and hence
\[ D_{40}^*\cap D_{40}^+=\{6\}. \]
As we can see, $\mathbb{N}(\Delta)\neq\varnothing$. A very natural question arises: is $\mathbb{N}(\Delta)$ finite or infinite? In the next section, we will answer this question in a very simple and elegant way (spoiler: $\mathbb{N}(\Delta)$ is infinite!).

\begin{proposition}
	\label{Delta.prop.caracterizacion}
	A natural number $n$ satisfies the $\Delta$ condition if and only if there exists a triple $\{x,y,z\}\subseteq D_n$ such that $1<x<y\leq z<\sqrt{n}$ and
	\begin{equation}
		\frac{n}{x}-x = \frac{n}{y}-y + \frac{n}{z}-z.
		\label{ecuacion.Delta.prop}
	\end{equation}
\end{proposition}

\begin{proof}
	A number $c$ belongs to $D^*_{n}\cap D^+_{n}$ if and only if $c\in D^*_{n}$ and $c=a+b$, for some $a,b\in D^*_{n}-\{0\}$. Since $a,b,c\in D^*_{n}$, we can write them as
	\begin{equation*}
		a=\frac{n}{z}-z, \ b=\frac{n}{y}-y, \ c=\frac{n}{x}-x,
	\end{equation*}
	where $x,y,z$ are divisors of $n$. As $a,b>0$, then also $c>0$. Thus, $x^2,y^2,z^2<n$, so $x,y,z<\sqrt{n}$. Since $c=a+b$ and $a,b>0$, clearly $a,b<c$, that is, $x<y,z$. By symmetry on the right-hand side of \eqref{ecuacion.Delta.prop}, we may assume without loss of generality that $y\leq z$. Finally, suppose $x=1$. Since $y,z\geq 2$, we have
	\[ n-1=c=a+b\leq 2\Big(\frac{n}{2}-2\Big)=n-4, \]
	which is a contradiction. Therefore, it must be $x>1$.
\end{proof}

The condition \eqref{ecuacion.Delta.prop} can be rewritten as
\begin{equation}
	(xy+xz-yz)n=xyz(z+y-x),
	\label{ecuacion2.Delta.prop}	
\end{equation}
through elementary algebraic manipulations. The triple $\{x,y,z\}$ referred to in Proposition \ref{Delta.prop.caracterizacion} is $\{2,3,3\}$ for $n=24$ and $\{4,5,5\}$ for $n=40$. Thus,
\[ \frac{24}{2}-2=10=2\Big(\frac{24}{3}-3\Big), \  \frac{40}{4}-4=6=2\Big(\frac{40}{5}-5\Big). \]
The next two results establish some arguments used in the proof of Proposition \ref{Delta.prop.caracterizacion}.

\begin{proposition}
	\label{cotasup.D_n^*}
	Let $n$ be a natural number and let $x,y\in D_n$ such that $x\leq y\leq\sqrt{n}$. Then
	\begin{equation}
		\label{eq15}
		\frac{n}{y}-y\leq\frac{n}{x}-x,
	\end{equation} 
	and equality holds if and only if $x=y$. In particular, we have
	\[ \max\big(D_n^*-\{n-1\}\big)\leq\frac{n}{2}-2. \]
\end{proposition}

\begin{proof}
	The hypothesis $x\leq y$ implies $-y\leq -x$ and $\frac{n}{y}\leq\frac{n}{x}$, which added together yield inequality \eqref{eq15}. To obtain the latter bound, we first recall that $\max(D_n^*)=n-1$, and then take $x=2$ in \eqref{eq15}.
\end{proof}

\begin{corollary}
	\label{cotasup.D^*capD^+}
	For all $n\in\mathbb{N}(\Delta)$:
	\[ \max(D_n^*\cap D_n^+)\leq \frac{2n}{3}-6. \]
\end{corollary}

\begin{proof}
	Since $D_n^*\cap D_n^+\neq\varnothing$, there exists a triple $\{x,y,z\}\subseteq D_n$ that satisfies \eqref{ecuacion.Delta.prop}, such that $1<x<y\leq z<\sqrt{n}$ and
	\[ \frac{n}{x}-x = \max(D_n^*\cap D_n^+)=m.\]
	Since $y,z\geq 3$, it follows from Proposition \ref{cotasup.D_n^*} that $m\leq 2\big(\frac{n}{3}-3\big)$.
\end{proof}

A curious consequence of Proposition \ref{cotasup.D_n^*} is the exotic characterization of perfect squares we now present.

\begin{corollary}
	Let $D^*_n +D^*_n =\{x+y:x,y\in D^*_n\}$. An integer $n>1$ is a perfect square if and only if 
	\[ D^*_n \subset D^*_n+D^*_n. \]
\end{corollary}

\begin{proof}
	If $n$ is a square, then $0\in D^*_n$ and thus $x+0\in D^*_n+D^*_n$ for all $x\in D^*_n$. Hence, $D^*_n \subseteq D^*_n+D^*_n$. To see that this inclusion is always strict, note that $n-1\in D^*_n$ but $2n-2=(n-1)+(n-1)\in (D^*_n+D^*_n)-D^*_n$, since $n>1$.
	
	For the converse, suppose $D^*_n \subset D^*_n+D^*_n$. This means, in particular, that $n-1\in D^*_n\cap(D^*_n+D^*_n)$. Then, there exist $y,z\in D_n$ such that $y^2,z^2\leq n$ and
	\begin{equation}
		\frac{n}{z}-z+\frac{n}{y}-y=n-1.
		\label{ecuac.Delta.prop.x=1}	
	\end{equation} 
	If $y,z\geq 2$, then by Proposition \ref{cotasup.D_n^*}, we would have
	\[ \frac{n}{z}-z, \frac{n}{y}-y \leq\frac{n}{2}-2, \]
	and thus (\ref{ecuac.Delta.prop.x=1}) could not hold. Therefore, either $y=1$ or $z=1$. If $y=1$, then $n=z^2$. If $z=1$, then $n=y^2$.
\end{proof}

We now show the connection that exists between split graphs and the $\Delta$ property. The following theorem, whose proof is surprisingly simple, is very important because it provides a bridge from Graph Theory to Number Theory.

\begin{theorem}
	\label{triang.n-simple&tipo0.implica.Delta.prop}
	Let $S$ be a split graph.
	\begin{enumerate}
		\item If $\sigma_{uv}\neq 0$, then $|\deg_S(u)-\deg_S(v)|\in D_{\sigma_{uv}}^*$.
		\item If $S$ is $n$-simple and $\vec{\Phi}(S)=\Delta_0$ (see Figure \ref{triangulos.permitidos}), then $n\in\mathbb{N}(\Delta)$.
	\end{enumerate}
\end{theorem}
\begin{figure}[h]
	\centering
	\begin{tikzpicture}[scale=2, every node/.style={circle, draw, thick, minimum size=6pt, inner sep=2pt}]
		
		\node [label=above left:{$\Phi(S)$}](x1) at (0,0) {$a$};
		\node (x2) at (1,0) {$c$};
		\node (x3) at (0.5, {0.5*sqrt(3)}) {$b$};
		
		\draw[thick] (x1) -- (x2) node[midway, draw=none, fill=white, inner sep=1pt] {$n$};
		\draw[thick] (x2) -- (x3) node[midway, draw=none, fill=white, inner sep=1pt] {$n$};
		\draw[thick] (x3) -- (x1) node[midway, draw=none, fill=white, inner sep=1pt] {$n$};
		
		\draw[bend left=15] (x1) to (x2) ;
		\draw[bend right=15] (x1) to (x3) ;
		\draw[bend left=15] (x2) to (x3) ;
		
		\node (a) at (2,0) {$d_a$};
		\node [label=above right:{$S$}](c) at (3,0) {$d_c$};
		\node (b) at (2.5, {0.5*sqrt(3)}) {$d_b$};
		
		\draw[->, thick, >=stealth, line width=2pt, scale=1.5] (a) -- (c); 
		\draw[->, thick, >=stealth, line width=2pt, scale=1.5] (a) -- (b);
		\draw[->, thick, >=stealth, line width=2pt, scale=1.5] (b) -- (c);
	\end{tikzpicture}
	\caption{Hypothesis of Theorem \ref{triang.n-simple&tipo0.implica.Delta.prop}.}
\end{figure}
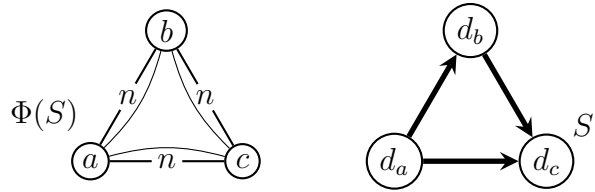

\begin{proof}
	(1). Since $0\neq\sigma_{uv}=(d_u-\eta_{uv})(d_v-\eta_{uv})$, we have that $d_u-\eta_{uv}$ and $d_v-\eta_{uv}$ are complementary divisors of $\sigma_{uv}$. Therefore,
	\[ |(d_u-\eta_{uv})-(d_v-\eta_{uv})|=|d_u-d_v|\in D_{\sigma_{uv}}^*. \]
	
	(2). From (1), we have $|d_u-d_v|\in D_n^*$ for every edge $uv\in\Phi$. Since $\vec{\Phi}=\Delta_0$, it follows that $d_c>d_a$, $d_b>d_a$, and $d_c>d_b$. Moreover, the arcs of $\vec{\Phi}$ indicate that the degree increase (in $S$) from $a$ to $c$ equals the increase from $a$ to $b$ plus the increase from $b$ to $c$. Hence,
	\[ d_c-d_a =(d_b-d_a)+(d_c-d_b)\in D^*_n \cap D^+_n, \] 
	which shows that $D^*_n \cap D^+_n\neq\varnothing$.
\end{proof}

Theorem \ref{Delta.prop.caracterizacion} has many corollaries that relate the set $D_n^*$ to split graphs, especially those that are $n$-simple and $\varepsilon$-linear.

\begin{corollary}
	\label{triang.n-simple&n-nonsquare.implica.Delta.prop}
	Let $S$ be a split graph such that $\Phi(S)$ is an $n$-simple triangle. If $n$ is not a perfect square, then $n\in\mathbb{N}(\Delta)$. 
\end{corollary}

\begin{proof}
	By Corollary \ref{S.squarefree.implica.triang_tipo0}, $\vec{\Phi}$ must be of type 0. Then, by item (2) of Theorem \ref{triang.n-simple&tipo0.implica.Delta.prop}, we conclude that $n\in\mathbb{N}(\Delta)$.
\end{proof}

\begin{corollary}
	\label{|d_u-d_v|<=sigma_uv-1}
	If $S$ is a split graph, then for every edge $uv\in\Phi(S)$ we have
	\[|\deg_S(u)-\deg_S(v)|\leq\sigma_{uv}(S)-1.\]
\end{corollary}

\begin{proof}
	Recalling that $\max(D_n^*)= n-1$ for every $n\in\mathbb{N}$, the claim follows immediately from item (1) of Theorem \ref{triang.n-simple&tipo0.implica.Delta.prop}.
\end{proof}

\begin{corollary}
	If $S$ is an $n$-simple split graph, then 
	\begin{equation}
		\label{eq16}
		\{|\deg_S(u)-\deg_S(v)|: \sigma_{uv}(S)\neq 0\}\subseteq D_n^*.
	\end{equation}
	In particular, if $\Phi(S)$ has two adjacent vertices with the same degree in $S$, then $n$ is a perfect square. 
\end{corollary}

\begin{proof}
	The inclusion \eqref{eq16} is a direct consequence of item (1) of Theorem \ref{triang.n-simple&tipo0.implica.Delta.prop}. If $\Phi$ has two adjacent vertices with the same degree in $S$, then by \eqref{eq16}, we have $0\in D_n^*$, which is equivalent to $n$ being a square. 
\end{proof}

\begin{corollary}
	\label{split.lineal.relacion.epsilon.sigma_uv.D^*}
	Let $(S,K,I)$ be an $\varepsilon$-linear split graph. Then, we have the following:
	\begin{enumerate}
		\item $\varepsilon\in\bigcap_{uv\in\Phi} D_{\sigma_{uv}}^*$.
		\item $\varepsilon\leq\min\{\sigma_{uv}>0:u,v\in I\}-1$.
		\item If $\sigma_{ab}=p$ and $\sigma_{cd}=q$ for some $\{a,b\},\{c,d\}\subseteq I$ and some primes $p$ and $q$, then $p=q$ and $\varepsilon=p-1$.
		\item $\sigma_{uv}\neq 1$, for all $\{u,v\}\subseteq I$. In other words, $\Phi(S)$ cannot have simple edges.
		\item If $n=\max\{\sigma_{uv}>0:u,v\in I\}$ and $\varepsilon=n-1$, then $S$ is $n$-simple. 
	\end{enumerate}
\end{corollary}

\begin{proof}
	(1). It follows immediately from Theorem \ref{triang.n-simple&tipo0.implica.Delta.prop}, since $|d_u-d_v|=\varepsilon$ for every edge $uv\in\Phi$. 
	
	(2). It follows immediately from Corollary \ref{|d_u-d_v|<=sigma_uv-1}, since $|d_u-d_v|=\varepsilon$ for all edges $uv\in\Phi$.
	
	(3). Using (1), we get that $\varepsilon\in\bigcap_{uv\in\Phi} D_{\sigma_{uv}}^*\subseteq D_p^*\cap D_q^*=\{p-1\}\cap\{q-1\}$. Thus, $p-1=\varepsilon=q-1$.
	
	(4). If $\sigma_{ab}=1$ for some $\{a,b\}\subseteq I$, then by Corollary \ref{|d_u-d_v|<=sigma_uv-1}, we have $\varepsilon\leq\sigma_{ab}-1=0$, which is a contradiction.
	
	(5). Let $A=\{\sigma_{uv}>0:u,v\in I\}$. Using (2), we get $n-1=\varepsilon\leq\min A-1\leq\max A-1=n-1$. Hence, $\min A=\max A=n$, so $A=\{n\}$.
\end{proof}

Before proceeding, we would like to share some remarks on Corollary \ref{split.lineal.relacion.epsilon.sigma_uv.D^*}. Items (1) and (3) can be interpreted as strong constraints on the structure of a linear split graph $S$ and its associated graph $\Phi(S)$. Item (1) limits, in a sense, the size of $\{\sigma_{uv}:uv\in\Phi\}$, that is, the number of distinct multiplicities of the edges in $\Phi$. The intuition behind this is that, with too many multiplicities, it becomes more likely that $\bigcap_{uv\in\Phi} D_{\sigma_{uv}}^*=\varnothing$, which would clearly forbid the existence of an $S$ with the required properties. Item (3) tells us that the edges of $\Phi$ can have at most one prime multiplicity. Moreover, it is enough for a single edge to have such a multiplicity, say $p$, for $\varepsilon$ to be forced to equal $p-1$.

If Theorem \ref{triang.n-simple&tipo0.implica.Delta.prop} can be thought of as a bridge from Graphs to Numbers, the following result goes in the opposite direction, that is, from Number Theory back to Graph Theory. For this reason, we consider it another key theorem of this section.

\begin{theorem}
	\label{n.con.propDelta.implica.existencia.de.T_0.n-simple}
	If $n\in\mathbb{N}(\Delta)$, let $(x,y,z)$ be a triple of divisors of $n$ satisfying (\ref{ecuacion.Delta.prop}). Then, there exists a balanced and $n$-simple split graph $(S,K,I)$ such that $\vec{\Phi}(S)=\Delta_0$ (see Figure \ref{triangulos.permitidos}). Moreover, if $y\leq z$, then $S$ has the following properties:
	\begin{enumerate}
		\item $d_b=d_a+\frac{n}{z}-z, \quad d_c=d_a+\frac{n}{x}-x$;
		\item $\eta_{ab}=d_a-z, \quad \eta_{bc}=d_a+\frac{n}{z}-z-y, \quad \eta_{ac}=d_a-x$;
		\item $d_a\geq z$;
		\item $|K|=\frac{n}{x}+z+y+\eta_{abc}$, where $\eta_{abc}=|N_a\cap N_b\cap N_c|\geq d_a-x-z$;
		\item if $S$ is active, then $d_a\leq x+z$;
		\item if $d_a=z$, then $S$ is active.
	\end{enumerate}
\end{theorem}

\begin{proof}
	We construct a split graph $(S,K,I)$ with $I=\{a,b,c\}$ and $K=N_a\cup N_b\cup N_c$. For $S$ to have the desired properties, it suffices to express $d_a, d_b, d_c, \eta_{ab},\eta_{bc}$ and $\eta_{ac}$ in terms of $n, x, y$ and $z$. If $\{u,v\}\subseteq I$, then $n=\sigma_{uv}=(d_u-\eta_{uv})(d_v-\eta_{uv})$. Thus, $d_u-\eta_{uv}$ and $d_v-\eta_{uv}$ are complementary divisors of $n$. Now take:
	\[ d_a-\eta_{ab}=z, \quad d_b-\eta_{ab}=n/z, \]
	\[ d_b-\eta_{bc}=y, \quad d_c-\eta_{bc}=n/y, \]
	\[ d_c-\eta_{ac}=x, \quad d_c-\eta_{ac}=n/x. \]
	These relations make it clear that $S$ is $n$-simple. Then, we solve for $d_b, d_c, \eta_{ab}, \eta_{bc}$ and $\eta_{ac}$ in terms of $d_a,n,x,y$ and $z$. Note that $d_a$ acts as a ``free parameter". This yields (1) and (2). In particular, $d_a<d_b<d_c$, which shows that $\vec{\Phi}=\Delta_0$.
	
	Inequality (3) follows from three facts: $\eta_{uv}\geq 0$, $z>x$, and $z>z-(\frac{n}{z}-y)$. To clarify the last of these inequalities, recall that $y^2,z^2<n$ (see Theorem \ref{Delta.prop.caracterizacion}); hence, $y^2z^2<n^2$, and thus $y<n/z$.
	
	To obtain (4), we use (1) and (2) in formula (\ref{formulafamosa.union.intersec}).
	
	Next, we show that $S$ is balanced, that is, it has no swing vertices. By construction, it is clear that $d_w\geq |K|$ for all $w\in K$. This means $N_w\neq K-w$ for all $w\in K$, so no clique vertex is swing in $S$. On the other hand, if $d_v=|K|$ for some $v\in I$, then necessarily $v=c$ (since $d_a<d_b<d_c$); hence, $N_c=K=N_a\cup N_b$ and $\eta_{abc}=\eta_{ab}$. Consequently, we get:
	\begin{equation}\label{eq14}
		d_a+\frac{n}{x}-x = d_c=|K|=\frac{n}{x}+z+y+\eta_{ab}.
	\end{equation}
	Given that $\eta_{ab}=d_a-z$, by (2), equality \eqref{eq14} becomes $-x=y$, which is absurd. Hence, $d_c<|K|$ and $S$ is balanced.
	
	(5). If $S$ is active, it clearly has no universal vertices (since they are always inactive). Then, $\eta_{abc}=0$, by Proposition \ref{split.caract.vert.activos}. Combining (2) with formula (\ref{formulafamosa.union.intersec}) and some basic set properties, we deduce:
	\[
	|N_a\cap(N_b\cup N_c)|=|(N_a\cap N_b)\cup(N_a\cap N_c)|=\eta_{ab}+\eta_{ac}-\eta_{abc}= 
	\]	
	\[
	2d_a-x-z-\eta_{abc}\leq |N_a|=d_a,
	\]
	so
	\[
	\eta_{abc}\geq d_a-x-z.
	\]
	If $d_a>x+z$, then $\eta_{abc}\geq 1$, which means there are universal vertices in $K$. In other words, $\eta_{abc}=0$ implies $d_a\leq x+z$.
	
	(6). If $d_a=z$, then $\eta_{ab}=0$, and hence $\eta_{abc}=0$. Then, $S$ is active by Lemma \ref{Phi.completo.implica.inactivos.universales}.
\end{proof}

Let us now look at a concrete example applying Theorem \ref{n.con.propDelta.implica.existencia.de.T_0.n-simple}. We take $n=24$, which, as we already know, is the smallest number satisfying the $\Delta$ condition. Since $(\frac{24}{3}-3)+(\frac{24}{3}-3)=\frac{24}{2}-2$, we have $z=y=3$ and $x=2$. Therefore, we want to construct a balanced split graph $S$ with the following properties:
\begin{enumerate}
	\item $d_b=d_a+5, \quad d_c=d_a+10$;
	\item $\eta_{ab}=d_a-3, \quad \eta_{bc}=d_a+2, \quad \eta_{ac}=d_a-2$;
	\item $d_a\geq 3$;
	\item $|K|=18+\eta_{abc}$, where $\eta_{abc}\geq d_a-5$;
	\item if $S$ is active, then $d_a\leq 5$;
	\item if $d_a=3$, then $S$ is active.
\end{enumerate}
We see that everything depends on the value assigned to $d_a$, as long as it is at least 3. For example, taking $d_a=3$, we obtain an active split graph $S$ with $|K|=18$, $d_b=8$, $d_c=13$, $\eta_{ab}=0$, $\eta_{bc}=5$, and $\eta_{ac}=1$. The following neighborhoods for $a,b$, and $c$ are compatible with these parameters: if $K=[18]$, define $N_a=[3], N_b=\{4,\ldots,11\}$, and $N_c=\{3,\ldots,8,12,\ldots,18\}$. Thus, $N_a\cap N_b=\varnothing, N_b\cap N_c=\{4,\ldots,8\}$, and $N_a\cap N_c=\{3\}$. Since $\Phi(S)$ is a $24$-simple triangle, we should check that $\sigma_{uv}=(d_u-\eta_{uv})(d_v-\eta_{uv})=24$ for every $uv\in\Phi$. Indeed, $\sigma_{ab}=(3-0)(8-0)=24$, $\sigma_{bc}=(8-5)(13-5)=24$, and $\sigma_{ac}=(3-1)(13-1)=24$. In Figure \ref{split.24simple.Delta.tipo0}, we show a simplified version of the split graph $S$ we have just constructed, omitting all edges between clique vertices (in black).

\begin{figure}[ht]
	\centering
	\begin{tikzpicture}[scale=0.5, every node/.style={circle, draw, minimum size=0.25cm,inner sep=1pt}]
		\coordinate (center) at (0,0);
		
		\foreach \i [count=\j from -9] in {1,...,18} {
			\node[fill=black] (K\i) at (\j*1.2, 0) {};
		}
		
		\node[label=right:{$a$}] (a) at (-7, 4) {};
		\node[label=right:{$b$}] (b) at (0, 4) {};
		\node[label=below:{$c$}] (c) at (0, -4) {};
		
		\foreach \i in {1,2,3} {
			\draw (a) -- (K\i);
		}
		
		\foreach \i in {4,5,6,7,8,9,10,11} {
			\draw (b) -- (K\i);
		}
		
		\foreach \i in {3,4,5,6,7,8,12,13,14,15,16,17,18} {
			\draw (c) -- (K\i);
		}
	\end{tikzpicture}
	\caption{Application of Theorem \ref{n.con.propDelta.implica.existencia.de.T_0.n-simple} to the case $n=24$.}
	\label{split.24simple.Delta.tipo0}
\end{figure}
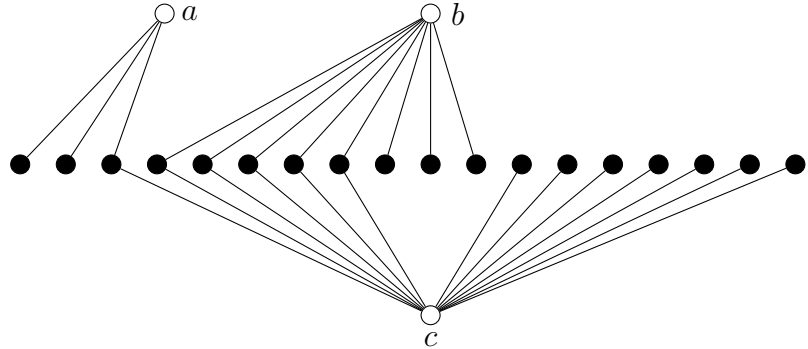

As we have just seen, all parameters in Theorem \ref{n.con.propDelta.implica.existencia.de.T_0.n-simple} defining $S$ depend on $d_a$, the degree of vertex $a$ in $S$. Indeed, according to item (3) of that theorem, it is enough to fix any value $d_a\geq z$ to obtain essentially different balanced split graphs for different choices of $d_a$. For $d_a=z$, item (6) ensures that $S$ is active. However, item (5) restricts the number of active graphs we can construct to a finite quantity, since $S$ will have universal vertices if $d_a>x+z$. These observations are summarized in the following corollary.

\begin{corollary}
	For each $n\in\mathbb{N}(\Delta)$:
	\begin{enumerate}
		\item there exists an infinite number of balanced and $n$-simple split graphs $S$ such that $\vec{\Phi}(S)=\Delta_0$;
		\item there exists a finite, but nonzero, number of active and $n$-simple split graphs $S$ such that $\vec{\Phi}(S)=\Delta_0$.
	\end{enumerate}
\end{corollary}

\begin{proof}
	It follows from the previous discussion.
\end{proof}


\section{$\Delta$-primitive numbers}\label{sec:Delta.prim}

In this section, we observe that numbers of the form $\alpha^2n$, with $n\in\mathbb{N}(\Delta)$ and $\alpha\in\mathbb{N}$, all have the $\Delta$ property. This fact motivates the search for those members of $\mathbb{N}(\Delta)$ that are not nontrivial square multiples of any other element in $\mathbb{N}(\Delta)$. This leads to the concept of a $\Delta$-primitive number. We see that every element of $\mathbb{N}(\Delta)$ is a square multiple of some primitive, and that the set $\mathbb{N}(\Delta^2)$, consisting of all squares with the $\Delta$ property, forms an abelian semigroup under ordinary multiplication. Subsequently, we exhibit infinite examples of $\Delta$ squares that admit more than one representation as $\alpha^2m$, where $m$ is $\Delta$-primitive. Finally, we construct a polynomial generator of numbers with the $\Delta$ property, through which we prove that there are infinitely many $\Delta$-primitives. \\

\begin{proposition}
	\label{multiplos.cuadr.prop.delta}
	If $n\in\mathbb{N}(\Delta)$, then $\alpha^2 n\in\mathbb{N}(\Delta)$ for all $\alpha\in\mathbb{N}$. In particular, every odd power of $n$ satisfies the $\Delta$ condition.
\end{proposition}

\begin{proof}
	It suffices to observe that condition \eqref{ecuacion.Delta.prop} from Proposition \ref{Delta.prop.caracterizacion} can be rewritten as
	\begin{equation*}
		\frac{\alpha^2 n}{\alpha z}-\alpha z + \frac{\alpha^2 n}{\alpha y}-\alpha y =\frac{\alpha^2 n}{\alpha x}-\alpha x.
	\end{equation*}
	Indeed, since $1<x<y\leq z<\sqrt{n}$ implies
	\[ 1\leq\alpha<\alpha x<\alpha y\leq\alpha z<\alpha\sqrt{n}=\sqrt{\alpha^2n}, \]
	and moreover $\alpha\{x,y,z\}\subseteq D_{\alpha^2n}$, we conclude that $\alpha^2 n\in\mathbb{N}(\Delta)$. In particular, taking $\alpha=n^\beta$, we see that $(n^\beta)^2 n=n^{2\beta+1}\in\mathbb{N}(\Delta)$ for all $\beta\in\mathbb{N}$.
\end{proof}

Thanks to Proposition \ref{multiplos.cuadr.prop.delta}, we see that $|\mathbb{N}(\Delta)|=\infty$. Another immediate consequence of this proposition is that if $n\in\mathbb{N}(\Delta)$ and $n^{2\beta}\in\mathbb{N}(\Delta)$ for some $\beta\in\mathbb{N}$, then $\{n^k:2\beta\leq k\in\mathbb{N}\}\subset\mathbb{N}(\Delta)$. For example, both $84$ and $84^2$ have the $\Delta$ property, so all higher powers of 84 do as well.

Proposition \ref{multiplos.cuadr.prop.delta} allows us to trivially obtain infinitely many elements of $\mathbb{N}(\Delta)$ from others. This motivates the search and study of those numbers with the $\Delta$ property that cannot be obtained in this way. We say that a natural number $n$ is $\Delta$-\textbf{primitive} if $n\in\mathbb{N}(\Delta)$ and there is no pair of numbers $\alpha,m\geq 2$ such that $n=\alpha^2m$ and $m\in\mathbb{N}(\Delta)$. For example, 24 and 40 are $\Delta$-primitive numbers. On the other hand, 96 is not, since we can write it as $2^2 \cdot 24$.

A natural number $n$ is said to be \textbf{square-free} if there is no prime number $p$ such that $p^2|n$. In other words, all primes appearing in the factorization of $n$ have exponent 1. Clearly, if $n\in\mathbb{N}(\Delta)$ and $n$ is square-free, then $n$ is $\Delta$-primitive. Examples of numbers satisfying the $\Delta$ condition and which are square-free include 105 and 385. Therefore, 105 and 385 are $\Delta$-primitive. Moreover, 105 is the smallest odd number with the $\Delta$ property.

\begin{proposition}
	\label{n=alfa^2 m,m.delta.prim}
	Every $n\in\mathbb{N}(\Delta)$ is either $\Delta$-primitive or can be written as $\alpha^2m$ for some $\alpha\geq 2$ and some $\Delta$-primitive $m$.
\end{proposition}

\begin{proof}
	If $n\in\mathbb{N}(\Delta)$ but is not primitive, then $n=\alpha_1^2m_1$ for some $\alpha_1\geq 2$ and some $m_1\in\mathbb{N}(\Delta)$. If $m_1$ is primitive, we are done. Otherwise, $m_1=\alpha_2^2m_2$ for some $\alpha_2\geq 2$ and $m_2\in\mathbb{N}(\Delta)$. If $m_2$ is primitive, we are done. Otherwise, $m_2=\alpha_3^2m_3$, and so on. Continuing this way generates a strictly decreasing sequence $\{m_1,m_2,m_3,\ldots\}\subseteq\mathbb{N}(\Delta)$. This process cannot continue indefinitely, so there must exist some $k\in\mathbb{N}$ such that $m_k$ is $\Delta$-primitive, hence $n=(\alpha_1\cdots\alpha_k)^2m_k$.  
\end{proof}

An interesting consequence of Proposition \ref{multiplos.cuadr.prop.delta} is that the set $\mathbb{N}(\Delta^2)$ of all squares satisfying the $\Delta$ condition is closed under multiplication.

\begin{corollary}
	\label{delta.cuadrados.cerrado}
	If $m^2,n^2\in\mathbb{N}(\Delta)$, then $m^2n^2\in\mathbb{N}(\Delta)$. In other words, the set $\mathbb{N}(\Delta^2)$, under the usual multiplication, forms an abelian semigroup.
\end{corollary}

\begin{proof}
	It follows immediately from Proposition \ref{multiplos.cuadr.prop.delta}.
\end{proof}

With the help of a computer, one can verify that $30^2$ is the first element of $\mathbb{N}(\Delta^2)$. Then, $|\mathbb{N}(\Delta^2)|=\infty$, by Proposition \ref{multiplos.cuadr.prop.delta}. However, the following questions remain open: Are there infinitely many $\Delta$-primitive numbers? Are there infinitely many $\Delta$-primitive squares? Is the decomposition in Proposition \ref{n=alfa^2 m,m.delta.prim} unique?

Regarding the last question, we observe that, given two distinct $\Delta$-primitive squares, it is quite easy to construct a number $n$ for which the decomposition of Proposition \ref{n=alfa^2 m,m.delta.prim} is not unique. Let us illustrate this with an example.

We already know that $30^2$ is $\Delta$-primitive. With the help of a computer, one can determine that $84^2$ is the next smallest $\Delta$-primitive square after $30^2$. By Proposition \ref{multiplos.cuadr.prop.delta}, we also know that $30^2a^2, 84^2b^2\in\mathbb{N}(\Delta)$ for all $a,b\in\mathbb{N}$. Therefore, if we find a pair $(a,b)\in\mathbb{N}^2$ that satisfies the equation
\begin{equation}
	\label{eq27}
	30^2a^2=84^2b^2,
\end{equation}
then $n=30^2a^2$ is the number we are looking for. Obviously, \eqref{eq27} has the same solutions as
\begin{equation}
	\label{eq28}
	5a=14b,
\end{equation}
in $\mathbb{N}^2$. Since 5 and 14 are coprime, it follows from \eqref{eq28} that $a=14a_1$ and $b=5b_1$, for some $a_1,b_1\in\mathbb{N}$. Substituting this into \eqref{eq28} gives $a_1=b_1$. Therefore, the solutions $(a,b)$ to \eqref{eq27} are of the form $k(14,5)$ for $k\in\mathbb{N}$. Finally, we find that $n=420^2k^2$. We generalize this procedure in the following proposition.

\begin{proposition}
	If $m$ and $l$ are perfect squares, then there exist $a,b\in\mathbb{N}$ such that $a^2m=b^2l$. In particular, if $m$ and $l$ are $\Delta$-primitive, then $n=a^2m$ is a square with the $\Delta$ property that admits more than one representation as the product of a square and a $\Delta$-primitive.
\end{proposition}

\begin{proof}
	To prove the claim, it is necessary to find the integer solutions of the equation $a^2m=b^2l$ in the variables $a$ and $b$. Clearly, in $\mathbb{N}^2$, this is equivalent to solving
	\begin{equation}
		\label{eq26}
		am_1=bl_1,
	\end{equation}
	where $m_1=m/\gcd(\sqrt{m},\sqrt{l})$ and $l_1=l/\gcd(\sqrt{m},\sqrt{l})$. 
	
	Since $\gcd(m_1,l_1)=1$, it follows that $m_1|b$ and $l_1|a$. Then, $b=m_1b_1$ and $a=l_1a_1$ for some $a_1,b_1\in\mathbb{N}$. Substituting into \eqref{eq26}, we get $a_1=b_1$. Hence, the solution set of \eqref{eq26} is 
	\[ \{ (a,b)\in\mathbb{N}^2: a=l_1k, b=m_1k, k\in\mathbb{N} \}. \]
\end{proof}

\begin{proposition}
	\label{Delta-num.y=z}
	If $x$ is a natural number $\geq 2$, then
	\[ n=x(2x-1)(3x-2)\in\mathbb{N}(\Delta). \]
\end{proposition}

\begin{proof}
	To prove that $n\in\mathbb{N}(\Delta)$, consider the number $y=2x-1$. Observe that $x,y\in D_n$, $1=2x-y$, and $3x-2=2y-x$. Substituting all of this into the identity $1\cdot n=x(2x-1)(3x-2)$ gives
	\begin{equation*}
		(2x-y)n=xy(2y-x).
	\end{equation*}  
	But this is exactly condition \eqref{ecuacion2.Delta.prop} from Proposition \ref{Delta.prop.caracterizacion} for $z=y$. We see that $2\leq x <y\leq z$, in accordance with Proposition \ref{Delta.prop.caracterizacion}. According to that proposition, it only remains to verify that $z<\sqrt{n}$, i.e., that $(2x-1)^2<n$. This is equivalent to the inequality $3x^2-4x+1>0$, which is true for all $x\geq 2$. Therefore, $n\in\mathbb{N}(\Delta)$.
\end{proof}

\begin{proposition}
	\label{Delta-primitivos.son.infinitos}
	There are infinitely many $\Delta$-primitive numbers.
\end{proposition}

\begin{proof}
	We will prove that there are infinitely many $\Delta$-primitives of the form $f(x)=x(2x-1)(3x-2)$, where $x$ is an integer $\geq 2$ (see Proposition \ref{Delta-num.y=z}).
	
	Suppose that $P=\{f(x_i): i\in[k]\}$ is the set of all $\Delta$-primitives of this form, with $x_1<x_2<\ldots<x_k$. Then $f(x_1)<f(x_2)<\ldots<f(x_k)$, since $f$ is increasing on $[2,\infty)$. Now consider an (odd) prime $p>f(x_k)$. Since $f(p)>f(x_k)$, the number $f(p)$ cannot be $\Delta$-primitive, i.e.,
	\begin{equation}
		f(p)=\alpha^2f(x_j), 
		\label{eq1}
	\end{equation}
	for some $\alpha\geq 2$ and $j\in[k]$. For simplicity, call $x_j=n$. Then, equality (\ref{eq1}) becomes
	\begin{equation}
		p(2p-1)(3p-2)=\alpha^2n(2n-1)(3n-2).
		\label{eq2}
	\end{equation} 
	We see that $p|\alpha^2f(n)$, but clearly $p$ cannot divide $n$, $2n-1$, or $3n-2$. Consequently, $p|\alpha$, so we can write $\alpha=p\beta$ for some $\beta\in\mathbb{N}$. Substituting into (\ref{eq2}) and simplifying, we get
	\begin{equation}
		(2p-1)(3p-2)=p\beta^2f(n),
		\label{eq3}
	\end{equation} 
	which is absurd since the right-hand side of (\ref{eq3}) is divisible by $p$, while the left-hand side is not.  
\end{proof}

\begin{conjecture}
	There are infinitely many $\Delta$-primitive squares.
\end{conjecture}


\section{Generating polynomials} \label{sec:polinom.generadores}

In this section, we revisit and expand upon an idea that arose in Proposition \ref{Delta-num.y=z}, developing a method for creating generating polynomials for numbers with the $\Delta$ property. \\

If $x$ is a natural number $>1$, consider
\begin{equation}
	\label{eq34}
	n=n(x)=x(x+2)(2x+1).
\end{equation}
Now observe that
\[ x(2x+1)-(x+2) = 2(x(x+2)-(2x+1)), \]
which is equivalent to
\begin{equation}
	\label{eq29}
	\frac{n}{x+2}-(x+2)=2\Big( \frac{n}{2x+1}-(2x+1) \Big).
\end{equation}
Since for all $x>1$ we have that $x+2, 2x+1\in D_n$ and
\[ 1<x+2<2x+1<\sqrt{n}, \]
it follows from \eqref{eq29} that $n\in\mathbb{N}(\Delta)$ for all $x>1$, since all requirements of Proposition \ref{Delta.prop.caracterizacion} are satisfied. These simple calculations, together with what we saw in Proposition \ref{Delta-num.y=z}, suggest the possibility of finding a more general method for constructing generating polynomials for numbers with the $\Delta$ property.

The idea is to consider polynomials of the form
\[ n(x)=\prod_{i=1}^{k}(a_ix+b_i), \]
with $a_i,b_i\in\mathbb{Z}$, such that each linear factor $a_ix+b_i$ of $n$ is an element of $D_n$. Clearly, we must have $k\geq 2$, since we always need $n$ to have at least two nontrivial divisors for sufficiently large $x$. If $k=2$, then suppose that $n=(a_1x+b_1)(a_2x+b_2)$ satisfies the $\Delta$ property via its divisors $a_1x+b_1$ and $a_2x+b_2$. Without loss of generality, this can only happen in the following way:
\begin{equation}
	\label{eq30}
	\frac{n}{a_1x+b_1}-(a_2x+b_2)=2\Big( \frac{n}{a_2x+b_2}-(a_1x+b_1) \Big).
\end{equation}
Simplifying expression \eqref{eq30}, we obtain
\begin{equation}
	\label{eq31}
	(a_2-a_1)x+(b_2-b_1)=0.
\end{equation}
The key to the method we wish to present is this: if the left-hand side of \eqref{eq31} is the identically zero polynomial, then \eqref{eq30} holds. Therefore, we take $a_2-a_1=0=b_2-b_1$, that is, $a_2=a_1$ and $b_2=b_1$. However, now $n=(a_1x+b_1)^2$, so $a_1x+b_1=\sqrt{n}$, which is not acceptable by Proposition \ref{Delta.prop.caracterizacion}. Hence, we must have $k\geq 3$. We now define the following polynomial:
\begin{equation}
	\label{eq32}
	f(x) = \frac{n}{a_1x + b_1} - (a_1x + b_1) - 2\left(\frac{n}{a_2x + b_2} - (a_2x + b_2)\right).
\end{equation}
Let $k=3$. Observe that the conditions $0<a_1<a_2$ and $a_3>0$ are sufficient to guarantee the existence of some $n_0\in\mathbb{N}$ such that the inequalities
\begin{equation*}
	1<a_1x+b_1<a_2x+b_2<\sqrt{n(x)}
\end{equation*}
hold simultaneously for $x\in[n_0,\infty)$. Moreover, since all $a_i$ are positive, we also ensure that $\deg(n)=3$ and that $n$ is eventually increasing. Through basic manipulations, we rewrite \eqref{eq32} as $f(x)=c_2x^2+c_1x+c_0$, where
\[ c_2= a_2a_3 - 2a_1a_3, \]
\[ c_1= a_2b_3 + a_3b_2 - a_1 - 2a_1b_3 - 2a_3b_1 + 2a_2, \]
\[ c_0= b_2b_3 - b_1 - 2b_1b_3 + 2b_2. \]
Since $f=0$ implies $n\in\mathbb{N}(\Delta)$ for all $x\geq n_0$, we must find integer solutions of the system $c_2=c_1=c_0=0$ in the six unknowns $a_i,b_i$. From $c_2=0$ we deduce that
\begin{equation}
	\label{eq33}
	a_2=2a_1,
\end{equation}
which is compatible with assuming $a_1<a_2$. Substituting \eqref{eq33} into $c_1=0$ and $c_0=0$, we obtain respectively that $(2b_1-b_2)a_3=3a_1$ and $(2b_1-b_2)b_3=2b_2-b_1$. Since $a_1,a_3\neq 0$, it follows that $2b_1-b_2\neq 0$, which allows us to isolate:
\begin{equation*}
	a_3=\frac{3a_1}{2b_1-b_2}, \quad b_3=\frac{2b_2-b_1}{2b_1-b_2}.
\end{equation*}
We note that $a_1,a_3>0$ implies $2b_1>b_2$; in particular, $(b_1,b_2)\neq(0,0)$ and $b_3\geq 0$ if and only if $2b_2\geq b_1$. We see that $a_1,b_1$, and $b_2$ end up being free parameters. Finally, for $a_3,b_3\in\mathbb{Z}$, it is sufficient that $2b_1-b_2$ divides $\gcd(3a_1,2b_2-b_1)$. We summarize all this in the following proposition.

\begin{proposition}
	\label{polinom.Delta-generadores,deg=3}
	Let $a,b,c\in\mathbb{Z}$ and define
	\[ \alpha=\frac{3a}{2b-c}, \quad \beta=\frac{2c-b}{2b-c}, \]
	where $a,\alpha>0$. Consider the polynomial $n(x)=(ax+b)(2ax+c)(\alpha x+\beta)$ and define $n_0$ as the smallest natural number such that the inequalities
	\[ 1<ax+b<2ax+c<\sqrt{n(x)} \]
	hold simultaneously for $x\in[n_0,\infty)$. If $x\in\mathbb{N}$ and
	\[ (2b-c)\mid\gcd(3a,2c-b), \]
	then $n(x)\in\mathbb{N}(\Delta)$ for all $x\geq n_0$.
\end{proposition}

\begin{proof}
	It follows from the preceding discussion.
\end{proof}

Let us see some examples of how to use Proposition \ref{polinom.Delta-generadores,deg=3} to construct generating polynomials for numbers with the $\Delta$ property.

If $\gcd(3a,2c-b)$ is a prime $p$, then $2b-c\in\{1,p\}$. If $2b-c=1$, then $c=2b-1$, $\alpha=3a$, and $\beta=3b-2$, yielding
\[ n(x)=(ax+b)(2ax+2b-1)(3ax+3b-2), \]
for $a\in\mathbb{N}$ and $b\in\mathbb{Z}$. For example, if $(a,b)=(3,4)$, then
\[ (3x+4)(6x+7)(9x+10)\in\mathbb{N}(\Delta), \]
for all $x\in\mathbb{N}$. If instead $(a,b)=(1,0)$, we recover the polynomial $x(2x-1)(3x-2)$ from Proposition \ref{Delta-num.y=z}.

If $2b-c=p$, then $c=2b-p$, $\alpha=\frac{3a}{p}$, and $\beta=\frac{3b}{p}-2$. Then: either $p=3$, or $p\neq 3$ and $p\mid\gcd(a,b)$. If $p=3$, then $c=2b-3$, $\alpha=a$, and $\beta=b-2$, so
\[ n(x)=(ax+b)(2ax+2b-3)(ax+b-2). \]
For example, if $(a,b)=(1,-1)$, then
\[ (x-1)(2x-5)(x-3)\in\mathbb{N}(\Delta), \]
for all $x\geq 5$, $x\in\mathbb{N}$. If instead $(a,b)=(1,2)$, we recover the polynomial \eqref{eq34}. If $p\neq 3$ and $(a,b)=(2,-4)$, then $p=2$, $c=-10$, $\alpha=3$, and $\beta=-8$. Hence,
\[ 4(x-2)(2x-5)(3x-8)\in\mathbb{N}(\Delta), \]
for all $x\geq 4$, $x\in\mathbb{N}$.


\section{Numbers without the $\Delta$ property}\label{sec:Num.sin.peop.Delta}

In this final section of Chapter \ref{cap:La.Prop.Delta}, we present a long list of families of numbers that do not satisfy the $\Delta$ condition. This class of numbers is infinite, as seen from the first result of the section. In fact, it states that no even integer not divisible by 4 has the $\Delta$ property. A key insight is that $pk \notin \mathbb{N}(\Delta)$ when $p$ is a ``large" prime compared to $k$. In rigorous terms: $p \geq 2k - 1$. A similar phenomenon occurs for numbers $n$ of the form $p^x q^y$ ($p, q$ primes): if $q > p^x$, then $n \notin \mathbb{N}(\Delta)$.

Another observed behavior is that numbers with very few prime factors in their factorization are less likely to satisfy the $\Delta$ condition. For example, powers of a prime or products of certain powers of two primes, such as $pq, pq^2, p^2 q^2$, and $p^k q$. We conjecture that if $n$ is a $\Delta$-primitive of the form $p^x q^y$ ($p, q$ primes), then $n \in \{24, 40\}$, and we characterize when a product of three primes $pqr \in \mathbb{N}(\Delta)$. In this part of the section, we note that having information about the ordering of the elements of $D_n$ is crucial to prove that $n \notin \mathbb{N}(\Delta)$.

After all this analysis, we return to split graphs, establishing a result about the $n$-simple induced cycles of $\Phi$, when $n$ does not satisfy the $\Delta$ condition and is not a square.

\begin{theorem}
	\label{2impar.no.tiene.prop.Delta}
	If $k$ is odd, then $2k \notin \mathbb{N}(\Delta)$.
\end{theorem}

\begin{proof}
	If $2k = ab$, then $a$ and $b$ cannot have the same parity, since otherwise $2k$ would be odd or a multiple of 4. Consequently, each element of $D^*_{2k}$ is odd, making every member of $D^+_{2k}$ even. Thus, $D^*_{2k} \cap D^+_{2k} = \varnothing$.  
\end{proof}

Proposition \ref{2impar.no.tiene.prop.Delta} has two direct consequences worth highlighting. The first is that $|\mathbb{N} - \mathbb{N}(\Delta)| = \infty$. The second is that $\mathbb{N}(\Delta)$ contains no even square-free numbers. In other words, every even number in $\mathbb{N}(\Delta)$ is a multiple of 4.

The next result is of notable importance because it tells us that a number cannot satisfy the $\Delta$ condition if it is divisible by a ``sufficiently large" prime.

\begin{theorem}
	\label{p>2k-2,entonces.pk.no tiene.prop.Delta}
	If $k \in \mathbb{N}$ and $p$ is a prime $\geq 2k - 1$, then $pk \notin \mathbb{N}(\Delta)$.
\end{theorem}

\begin{proof}
	Suppose $k \in \mathbb{N}$, $p$ is a prime $\geq 2k - 1$, but $pk \in \mathbb{N}(\Delta)$. Then, by Proposition \ref{Delta.prop.caracterizacion}, there exists a triple $T = \{x, y, z\} \subseteq D_{pk} = D_k \cup p D_k$ such that $2 \leq x < y \leq z < \sqrt{pk}$ and 
	\begin{equation}
		\label{eq19}
		\frac{pk}{x} - x = \frac{pk}{y} - y + \frac{pk}{z} - z.
	\end{equation}
	Since $k \leq 2k - 1 \leq p$, we have $\sqrt{pk} \leq \sqrt{p^2} = p$. Thus, $x, y, z < p$, which shows that $T \subseteq D_k$. We rewrite \eqref{eq19} as:
	\begin{equation}
		\label{eq20}
		p \left( \frac{k}{z} + \frac{k}{y} - \frac{k}{x} \right) = z + y - x
	\end{equation}
	Since the left-hand side of \eqref{eq20} is a product of two integers, it follows that $p | (z + y - x)$. Therefore, 
	\[ 2k - 1 \leq p \leq z + (y - x) \leq k + (k - 2) < 2k - 1, \]
	which is absurd. Hence, $pk \notin \mathbb{N}(\Delta)$.
\end{proof}

Proposition \ref{p>2k-2,entonces.pk.no tiene.prop.Delta} is useful for generating numbers without the $\Delta$ property, since for each $k \in \mathbb{N}$, we have infinitely many choices for $p$. It is interesting to note what the contrapositive of this proposition tells us: if $pk \in \mathbb{N}(\Delta)$, then $p < 2(k - 1)$. In other words, there are only finitely many prime multiples of a number that satisfy the $\Delta$ property. Another important observation is that Proposition \ref{p>2k-2,entonces.pk.no tiene.prop.Delta} implicitly provides a recipe to construct, for each prime $k_1$, an infinite sequence $\ell(k_1) = (k_i)_{i \in \mathbb{N}}$ of square-free numbers that do not have the $\Delta$ property. Let's see how to do this with an example. We want to determine the first terms of $\ell(2)$. Since $k_1 = 2$, let $p_1 = 3 \geq 2k_1 - 1$. Then, $k_2 = p_1 k_1 = 6$. If $p_2 = 11 \geq 2k_2 - 1$, then $k_3 = p_2 k_2 = 66$. If $p_3 = 131 \geq 2k_3 - 1$, then $k_4 = p_3 k_3 = 8646$. And so on. In general, $k_{i+1} = p_i k_i$, where $p_i$ is a prime $\geq 2k_i - 1$.

The next result we prove, closely resembles Proposition \ref{p>2k-2,entonces.pk.no tiene.prop.Delta} in style. It states that a number cannot have the $\Delta$ property if its factorization contains exactly two primes and one of them is ``much larger" than the other.

\begin{theorem}
	\label{p^xq^y.no.tiene.prop.Delta.q>p^x}
	Let $x, y \in \mathbb{N}$, with $y \geq 2$, and let $p, q$ be distinct primes. If $n = p^x q^y$ and $q > p^x$, then $n \notin \mathbb{N}(\Delta)$.
\end{theorem}

\begin{proof}
	We proceed by contradiction, first assuming that $n$ is $\Delta$-primitive. Then, there exist non-negative integers $a, c, e \leq x$ and $b, d, f \leq y$ such that 
	\begin{equation}
		\frac{n}{p^a q^b} - p^a q^b = \frac{n}{p^c q^d} - p^c q^d + \frac{n}{p^e q^f} - p^e q^f,
		\label{eq11}
	\end{equation}  
	where 
	\begin{equation*}
		1 < p^a q^b < p^c q^d \leq p^e q^f \leq \max \{ \alpha \in D_n : \alpha < \sqrt{n} \} = p^z q^{\lfloor y/2 \rfloor},
	\end{equation*}
	for some $z \leq x$. Note that $f \leq \lfloor y/2 \rfloor$. Indeed, if $f > \lfloor y/2 \rfloor$, we would have $p^e q^{f - \lfloor y/2 \rfloor} \leq p^z < q$, which is absurd. Using the same argument, we can also see that $b \leq d \leq f$. Furthermore, since $\lfloor y/2 \rfloor \leq y/2 < y$, it follows that $y - b, y - d$, and $y - f$ are all positive. Now, we rewrite equality \eqref{eq11} as
	\begin{equation}
		p^{x - a} q^{y - b} - p^a q^b = p^{x - c} q^{y - d} - p^c q^d + p^{x - e} q^{y - f} - p^e q^f.
		\label{eq12}
	\end{equation} 
	
	If $b, d, f > 0$, then we can divide both sides of \eqref{eq12} by $q$, obtaining that $m = p^x q^{y - 2} \in \mathbb{N}(\Delta)$ (recall that $y - 2 \geq 0$ by hypothesis). Since $n = q^2 m$ is $\Delta$-primitive, it follows that $q = 1$, a contradiction. Hence, $0 \in \{b, d, f\}$. 
	
	If $f = 0$, then $b = d = 0$ and \eqref{eq12} becomes
	\[ q^y (p^{x - a} - p^{x - c} - p^{x - e}) = p^a - p^c - p^e. \]
	Thus, $q^y$ divides $|p^a - p^c - p^e|$. Since $p^a q^0 < p^c q^0$, it follows that $|p^a - p^c - p^e| = (p^c - p^a) + p^e \leq p^c + p^e \leq 2q$. But then $q^y \leq 2q$, which is absurd. Therefore, it must be that $f > 0$.
	
	If $d = 0$, then $b = 0$ as well. Since $p^a q^0 < p^c q^0$, it follows that $a < c$ and, thus, $x - c < x - a$, so $p^{x - c} < p^{x - a}$. With these observations in mind, we rewrite \eqref{eq12} appropriately, obtaining the following contradiction:
	\[ 0 < q^{y - f} [q^f (p^{x - a} - p^{x - c}) - p^{x - e}] = -p^e q^f - (p^c - p^a) < 0. \]
	Therefore, $d > 0$. 
	
	Since $0 \in \{b, d, f\}$ but $d, f > 0$, necessarily $b = 0$. But then, taking congruences modulo $q$ in \eqref{eq12}, we deduce that $p^a \equiv 0 \pmod q$, which is impossible since $p$ and $q$ are distinct primes. Finally, we can conclude that $n$ cannot be $\Delta$-primitive.
	
	If $n \in \mathbb{N}(\Delta)$ but is not $\Delta$-primitive, then we can write $n = \alpha^2 m$ for some $\alpha \geq 2$ and some $\Delta$-primitive $m$. Note that $m = p^{x'} q^{y'}$, where $x' \leq x, y' \leq y$ but $(x', y') \neq (x, y)$. Since $q > p^{x'}$ and $m$ is $\Delta$-primitive, it follows from the first part of the proof that $m \notin \mathbb{N}(\Delta)$, which is absurd.   
\end{proof}

\begin{corollary}
	\label{contrarr.p^xq^y.no.tiene.prop.Delta.q>p^x}
	Let $x, y \in \mathbb{N}$, with $y \geq 2$, and let $p, q$ be distinct primes. For each pair $(p, x)$, there are only finitely many primes $q$ such that $p^x q^y \in \mathbb{N}(\Delta)$.
\end{corollary} 

\begin{proof}
	If $p^x q^y \in \mathbb{N}(\Delta)$, then $q < p^x$, by Proposition \ref{p^xq^y.no.tiene.prop.Delta.q>p^x}. Thus, fixing $p$ and $x$, we have only finitely many options for $q$.
\end{proof}

As an example application of Corollary \ref{contrarr.p^xq^y.no.tiene.prop.Delta.q>p^x}, consider a number $n$ of the form $9q^y$, with $y \geq 2$. If $n \in \mathbb{N}(\Delta)$, then it must be that $q < 9$, i.e., $q \in \{2, 5, 7\}$. 

We now prove that no power of a prime has the $\Delta$ property. A preliminary step is to verify that this claim holds for powers of order 1 and 2.

\begin{lemma}
	\label{p,p^2.no.tienen.prop.Delta}
	If $k \in \{1, 2\}$ and $p$ is prime, then $p^k \notin \mathbb{N}(\Delta)$.
\end{lemma}

\begin{proof}
	If $k = 1$, then $D^*_p = \{p - 1\}$ and $D^+_p = \{2(p - 1)\}$ are clearly disjoint. For $k = 2$, let $n = p^2$. Then, $D_n^* = \{n - 1, 0\}$ and $D_n^+ = \{2(n - 1)\}$. Hence, it is again evident that $D_n^* \cap D_n^+ = \varnothing$.
\end{proof}

As a mere curiosity, note that it can be proven that no prime satisfies the $\Delta$ condition simply by taking $k = 1$ in Proposition \ref{p>2k-2,entonces.pk.no tiene.prop.Delta}.  

\begin{theorem}
	\label{p^k.no.tiene.prop.Delta}
	If $p$ is prime and $k \in \mathbb{N}$, then $p^k \notin \mathbb{N}(\Delta)$. 
\end{theorem}

\begin{proof}
	First, observe that if $n = p^k$, then:
	\begin{equation*}
		D^*_n - 0 = \{p^{k - i} - p^i : 0 \leq 2i < k\}.
	\end{equation*}
	Thanks to Lemma \ref{p,p^2.no.tienen.prop.Delta}, we know the claim holds for $k = 1, 2$. Thus, we can assume $k \geq 3$.
	
	We proceed by contradiction. Suppose $n$ is $\Delta$-primitive. By Theorem \ref{Delta.prop.caracterizacion} (condition \ref{ecuacion.Delta.prop}), we have 
	\begin{equation}
		p^{k - z} - p^z + p^{k - y} - p^y = p^{k - x} - p^x,
		\label{eq13}    
	\end{equation}
	for certain non-negative integers $x, y, z$ such that $2x, 2y, 2z < k$. Then, $k-x, k-y, k-z > 0$. If $0 \in \{x, y, z\}$, $n - 1$ would participate in equality \eqref{eq13}, contradicting Corollary \ref{cotasup.D^*capD^+}. Thus, $x, y, z > 0$, and we can divide both sides of \eqref{eq13} by $p$. But then $n = p^2 m$, where $2 \leq m = p^{k - 2} \in \mathbb{N}(\Delta)$ (recall that $k \geq 3$), contradicting the primitivity of $n$. Therefore, we conclude that $n$ cannot be $\Delta$-primitive.
	
	If $n \in \mathbb{N}(\Delta)$ but is not $\Delta$-primitive ($k \geq 3$), then $n = \alpha^2 m$, where $\alpha \geq 2$ and $m$ is $\Delta$-primitive. But $m = p^h \in \mathbb{N}(\Delta)$, for some $h \in [k - 2]$, which contradicts what was shown in the first part of the proof or in Lemma \ref{p,p^2.no.tienen.prop.Delta}.   
\end{proof}

Over the course of the next three lemmas, we prove that for every pair of primes $p, q$, no number of the form $p^2 q^2, p q^2$, or $p q$ satisfies the $\Delta$ condition.

\begin{lemma}
	\label{p^2q^2.no.tiene.prop.Delta}
	If $p$ and $q$ are primes, then $p^2 q^2 \notin \mathbb{N}(\Delta)$.
\end{lemma}

\begin{proof}
	If $p = q$, we use Proposition \ref{p^k.no.tiene.prop.Delta}. If $p \neq q$, we can assume without loss of generality that $q < p$ (in particular, $p$ is odd). If $q > p^2$ or $p > q^2$, the result follows immediately from Proposition \ref{p^xq^y.no.tiene.prop.Delta.q>p^x}. The complementary case remains to be analyzed, i.e., when $q < p^2$ and $p < q^2$. Thus, $q < p < q^2 < p^2$. Under these restrictions, we have 
	\[ D_n = \{1, q, p, q^2, p q, p^2, p q^2, p^2 q, n\}, \]
	\[ D_n^* = \{n - 1, p^2 q - q, p q^2 - p, p^2 - q^2, 0\}, \]
	where the elements of $D_n$ are listed in increasing order and those of $D_n^*$ in decreasing order. Explicitly, $D_n^+ =$
	\[ \{p^2 q - q + p q^2 - p, p^2 q - q + p^2 - q^2, p q^2 - p + p^2 - q^2\} \cup \]
	\[\cup (D_n^* + (n - 1)) \cup 2 D_n^*. \]
	For every triple of elements $a, b, c \in D_n^* - 0$, we must prove that $a + b \neq c$. Clearly, $a + b \neq a$. Thanks to this observation and Corollary \ref{cotasup.D^*capD^+}, we can immediately discard many cases, leaving only the following:
	\begin{enumerate}
		\item $p^2 q - q + p q^2 - p = p^2 - q^2$,
		\item $p^2 q - q + p^2 - q^2 = p q^2 - p$,
		\item $p q^2 - p + p^2 - q^2 = p^2 q - q$,
		\item $2(p^2 q - q) \in \{p q^2 - p, p^2 - q^2\}$,
		\item $2(p q^2 - p) \in \{p^2 q - q, p^2 - q^2\}$,
		\item $2(p^2 - q^2) \in \{p^2 q - q, p q^2 - p\}$.
	\end{enumerate}
	Equalities (1), (2), (4), and the equality $2(p q^2 - p) = p^2 - q^2$ in (5) are evidently impossible due to the fact that we could fully specify the ordering of the elements in $D_n^*$.
	
	With basic manipulations, we can rewrite (3) as $p = q(p + 1) - q^2$. But this means $q | p$, which is absurd.
	
	The equality $2(p q^2 - p) = p^2 q - q$ in case (5) implies $2p \equiv 0 \pmod q$, forcing $q$ to be $2$. Then, we obtain $p^2 = 3p + 1$, but this is not possible since $p$ is odd.
	
	Only case (6) remains to be analyzed. If $2(p^2 - q^2) = p^2 q - q$, it follows that $2p^2 \equiv 0 \pmod q$, and thus $q = 2$. But then $p^2 - 4 = p^2 - 1$, which is absurd. If $2(p^2 - q^2) = p q^2 - p$, we deduce that $2q^2 \equiv 0 \pmod p$. But then it must be that $p = 2$, contradicting the hypothesis that $p$ is odd.
\end{proof}

\begin{lemma}
	If $p$ and $q$ are primes, then $p q^2 \notin \mathbb{N}(\Delta)$.
	\label{pq^2.no.tiene.prop.Delta}
\end{lemma}

\begin{proof}
	Let $n = p q^2$. If $p = q$, we use Proposition \ref{p^k.no.tiene.prop.Delta}. Then, suppose that $p\neq q$. Since $D_n = \{1, q, p, q^2, p q, n\}$, we have 
	\[ D_n^* = \{n - 1, p q - q, |q^2 - p|\}, \]
	and thus 
	\[ D_n^+ = \{p q - q + |q^2 - p|\} \cup (D_n^* + (n - 1)) \cup 2 D_n^*. \]
	Obviously, $p q - q + |q^2 - p|\notin \{2(pq-q),2|q^2-p|\}$ and $2d\neq d$ for all $d\in D_n^*$. Moreover, 
	\[ \{p q - q + |q^2 - p|, 2(pq-q),2|q^2-p|\}\cap\{n-1\}=\varnothing, \]
	by Corollary \ref{cotasup.D^*capD^+}. If $2(pq-q)=|q^2-p|$ or $2|q^2-p|=pq-q$, then $|q^2-p|\equiv 0\pmod{q}$, which is absurd. Thus, $D^*_n \cap D^+_n = \varnothing$. 
\end{proof}

\begin{lemma}
	\label{pq.no.tiene.prop.Delta}
	If $p$ and $q$ are primes, then $p q \notin \mathbb{N}(\Delta)$.
\end{lemma}

\begin{proof}
	If $p = q$, we use Proposition \ref{p^k.no.tiene.prop.Delta}. If $p \neq q$, we can assume without loss of generality that $p < q$. Then, with $n = p q$, $D^*_n = \{n - 1, q - p\}$ and
	\begin{equation*}
		D^+_n = \{2(n - 1), n - 1 + q - p, 2(q - p)\}.    
	\end{equation*}
	At this point, using the usual arguments, it is quickly verified that $D^*_n \cap D^+_n = \varnothing$.
\end{proof}

\begin{theorem}
	\label{pq,pq^2,p^2q^2.no.tiene.prop.Delta}
	If $p$ and $q$ are prime numbers, then
	\[ \{p q, p^2q, pq^2, p^2 q^2\} \subset \mathbb{N} - \mathbb{N}(\Delta). \]
\end{theorem}

\begin{proof}
	This is the content of Lemmas \ref{p^2q^2.no.tiene.prop.Delta}, \ref{pq^2.no.tiene.prop.Delta} and \ref{pq.no.tiene.prop.Delta}.
\end{proof}

The following is a technical lemma we need later to prove that 24 and 40 are the only $\Delta$-primitives of the form $p^k q$, where $p$ and $q$ are primes.

\begin{lemma}
	\label{2^i-1q-2^j-1=2^i+j-q}
	Let $i, j \in \mathbb{N}$, and let $q \geq 3$ be an odd number. If
	\begin{equation}
		|2^{i - 1} q - 2^{j - 1}| = |2^{i + j} - q|,
		\label{eq9}
	\end{equation}
	then $(i, j, q) \in \{(1, 2, 5), (2, 1, 3)\}$.   
\end{lemma}

\begin{proof}
	If $q > 2^{i + j}$ in \eqref{eq9}, then $2^{i - 1} q - 2^{j - 1} > 2^{2i + j - 1} - 2^{j - 1} > 0$. Thus,
	\[ 0 < 2^{j - 1} (2^{i + 1} - 1) = q (1 - 2^{i - 1}) \leq 0, \]
	which is absurd. Therefore, it follows that $2^{i + j} > q$ in \eqref{eq9}. Successively, using similar arguments, it is easy to also eliminate the absolute value on the left-hand side of \eqref{eq9}. Finally, \eqref{eq9} is equivalent to
	\begin{equation}
		2^{i - 1} q - 2^{j - 1} = 2^{i + j} - q.
		\label{eq10}
	\end{equation}
	Since the right-hand side of \eqref{eq10} is odd, necessarily one of the two terms on the left-hand side must be odd. Therefore, it must be either $i = 1$ or $j = 1$. Substituting $i = 1$ into \eqref{eq10} and simplifying, we obtain $2q = 2^{j - 1} 5$, which implies $j = 2$ and $q = 5$. If instead we substitute $j = 1$ into \eqref{eq10}, we arrive at $q - 1 = 2^{i - 1} (4 - q)$. Since $q - 1 > 0$, it must also be that $4 - q > 0$. Given that by hypothesis $q$ is an odd number $\geq 3$, the only option is $q = 3$, which in turn implies $i = 2$.   
\end{proof}

\begin{lemma}
	\label{p^kq.prim.24.40}
	Let $n = p^k q$, where $p$ and $q$ are distinct primes, $q$ is odd, and $k \geq 2$. If $n$ is $\Delta$-primitive, then $n \in \{24, 40\}$. 
\end{lemma}

\begin{proof}
	Since $n \in \mathbb{N}(\Delta)$, there exist non-negative integers $e_1, \ldots, e_6$ such that $e_1 + e_2 = e_3 + e_4 = e_5 + e_6 = k$ and
	\begin{equation}
		|p^{e_1} q - p^{e_2}| = |p^{e_3} q - p^{e_4}| + |p^{e_5} q - p^{e_6}|.
		\label{eq4}
	\end{equation}
	If $e_i > 0$ for all $i$, we can divide both sides of \eqref{eq4} by $p$, obtaining that $m = p^{k - 2} q \in \mathbb{N}(\Delta)$ (recall that $k \geq 2$). But then $n = p^2 m$, which contradicts the primitivity of $n$. Thus, $0 \in \{e_1, \ldots, e_6\}$. Note that necessarily $e_2, e_4, e_6 > 0$ in \eqref{eq4}, thanks to Corollary \ref{cotasup.D^*capD^+}. 
	
	If $e_1 = 0$, then \eqref{eq4} becomes
	\begin{equation}
		|q - p^k| = |p^{e_3} q - p^{e_4}| + |p^{e_5} q - p^{e_6}|. 
		\label{eq5}
	\end{equation}
	If $e_3 = 0$ or $e_5 = 0$, we find $|q - p^k|$ as a summand on the right-hand side of \eqref{eq5}. Both cases lead to clear absurdities. If $e_3, e_5 > 0$, then the right-hand side of \eqref{eq5} becomes a multiple of $p$: this represents another absurdity because the left-hand side of \eqref{eq5} is not. Therefore, we can conclude that $e_1 > 0$ in \eqref{eq4}, i.e., the left-hand side of \eqref{eq4} is a multiple of $p$, and moreover $0 \in \{e_3, e_5\}$. 
	
	Suppose now that $e_3 = 0$ and $e_5 > 0$. After reducing \eqref{eq4} modulo $p$, we get that $p|q$, which is obviously false. The same would happen if $e_3 > 0$ and $e_5 = 0$. Hence, the only possibility is that $e_3 = e_5 = 0$. After all these steps, we have converted \eqref{eq4} into
	\begin{equation}
		|p^{e_1} q - p^{e_2}| = 2 |q - p^{e_1 + e_2}|.
		\label{eq6}
	\end{equation}
	Reducing \eqref{eq6} modulo $p$, we obtain $2q \equiv 0 \pmod p$, which is true only if $p = 2$. We can then rewrite \eqref{eq6} as
	\begin{equation}
		|2^{e_1 - 1} q - 2^{e_2 - 1}| = |q - p^{e_1 + e_2}|.
		\label{eq7}
	\end{equation}
	We have shown that if $n$ is $\Delta$-primitive, then $n = 2^{e_1 + e_2} q$, where the integers $e_1, e_2$ and $q$ satisfy \eqref{eq7}. By Lemma \ref{2^i-1q-2^j-1=2^i+j-q}, the only triples $(e_1, e_2, q)$ of natural numbers satisfying \eqref{eq7} are $(1, 2, 5)$ and $(2, 1, 3)$. Therefore, we conclude that $n \in \{24, 40\}$. 
\end{proof}

\begin{theorem}
	\label{p^kq.tiene.prop.Delta.iff...}
	Let $n = p^k q$, where $p$ and $q$ are primes and $k \in \mathbb{N}$. If $n \in \mathbb{N}(\Delta)$, then $n \in \{2^{2h + 1} q : q \in \{3, 5\}, h \in \mathbb{N}\}$. 
\end{theorem}

\begin{proof}
	If $k = 1$, we apply Lemma \ref{pq.no.tiene.prop.Delta}. If $k \geq 2$, suppose $p \neq q$, $q$ is odd, and $n \in \mathbb{N}(\Delta)$. If $n$ is $\Delta$-primitive, we use Lemma \ref{p^kq.prim.24.40}. If $n$ is not $\Delta$-primitive, then there exists an integer $e > 0$ such that $n = p^{2e} p^{k - 2e} q$ and $m = p^{k - 2e} q$ is $\Delta$-primitive. But again, by Lemma \ref{p^kq.prim.24.40}, it follows that $m \in \{24, 40\}$, i.e., $p = 2, k - 2e = 3$ and $q \in \{3, 5\}$. 
\end{proof}

\begin{corollary}
	If $p$ and $q$ are distinct odd primes, consider a number $n$ of the form $p^2 q^3, p^3 q^3$, or $p^2 q^4$. If $n \in \mathbb{N}(\Delta)$, then $n$ is $\Delta$-primitive.
\end{corollary}

\begin{proof}
	Let $n \in \{p^2 q^3, p^3 q^3, p^2 q^4\}$. Suppose $n$ has the $\Delta$ property but is not $\Delta$-primitive. Then $n = \alpha^2 m$, where $\alpha \in \{p, q, p q\}$ and $m$ is $\Delta$-primitive. Thus,
	\[ m \in \{q, q^3, p^2 q, p q, p^3 q, p q^3, q^2, q^4, p^2 q^2\}. \]
	Thanks to Propositions \ref{p^k.no.tiene.prop.Delta}, \ref{pq,pq^2,p^2q^2.no.tiene.prop.Delta}, and \ref{p^kq.tiene.prop.Delta.iff...}, we know this is not possible.
\end{proof}

Based on experimental evidence, we conjecture that 24 and 40 are the only $\Delta$-primitive numbers using at most two primes in their decomposition.

\begin{conjecture}
	Let $n$ be a natural number of the form $p^x q^y$, where $p$ and $q$ are primes and $x, y \geq 0$. If $n$ is $\Delta$-primitive, then $n \in \{24, 40\}$. 
\end{conjecture}

Lemma \ref{pq.no.tiene.prop.Delta} cannot be extended to a product $p q r$ of three distinct primes: numbers like $105 = 3 \cdot 5 \cdot 7$ and $385 = 5 \cdot 7 \cdot 11$ would be counterexamples. However, we will analyze this new situation to understand in detail how the case $p q$ differs from the case $p q r$.

Let $p, q$, and $r$ be primes such that $p < q < r$. If $n = p q r$, then 
\[ D_n = \{1, p, q, r, p q, p r, q r, n\}, \]
\[ D_n^* = \{n - 1, q r - p, p r - q, |p q - r|\}, \]
\[ D_n^+ = \{q r - p + p r - q, q r - p + |p q - r|, p r - q + |p q - r|\} \cup \]
\[ \cup (D_n^* + (n - 1)) \cup 2 D_n^*, \]
where the elements of $D_n^*$ are listed in decreasing order. This complete ordering is due to $1 < p < q < r, p q < p r < q r < n$. Imagine having to prove that $a + b \neq c$ for every triple of elements $a, b, c \in D_n^* - 0$. Clearly, $a + b \neq a$. Thanks to this observation and Corollary \ref{cotasup.D^*capD^+}, we can immediately discard many cases, leaving only the following:
\begin{enumerate}
	\item $q r - p + p r - q = |p q - r|$,
	\item $q r - p + |p q - r| = p r - q$,
	\item $p r - q + |p q - r| = q r - p$,
	\item $2(q r - p) \in \{p r - q, |p q - r|\}$,
	\item $2(p r - q) = q r - p$,
	\item $2(p r - q) = |p q - r|$,
	\item $2|p q - r| = q r - p$,
	\item $2|p q - r| = p r - q$.
\end{enumerate}
The impossibility of (1), (2), (4), and (6) is evident due to the complete ordering of the elements of $D_n^*$. 

If $r < p q$, then (7) can be rewritten as $p (2q + 1) = r (q + 2)$, so $q = k p - 2$, for some $k \in \mathbb{N}$. Then, $k (2p - r) = 3$, and thus $k \in \{1, 3\}$. If $k = 1$, then $q = p - 2 < p$. If $k = 3$, then $r = 2p - 1 < 3p - 2 = q$. Both cases contradict the hypotheses. If $p q < r$, then (7) is equivalent to $r (2 - q) = p (2q - 1)$, which is obviously false.

If $r < p q$, equality (8) can be rewritten as $q (2p + 1) = r (p + 2)$. Since $q (2p + 1)$ is odd, necessarily $p \neq 2$. Moreover, $q | (p + 2)$, which implies $q \leq p + 2$. But then $q = p + 2$, since $3 \leq p < p + 2 \leq q$, and thus $r = 2p + 1$. If $p q < r$, then we can convert (8) into $r (2 - p) = q (2p - 1)$, which is clearly absurd. 

We can rewrite (5) as $p (2r + 1) = q (r + 2)$. Then, $2r + 1 = k q$ and $r + 2 = h p$, for certain $k, h \in \mathbb{N}$. But then $p k q = q h p$, i.e., $k = h$, from which we deduce that $k (2p - q) = 3$. Hence, $k \in \{1, 3\}$. If $k = 1$, then $r = p - 2 < p$, a contradiction. If $k = 3$, then $r = 3p - 2$ and $q = 2p - 1$.

Regarding equality (3), note that it can be rewritten as
\[ (r + p + 1)(q - p - 1) = -p^2 - p - 1, \]
if $p q < r$, and
\[ (r - p + 1)(q - p + 1) = p^2 - p + 1, \]
if $r < p q$. The first case is immediately discarded. Regarding the second, we can say that $d_1 = q - p + 1$ and $d_2 = r - p + 1$ are complementary divisors of the integer $m = d_1 d_2 = p^2 - p + 1$. Therefore, by fixing $p$, we will have a finite number of pairs $(q, r)$ such that $p q r \in \mathbb{N}(\Delta)$. We also note that: 1) $1 < d_1 < \sqrt{m}$; 2) $m, d_1$, and $d_2$ are odd for $p > 2$. We can summarize all this analysis in the following proposition.

\begin{theorem}
	\label{pqr.prop.Delta}
	Let $p, q$, and $r$ be prime numbers such that $p < q < r$. Then $p q r \in \mathbb{N}(\Delta)$ if and only if one of these conditions is satisfied:
	\begin{enumerate}
		\item $(r - p + 1)(q - p + 1) = p^2 - p + 1$,
		\item $(q, r) \in \{(p + 2, 2p + 1), (2p - 1, 3p - 2)\}$.
	\end{enumerate}
	In particular, once $p$ is fixed, there are only finitely many pairs $(q, r)$ such that $p q r \in \mathbb{N}(\Delta)$. 
\end{theorem} 

\begin{proof}
	It follows from the previous discussion.
\end{proof}

Let's see some examples of applying Proposition \ref{pqr.prop.Delta}. If $p = 2$, then $2 q r \in \mathbb{N}(\Delta)$ if and only if $(r - 1)(q - 1) = 3$ (1) or $(q, r) \in \{(4, 5), (3, 4)\}$ (2). Note that (1) implies $q - 1 = 1$, which contradicts $p < q$. Regarding (2), we see there are no pairs of primes. Therefore, we conclude that $2 q r \notin \mathbb{N}(\Delta)$. This is consistent with Proposition \ref{2impar.no.tiene.prop.Delta}, of which Proposition \ref{pqr.prop.Delta} is a particular case when $p = 2$. If $p = 3$, then $3 q r \in \mathbb{N}(\Delta)$ if and only if $(r - 2)(q - 2) = 7$ (1) or $(q, r) = (5, 7)$ (2). We see that (1) implies $q - 2 = 1$, which contradicts $p < q$. Thus, $3 q r \in \mathbb{N}(\Delta)$ if and only if $(q, r) = (5, 7)$, by (2). If $p = 5$, then $5 q r \in \mathbb{N}(\Delta)$ if and only if $(r - 4)(q - 4) = 21$ (1) or $(q, r) \in \{(7, 11), (9, 13)\}$ (2). Since $1 < q - 4 < \sqrt{21} < 5$, it follows from (1) that $q - 4 = 3$ and $r - 4 = 7$. Hence, $5 q r \in \mathbb{N}(\Delta)$ if and only if $(q, r) = (7, 11)$. Similarly, it is very easy to verify that $7 q r \in \mathbb{N}(\Delta)$ if and only if $(q, r) = (13, 19)$. We can summarize all this in the following corollary.

\begin{corollary}
	Let $p, q$, and $r$ be prime numbers such that $p < q < r$, and let $n = p q r$. If $p \leq 7$ and $n \in \mathbb{N}(\Delta)$, then $n \in \{105, 385, 1729\}$.
\end{corollary} 

\begin{proof}
	It follows from the previous discussion.  
\end{proof}

The next corollary generalizes some arguments used in the previous examples. 

\begin{corollary}
	\label{p^2-p+1.primo...}
	Let $p, q$, and $r$ be prime numbers such that $p < q < r$. If $p^2 - p + 1$ is prime and the set $\{p + 2, 2p \pm 1, 3p - 2\}$ contains at most one prime, then $p q r \notin \mathbb{N}(\Delta)$. 
\end{corollary}

\begin{proof}
	Condition (1) of Proposition \ref{pqr.prop.Delta} cannot be fulfilled acceptably if $p^2 - p + 1$ is prime, since this implies $q - p + 1 = 1$, contradicting the hypothesis that $p < q$.
	
	If the set $\{p + 2, 2p \pm 1, 3p - 2\}$ contains at most one prime, then it is impossible to form any pair of primes $(q, r)$ as required by condition (2) of Proposition \ref{pqr.prop.Delta}.
\end{proof}

The numbers 13, 67, and 79 are the smallest primes satisfying the hypotheses of Corollary \ref{p^2-p+1.primo...}.

Finally, we close this section with an important theorem about the $n$-simple induced cycles of $\Phi$, when $n \notin \mathbb{N}(\Delta)$ and is not a square.

\begin{theorem}
	Let $S$ be a split graph, and let $C$ be an $n$-simple induced cycle in $\Phi(S)$. If $n$ is not a square and $n \notin \mathbb{N}(\Delta)$, then $|C| = 4$. 
\end{theorem}

\begin{proof}
	We already know from Theorem \ref{ciclos.inducidos.en.Phi} that $|C| \in \{3, 4\}$. If $|C| = 3$, then $C$, viewed in $\vec{\Phi}(S)$, would have to be of type 0 (see Figure \ref{triangulos.permitidos}), by Corollary \ref{S.squarefree.implica.triang_tipo0}. Since $C$ is $n$-simple, it follows by Theorem \ref{triang.n-simple&tipo0.implica.Delta.prop} that $n \in \mathbb{N}(\Delta)$, which contradicts the hypothesis.
\end{proof}


\chapter{Active graphs of low degree} \label{cap:grafos.activos.con.deg<5}

In this chapter, we apply all the machinery developed earlier with the aim of:
\begin{enumerate}[(1).]
	\item classifying active graphs of degrees 1, 2, and 3;
	\item classifying the realization spaces associated with graphs of degree $\leq 2$;
	\item proving that the realization spaces associated with graphs of degree 3 are ``almost all" 3-regular;
	\item classifying prime split graphs of degree 4.
\end{enumerate}
In the first section, we resolve (1) for degrees 1 and 2, obtaining (2) as a consequence. In the second section, we classify active graphs of degree 3, thereby achieving (3). In the third and final section, we resolve (4).


\section{Active graphs of degree 1 and 2}

The following immediate corollaries of Theorem \ref{degreeofG} allow us to characterize graphs $G$ of degree $\leq 2$ based on the members of $Q_G^*$.

\begin{corollary}
	Let $G$ be a non-null graph. The following statements are equivalent:
	\begin{enumerate}
		\item $\deg(G)=0$.
		\item $G$ has no induced subgraphs isomorphic to $2K_2$, $C_4$, or $P_4$.
		\item $G$ is threshold.
		\item $G$ is inactive.
		\item $G^* = K_0$.
	\end{enumerate}
\end{corollary}

\begin{corollary}
	A graph $G$ has degree 1 if and only if $G$ has exactly one induced subgraph isomorphic to $P_4$ and no induced subgraphs isomorphic to $2K_2$ or $C_4$.
\end{corollary}

\begin{corollary}
	A graph $G$ has degree 2 if and only if exactly one of the following conditions holds:
	\begin{enumerate}
		\item $G$ has exactly two induced subgraphs isomorphic to $P_4$ and no induced subgraphs isomorphic to $2K_2$ or $C_4$.
		\item $G$ has exactly one induced subgraph isomorphic to $C_4$ and no induced subgraphs isomorphic to $2K_2$ or $P_4$.
		\item $G$ has exactly one induced subgraph isomorphic to $2K_2$ and no induced subgraphs isomorphic to $C_4$ or $P_4$.
	\end{enumerate}
\end{corollary}

\begin{theorem}
	\label{|G^*|}
	If $G$ is an active graph, then
	\begin{equation*}
		4 \leq |G| \leq 4\deg(G).
	\end{equation*}
\end{theorem}

\begin{proof}
	Since $G$ is active, it is clear that $Q_G^* \neq \varnothing$, as $\deg(G) \geq 1$. Thus, $|G| \geq 4$. On the other hand:
	\begin{equation*}
		|G| = \left|\bigcup_{H \in Q_G^*} V(H)\right| \leq \sum_{H \in Q_G^*} 4 \leq 4 \sum_{H \in Q_G^*} \deg(H) = 4\deg(G).
	\end{equation*}
\end{proof}

An immediate consequence of Theorem \ref{|G^*|} is that $P_4$ is the only active graph of degree 1.

\begin{corollary}
	If $s$ is a graphical sequence, then
	\begin{equation*}
		|\text{act}(s)| \leq 4\min\{\deg(G) : G \in \mathcal{G}(s)\}.
	\end{equation*}
\end{corollary}

\begin{theorem}
	\label{tau.preserves.deg1}
	If $\deg(G)=1$, then $\deg(\tau(G))=1$ for every 2-switch $\tau$. In other words, the 2-switch preserves degree 1.
\end{theorem}

\begin{proof}
	Using Theorem \ref{|G^*|}, we deduce that $G^* \approx P_4$. Letting $s=s(G)$, we have $\mathcal{G}(s) \approx \mathcal{G}(P_4) \approx K_2$, by Theorem \ref{isomorfismo.espacio.activo}.
\end{proof}

From Theorem \ref{tau.preserves.deg1}, we immediately obtain the following corollary.

\begin{corollary}
	If $\mathcal{G}(s)$ has a leaf, then $\mathcal{G}(s) \approx K_2$.
\end{corollary}

\begin{proposition}
	\label{|I|,|K|<=2deg(S)}
	Let $(S,K,I)$ be a split graph.
	\begin{enumerate}
		\item If $S$ is active, then
		\[ 2 \leq |I|, |K| \leq 2\deg(S). \]
		\item If $S$ is prime, then
		\[ 2 \leq |I|, |K| \leq \deg(S)+1. \]
	\end{enumerate}
\end{proposition}

\begin{proof}
	\begin{enumerate}[(1).]
		\item If $S$ is active, then
		\[ |I| = \left|\bigcup_{H \in Q_S(P_4)} (V(H) \cap I)\right| \leq \sum_{H \in Q_S(P_4)} |V(H) \cap I| \]
		\[ \leq \sum_{H \in Q_S(P_4)} 2 = 2|Q_S(P_4)| = 2\deg(S). \]
		The above arguments remain valid if we replace $I$ with $K$.
		\item If $S$ is prime, then $\Phi(S)$ is connected, by Theorem \ref{S.primo.iff.Phi(S).conexo}. Among all connected multigraphs of fixed size $\deg(S)$, those that maximize the number of vertices are clearly the simple trees, which have order $\deg(S)+1$. Thus, $|I| \leq \deg(S)+1$. Since $S$ is prime, $(\overline{S},I,K)$ is also prime (see \cite{tyshkevich2000decomposition}, page 14). Therefore, reusing the same argument, we obtain $|K| \leq \deg(S)+1$.
	\end{enumerate}
\end{proof}

\begin{lemma}
	\label{5>degG=|{P_4}|.implica.G.split}
	If $G$ is a graph such that $\deg(G)=|Q_G(P_4)| \leq 4$, then $G$ is split.
\end{lemma}

\begin{proof}
	Recall that a graph is split if and only if it contains no induced subgraphs isomorphic to $C_4$, $2K_2$, or $C_5$. The hypotheses imply that $G$ contains no induced subgraphs isomorphic to $C_4$ or $2K_2$. If $C_5 \preceq G$, then $4 \geq \deg(G) \geq \deg(C_5)=5$, by Proposition \ref{deg.respects.ind.inc}. Thus, $G$ cannot contain induced subgraphs isomorphic to $C_5$ either.
\end{proof}

Consider an active graph $G$ such that $\deg(G)=|Q_G(P_4)|=2$. Thanks to Lemma \ref{5>degG=|{P_4}|.implica.G.split}, we know that $G$ is split. Let $n=|G|$. Applying Theorem \ref{|G^*|}, we obtain $4 \leq n \leq 8$. However, it is clear that $n \neq 4$, since otherwise we would have $|Q_G(P_4)|=1$. Now, suppose $G=S \circ H$ for some graphs $S$ and $H$. Since $G$ is active, it follows from Proposition \ref{SoG.act.iff.S,G.act} that $S$ and $H$ are active. Consequently, $|S|, |H| \geq 4$, by Theorem \ref{|G^*|}. Since $|S|+|H|=n \leq 8$, we conclude that $S, H \approx P_4$. Therefore:
\begin{enumerate}[(1).]
	\item if $n=8$, then $G \approx P_4^2$;
	\item if $n \in \{5,6,7\}$, then $(G,K,I)$ is prime.
\end{enumerate}
In general, to find all prime split graphs $S$ of a certain degree, we can first assume that $(S,K,I)$ satisfies
\begin{equation}
	\label{eq36}
	\alpha(S) \leq \omega(S).
\end{equation}
When we say ``all", we clearly mean all isomorphism classes (which, as we know, are unlabeled graphs). Once we find all $S$ satisfying \eqref{eq36}, it suffices to complement each of them to complete the classification. This procedure is valid for the following reasons:
\begin{enumerate}[(1).]
	\item a graph is prime if and only if its complement is prime (\cite{tyshkevich2000decomposition}, page 14);
	\item every graph has the same degree as its complement (Theorem \ref{degG=deg(G.complemento)});
	\item $(\overline{S},I,K)$ is the bipartition of $\overline{S}$ and $\alpha(\overline{S}) \geq \omega(\overline{S})$.
\end{enumerate}

Let us continue our analysis of $G$. If $n \in \{5,6,7\}$, then $\Phi(G)$ is connected, by Corollary \ref{S.primo.iff.Phi(S).conexo}. This means that $\Phi(G)$ is isomorphic to $K_2$ or $P_3$, ignoring multiplicities. Assume $\alpha=\alpha(G) \leq \omega(G)=\omega$. Remember that $\omega=|K|$ and $\alpha=|I|$ (since $G$ is balanced). As $5 \leq \alpha+\omega=n \leq 7$ and $2 \leq \alpha \leq \omega$, it follows that $\alpha \in \{2,3\}$ and $\omega \neq 2$. Moreover, $\omega \leq 3$, by Proposition \ref{|I|,|K|<=2deg(S)}. Thus, $\omega=3$.

Let $I=\{a,b\}$. Clearly, $\eta_{ab}=0$, since $G$ has no universal vertices. Thus, $3=\omega=d_a+d_b$, and hence $(d_a,d_b)=(1,2)$. Up to isomorphism, there is a unique split graph with these characteristics; we call it $D_5$ and it can be seen in Figure \ref{grafos.activos.deg<3}.

If $\Phi$ is the path $abc$, we deduce from Theorem \ref{caract.Phi.simple.conexo} that $G$ is isomorphic to $(R,[3],I)$ or $\overline{R^{\iota}}$, where $N_R(a)=N_R(c)=\{1\}$ and $N_R(b)=\{2\}$. This is impossible because $G$ is active, but neither $R$ nor $\overline{R^{\iota}}$ are (vertex 3 is swing in $R$ and universal in $\overline{R^{\iota}}$).

So far, we have classified all active graphs $G$ of degree 2 such that $Q_G^*=Q_G(P_4)$. To classify all active graphs of degree 2, it remains to analyze the case
\begin{equation}
	\label{eq35}
	Q_G^*=Q_G(C_4) \cup Q_G(2K_2).
\end{equation}
Recall that $4 \leq n \leq 8$, by Theorem \ref{|G^*|}. Obviously, $G \approx C_4$ or $2K_2$ if $n=4$. On the other hand, if $n \geq 5$, then $|Q_G^*| \geq 2$. This would imply $\deg(G) \geq 4$, by \eqref{eq35}, contradicting the hypothesis that $\deg(G)=2$. It is very easy to verify that $C_4$ and $2K_2$ are indecomposable. Indeed, if there were graphs $S$ and $H$ such that $S \circ H \approx C_4$, both factors would have to be active, by Theorem \ref{SoG.act.iff.S,G.act}. But then applying Theorem \ref{|G^*|} leads to the following absurdity: $4=|C_4|=|S|+|H| \geq 8$. The same argument applies to $2K_2$.

\begin{theorem}
	\label{clasificacion.activo,deg=2}
	Let $G$ be an active graph.
	\begin{enumerate}
		\item If $\deg(G)=1$, then $G \approx P_4$.
		\item If $\deg(G)=2$ and $G$ is decomposable, then $G \approx P_4^2$.
		\item If $\deg(G)=2$ and $G$ is indecomposable, then $G$ is isomorphic to one of these graphs: $D_5$, $\overline{D_5}$, $C_4$, $2K_2$.
	\end{enumerate}
	All graphs mentioned here are represented in Figure \ref{grafos.activos.deg<3}.
\end{theorem}
\begin{figure}[h]
	\centering
	\includegraphics[scale=0.8]{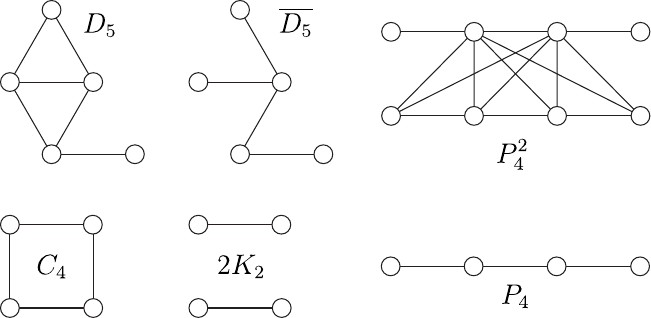}
	\caption{The 6 active graphs of degree $\leq 2$.}
	\label{grafos.activos.deg<3}
\end{figure}
\begin{proof}
	It follows from the previous discussion.
\end{proof}

\begin{theorem}
	\label{2switch.preserva.deg=2}
	Let $G$ be a graph. If $\deg(G)=2$, then $\deg(\tau(G))=2$ for every 2-switch $\tau$. In other words, the 2-switch preserves degree 2.
\end{theorem}

\begin{proof}
	If $X$ is any graph of degree 2, we know that $\deg(X^*)=2$ as well. Thanks to Theorem \ref{clasificacion.activo,deg=2}, we have $X^* \approx D_5$, $\overline{D_5}$, $C_4$, $2K_2$, or $P_4^2$ (see Figure \ref{grafos.activos.deg<3}). Thus, it is evident that $\tau(X^*) \approx X^*$ for every 2-switch $\tau$, which means that $\mathcal{G}(X^*)$ is 2-regular. Finally, the result follows from Theorem \ref{isomorfismo.espacio.activo}.
\end{proof}

\begin{corollary}
	If $\mathcal{G}(s)$ has a vertex of degree 2, then $\mathcal{G}(s)$ is isomorphic to $K_3$ or $C_4$.
\end{corollary}

\begin{proof}
	From Theorems \ref{berge's.theorem} and \ref{2switch.preserva.deg=2}, it follows respectively that $\mathcal{G}(s)$ is connected and 2-regular. Thus, $\mathcal{G}(s)$ is a cycle. To complete the proof, thanks to Theorems \ref{isomorfismo.espacio.activo} and \ref{dual.spaces.iso}, it suffices to consider $X \in \{ \overline{D_5}, 2K_2, P_4^2 \}$ (see Figure \ref{grafos.activos.deg<3}) and compute $c=|\mathcal{G}(X)|$. Due to the high symmetry of $X$, it is easy in this case to find the exact value of $c$.
	
	If $X=\overline{D_5}$, we see that there are only 3 ways to label $X$ while maintaining $s(X)$, since the 2-switch only exchanges the adjacencies of its leaves. Thus, $c=3$. If $Q$ is a set of 4 elements, there are exactly 3 ways to write $Q$ as $A \dot{\cup} B$, with $|A|=2=|B|$. This shows that $c=3$ if $X=2K_2$. Finally, let $X=P_4^2$. In this case, we note that the 2-switch acts on each of the $P_4$ factors as if the rest of the graph did not exist, exchanging the adjacencies of their leaves. Therefore, there are two distinct graphs with the same degree sequence for each $P_4 \preceq X$, which implies $c=4$.
\end{proof}


\section{Active graphs of degree 3}

Consider an active graph $G$ such that $\deg(G)=|Q_G(P_4)|=3$. Thanks to Lemma \ref{5>degG=|{P_4}|.implica.G.split}, we know that $G$ is split. Suppose $G$ is decomposable. Since $G$ is active, it follows from Theorems \ref{SoG.act.iff.S,G.act} and \ref{deg(SoG)=deg(S)+deg(G)} that each factor of $G$ is an active split graph of degree 1 or 2. Consequently,
\[ G \in \{P_4 \circ D_5, D_5 \circ P_4, P_4 \circ \overline{D_5}, \overline{D_5} \circ P_4, P_4^3\}, \]
by Theorem \ref{clasificacion.activo,deg=2}.

If instead $(G,K,I)$ is indecomposable, assume $\alpha=|I|\leq|K|=\omega$. Let $n=|G|$. Applying Theorem \ref{|G^*|}, we obtain $4 \leq n \leq 12$. However, it is clear that $n \geq 5$ and $\omega \geq 3$, since otherwise we would have $|Q_G(P_4)|=1$. Thanks to Proposition \ref{|I|,|K|<=2deg(S)}, we can then state that, in principle, $2 \leq \alpha \leq 4$ and $\omega \in \{3,4\}$. Applying Corollary \ref{S.primo.iff.Phi(S).conexo}, we have
\[ \Phi(G) \in \{K_2, K_3, P_3, S_4, P_4\}, \]
ignoring multiplicities if any (Figure \ref{Phi.posibles.deg=3}).
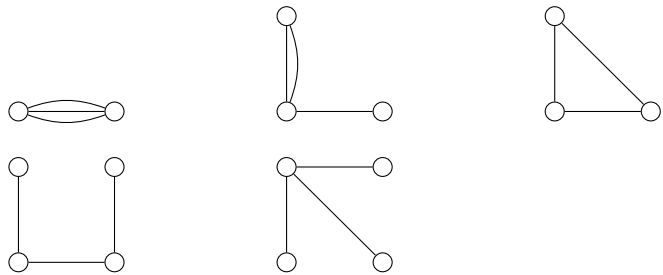
\begin{figure}[h]
	\centering
	\begin{tikzpicture}[scale=1, every node/.style={draw, circle, fill=white, minimum size=0.25cm, inner sep=1pt}]
		
		\node (a1) at (0,0) {};
		\node (a2) [right=of a1] {};
		\draw (a1) -- (a2);
		\draw[bend left=20] (a1) to (a2);
		\draw[bend right=20] (a1) to (a2);
		
		\node (b1) [right=2cm of a2] {};
		\node (b2) [right=of b1] {};
		\node (b3) [above=of b1] {};
		\draw (b1) -- (b2);
		\draw (b1) -- (b3);
		\draw[bend right=20] (b1) to (b3);
		
		\node (c1) [right=2cm of b2] {};
		\node (c2) [right=of c1] {};
		\node (c3) [above=of c1] {};
		\draw (c1) -- (c2);
		\draw (c1) -- (c3);
		\draw (c2) -- (c3);
		
		\node (d1) at (0,-2) {};
		\node (d2) [right=of d1] {};
		\node (d3) [above=of d1] {};
		\node (d4) [above=of d2] {};
		\draw (d1) -- (d2);
		\draw (d1) -- (d3);
		\draw (d2) -- (d4);
		
		\node (e1) [right=2cm of d2] {};
		\node (e2) [right=of e1] {};
		\node (e3) [above=of e1] {};
		\node (e4) [above=of e2] {};
		\draw (e4) -- (e3);
		\draw (e1) -- (e3);
		\draw (e2) -- (e3);
	\end{tikzpicture}
	\caption{The 5 connected multigraphs of size 3.}
	\label{Phi.posibles.deg=3}
\end{figure}
If $\Phi=\Phi(G) \in \{K_3, S_4, P_4\}$, then $G$ is homogeneous, by Theorem \ref{Phi.simple.conexo.1}. Theorem \ref{S.homogeneo.implica...} therefore prohibits $\Phi=P_4$. Let $d$ be the degree in $G$ of the vertices in $I$. Suppose $\Phi=S_4$. Thanks to Theorem \ref{S.homogeneo.implica...}, we know that $N_1=N_2=N_3$. On the other hand, $0=|\bigcap_{i=1}^4 N_i|=\eta_{14}$, since $G$ has no universal vertices. This means that $d=1$, by Proposition \ref{prop.basicas.sigma_uv}. But then $\omega=2$, contradicting $\alpha \leq \omega$.

If $\Phi=K_3$, Theorem \ref{Phi.simple.completo.implica...} tells us that $\omega=3$ and $d \in \{1,2\}$. Both possibilities for $d$ are realized, forming a pair of graphs that are complements of each other. We name $T_6$ the graph resulting from taking $d=1$. These are shown in Figure \ref{grafos.primos.deg=3}.

Let $I=\{a,b\}$, with $d_a \leq d_b$. Since $G$ has no universal vertices, we have $\eta_{ab}=0$. Thus, $d_a=1$ and $d_b=3$, by Proposition \ref{prop.basicas.sigma_uv}, which implies $\omega=4$. Up to isomorphism, there is a unique split graph with these characteristics; we name it $D_6$ and it can be seen in Figure \ref{grafos.primos.deg=3}.

Finally, let $\Phi=abc$, with $\sigma_{ab}=2$. Here two cases arise: 1) $d_a < d_b$; 2) $d_a > d_b$. Suppose $d_a < d_b$. By Proposition \ref{prop.basicas.sigma_uv}, we have $d_c = d_b = d_a + 1$ and $N_a \subset N_c$. Note that $\eta_{ac}=d_a$. Since $G$ contains no universal vertices, it follows that $\varnothing = N_a \cap N_b \cap N_c = N_a \cap N_b$, and thus $\eta_{bc}=d_a=1$, again by Proposition \ref{prop.basicas.sigma_uv}. Then, applying formula \eqref{formulafamosa.union.intersec}, it is easy to deduce that $\omega=3$. All these characteristics determine a graph we name $U_6$, which can be seen in Figure \ref{grafos.primos.deg=3}. If $d_a > d_b$, an almost identical reasoning leads us to $\overline{U_6}$.

So far, we have classified active graphs $G$ of degree 3 such that $Q_G^*=Q_G(P_4)$. To complete the classification of all active graphs of degree 3, it remains to analyze the following case:
\begin{equation}
	\label{eq38}
	|Q_G(C_4)| + |Q_G(2K_2)| = 1 = |Q_G(P_4)|.
\end{equation}
Recall something very important about how Tyshkevich composition is defined: if $G=X \circ Y$, then $X$ must be split. With this in mind, if $G$ is decomposable and satisfies \eqref{eq38}, then $G \in \{P_4 \circ C_4, P_4 \circ 2K_2\}$, by Propositions \ref{SoG.act.iff.S,G.act}, \ref{deg(SoG)=deg(S)+deg(G)} and Theorem \ref{clasificacion.activo,deg=2}. In the following lemmas, we prove that there are no prime graphs satisfying \eqref{eq38}.

\begin{lemma}
	\label{no.existe.deg3.primo.5,8}
	There exists no prime graph $G$ of order 5 or 8 such that
	\begin{equation}
		\label{eq39}
		|Q_G(C_4)| = 1 = |Q_G(P_4)|, \ Q_G(2K_2) = \varnothing.
	\end{equation}
\end{lemma}

\begin{proof}
	Suppose there exists a graph $G$ of order 5 or 8 satisfying \eqref{eq39}. Then, $Q_G^*=\{P,C\}$, where $P \approx P_4$, $C \approx C_4$ and $V(P) \cup V(C) = V(G)$. Moreover, thanks to Theorem \ref{indecomp.characterization}, we know that $A_4=A_4(G)$ is connected and that both $V(P)$ and $V(C)$ are cliques of size 4 in $A_4$.
	
	If $|G|=8$, then $V(P) \cap V(C) = \varnothing$. Since $A_4$ is connected, there must be an edge $ab \in A_4$ such that $a \in P$ and $b \in C$. This means there exists $H \in Q_G^*$ such that $a,b \in V(H)$, which is clearly impossible.
	
	If $|G|=5$, then $|V(P) \cap V(C)|=3$. Thus, $P=ax_1x_2x_3$ and $C=bx_1x_2x_3b$. Now observe that necessarily $ab \in E(G)$. Otherwise, $ax_1bx_3 \preceq G$, which would contradict the hypothesis about $Q_G^*$. Consequently, $ax_3 \in E(G)$ or $bx_2 \in E(G)$, since otherwise we would have $abx_3x_2 \preceq G$. But this contradicts that $P,C \preceq G$.
\end{proof}

\begin{lemma}
	\label{no.existe.deg3.primo.7}
	There exist no prime graphs of order 7 satisfying \eqref{eq39}.
\end{lemma}

\begin{proof}
	Suppose there exists a prime graph $G$ of order 7 satisfying \eqref{eq39}. Then, $Q_G^*=\{P,C\}$, where $P \approx P_4$, $C \approx C_4$ and $V(P) \cup V(C) = V(G)$. Since $|G|=7$, it follows that $P$ and $C$ share exactly one vertex, call it $x$. Therefore, we essentially have two cases to analyze: 1) $P=xa_1a_2a_3$; 2) $P=a_1xa_2a_3$. Let $C=xb_1b_2b_3x$.
	\begin{enumerate}[(1).]
		\item Observe that $a_1b_1, a_1b_3 \in E(G)$ (otherwise, $\{b_1xa_1a_2, b_3xa_1a_2\} \subseteq Q_G^*$). Consequently, also $a_1b_2 \in E(G)$ (otherwise, $a_1b_1b_2b_3a_1 \in Q_G^*$). This in turn implies the presence of $a_2b_1$ in $G$ (otherwise, $b_1a_1a_2a_3 \in Q_G^*$). Since the path $a_2b_1xb_3$ cannot be induced in $G$, at least one edge between $b_1b_3$ and $a_2x$ is forced to be in $G$. However, this contradicts that $P,C \preceq G$.
		\item Note that $a_2b_1, a_2b_3 \in E(G)$ (otherwise, $\{b_3xa_2a_3, a_2xb_1b_2\} \subseteq Q_G^*$). Then, also $a_2b_2 \in E(G)$ (otherwise, $a_2b_1b_2b_3a_2 \in Q_G^*$). Since the path $a_1xa_2b_2$ cannot be induced in $G$, at least one edge between $a_1a_2$ and $b_2x$ is forced to be in $G$. However, this contradicts that $P,C \preceq G$.
	\end{enumerate}
\end{proof}

\begin{lemma}
	\label{no.existe.deg3.primo.6}
	There exist no prime graphs of order 6 satisfying \eqref{eq39}.
\end{lemma}

\begin{proof}
	Suppose there exists a prime graph $G$ of order 6 satisfying \eqref{eq39}. Then, $Q_G^*=\{P,C\}$, where $P \approx P_4$, $C \approx C_4$ and $V(P) \cup V(C) = V(G)$. Since $|G|=6$, it follows that $V(P) \cap V(C) = \{x_1,x_2\}$. We see that $x_1$ and $x_2$ can be positioned in $P$ essentially in 4 ways: 1) $a_1x_1x_2a_2$; 2) $a_1a_2x_1x_2$; 3) $a_1x_1a_2x_2$; 4) $x_1a_1a_2x_2$.
	\begin{enumerate}[(1).]
		\item Let $C=b_1b_2x_2x_1b_1$. Note that necessarily $a_1x_2 \in G$ or $b_2x_1 \in G$ (otherwise, $a_1x_1x_2b_2 \preceq G$). This contradicts that $P,C \preceq G$.
		\item Let $C=b_1b_2x_2x_1b_1$. Observe that necessarily $a_2x_2 \in G$ or $b_2x_1 \in G$ (otherwise, $a_2x_1x_2b_2 \preceq G$). This contradicts that $P,C \preceq G$.
		\item Let $C=b_1x_1b_2x_2b_1$. Note that necessarily $a_1b_1, a_1b_2 \in G$ (otherwise, $\{a_1x_1b_1x_2, a_1x_1b_2x_2\} \subseteq Q_G^*$). This forces $b_1b_2$ or $a_1x_2$ to be in $G$ (otherwise, $a_1b_1x_2b_2a_1 \preceq G$), which contradicts that $P,C \preceq G$.
		\item Let $C=b_1x_1b_2x_2b_1$. Observe that necessarily $a_2b_1, a_2b_2 \in G$ (otherwise, $\{a_1a_2x_2b_1, a_1a_2x_2b_2\} \subseteq Q_G^*$). This forces $b_1b_2$ or $a_2x_1$ to be edges of $G$ (otherwise, $a_2b_2x_1b_1a_2 \preceq G$), which contradicts that $P,C \preceq G$.
	\end{enumerate}
\end{proof}

\begin{lemma}
	\label{no.existe.deg3.primo}
	There exists no prime graph satisfying \eqref{eq38}.
\end{lemma}

\begin{proof}
	Suppose $G$ is a prime graph satisfying \eqref{eq38}. Since $|Q_G^*|=2$, it is clear that $5 \leq |G| \leq 8$, since $G$ is active and $\deg(G)>1$. We have two cases to analyze: 1) $|Q_G(C_4)|=1$; 2) $Q_G(C_4)=\varnothing$.
	\begin{enumerate}[(1).]
		\item Since $G$ satisfies \eqref{eq39}, Lemmas \ref{no.existe.deg3.primo.5,8}, \ref{no.existe.deg3.primo.7} and \ref{no.existe.deg3.primo.6} are contradicted.
		\item Since $\overline{2K_2}=C_4$ and $\overline{P_4}=P_4$, it follows that $\overline{G}$ satisfies \eqref{eq39}. This contradicts Lemmas \ref{no.existe.deg3.primo.5,8}, \ref{no.existe.deg3.primo.7} and \ref{no.existe.deg3.primo.6}.
	\end{enumerate}
\end{proof}

\begin{proposition}
	\label{clasificacion.activo,deg=3}
	Let $G$ be an active graph of degree 3.
	\begin{enumerate}
		\item If $G$ is decomposable, then $G \in$ \[ \{P_4 \circ D_5, D_5 \circ P_4, P_4 \circ \overline{D_5}, \overline{D_5} \circ P_4, P_4^3, P_4 \circ C_4, P_4 \circ 2K_2\}, \]
		\item If $G$ is indecomposable, then
		\[ G \in \{T_6, \overline{T_6}, D_6, \overline{D_6}, U_6, \overline{U_6}\}, \]
		all of which are split (see Figure \ref{grafos.primos.deg=3}).
	\end{enumerate}
\end{proposition}
\begin{figure}[h]
	\centering
	\includegraphics[scale=0.8]{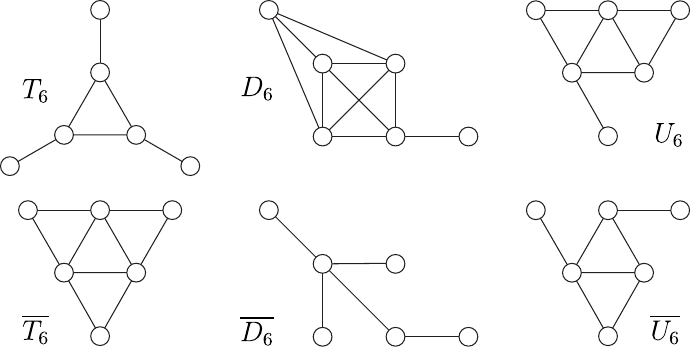}
	\caption{The 6 prime graphs of degree $3$.}
	\label{grafos.primos.deg=3}
\end{figure}
\begin{proof}
	It follows from the previous discussion.
\end{proof}

\begin{theorem}
	\label{2switch.casi-preserva.deg3}
	Let $X$ be a graph of degree 3 such that neither $U_6$ nor $\overline{U_6}$ (see Figure \ref{grafos.primos.deg=3}) appear as factors in the Tyshkevich decomposition of $X$. Then, $\deg(\tau(X))=3$ for every 2-switch $\tau$. In other words, $\mathcal{G}(X)$ is 3-regular.
\end{theorem}

\begin{proof}
	Thanks to Theorem \ref{isomorfismo.espacio.activo}, we know that $\mathcal{G}(X) \approx \mathcal{G}(X^*)$. Since $\deg(X^*)=3$ and $X^*$ is active, we can apply Proposition \ref{clasificacion.activo,deg=3} and easily see that $\tau(X^*) \approx X^*$ for every 2-switch $\tau$ and for each of the 11 graphs to which $X^*$ can be isomorphic. Therefore, $\mathcal{G}(X)$ is 3-regular.
\end{proof}

Regarding Theorem \ref{2switch.casi-preserva.deg3}, it is important to note that it ceases to be true if $U_6$ or $\overline{U_6}$ appear as factors in the decomposition of $X$. This is because $U_6$ has a neighbor of degree 4 in $\mathcal{G}(U_6)$ (see Figure \ref{G(U_6)}). 
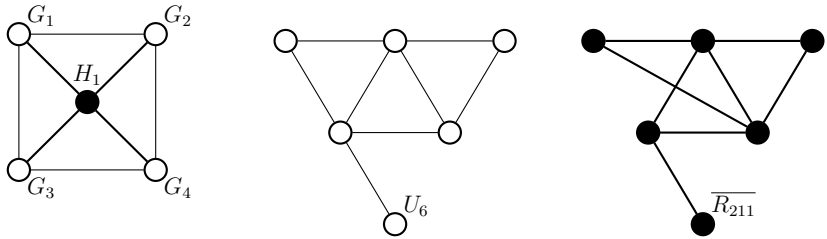
\begin{figure}[h]
	\label{G(U_6)}
	\centering
	\begin{tikzpicture}[scale=0.9, every node/.style={scale=0.75}]
		\begin{scope}[shift={(-4.5,0)}]
			\node[draw, circle, fill=white, thick] (G_1) at (-1,1) {};
			\node[color=black,above right] at (G_1) {$G_1$};
			\node[draw, circle, fill=white, thick] (G_2) at (1,1) {};
			\node[above right] at (G_2) {$G_2$};
			\node[draw, circle, fill=white, thick] (G_3) at (-1,-1) {};
			\node[below right] at (G_3) {$G_3$};
			\node[draw, circle, fill=white, thick] (G_4) at (1,-1) {};
			\node[below right] at (G_4) {$G_4$};
			\node[color=black,draw, circle, fill=black, thick] (H_1) at (0,0) {};
			\node at ($(H_1)+(0,0.4)$) {\textcolor{black}{$H_1$}};
			\draw (G_1) -- (G_2);
			\draw (G_1) -- (G_3);
			\draw (G_2) -- (G_4);
			\draw (G_3) -- (G_4);
			\draw[thick] (G_1) -- (H_1);
			\draw[thick] (G_2) -- (H_1);
			\draw[thick] (G_3) -- (H_1);
			\draw[thick] (G_4) -- (H_1);
		\end{scope}
		
		\begin{scope}[shift={(0,0)}]
			\node[draw, circle, fill=white, thick] (G0) at (0,0.9) {};
			\node[draw, circle, fill=white, thick] (G1) at (-0.8,-0.45) {};
			\node[draw, circle, fill=white, thick] (G2) at (0.8,-0.45) {};
			\node[draw, circle, fill=white, thick] (G3) at (0,-1.8) {};
			\node[above right] at (G3) {$U_6$};
			\node[draw, circle, fill=white, thick] (G4) at (1.6,0.9) {};
			\node[draw, circle, fill=white, thick] (G5) at (-1.6,0.9) {};
			\draw (G0) -- (G1);
			\draw (G1) -- (G2);
			\draw (G0) -- (G2);
			\draw (G3) -- (G1);
			\draw (G4) -- (G2);
			\draw (G4) -- (G0);
			\draw (G5) -- (G1);
			\draw (G5) -- (G0);
		\end{scope}
		
		\begin{scope}[shift={(4.5,0)}]
			\node[draw, circle, fill=black, thick] (H0) at (0,0.9) {};
			\node[draw, circle, fill=black, thick] (H1) at (-0.8,-0.45) {};
			\node[draw, circle, fill=black, thick] (H2) at (0.8,-0.45) {};
			\node[draw, circle, fill=black, thick] (H3) at (0,-1.8) {};
			\node[above right] at (H3) {$\overline{R_{211}}$};
			\node[draw, circle, fill=black, thick] (H4) at (1.6,0.9) {};
			\node[draw, circle, fill=black, thick] (H5) at (-1.6,0.9) {};
			\draw[thick] (H0) -- (H1);
			\draw[thick] (H1) -- (H2);
			\draw[thick] (H0) -- (H2);
			\draw[thick] (H3) -- (H1);
			\draw[thick] (H4) -- (H2);
			\draw[thick] (H4) -- (H0);
			\draw[thick] (H5) -- (H2);
			\draw[thick] (H5) -- (H0);
		\end{scope}
	\end{tikzpicture}
	\caption{Left: $\mathcal{G}(U_6)$; $G_i\approx U_6$ (center), \textcolor{black}{$H_1\approx\overline{R_{211}}$} (right).}
\end{figure}


\section{Prime split graphs of degree 4}

\begin{theorem}
	For every $k \geq 4$, there exists a graph $G_k$ and a 2-switch $\tau$ such that $\deg(\tau(G_k)) \neq \deg(G_k) = k$. In other words, for $k \geq 4$, the 2-switch does not generally preserve degree $k$.
\end{theorem}

\begin{proof}
	Let $\tau = {{a \ b} \choose {d \ c}}$. If $G_4 = abxcd \approx P_5$, then $\tau(G_4) \approx K_3 \dot{\cup} K_2$ and therefore $\deg(\tau(G_4)) = 6 \neq \deg(G_4) = 4$. Then, for $k \geq 5$, consider the graph $G_k = Y_{k-1} \circ G_4$, where $Y_{k-1}$ is the tree of order $k-1$ obtained by identifying a vertex of $K_2$ with a leaf of $S_{k-2}$. Clearly, $\deg(Y_{k-1}) = k-4$. Then, by Theorem \ref{deg(SoG)=deg(S)+deg(G)}, we have that $\deg(G_k) = k$ and $\deg(\tau(G_k)) = \deg(Y_{k-1} \circ \tau(G_4)) = k+2$.
\end{proof}

Consider a prime graph $S$ such that $\deg(G) = |Q_G(P_4)| = 4$. Thanks to Lemma \ref{5>degG=|{P_4}|.implica.G.split}, we know that $S$ is split. Assume $S$ has bipartition $(S,K,I)$ and $\alpha = |I| \leq |K| = \omega$. Let $n = |S|$. Applying Theorem \ref{|G^*|}, we obtain $4 \leq n \leq 16$. However, it is clear that $n \geq 5$ and $\omega \geq 3$, since otherwise we would have $|Q_G(P_4)| = 1$. Thanks to Proposition \ref{|I|,|K|<=2deg(S)}, we can then state that, in principle, $2 \leq \alpha \leq 8$ and $3 \leq \omega \leq 8$. Applying Corollary \ref{S.primo.iff.Phi(S).conexo}, we have that $\Phi = \Phi(S)$ is isomorphic to one of the 12 multigraphs in Figure \ref{Phi.posibles.deg=4}.
\begin{figure}[h]
	\centering
	\begin{tikzpicture}[scale=1, every node/.style={draw, circle, fill=white, minimum size=0.25cm, inner sep=1pt}]
		
		\node (a1) at (0,0) {};
		\node (a2) [right=of a1] {};
		\draw (a1) -- (a2);
		\draw[bend left=15] (a1) to (a2);
		\draw[bend left=35] (a1) to (a2);
		\draw[bend right=15] (a1) to (a2);
		
		\node (b1) [right=2cm of a2] {};
		\node (b2) [right=of b1] {};
		\node (b3) [above=of b1] {};
		\draw (b1) -- (b2);
		\draw (b1) -- (b3);
		\draw[bend left=20] (b1) to (b2);
		\draw[bend right=20] (b1) to (b3);
		
		\node (c1) [right=2cm of b2] {};
		\node (c2) [right=of c1] {};
		\node (c3) [above=of c1] {};
		\draw (c1) -- (c2);
		\draw (c1) -- (c3);
		\draw (c2) -- (c3);
		\draw[bend left=20] (c2) to (c3);
		
		\node (d1) at (0,-2) {};
		\node (d2) [right=of d1] {};
		\node (d3) [above=of d1] {};
		\node (d4) [above=of d2] {};
		\draw (d1) -- (d2);
		\draw (d1) -- (d3);
		\draw (d2) -- (d4);
		\draw (d3) -- (d4);
		
		\node (e1) [right=2cm of d2] {};
		\node (e2) [right=of e1] {};
		\node (e3) [above=of e1] {};
		\node (e4) [above=of e2] {};
		\draw (e4) -- (e3);
		\draw (e1) -- (e3);
		\draw[bend left=20] (e1) to (e3);
		\draw (e2) -- (e3);
		
		\node (f1) [right=2cm of e2] {};
		\node (f2) [right=of f1] {};
		\node (f3) [above=of f1] {};
		\draw (f1) -- (f3);
		\draw (f2) -- (f3);
		\draw[bend left=20] (f3) to (f1);
		\draw[bend right=20] (f3) to (f1);
		
		\node (g1) at (0,-4) {};
		\node (g2) [right=of g1] {};
		\node (g3) [above=of g1] {};
		\node (g4) [above=of g2] {};
		\draw (g1) -- (g2);
		\draw (g1) -- (g3);
		\draw (g3) -- (g4);
		\draw[bend left=20] (g3) to (g1);
		
		\node (h1) [right=2cm of g2] {};
		\node (h2) [right=of h1] {};
		\node (h3) [above=of h1] {};
		\node (h4) [above=of h2] {};
		\draw (h1) -- (h2);
		\draw (h1) -- (h3);
		\draw (h2) -- (h4);
		\draw (h2) -- (h3);
		
		\node (i1) [right=2cm of h2] {};
		\node (i2) [right=of i1] {};
		\node (i3) [above=of i1] {};
		\node (i4) [above=of i2] {};
		\draw (i1) -- (i2);
		\draw (i1) -- (i3);
		\draw (i2) -- (i4);
		\draw[bend left=20] (i1) to (i3);
		
		\node (j1) at (0,-6) {};
		\node (j2) [right=of j1] {};
		\node (j3) [above=of j1] {};
		\node (j4) [above=of j2] {};
		\node (j5) [below=of j1] {};
		\draw (j1) -- (j5);
		\draw (j1) -- (j3);
		\draw (j2) -- (j4);
		\draw (j3) -- (j4);
		
		\node (k1) [right=2cm of j2] {};
		\node (k2) [right=of k1] {};
		\node (k3) [above=of k1] {};
		\node (k4) [above=of k2] {};
		\node (k5) [below=of k1] {};
		\draw (k1) -- (k2);
		\draw (k1) -- (k3);
		\draw (k1) -- (k4);
		\draw (k1) -- (k5);
		
		\node (l1) [right=2cm of k2] {};
		\node (l2) [right=of l1] {};
		\node (l3) [above=of l1] {};
		\node (l4) [above=of l2] {};
		\node (l5) [below=of l1] {};
		\draw (l5) -- (l2);
		\draw (l1) -- (l3);
		\draw (l1) -- (l4);
		\draw (l1) -- (l5);
	\end{tikzpicture}
	\caption{The 12 connected multigraphs of size 4.}
	\label{Phi.posibles.deg=4}
\end{figure}
Thanks to Theorems \ref{Phi.simple.conexo.1} and \ref{S.homogeneo.implica...}, we can immediately rule out that $\Phi$ is isomorphic to $U_4$, $P_5$, or $\overline{D_5}$. Applying instead Theorem \ref{prohibido.P_4.sigma_23=1}, we can also exclude that $\Phi = v_1v_2v_3v_4$, with $\sigma_{23} = 1$.

If $\Phi = abca$, with $\sigma_{ac} = 2$, then by Proposition \ref{prop.basicas.sigma_uv} we would have $d_a = d_b = d_c = d_a \pm 1$, which is absurd.

If $\Phi \approx S_5$, let $v_5$ be the vertex of degree 4 in $S_5$ and $v_i$ its leaves, for $1 \leq i \leq 4$. Applying Theorems \ref{Phi.simple.conexo.1} and \ref{S.homogeneo.implica...}, we have that the $v_i$ are all twins in $S$, so $K = N_1 \cup N_5$. On the other hand, since $S$ has no universal vertices, we have $\varnothing = \bigcap_{j=1}^5 N_j = N_1 \cap N_5$, and consequently, $d_1 = 1 = d_5$, by Proposition \ref{prop.basicas.sigma_uv}. Then, $5 = \alpha \leq \omega = 2$, an absurdity.

If $\Phi \approx C_4 = v_1v_2v_3v_4v_1$, then $N_1 = N_3$ and $N_2 = N_4$, by Theorems \ref{Phi.simple.conexo.1} and \ref{S.homogeneo.implica...}, so $K = N_1 \cup N_2$. Since $S$ has no universal vertices, we have $\varnothing = \bigcap_{i=1}^4 N_i = N_1 \cap N_2$, and consequently, $d_1 = 1 = d_2$, by Proposition \ref{prop.basicas.sigma_uv}. Then, $4 = \alpha \leq \omega = d_1 + d_5 = 2$, which is impossible.

If $\Phi = v_1v_2v_3v_4$, with $\sigma_{23} = 2$, then we can assume without loss of generality that $d_2 \leq d_3$. Then, $d_1 = d_2$, $d_3 = d_1 + 1 = d_4$, $N_2 \subset N_4$, and $N_1 \subset N_3, N_4$, by Proposition \ref{prop.basicas.sigma_uv}. Thus, $K = N_3 \cup N_4$. Since $S$ has no universal vertices, we have $\varnothing = \bigcap_{i=1}^4 N_i = N_1 \cap N_2$, and consequently, $d_1 = 1$, by Proposition \ref{prop.basicas.sigma_uv}. Thus, $d_3 = 2$ and $\eta_{34} = 1$, again by Proposition \ref{prop.basicas.sigma_uv}. Finally, we obtain $4 = \alpha \leq \omega = d_3 + d_4 - \eta_{34} = 3$, an absurdity.

If $E(\Phi) = \{2v_1v_2, v_2v_3, v_2v_4\}$, then $d_3 = d_2 = d_4$, $|d_2 - d_1| = 1$, and $N_3 = N_4$, by Proposition \ref{prop.basicas.sigma_uv}. If $d_2 = d_1 + 1$, then $N_1 \subset N_3$, so $K = N_2 \cup N_3$. Since $S$ has no universal vertices, we have $\varnothing = \bigcap_{i=1}^4 N_i = N_1 \cap N_2$, and consequently, $d_1 = 1$, by Proposition \ref{prop.basicas.sigma_uv}. Thus, $\eta_{23} = 1$, again by Proposition \ref{prop.basicas.sigma_uv}. Finally, we obtain $4 = \alpha \leq \omega = d_2 + d_3 - \eta_{23} = 3$, which is impossible. The case $d_1 = d_2 + 1$ is reduced to an absurdity with similar arguments.

If $I = \{a, b\}$, then $N_a \cap N_b = \varnothing$, since $S$ has no universal vertices. Thus, $\omega = d_a + d_b$ and $4 = \sigma_{ab} = d_a d_b$, by Proposition \ref{prop.basicas.sigma_uv}. With $d_a \leq d_b$, we conclude that $(d_a, d_b, \omega) \in \{(1, 4, 5), (2, 2, 4)\}$. The realizations of these triples are called $D_{41}$ and $D_{22}$, respectively; they are shown along with their complements in Figure \ref{split.primos.deg=4.fig.1}.
\begin{figure}[h]
	\centering
	\includegraphics[scale=0.8]{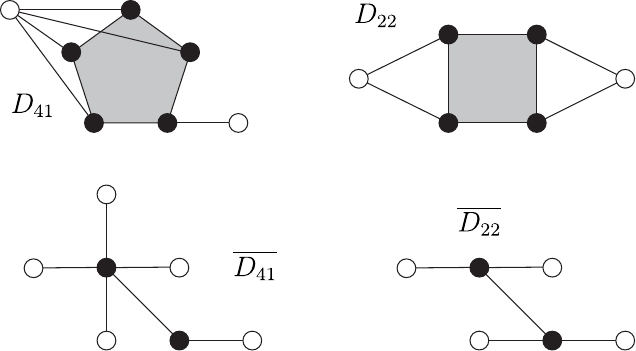}
	\caption{The prime split graphs $D_{41}, D_{22}$ and their complements.}
	\label{split.primos.deg=4.fig.1}
\end{figure}

If $\Phi = abc$, with $\sigma_{bc} = 3$, then $d_a = d_b$ and $|d_c - d_b| = 2$, by Proposition \ref{prop.basicas.sigma_uv}. If $d_c = d_b + 2$, it follows by Proposition \ref{prop.basicas.sigma_uv} that $N_a \subset N_c$, so $K = N_b \cup N_c$. Since $S$ has no universal vertices, we have $\varnothing = N_a \cap N_b \cap N_c = N_a \cap N_b$, and consequently, $d_a = 1$, by Proposition \ref{prop.basicas.sigma_uv}. Thus, $\eta_{bc} = 0$, again by Proposition \ref{prop.basicas.sigma_uv}. Then, $\omega = d_b + d_c = 4$, so we have a realization for $S$ called $F_{311}$. If $d_b = d_c + 2$, it follows by Proposition \ref{prop.basicas.sigma_uv} that $N_c \subset N_a$, so $K = N_a \cup N_b$. Since $S$ has no universal vertices, we have $\varnothing = N_a \cap N_b \cap N_c = N_b \cap N_c$, and consequently, $d_c = 1$, by Proposition \ref{prop.basicas.sigma_uv}. Thus, $\eta_{ab} = 2$, again by Proposition \ref{prop.basicas.sigma_uv}. Then, $\omega = d_a + d_b - \eta_{ab} = 4$, so we obtain another realization for $S$ called $F_{331}$. The graphs $F_{311}, F_{331}$ and their complements are shown in Figure \ref{split.primos.deg=4.fig.2}.
\begin{figure}[h]
	\centering
	\includegraphics[scale=0.8]{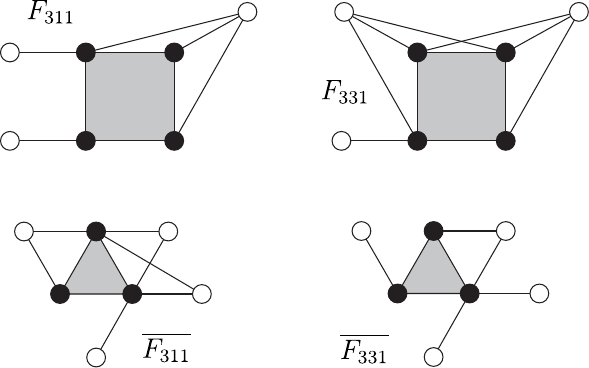}
	\caption{The prime split graphs $F_{311}, F_{331}$ and their complements.}
	\label{split.primos.deg=4.fig.2}
\end{figure}

Let $\Phi = abc$, with $\sigma_{ab} = 2$. Then $|d_b - d_a| = 1 = |d_c - d_b|$, by Proposition \ref{prop.basicas.sigma_uv}. If $d_b = d_a + 1$ and $d_c = d_a$, it follows by Proposition \ref{prop.basicas.sigma_uv} that $N_a = N_c$, so $K = N_a \cup N_b$. Since $S$ has no universal vertices, we have $\varnothing = N_a \cap N_b \cap N_c = N_a \cap N_b$, and consequently, $d_a = 1$, by Proposition \ref{prop.basicas.sigma_uv}. Then, $\omega = d_a + d_b = 3$, so we have a realization for $S$ called $R_{211}$. If $d_b = d_a + 1$ and $d_c = d_a + 2$, it follows by Proposition \ref{prop.basicas.sigma_uv} that $N_a \subset N_c$, so $K = N_b \cup N_c$. Since $S$ has no universal vertices, we have $\varnothing = N_a \cap N_b \cap N_c = N_a \cap N_b$, and consequently, $d_a = 1$, by Proposition \ref{prop.basicas.sigma_uv}. Thus, $\eta_{bc} = 1$, again by Proposition \ref{prop.basicas.sigma_uv}. Then, $\omega = d_b + d_c - \eta_{bc} = 4$, so we obtain another realization for $S$ called $R_{321}$. Finally, if $d_a = d_b + 1$ and $d_a = d_c$, then analogous reasoning yields $\overline{R_{211}}$. The graphs $R_{211}, R_{321}$ and their complements are shown in Figure \ref{split.primos.deg=4.fig.3}.
\begin{figure}[h]
	\centering
	\includegraphics[scale=0.8]{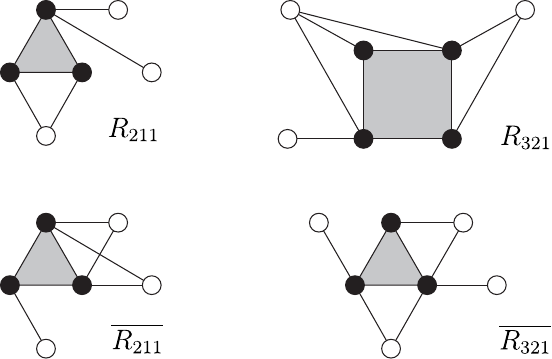}
	\caption{The prime split graphs $R_{211}, R_{321}$ and their complements.}
	\label{split.primos.deg=4.fig.3}
\end{figure}


\chapter{Conclusions}

Among all the results obtained throughout this work, we wish to highlight the following.

\begin{theorem}
	Every transition space $\mathcal{X}\preceq\mathcal{G}(s)$ is isomorphic to its associated active space:
	\begin{equation*}
		\mathcal{X}\approx\mathcal{X}^*.
	\end{equation*}
\end{theorem}

\begin{theorem}
	Let $k\leq 3$ and let $X$ be a graph of degree $k$ such that neither $U_6$ nor $\overline{U_6}$ (see Figure \ref{grafos.primos.deg=3}) appear as factors in the Tyshkevich decomposition of $X$. Then, $\deg(\tau(X))=k$ for every 2-switch $\tau$. In other words, $\mathcal{G}(X)$ is $k$-regular.
\end{theorem}

\begin{theorem}
	If $S$ is a split graph and $G$ is a graph, then:
	\begin{enumerate}
		\item $\deg(S\circ G)=\deg(S)+\deg(G)$;
		\item if $S$ is balanced, then $[S\circ G]=[S]\circ [G]$.
	\end{enumerate}
\end{theorem}

\begin{theorem}
	If $G$ is a graph, then
	\[ \deg(G)+\zeta_2(G)=\Vert G\Vert^2+2c_4(G)+3k_3(G)-4k_4(G). \]
\end{theorem}

\begin{theorem}
	If $S$ is an active split graph, then $S$ is prime if and only if $\Phi(S)$ is connected.
\end{theorem}

\begin{theorem}
	Every active and twin-free graph can be uniquely written as a composition of elementary graphs.
\end{theorem}

\begin{theorem}
	Let $(S,K,I)$ be a split graph such that $\Phi(S)$ is simple and complete, and let $U$ be the set of universal vertices of $S$. If $\bigcup_{v\in I}N_v = K$, $\omega=|K|$ and $\alpha=|I|\geq 2$, then:
	\begin{enumerate}
		\item $S$ is homogeneous;
		\item if $d$ is the degree in $S$ of the vertices in $I$, then $1\leq d\leq\omega -1$;
		\item $|U|\in\{d-1,d+1-\alpha\}$;
		\item $|U|=d-1$ if and only if $\omega=\alpha +d-1$;
		\item $|U|=d+1-\alpha$ if and only if $\omega =d+1$;
		\item $\omega =\alpha +|U|$ (in particular, $\omega\geq\alpha$);
		\item $S$ is active if and only if $U=\varnothing$;
		\item If $S$ is active, then $\omega=\alpha$ and $d\in\{1,\omega-1\}$;
		\item if $\omega=\alpha$ or $d=1$, then $S$ is active;
	\end{enumerate}
\end{theorem}

\begin{theorem}
	Given a split graph $(S,K,I)$ with $|K|\leq |I|$, consider the associated split graph $(R,K,I)$ such that $\deg_R(v)=1$ for every $v\in I$, where $u\cong_R v$ if and only if $\sigma_{uv}(S)=0$ for all $\{u,v\}\subseteq I$. Then $S$ is active and $\Phi(S)$ is simple and connected if and only if $S$ is isomorphic to $R$ or to $\overline{R^{\iota}}$.
\end{theorem}

\begin{theorem}
	Let $S$ be a split graph and let $\Phi=\Phi(S)$.
	\begin{enumerate}
		\item If $C$ is an induced cycle in $\Phi$, then $|C|\leq 4$.
		\item If $\Phi$ is connected, then $diam(\Phi)\leq\lceil (\deg(S)+1)/2 \rceil$.
	\end{enumerate}
\end{theorem}

\begin{theorem}
	If $S$ is an $n$-simple split graph and $\vec{\Phi}(S)=\Delta_0$ (see Figure \ref{triangulos.permitidos}), then $n\in\mathbb{N}(\Delta)$.
\end{theorem}

\begin{theorem}
	If $n\in\mathbb{N}(\Delta)$, let $(x,y,z)$ be a triple of divisors of $n$ satisfying (\ref{ecuacion.Delta.prop}). Then there exists a balanced and $n$-simple split graph $(S,K,I)$ such that $\vec{\Phi}(S)=\Delta_0$ (see Figure \ref{triangulos.permitidos}). Moreover, if $y\leq z$, we have that $S$ has the following properties:
	\begin{enumerate}
		\item $d_b=d_a+\frac{n}{z}-z, d_c=d_a+\frac{n}{x}-x$;
		\item $\eta_{ab}=d_a-z, \eta_{bc}=d_a+\frac{n}{z}-z-y, \eta_{ac}=d_a-x$;
		\item $d_a\geq z$;
		\item $|K|=\frac{n}{x}+z+y+\eta_{abc}$, where $\eta_{abc}=|N_a\cap N_b\cap N_c|\geq d_a-x-z$;
		\item if $S$ is active, then $d_a\leq x+z$;
		\item if $d_a=z$, then $S$ is active.
	\end{enumerate}
\end{theorem}

\begin{theorem}
	Let $S$ be a split graph and let $C$ be an induced $n$-simple cycle in $\Phi(S)$. If $n$ is not a square and $n\notin\mathbb{N}(\Delta)$, then $|C|=4$.
\end{theorem}

\begin{theorem}
	There are infinitely many $\Delta$-primitive numbers.
\end{theorem}

We now provide some ideas for future research.
\begin{enumerate}[(1).]
	\item Find an efficient method to classify by degree the prime graphs $G$ containing $C_4$ or $2K_2$ as induced subgraphs when $\deg(G)\geq 4$. Since under these conditions $G$ is not split, we cannot directly apply the tools developed around the factor graph $\Phi$. However, these ideas could be adapted to define a more general factor graph to be used in other graph families.
	
	\item Study and/or characterize the prime graphs $G$ whose degree is invariant under 2-switch, that is, $\deg(\tau(G))=\deg(G)$ for all 2-switch $\tau$. This would be crucial in determining whether an arbitrary transition space (or a given induced subspace of it) is a regular graph.
	
	\item Investigate how a 2-switch on a split graph $S$ affects the structure of $\Phi(S)$.
	
	\item The $\Delta$ property and the study of the natural numbers that satisfy it or not arise from the study of $n$-simple triangles in $\Phi$. An analogous property emerges when analyzing 4-cycles in $\Phi$, and we propose to investigate the natural numbers that satisfy this new condition and their relationship with the $\Delta$ property.
	
	\item Decide whether the following conjecture is true: there are infinitely many $\Delta$-primitive squares.
\end{enumerate}


\appendix

\chapter{Proof of Lemma \ref{inac.G.inac.tauG}}\label{proof.lema.2switch.preserva.vert.activos}

\textit{If $a$ is an inactive vertex of a graph $G$, then $a$ is also inactive in $\tau(G)$, for every 2-switch $\tau$.}

\begin{proof}
	Let $\tau={{x \ z}\choose{y \ u}}$. Suppose $a$ is inactive in $G$, but active in $\tau(G) \neq G$. Then $a \notin \{x, z, y, u\}$ (otherwise, $a$ would be active in $G$). Since $a$ is active in $\tau(G)$, there must exist a 2-switch $\tau'$ that replaces the edges $av$ and $e$ in $\tau(G)$, for some $v \in V(\tau(G))$ and $e \in E(\tau(G))$. We distinguish two main cases:\\
	
	(1) $e \in \{xy, zu\}$,
	
	(2) $e \notin \{xy, zu\}$.\\
	
	Case (1): Suppose without loss of generality that $e = xy$, and let $H$ be the subgraph of $G$ induced by $\{a, v, x, y\}$. If $av \notin G$, then $\tau$ necessarily activates $a$ (since $av \in \tau(G)$), contradicting that $a$ is inactive in $G$. Therefore, $av \in H$. Clearly, $\deg(H) = 0$ (since $a \in H \in Q_G$ and $a \notin act(G)$).
	
	Let $H'$ be the subgraph of $\tau(G)$ induced by $\{a, v, x, y\}$. Since $\deg(H') \neq 0$ (due to $\tau'$), we have four possibilities for $E(H')$ (without loss of generality):\\
	
	(1.1) $\{ax, xy, yv, va\}$,
	
	(1.2) $\{av, xy\}$,
	
	(1.3) $\{va, ax, xy\}$,
	
	(1.4) $\{av, vy, yx\}$.\\
	
	Clearly, $H \neq H'$ because $\deg(H) \neq \deg(H')$. Recall that $G = \tau(G) - xy - zu + xz + uy$. In particular, $xy, zu \notin H$. If $v \notin \{u, z\}$, then $zu, xz,$ and $uy$ cannot be in $H$ (by definition of $H$), and thus $H = H' - xy$. If $v = u$, then $vy = uy \in H$ and $xz \notin H$, so $H = H' - xy + uy$. If $v = z$, then $xv = xz \in H$ and $uy \notin H$, so $H = H' - xy + xz$.
	
	Case (1.1): If $v \neq z$, then $H \approx P_4$, and hence $\deg(H) \neq 0$, a contradiction. If $v = z$, focus on the edges $av$ and $uy$ in $G$. Since $a \notin act(G)$, the edge $ay$ or $uz$ must be in $G$, which contradicts $E(H) = \{ax, yx, ay, xz\}$ and the definition of $\tau$, respectively.
	
	Case (1.2): If $v \notin \{u, z\}$, focus on the edges $av$ and $xz$ in $G$. Since $a \notin act(G)$, at least one of the following four triangles must be in $G$: $azxa, vxzv, xvax, zvaz$. Since $ax, vx \notin G$, we can discard $azxa, vxzv,$ and $xvax$. Now consider the edges $az$ and $uy$ in $G$. Since $a \notin act(G)$, at least one of the triangles $ayua, zyuz, yzay, uzau$ must be in $G$, which is a contradiction because $ay, zu \notin G$.
	
	Case (1.3): If $v \notin \{u, z\}$, consider the edges $av$ and $xz$ in $G$. Since $a \notin act(G)$, either $az$ or $vx$ must be in $G$. But $E(H) = \{av, ax\}$, so $vx \in G$ is ruled out. Now consider $az$ and $uy$. Since $a \notin act(G)$, one of the triangles $auya, zuyz, yazy, uazu$ must be in $G$, which is a contradiction since $ay, zu \notin G$.
	
	Case (1.4): If $v \notin \{u, z\}$, consider the edges $av$ and $xz$ in $G$. Since $a \notin act(G)$, one of the triangles $azxa, xvax, vxzv, zvaz$ must be in $G$. The first three are discarded because $ax, vx,$ and $xz$ are not in $G$. Now consider $az$ and $uy$. Again, one of $auya, zuyz, yazy, uazu$ must be in $G$, which is absurd since $zu, ay \notin G$.\\
	
	Case (2): Suppose $e = bc \notin \{xy, zu\}$ and let $H$ be the subgraph of $G$ induced by $\{a, v, b, c\}$. Since $a \notin act(G)$ and $a \in H \in Q_G$, $H$ has degree 0. If $av \notin G$, then $\tau^{-1}$ activates $a$ (since $av \in \tau(G)$), contradicting that $a$ is inactive in $G$. Thus $av \in H$. Since $bc \notin \{xy, zu\}$ by hypothesis, $\tau^{-1}$ does not replace it, so $bc \in H$.
	
	Let $H'$ be the subgraph of $\tau(G)$ induced by $\{a, v, b, c\}$. Clearly, $\deg(H) \neq \deg(H')$ implies $H \neq H'$. Moreover, we cannot obtain $H$ by adding or removing two edges from $H'$ (otherwise, $\tau(H) = H' = H$, a contradiction). So $H$ must be obtained by adding or removing a single edge. Since $\deg(H') \neq 0$, we have (without loss of generality) four cases for $E(H')$ :\\
	
	(2.1) $\{av, bc\}$,
	
	(2.2) $\{ab, bc, cv, va\}$,
	
	(2.3) $\{av, vc, cb\}$,
	
	(2.4) $\{va, ab, bc\}$.\\
	
	Case (2.1): Since $av, bc \in H$, $\tau^{-1}$ must add one edge to $H'$ to get $H$. There are four ways to do this, but all result in $H \approx P_4$, which is impossible since $\deg(P_4) = 1$.
	
	Case (2.2): Removing any edge from $H'$ gives $H \approx P_4$. So $H = H' + bv$ or $H = H' - cv + bv$, since $\tau^{-1}$ does not activate $a$ (we can't add or remove $ac$ or $ab$ in $H'$).
	
	If $H = H' + bv$, clearly $bv \in \{xz, uy\}$. Assume $b = x$ and $v = z$. Then $u, y \notin H$ (otherwise the subgraph induced by $\{x, y, z, u\}$ in $G$ would have degree 0). Consider edges $az$ and $uy$ in $G$. Since $a \notin act(G)$, one of the triangles $ayua, zyuz, yzay, uzau$ must be in $G$. But $uz \notin G$, discarding $zyuz$ and $uzau$. If $ayua$ or $yzay$ are in $G$, consider $ay$ and $cx$. Since $a \notin act(G)$, we must have $ac \in G$ or $xy \in G$, a contradiction.
	
	If $H = H' - cv + bv$, then $cv \in \{xy, uz\}$ and $bv \in \{xz, uy\}$. We may assume $v = x$, so $H = H' - xy + xz$ and $u \notin H$. Again, $az$ or $uz$ must be in $G$, a contradiction.
	
	Case (2.3): Note $H' - cv \approx 2K_2$ and $H' - cv + bv \approx P_4$. So $H = H' + bv$, where we can take $bv = xz$, without loss of generality. If $c \in \{u, y\}$, then $\{x, y, z, u\}$ would induce degree 0. So $u, y \notin H$. If $b = x$, $v = z$, the argument is as in (2.2). If $b = z$, $v = x$, consider $ax$ and $uy$, and triangles $ayua, xyux, yxay, uxau$. Since $xy \notin G$, discard $xyux$ and $yxay$. If $ayua$ or $uxau$ is in $G$, consider $au$ and $cz$. Then, $azca, uczu, cauc, zuaz$ must be in $G$. Since $ac, uz \notin G$, we get a contradiction.
	
	Case (2.4): $ab$ cannot be removed by $\tau^{-1}$ (or $a$ would be used by $\tau$). So $H = H' + bv$, with $bv = xz$. If $b = x$, $v = z$ and $u, y \notin H$, continue as in (2.3). If $c = u$, consider $ax$ and $uy$: $ac \in G$ or $xy \in G$, a contradiction. If $b = z$, $v = x$, and $u, y \notin H$, we use the same argument as in (2.2). If $c = y$, consider $az$ and $uy$: $ac \in G$ or $uz \in G$, a contradiction.
\end{proof}


\chapter{On a special intersecting family}\label{familias.inserc.conjuntos}	

Intersecting families of sets have proven to be a fundamental tool in combinatorics. This theoretical framework allows for the description, analysis, and classification of complex structures based on intersection properties among finite sets. In particular, conditions on the size of intersections provide a solid foundation for deriving broader properties that can extend to structural contexts, such as that of graphs.

In this section, we explore families of sets of fixed size whose elements satisfy particular intersection constraints, such as the size of the intersection between two sets being constant. More specifically, we study families of sets $\{N_v:v\in I\}$ that satisfy these two properties:
\begin{enumerate}
	\item $|N_v|=d$, for some $d\geq 1$ and for all $v\in I$;
	\item $|N_u\cap N_v|=d-1$, for all $\{u,v\}\subseteq I$.
\end{enumerate}
We prove that the sizes of arbitrary intersections and the size of the union of all members of the family are completely determined by $|\bigcap_{v\in T}N_v|$, where $T$ is any triple of $I$. This analysis is essential in Section \ref{sec:caract.Phi.simples} for characterizing split graphs $S$ such that $\Phi(S)$ is simple and complete. We remark that the key theorems of this section are essentially equivalent to certain auxiliary results proved by M. Ramras and E. Donovan in \cite{ramras2011automorphism}. \\

The following two lemmas describe how the size of an arbitrary intersection of sets changes when a new intersecting set is added. In short, we show that it either remains the same or decreases by 1.

\begin{lemma}
	\label{pre.lema.aureo}
	Let $\mathcal{N}=\{N_v :v\in I\}$ be a finite family of finite sets of size $d$ such that $|N_u \cap N_v |=d-1$ for all $\{u,v\}\subseteq I$. If $A\subseteq I, |A|\geq 2$ and $b\in I-A$, then 
	\begin{equation*}
		\left|\bigcap_{v\in A\cup b}N_v \right| \in \left\{ \left| \bigcap_{v\in A}N_v \right| -1, \left| \bigcap_{v\in A}N_v \right| \right\} .
	\end{equation*}
	In particular:
	\begin{equation*}
		|N_u \cap N_v \cap N_w |\in \{d-2,d-1\},
	\end{equation*}
	for any triple $\{u,v,w\}\subseteq I$.
\end{lemma}

\begin{proof}
	Without loss of generality, we can take $A=[n-1]$ and $b=n$ ($n\geq 3$). We rename the intersections as $\bigcap_{v\in A}N_v =M_{n-1}$ and $N_n\cap M_{n-1} =M_{n}$. Obviously, $|M_n|\leq |M_{n-1}|$, since $M_n \subseteq M_{n-1}$. Suppose that $|M_n|\leq |M_{n-1}|-2$. Then, $|M_{n-1} -N_n|\geq 2$. Let $x,y\in M_{n-1}-N_n$. Since $x,y\in M_{n-1}$, it follows that in particular $x,y\in N_1$. Therefore, $\{x,y\}\subseteq N_1 -N_n$, which is absurd because the hypothesis implies that $|N_u -N_v|=1=|N_v -N_u|$ for all $\{u,v\}\subseteq I$. 
\end{proof}

\begin{lemma}
	\label{pre.pre.lema.aureo}
	Let $\mathcal{N}=\{N_v :v\in I\}$ be a finite family of finite sets of size $d$ such that $|N_u \cap N_v |=d-1$ for all $\{u,v\}\subseteq I$. If $W$ is an arbitrary subset of $I$ with at least 4 elements, then:
	\begin{enumerate}
		\item if $|\bigcap_{v\in W-x} N_v|=d-1$ for some $x\in W$, then \[\left|\bigcap_{v\in W} N_v \right|=\left|\bigcap_{v\in W-x} N_v \right|;\] 
		\item if $|\bigcap_{v\in W-x} N_v|=d+2-|W|$ for some $x\in W$, then \[\left|\bigcap_{v\in W} N_v \right|=\left|\bigcap_{v\in W-x} N_v \right|-1.\]
	\end{enumerate}
\end{lemma}

\begin{proof}
	Without loss of generality, we take $W=[n]$ and $x=n$, for $n\geq 4$, so that $W-x=[n-1]$. If $M_k=\bigcap_{v=1}^k N_v$ ($k\in[n]$), then $|M_n|\in\{|M_{n-1}|-1,|M_{n-1}|\}$, by Lemma \ref{pre.lema.aureo}. 
	
	Suppose that $|M_{n-1}|=d-1$, but $|M_n |=|M_{n-1}|-1$. Then, we can find elements $a,b\in N_n -M_{n-1}$. Since $a,b\notin M_{n-1}$, there exist $i,j\in [n-1]$ such that $a\notin N_i$ and $b\notin N_j$. If $i=j$, then $\{a,b\}\subseteq N_n -N_i$, which is a contradiction. Therefore, it must be $i\neq j$, and we can assume without loss of generality that $i=1$ and $j=2$. That is, $a\in N_n -N_1$ and $b\in N_n -N_2$. Let $x_1$ be the unique element in $N_1 -M_{n-1}$. 
	If $x_1 \neq b$, then $\{a,b\}\subseteq N_1 -N_n$, which is absurd. Therefore, $x_1=b$ and $N_1 =b\dot{\cup}M_{n-1}$. Using similar arguments, it is easy to deduce that $N_2 =a\dot{\cup}M_{n-1}$. Let $x_3$ be the unique element in $N_3 -M_{n-1}$. Then, $N_3 =x_3\dot{\cup}M_{n-1}$. If $x_3\in\{a,b\}$, then $N_3 =N_1$ or $N_3 =N_2$, which is absurd. If $x_3\notin\{a,b\}$, then $\{a,b\}\subseteq N_n -N_3$, another contradiction. With this, item (1) is proven.
	
	Now suppose that $|M_{n-1}|=d+2-n$, but $|M_n|=|M_{n-1}|$. Then, $M_n =M_{n-1}$, because $M_n \subseteq M_{n-1}$ and $|M_n|=|M_{n-1}|$. Observe that $N_n -M_n =\bigcup_{v=1}^{n-1}(N_n -N_v)$. Since $|\bigcup_{v=1}^{n-1}(N_n -N_v)|=|N_n|-|M_n|=n-2$ and $|N_n -N_v|=1$ for all $v\in[n-1]$, there is exactly one repetition in the union of these singletons. Therefore, we can assume without loss of generality that $N_n-N_1=\{a\}=N_n-N_2$ and $|N|=n-3$, where $N=\bigcup_{v=3}^{n-1}(N_n -N_v)$. Since $|N_1\cap N_n|=d-1$ and $|N_1|=d$, we have that $N_1=x_1\dot{\cup}M_n\dot{\cup}N$, for some element $x_1 \notin N_n$. Similarly: $N_2=x_2\dot{\cup}M_n\dot{\cup}N$ for some $x_2 \notin N_n$. Since $N_n =a\dot{\cup}M_n\dot{\cup}N$, it is evident that $|N_1\cap N_2\cap N_n|=|M_n\dot{\cup}N|=d-1$. Then, $|M_n|=d-1$, by (1), and consequently $n=3$, which is a contradiction. Finally, item (2) is also proven.
\end{proof}

The effort made in the first lemmas of this section is well rewarded in the results that follow. Below, we show that the size of arbitrary intersections is restricted to two possible values.

\begin{proposition}
	\label{lema.aureo.conjuntista}
	Let $\mathcal{N}=\{N_v :v\in I\}$ be a finite family of finite sets of size $d$ such that $|N_u \cap N_v |=d-1$ for all $\{u,v\}\subseteq I$. If $W$ is a non-empty subset of $I$, then 
	\begin{equation*}
		\left|\bigcap_{v\in W} N_v \right| \in\{d+1-|W|,d-1\}.
	\end{equation*}
\end{proposition}

\begin{proof}
	We proceed by induction on $n=|W|$. The case $n\leq 2$ is part of the hypotheses. The case $n=3$ is true by Lemma \ref{pre.lema.aureo}. Now, suppose the claim holds for each $W'\subseteq I$ such that $4\leq|W'|<n$. Without loss of generality, we take $W=[n]$. If $M_k=\bigcap_{v=1}^k N_v$ ($k\in[n]$), then $|M_{n-1}|\in\{d-1, d+2-n\}$, by the inductive hypothesis. If $|M_{n-1}|=d-1$, then $|M_n|=d-1$, by item (1) of Lemma \ref{pre.pre.lema.aureo}. If $|M_{n-1}|=d+2-n$, then $|M_n|=d+1-n$, by item (2) of Lemma \ref{pre.pre.lema.aureo}. 
\end{proof}

The following proposition tells us that the size of an arbitrary intersection of 3 or more sets is determined by the size of the intersection of only 3 of them, and vice versa. Moreover, intersections of the same number of sets have the same size. 

\begin{proposition}
	\label{lema.aureo.conjuntista.2}
	Let $\mathcal{N}=\{N_v :v\in I\}$ be a finite family of finite sets of size $d$ such that $|N_x \cap N_y |=d-1$ for all $\{x,y\}\subseteq I$. Then, we have the following:
	\begin{enumerate}
		\item if $A\subseteq I$ and $|A|\geq 3$, then
		\begin{equation*}
			\left|\bigcap_{v\in A}N_v \right|=d-1
		\end{equation*}
		if and only if there exists a triple $T\subseteq A$ such that $|\bigcap_{v\in T}N_v|=d-1$;
		\item if $A\subseteq I$ and $|A|\geq 3$, then
		\begin{equation*}
			\left|\bigcap_{v\in A}N_v \right|=d+1-|A|
		\end{equation*}
		if and only if there exists a triple $T\subseteq A$ such that $|\bigcap_{v\in T}N_v|=d-2$;
		\item if $A,B\subseteq I$ and $|A|=|B|\geq 1$, then 
		\begin{equation*}
			\left|\bigcap_{v\in A} N_v \right| = \left|\bigcap_{v\in B}N_v \right|.
		\end{equation*}
	\end{enumerate} 
\end{proposition}

\begin{proof}
	(1)-(2). If $|A|=3$, then both (1) and (2) are trivially true. Let then $|A|\geq 4$. The converse in both items is easily obtained by induction on $|A|$ using Lemma \ref{pre.pre.lema.aureo}. Let us now prove the forward direction of (1). Suppose no triple of $A$ satisfies the requirement. Then, by Lemma \ref{pre.lema.aureo}, we have that $|\bigcap_{v\in T}N_v|=d-2$ for every triple $T\subseteq A$. Since $\bigcap_{A}N_v\subseteq\bigcap_{T}N_v$ implies $|\bigcap_{A}N_v|\leq|\bigcap_{T}N_v|$, we finally get a contradiction. For the forward direction of (2), suppose again that no triple of $A$ satisfies the requirement. Then, by Lemma \ref{pre.lema.aureo}, we have that $|\bigcap_{T}N_v|=d-1$ for every triple $T\subseteq A$, and thus $|\bigcap_{A}N_v|=d-1$, by (1). But then $d+1-|A|=d-1$ implies $|A|=2$, a contradiction.
	
	(3). Let $n=|A|=|B|$. The case $n\leq 2$ is the general hypothesis. If $n\geq 3$, suppose that $|\bigcap_{A} N_v |=d-1$ and $|\bigcap_{B}N_v |=d+1-n$. By items (1) and (2), there then exist triples of elements $T_1$ and $T_2$ such that $T_1\subseteq A, T_2\subseteq B, |\bigcap_{T_1}N_v|=d-1$ and $|\bigcap_{T_2}N_v|=d-2$. Since $T_1,T_2\subseteq A\cup B$, we can again apply (1) and (2) to infer that 
	\[d-1=\left|\bigcap_{v\in A\cup B}N_v\right|=d+1-|A\cup B|.\]
	Then, $|A\cup B|=2$. This is impossible since $n\geq 3$ implies $|A\cup B|\geq 4$. 
\end{proof}

Before continuing, we need to recall a well-known result that, through the \textbf{inclusion-exclusion principle} (also known as the sieve principle), relates the size of the union of a finite family of finite sets to the sizes of the intersections of the same. This result states that, if $\{N_v :v\in [\alpha]\}$ is a family of finite sets, then
\begin{equation}
	\left|\bigcup _{v=1}^{\alpha}N_{v}\right|=\sum _{\varnothing \neq A\subseteq [\alpha]}(-1)^{|A|+1}\left|\bigcap _{v\in A}N_{v}\right|.
	\label{formulafamosa.union.intersec}
\end{equation}

\begin{lemma}
	\label{suma.binomial0}
	For every natural number $n$,
	\[ \sum_{k=0}^{n}\binom{n}{k}(-1)^k =0. \]
\end{lemma}

\begin{proof}
	It is a consequence of the binomial theorem:
	\[0=((-1)+1)^{n}=\sum_{k=0}^{n}\binom{n}{k}1^{n-k}(-1)^k.\]
\end{proof}

\begin{lemma}
	\label{suma.binomial1}
	If $\alpha$ and $d$ are two natural numbers, then
	\[ \sum_{v=0}^{\alpha}\binom{\alpha}{v}(d+1-v)(-1)^v =0. \] 
\end{lemma}

\begin{proof}
	First, observe that the following identity holds:
	\begin{align*}
		\binom{\alpha}{v}(d+1-v)&=\binom{\alpha}{v}\big((d+1-\alpha)+(\alpha-v)\big)\\
		&=\binom{\alpha}{v}(d+1-\alpha)+\frac{\alpha!}{v!(\alpha-v-1)!}\\
		&=\binom{\alpha}{v}(d+1-\alpha)+\alpha\binom{\alpha-1}{v}.
	\end{align*}
	Moreover, recall that $\binom{n}{k}=0$ if $k>n$. Finally, using Lemma \ref{suma.binomial0}, we obtain the required identity as follows:
	\[ \sum_{v=0}^{\alpha}\binom{\alpha}{v}(d+1-v)(-1)^v = \]
	\[ (d+1-\alpha)\sum_{v=0}^{\alpha}\binom{\alpha}{v}(-1)^v +\alpha\sum_{v=0}^{\alpha}\binom{\alpha-1}{v}(-1)^v = \]
	\[ \alpha \big( \binom{\alpha-1}{\alpha}(-1)^{\alpha}+\sum_{v=0}^{\alpha-1}\binom{\alpha-1}{v}(-1)^v \big) = 0. \]
\end{proof}

Through the technical lemmas just developed, we will now prove that the size of the union of all sets in the family is determined by the size of the intersection of only 3 of them, and vice versa.

\begin{proposition}
	\label{lema.relacion.omegalfa.conjuntista}
	Let $\{N_v :v\in [\alpha],\alpha\geq 3\}$ be a family of finite sets of size $d\geq 1$ such that $|N_x \cap N_y |=d-1$ for all $\{x,y\}\subseteq [\alpha]$. If $\omega=|\bigcup_{v=1}^{\alpha}N_v |$, then we have the following: 
	\begin{enumerate}
		\item \[ \omega=\alpha +d-1 \] 
		if and only if there exists a triple $T\subseteq[\alpha]$ such that $|\bigcap_{v\in T}N_v|=d-1$;
		\item \[ \omega=d+1 \] 
		if and only if there exists a triple $T\subseteq[\alpha]$ such that $|\bigcap_{v\in T}N_v|=d-2$.
	\end{enumerate}
\end{proposition}

\begin{proof}
	(1, $\Leftarrow$). Using Lemma \ref{lema.aureo.conjuntista.2}, we can immediately infer that $|\bigcap_{A}N_v|=d-1$ for all $A\subseteq [\alpha]$ such that $|A|\geq 2$. Then, applying formula \eqref{formulafamosa.union.intersec}, we obtain that
	\begin{equation}
		\label{igualdad0}
		\omega =\alpha d -(d-1)\sum_{v=2}^{\alpha}\binom{\alpha}{v}(-1)^{v}.
	\end{equation}
	Through Lemma \ref{suma.binomial0}, it immediately follows that
	\begin{equation}
		\label{igualdad1}
		\sum_{v=2}^{\alpha}\binom{\alpha}{v}(-1)^v=\alpha-1.
	\end{equation}
	The required formula is now obtained by substituting \eqref{igualdad1} into \eqref{igualdad0}.
	
	(2, $\Leftarrow$). Using Lemma \ref{lema.aureo.conjuntista.2}, we deduce that $|\bigcap_{v\in A}N_v|=d+1-|A|$ for all $A\subseteq [\alpha]$ such that $A\neq\varnothing$. Once again, we apply formula \eqref{formulafamosa.union.intersec}, and then combine it with Lemma \ref{suma.binomial1}, obtaining the required equality as follows:
	\[ \omega=-\sum_{v=1}^{\alpha}\binom{\alpha}{v}(d+1-v)(-1)^{v} = \]
	\[ \binom{\alpha}{0}(d+1-0)(-1)^0 -\sum_{v=0}^{\alpha}\binom{\alpha}{v}(d+1-v)(-1)^{v} = d+1.\]
	
	(1, $\Rightarrow$). Suppose that $\omega=\alpha+d-1$, but no triple in $[\alpha]$ satisfies the requirement. Then, $|\bigcap_{T}N_v|=d-2$ for every triple $T\subseteq[\alpha]$, by Lemma \ref{pre.lema.aureo}. Consequently, we have that $\omega=d+1$, by (2, $\Leftarrow$). But then $\alpha=2$, a contradiction.  
	
	(2, $\Rightarrow$). It follows from $(1,\Leftarrow)$ using the same reasoning as in the proof of $(1,\Rightarrow)$.
\end{proof}

The section culminates with two important theorems that gather all the partial progress made so far. We conclude that the size of the intersection of a single triple of sets determines the behavior of the entire family, leading to a high structural regularity of the same. This great combinatorial ``rigidity" then allows us to relate local properties (such as intersections of 2 or 3 sets) and global properties (such as the union of all sets) of the intersecting family.

\begin{theorem}
	\label{equivalencias.d-1.conjuntistas}
	Let $\{N_v :v\in [\alpha],\alpha\geq 3\}$ be a family of finite sets of size $d\geq 1$ such that $|N_x \cap N_y |=d-1$ for all $\{x,y\}\subseteq [\alpha]$. Moreover, let $\omega=|\bigcup_{v=1}^{\alpha}N_v|$. The following statements are equivalent:
	\begin{enumerate}
		\item there exists a triple $T\subseteq[\alpha]$ such that $|\bigcap_{v\in T}N_v|=d-1$;
		\item \[ \left|\bigcap_{v\in A}N_v \right|=d-1, \] for all $A\subseteq[\alpha]$ with $|A|\geq 2$;
		\item \[ \omega=\alpha+d-1. \] 
	\end{enumerate}
\end{theorem}

\begin{proof}
	(1$\iff$2). This is the content of item (1) of Lemma \ref{lema.aureo.conjuntista.2} together with one of the general hypotheses of the theorem.
	
	(1$\iff$3). This is the content of item (1) of Proposition \ref{lema.relacion.omegalfa.conjuntista}.
\end{proof}

\begin{theorem}
	\label{equivalencias.d+1-n.conjuntistas}
	Let $\{N_v :v\in [\alpha],\alpha\geq 3\}$ be a family of finite sets of size $d\geq 1$ such that $|N_x \cap N_y |=d-1$ for all $\{x,y\}\subseteq [\alpha]$. Let, moreover, $\omega=|\bigcup_{v=1}^{\alpha}N_v|$. The following statements are equivalent:
	\begin{enumerate}
		\item there exists a triple $T\subseteq[\alpha]$ such that $|\bigcap_{v\in T}N_v|=d-2$;
		\item \[ \left|\bigcap_{v\in A}N_v \right|=d+1-|A|, \] for all $A\subseteq[\alpha]$ with $A\neq\varnothing$;
		\item \[ \omega=d+1. \] 
	\end{enumerate}
\end{theorem}

\begin{proof}
	(1$\iff$2). This is the content of item (2) of Lemma \ref{lema.aureo.conjuntista.2} together with the general hypotheses of the theorem.
	
	(1$\iff$3). This is the content of item (2) of Proposition \ref{lema.relacion.omegalfa.conjuntista}.
\end{proof}


\bibliographystyle{abbrv}
\bibliography{citastesis}


\end{document}